\documentclass[12pt]{article}
\usepackage{mystyle} 

\newcommand{\deltamin}{\delta} 
\newcommand{\deltaone}{\delta_-} 
\newcommand{\deltatwo}{\delta} 

\renewcommand{\epsilon}{\varepsilon}

\newcommand{\supp}{\operatorname{supp}}

\newcommand{\spn}{\operatorname{span}}

\newcommand{\dd}{\, \mathrm{d}}

\renewcommand{\i}{\mathrm i}

\newcommand{\1}{\mathbbm{1}}

\renewcommand{\P}{\mathbb{P}}

\newcommand{\bS}{\mathbf{S}}

\newcommand{\bPi}{\mathbf{\Pi}}

\newcommand{\cC}{\mathcal{C}}
\newcommand{\cD}{\mathcal{D}}
\newcommand{\cE}{\mathcal{E}}
\newcommand{\cF}{\mathcal{F}}
\newcommand{\cG}{\mathcal{G}}
\newcommand{\cH}{\mathcal{H}}

\newcommand{\cJ}{\mathcal{J}}
\newcommand{\cK}{\mathcal{K}}

\newcommand{\cM}{\mathcal{M}}
\newcommand{\cN}{\mathcal{N}}

\newcommand{\cP}{\mathcal{P}}
\newcommand{\cQ}{\mathcal{Q}}
\newcommand{\cR}{\mathcal{R}}
\newcommand{\cS}{\mathcal{S}}
\newcommand{\cT}{\mathcal{T}}
\newcommand{\cU}{\mathcal{U}}
\newcommand{\cV}{\mathcal{V}}

\newcommand{\cZ}{\mathcal{Z}}

\newcommand{\fa}{\mathfrak{a}}
\newcommand{\fb}{\mathfrak{b}}

\newcommand{\fe}{\mathfrak{e}}
\newcommand{\ff}{\mathfrak{f}}

\newcommand{\fm}{\mathfrak{m}}
\newcommand{\fn}{\mathfrak{n}}
\newcommand{\fo}{\mathfrak{o}}
\newcommand{\fp}{\mathfrak{p}}

\newcommand{\fs}{\mathfrak{s}}
\newcommand{\ft}{\mathfrak{t}}

\newcommand{\fB}{\mathfrak{B}}

\newcommand{\fF}{\mathfrak{F}}
\newcommand{\fG}{\mathfrak{G}}

\newcommand{\fI}{\mathfrak{I}}

\newcommand{\fK}{\mathfrak{K}}
\newcommand{\fL}{\mathfrak{L}}
\newcommand{\fM}{\mathfrak{M}}

\newcommand{\fT}{\mathfrak{T}}
\newcommand{\fU}{\mathfrak{U}}
\newcommand{\fV}{\mathfrak{V}}

\newcommand{\fX}{\mathfrak{X}}

\newcommand{\sA}{\mathscr{A}}

\newcommand{\sD}{\mathscr{D}}

\newcommand{\sG}{\mathscr{G}}
\newcommand{\sH}{\mathscr{H}}
\newcommand{\sI}{\mathscr{I}}
\newcommand{\sJ}{\mathscr{J}}
\newcommand{\sK}{\mathscr{K}}

\newcommand{\sM}{\mathscr{M}}
\newcommand{\sN}{\mathscr{N}}

\newcommand{\sR}{\mathscr{R}}

\newcommand{\sT}{\mathscr{T}}

\newcommand{\sZ}{\mathscr{Z}}

\newcommand{\arroweps}{\xrightarrow{\epsilon \downarrow 0}}

\newcommand{\limn}{\lim_{n\rightarrow \infty}}

\newcommand{\limN}{\lim_{N\rightarrow \infty}}

\newcommand{\liminfm}{\liminf_{m\rightarrow \infty}}
\newcommand{\liminfn}{\liminf_{n\rightarrow \infty}}

\newcommand{\supn}{\sup_{n\in \N}}

\usepackage{our_comment_disable}
\usepackage[a4paper,bindingoffset=0.2in,
            left=1in,right=1in,
            top=1in,bottom=1in,
            footskip=0.4in]{geometry}

\usepackage[style=numeric, backend=biber, doi=false, url=false, isbn=false, maxbibnames=99,
uniquename=false, giveninits]{biblatex}
\AtEveryBibitem{\clearlist{language}}
\renewbibmacro{in:}{
  \ifentrytype{article}{}{\printtext{\bibstring{in}\intitlepunct}}}
\usepackage[nottoc]{tocbibind}
\newcommand{\tymchange}[1]{#1}
\newcommand{\change}[1]{#1}

\usepackage[symbols,
   record 
]{glossaries-extra}
\glsdisablehyper

\GlsXtrLoadResources[
 src={reg-symbols},
 type=symbols,
]
\tikzset{
black node/.style={circle,draw=black,inner sep=2pt,fill=black},
blue node/.style={circle,draw=blue,inner sep=2pt,fill=blue},
red node/.style={circle,draw=red,inner sep=2pt,fill=red},
magenta node/.style={circle,draw=magenta,inner sep=2pt,fill=magenta},
orange node/.style={circle,draw=orange,inner sep=2pt,fill=orange},
purple node/.style={circle,draw=purple,inner sep=2pt,fill=purple},
noise node/.style={circle,draw=black,inner sep=2pt,fill=white},
}
\newenvironment{bigdrawtree}
  {
    \begin{tikzpicture}[very thick,baseline={([yshift=-.5ex]current bounding box.center)}]
      \tikzstyle{level 1}=[level distance=10mm,sibling distance=30mm]
      \tikzstyle{level 2}=[level distance=10mm,sibling distance=15mm]
      \tikzstyle{level 3}=[level distance=10mm,sibling distance=8mm]
  }
  {
  \end{tikzpicture}
  }
\newenvironment{drawtree}
  {
    \begin{tikzpicture}[very thick,baseline={([yshift=-.5ex]current bounding box.center)}]
      \tikzstyle{level 1}=[level distance=10mm,sibling distance=15mm]
      \tikzstyle{level 2}=[level distance=10mm,sibling distance=10mm]
      \tikzstyle{level 3}=[level distance=10mm,sibling distance=5mm]
  }
  {
  \end{tikzpicture}
  }
\newenvironment{smalldrawtree}
  {
    \tikzset{
    black node/.style={circle,draw=black,inner sep=1.5pt,fill=black},
    blue node/.style={circle,draw=blue,inner sep=2pt,fill=blue},
    red node/.style={circle,draw=red,inner sep=2pt,fill=red},
    magenta node/.style={circle,draw=magenta,inner sep=2pt,fill=magenta},
    orange node/.style={circle,draw=orange,inner sep=2pt,fill=orange},
    purple node/.style={circle,draw=purple,inner sep=2pt,fill=purple},
    noise node/.style={circle,draw=black,inner sep=1.5pt,fill=white},
    }
    \begin{tikzpicture}[very thick,  baseline={([yshift=-.5ex]current bounding box.center)}]
      \tikzstyle{level 1}=[level distance=5mm,sibling distance=10mm]
      \tikzstyle{level 2}=[level distance=5mm,sibling distance=7mm]
      \tikzstyle{level 3}=[level distance=5mm,sibling distance=3mm]
  }
  {
  \end{tikzpicture}
  }

\usepackage{scalerel} 
\usepackage{cancel}

\newcommand{\zzone}{\text{\resizebox{.7em}{!}{\includegraphics{trees-1.eps}}}}
\newcommand{\zztwo}{\text{\resizebox{.7em}{!}{\includegraphics{trees-2.eps}}}}
\newcommand{\zzthree}{\text{\resizebox{.7em}{!}{\includegraphics{trees-3.eps}}}}
\newcommand{\zzfour}{\text{\resizebox{1em}{!}{\includegraphics{trees-4.eps}}}}

\newcommand{\dir}{\scaleobj{0.8}{\mathrm{D}}}
\newcommand{\neu}{\scaleobj{0.8}{\mathrm{N}}}

\newcommand{\normresconv}{\xrightarrow{\textnormal{NR}}}
\newcommand{\kAH}{\boldsymbol{\fb}}
\newcommand{\aAH}{\boldsymbol{\fa}}
\newcommand{\aAHeps}{\boldsymbol{\fa}^\epsilon}
\newcommand{\WAH}{W}
\newcommand{\WAHeps}{W^{\epsilon}}
\newcommand{\XAH}{X}
\newcommand{\XAHeps}{X^{\epsilon}}
\newcommand{\XAHepsnm}{X^{\epsilon_{n_m}}}
\newcommand{\YAH}{Y}
\newcommand{\YAHeps}{Y^{\epsilon}}
\newcommand{\YAHepsnm}{Y^{\epsilon_{n_m}}}
\newcommand{\hatYAH}{\widehat{Y}}
\newcommand{\hatYAHeps}{\widehat{Y}^{\epsilon}}
\newcommand{\cZAH}{\cZ}
\newcommand{\cZAHeps}{\cZ^{\epsilon}}
\newcommand{\hatcZAH}{\widehat{\cZ}}
\newcommand{\hatcZAHeps}{\widehat{\cZ}^{\epsilon}}
\newcommand{\tildecZAH}{\tilde{\cZ}}
\newcommand{\tildecZAHeps}{\tilde{\cZ}^{\epsilon}}
\newcommand{\cZnAH}{\cZ^{\neu}}
\newcommand{\cZnAHeps}{\cZ^{\neu,\epsilon}}
\newcommand{\HAH}{\cH}
\newcommand{\HAHeps}{\cH_\epsilon}
\newcommand{\HAHepsnm}{\cH_{\epsilon_{n_m}}}
\newcommand{\MAH}{M}
\newcommand{\NAH}{\boldsymbol{N}}

\newcommand{\cNAH}{\boldsymbol{\cN}}
\newcommand{\NAHeps}{\boldsymbol{N}_\epsilon}
\newcommand{\NAHepsnm}{\boldsymbol{N}_{\epsilon_{n_m}}}

\newcommand{\cNAHeps}{\boldsymbol{\cN}_\epsilon}
\newcommand{\artNAH}{\overline{\boldsymbol{N}}^{\neu}}
\newcommand{\artNAHeps}{\overline{\boldsymbol{N}}^{\neu}_\epsilon}
\newcommand{\artcNAH}{\overline{\boldsymbol{\cN}}^{\neu}}
\newcommand{\artcNAHeps}{\overline{\boldsymbol{\cN}}^{\neu}_\epsilon}

\newcommand{\artAHU}{\overline{\cH}^{\neu,U}}

\newcommand{\artevk}{\overline{\lambda}^{\neu}_k}
\newcommand{\artevkeps}{\overline{\lambda}^{\neu}_{k;\epsilon}}

\newcommand{\lambdaAHk}{\lambda_k}
\newcommand{\lambdaAHone}{\lambda_1}

\newcommand{\lambdaAHkeps}{\lambda_{k;\epsilon}}
\newcommand{\lambdaAHkepsnm}{\lambda_{k;\epsilon_{n_m}}}

\newcommand{\minjJ}{\min_{j\in\{1,\dots,J\}}}

\newcommand{\ecf}{\boldsymbol{N}}

\newcommand{\w}{w}
\newcommand{\LL}{\mathbb{L}}

\begin{document}
\newcommand*{\TitleFont}{%
      \usefont{\encodingdefault}{\rmdefault}{b}{n}%
      \fontsize{16}{20}%
      \selectfont}

\title{Anderson Hamiltonians with singular potentials}

\author{Toyomu Matsuda\thanks{Independent,
\href{mailto:toyomumatsudawork@gmail.com}{toyomumatsudawork@gmail.com}}
\and Willem van Zuijlen\thanks{WIAS Berlin, Mohrenstra{\ss}e 39, 10117 Berlin, Germany,
\href{mailto:vanzuijlen@wias-berlin.de}{vanzuijlen@wias-berlin.de}}}

\date{}

\maketitle



\begin{abstract}
We construct random Schr\"odinger operators, called Anderson Hamiltonians, with Dirichlet and Neumann boundary conditions for a fairly general class of singular random potentials
on bounded domains.
Furthermore, we construct the integrated density of states of these Anderson Hamiltonians, and
 we relate the Lifschitz tails (the asymptotics of the left tails of the integrated density of states) to the left tails of the principal eigenvalues.

\bigskip


\noindent
\emph{Acknowledgements.}
We would like to thank Nicolas Perkowski for his support since the beginning of the work.
We are also grateful for Kazuhiro Kuwae for a useful discussion.
Moreover, we are grateful to the referees for their valuable feedback. 
TM was supported by the German Science Foundation (DFG) via the IRTG 2544.



%
\end{abstract}

\tableofcontents

\section{Introduction}

In this paper, we consider random Schrödinger operators of the form
\begin{equation}\label{eq:AH}
  -\Delta - \xi,
\end{equation}
where $\Delta = \sum_{i=1}^d \partial_i^2$ is the Laplacian on $\R^d$ and $\xi$ is a random potential.
Such operators are also called \emph{Anderson Hamiltonians}. 
This name is due to the influential work by Anderson \cite{An58}.
We consider the construction of such operators for irregular potentials $\xi$ (also called singular potentials) that do not need to be functions,  hence there is --a priori-- no obvious interpretation of \eqref{eq:AH}.

After constructing the Anderson Hamiltonian it is natural to investigate its spectral properties.
One of the most studied objects in the theory of random Schrödinger operators is the \emph{integrated density of states} (IDS), see for example  \cite[Chapter VI]{carmona1990spectral} and \cite{ids2007} for overviews.
The IDS is a nonrandom, increasing and right-continuous function on $\R$ and is often characterized as the vague limit
of the normalized eigenvalue counting functions.
The left tail asymptotics of the IDS are called \emph{Lifschitz tails}, which capture disorder effect in the operator \eqref{eq:AH}.
Relating the Lifschitz tails to the tail
asymptotics of the principal eigenvalue is a classical result, see for example Kirsch and Martinelli \cite{Kirsch_1982} and
Simon \cite{simon_lifschitz_1985}.

The rest of our introduction is split as follows. 
In Section~\ref{sec:construction_AH} we discuss the previous works on the construction of Anderson Hamiltonians with singular potentials and how our construction relates to these works regarding the assumptions and techniques. 
In Section~\ref{sec:spectral} we discuss the study of the spectral properties of the Anderson Hamiltonians. 
In Section~\ref{subsec:techniques} we discuss the techniques that we use to derive our results.
In Section~\ref{sec:main_results} the main results are presented. 
In Section~\ref{subsec:strategy_ang_techniques} we give the ideas behind the construction of the operator and how we derive the stochastic terms. 
In Section~\ref{sec:outline} we describe the outline of the rest of the paper and in Section~\ref{sec:notation} we give an overview of some notation that is used throughout the paper.

\subsection{Construction of Anderson Hamiltonians with singular potentials}
\label{sec:construction_AH}

The  mathematical study of Anderson Hamiltonians with singular potentials dates back to the work
\cite{fukushima_spectra_1977} by Fukushima and Nakao.
They constructed the Anderson Hamiltonian  with a white noise potential and with Dirichlet boundary conditions
on the one dimensional domain $(-L, L)$,
as the self-adjoint operator associated to the closed symmetric form  on $ H^1_0((-L, L))$, (formally)  given by
\begin{equation*}
 (u, v) \mapsto \int_{(-L, L)} \nabla u \cdot \nabla v
  - \int_{(-L, L)} \xi u v.
\end{equation*}
 For $\xi$ being the white noise one has to make sense of the term $\int_{(-L, L)} \xi u v$. To do so, Fukushima and Nakao replaced it by
\begin{equation*}
  \int_{(-L, L)}  ( u v' + v u' ) B  ,
\end{equation*}
where $B$ is the Brownian motion on $(-L,L)$ (as $\xi$ is the derivative of $B$, this is an integration by parts identity).
In general,
for a bounded open set $U$ in $\R^d$ and a potential $V$ of regularity greater than $-1$,
 it is possible to make sense of
\begin{equation*}
  \int_{U} V u v
\end{equation*}
 for $u, v \in H^1_0(U)$ (we show this in Theorem~\ref{theorem:examples_bd_sym_with_weighted_input}~\ref{item:sym_form_on_domain}). 
Therefore, in that case, one
can construct the Anderson Hamiltonian by considering the associated symmetric form.

However, this approach fails to work if the regularity of $\xi$ is below $-1$.
The treatment of such singular $\xi$ became possible only after
the advent of the theory on \emph{singular stochastic partial differential equations}
(singular SPDEs), most notably the theory of
\emph{regularity structures}
 by Hairer \cite{hairer_theory_2014} 
and the theory of
\emph{paracontrolled distributions} by Gubinelli, Imkeller and Perkowski \cite{gubinelli_paracontrolled_2015}. 

Motivated by the theory of paracontrolled distributions, Allez and Chouk \cite{allez2015continuous} constructed the Anderson Hamiltonian with  white noise on the 2D torus as the limit of
\begin{equation*}
  - \Delta - \xi_{\epsilon} + c_{\epsilon},
\end{equation*}
where $\xi_{\epsilon}$ is a regularized potential and $c_{\epsilon}$ is a suitably chosen number such that $c_\epsilon \uparrow \infty$ as $\epsilon\downarrow 0$.
They obtained an explicit domain of the operator and its action.
Subsequently, Gubinelli, Ugurcan and Zachhuber \cite{gubinelli_semilinear_2020} constructed
the Anderson Hamiltonian with white noise on the 2D and 3D torus
and studied SPDEs whose linear part is given by the Anderson Hamiltonian \eqref{eq:AH}. 
Chouk and  van Zuijlen  \cite{chouk2020asymptotics} constructed the Anderson Hamiltonian with white noise and with either Dirichlet or Neumann boundary conditions on
2D boxes.
Mouzard \cite{mouzard2020weyl} constructed the Anderson Hamiltonian with white noise on 2D compact manifolds, using the theory of higher order paracontrolled distributions \cite{bailleul_bernicot_2019}.
This paper can also be viewed as a generalisation  of \cite{allez2015continuous}.
Additionally, he proved a Weyl law for the Anderson Hamiltonian. We also prove such a Weyl law in Proposition~\ref{prop:weyl}.
 Ugurcan \cite{Ugur} constructed the Anderson Hamiltonian on $\R^2$ using the methods of paracontrolled distributions.

The works \cite{chouk2020asymptotics, gubinelli_semilinear_2020, mouzard2020weyl, Ugur} 
mentioned above use the techniques of the theory of paracontrolled distributions \cite{gubinelli_paracontrolled_2015}.
Labbé \cite{labbe_continuous_2019}  used the theory of
regularity structures to construct the Anderson Hamiltonian with white noise
 on a $d$-dimensional box ($d \leq 3$) 
with Dirichlet or periodic boundary conditions. Instead of directly constructing the operator itself,
he constructed the resolvent operators $G_a = (a-\Delta -\xi)^{-1}$ with Dirichlet boundary conditions for large $a>0$ and defined the Anderson Hamiltonian 
by $G_a^{-1} -a$. 
Although this approach is robust,
the construction is abstract and the domain of the operator is implicit.

\medskip 

In this work we consider a fairly general class of irregular potentials under the minimal assumption on the regularity of the potential $\xi$, which means that we assume that the regularity of $\xi$ is $-2+\delta$ for some $\delta > 0$.
Typical examples of potentials that are within this regularity regime include the \emph{white noise},  namely the centered Gaussian field with delta correlation, 
in $d$-dimensions with $d \in \{1, 2, 3\}$. However we can go beyound white noise, as we can treat a Gaussian noise $\xi$ whose covariance is formally given by
\begin{equation*}
  \expect[\xi(x) \xi(y)] = c \abs{x - y}^{-\alpha}, \quad c \in (0, \infty), \,\, \alpha \in (0, \min\{d, 4\}).
\end{equation*}
Moreover, instead of working on a box, we consider a bounded domain $U$ in $\R^d$ and construct the Anderson Hamiltonian on $U$ with both Dirichlet as well as Neumann boundary conditions. For the latter, besides that the domain needs also to be Lipschitz, we have to impose more restrictive assumptions on the potential.
For example, these assumptions do not allow us to construct the Anderson Hamiltonian with Neumann boundary conditions for a white noise potential on a three dimensional domain. In order to construct this operator one expects -- due to the work of Hairer and Gerencsér \cite{GeHa2021boundary} for the parabolic Anderson model --  the need to perform an additional renormalisation, but then only on the boundary.

\subsection{Spectral properties of Anderson Hamiltonians}
\label{sec:spectral}

Fukushima and Nakao \cite{fukushima_spectra_1977} studied the integrated density of states (IDS) for the Anderson Hamiltonian with white noise potential in one dimension
and derived the explicit formula that was predicted by physicists.
The IDS for the  Anderson Hamiltonian with white noise potential on two dimensional boxes was constructed by  Matsuda  in \cite{matsuda2020integrated}.

Besides the study of the IDS, quite related are the studies of the asymptotics of the eigenvalues.
Chouk and van Zuijlen \cite{chouk2020asymptotics} showed the asymptotics of the eigenvalues in two dimensions for a white noise potential and Labbé and Hsu \cite{hsu2021asymptotic} extended this to three dimensions.
The asymptotics of the principal eigenvalues plays an important role in the mass asymptotics of the parabolic Anderson model 
\cite{konig2016parabolic,konig2020longtime,ghosal2023fractal}.
Most recently, Bailleul, Dang and Mouzard \cite{BaDaMo}
 studied different properties of the Anderson Hamiltonian and its spectrum, for example the corresponding heat kernel and heat kernel estimates are studied, estimates of the norms of the eigenfunctions in terms of the size of their corresponding eigenvalues are given and a lower estimate on the spectral gap is given.

We remark that in one dimension with white noise, beyond the asymptotics of the eigenvalues and the study of the IDS, more is known about the spectrum properties.

Namely, McKean \cite{mckean94} showed that appropriately shifted and rescaled principal eigenvalues converge,
as the segment size grows to infinity, to the Gumbel distribution in law.
Cambronero and McKean \cite{cambronero99} and Cambronero, Ramírez and Rider \cite{cambronero06} derived
precise tail asymptotics of the principal eigenvalue on the fixed torus.
Dumaz and Labbé investigated the detailed statistics of the eigenvalues and the eigenfunctions in a series of works
\cite{dumaz_localization_2020,Dumaz:2023aa, dumaz2021localization,dumaz2022anderson}.
No analogous results are known for singular potentials other than the white noise in one dimension  (see the conjectures in the introduction of \cite{hsu2021asymptotic}).

\medskip 

In this work we construct
the IDS of the Anderson Hamiltonian with a singular potential
and we relate its left tail to those of the principal eigenvalues. In particular, by applying the
work \cite{hsu2021asymptotic} by Hsu and Labbé, we derive the precise tail behaviour of the IDS for the white noise in
$d$ dimensions, for $d \in \{2, 3\}$.

\subsection{Techniques}
\label{subsec:techniques}

\underline{The techniques to construct the operators.}
Instead of directly constructing the operators themselves, we construct the corresponding symmetric forms.
In fact, we are inspired by Gubinelli, Ugurcan and Zachhuber \cite{gubinelli_semilinear_2020}, where
they figured out that the form domain of the Anderson Hamiltonian for a white noise potential on the 2D or 3D torus  is quite simple. 
The work \cite{kuwae_shioya_2003} by Kuwae and Shioya is important for us as it provides a correct notion of convergence of symmetric forms that are bounded from below.

To construct the symmetric forms, we combine an exponential transformation with an integration by parts formulae, see Section~\ref{subsec:strategy_ang_techniques} for a heuristic description. 
The exponential transformation is now a well-known technique in singular SPDEs. The most notable one is the Cole-Hopf transform of the
KPZ equation as used by Bertini and Giacomin \cite{Bertini97}. Hairer and Labbé \cite{hairer_Labbe_2015}  used
the exponential transformation to simplify the 2D parabolic Anderson model.
Gubinelli, Ugurcan and Zachhuber \cite{gubinelli_semilinear_2020}
used it to construct the Anderson Hamiltonian with
3D white noise.
Recently, Jagannath and Perkowski \cite{jagannath2021simple}
applied it to simplify the construction of the dynamical $\Phi^4_3$ model
and Zachhuber \cite{zachhuber2021finite} applied it to prove global well-posedness of multiplicative stochastic wave equations.
It is interesting to note that, unlike previous works, we can apply the trick of exponential transformation 
for the entire subcritical regime.
A major drawback of the exponential transformation is the lack of robustness. For instance, it does not work
if we replace the Laplacian with a fractional Laplacian.

\underline{The techniques to treat the IDS.}
There are two standard approaches to construct the IDS:
the path integral approach \cite[Section~VI.1.2]{carmona1990spectral} and the functional analytic approach
\cite[Section~VI.1.3]{carmona1990spectral}. In our framework, we cannot use the path integral approach.
Indeed, it was shown in \cite{matsuda2020integrated} that the 2D white noise is critical for this approach in that
the Laplace transform of the IDS is finite only for small parameters. Therefore, if the regularity of the potential
$\xi$ is lower than that of the 2D white noise, we expect the blow-up of the Laplace transform of the IDS for any parameter.
Hence, in this paper we adopt the functional analytic approach.
This approach, introduced by Kirsch and Martinelli \cite{Kirsch_1982}, is based on the
super-(sub-)additivity of the Dirichlet (Neumann) eigenvalue counting functions and the ergodic theorem
by Akcoglu and Krengel \cite{AkcogluM.A1981Etfs}. 

In order to apply this approach only under the assumptions that guarantee the existence of the Anderson Hamiltonian with Dirichlet boundary conditions, i.e., without assuming the additional Neumann assumption~\ref{assump:all_three_assump}~\ref{item:neumann_assump} needed for the  Anderson Hamiltonian with Neumann boundary conditions, we are introducing an --what we call-- artificial Neumann Anderson Hamiltonian (see Definition~\ref{def:artificial_Neumann_AH}). 
For the construction of this artificial operator, we rely on a rather explicit representation of the symmetric form associated to the Anderson Hamiltonian. 

Many technical estimates here are inspired by
Doi, Iwatsuka and Mine \cite{doi_iwatsuka_mine_2001}.

\subsection{Main results}
\label{sec:main_results}

In this section we state our three main results of the paper. These results do not require the introduction of technical notions, however their assumptions do. 
We therefore briefly give an idea about what kind of assumptions we assume below in Remark~\ref{rem:assumptions_brief} and postpone the detailed description of the assumptions to Section~\ref{sec:assumptions}, in which we also discuss in which cases these conditions are fulfilled. 
In Section~\ref{sec:function_spaces} we introduce the necessary notation and recall a few definitions. 
In Section~\ref{subsec:strategy_ang_techniques} we discuss the heuristic idea behind the construction of the operator and the idea on how to proof that these assumptions are fulfilled in general.

\begin{assumption}\label{assump:base}
We fix the dimension parameter $d \in \N \setminus \{1\}$.
 We let $\Omega \defby \S'(\R^d)$,
$\P$ a probability measure on the Borel-$\sigma$-algebra on $\Omega$. 
We define the random variable $\xi$ with values in $\cS'(\R^d)$ by $\xi(\omega) \defby \omega$ for $\omega \in \Omega$.
There exists a $\delta \in (0, 1)$ such that for all $\sigma \in (0,\infty)$ one has $\P(\xi \in \csp^{-2 + \delta, \sigma}(\R^d)) =1$,  where $\csp^{-2 + \delta, \sigma}(\R^d)$ is a weighted Besov-Hölder space,  see Definition~\ref{def:besov_spaces}.
  A smooth, symmetric function $\rho \in \S(\R^d)$  with $\int \rho =1$  is given and we set
  \begin{equation}\label{eq:xi_mollifier}
    \rho_{\epsilon}(x)
  \defby \epsilon^{-d} \rho(\epsilon^{-1} x) \quad \mbox{and} \quad \xi_{\epsilon} \defby \rho_{\epsilon} \conv \xi, \qquad x\in \R^d, \epsilon \in (0,\infty). 
  \end{equation}
\end{assumption}

\begin{remark}
\label{rem:d_not_equal_one}
  We do not allow $d$ to be equal to one, since we will work with Green's functions and for $d=1$ the Green function is not singular and one would need a different treatment. 
  
  However, the interesting potentials for $d = 1$ are given by the derivative of fractional Brownian motions, whose regularity is 
  greater than $-1$. Hence, they can be easily treated in the classical framework of symmetric forms, see Theorem~\ref{theorem:examples_bd_sym_with_weighted_input}~\ref{item:sym_form_on_domain}.
\end{remark}

\change{

\begin{remark}[The main assumptions]
\label{rem:assumptions_brief}
Beside the above Assumption~\ref{assump:base}, we introduce three other assumptions in Assumptions~\ref{assump:all_three_assump}, namely what we call the 
\begin{itemize}
\item \emph{Construction assumption}, Assumption~\ref{assump:all_three_assump}~\ref{item:construction_assump}. This imposes the existence of renormalisation constants $(c_\epsilon)_{\epsilon>0}$ in $\R$ under which (renormalised) stochastic objects converge. 
\item \emph{Neumann assumption}, Assumption~\ref{assump:all_three_assump}~\ref{item:neumann_assump}. This imposes certain convergence of the stochastic terms considering the boundary of the domain. 
\item \emph{Ergodic assumption}, Assumption~\ref{assump:all_three_assump}~\ref{item:ergodic_assump}. This imposes ergodicity on the noise. 
\end{itemize}
\end{remark}

}

Now we state our three main results of the paper.
By ``domain'' we mean a nonempty open subset of $\R^d$ (remember that we assume $d\in \N \setminus \{1\}$).

\begin{definition}
\label{def:AH_eps_with_dir_and_neu}
We impose the construction assumption~\ref{assump:all_three_assump}~\ref{item:construction_assump}. 
Let $\epsilon>0$. 
\begin{enumerate}
\item For a bounded  domain $U$ we define $\HAHeps^{\dir,U}$  to be the self-adjoint operator on $L^2(U)$,
\begin{align}
\label{eqn:regular_AH}
-\Delta - \xi_{\epsilon} + c_{\epsilon}
\end{align}
with Dirichlet boundary conditions.
\item For a bounded Lipschitz domain $U$ we define $\HAHeps^{\neu,U}$ to be the self-adjoint operator \eqref{eqn:regular_AH} on $L^2(U)$ with Neumann boundary conditions.
\end{enumerate}
\end{definition}

\begin{remark}
Actually, in Section~\ref{sec:eigenvalues_of_AH} we first define the operators $\HAHeps^{\dir,U}$ and $\HAHeps^{\neu,U}$ as those that correspond to symmetric forms given in terms of the stochastic terms that we describe in Section~\ref{sec:assumptions}.
Then we show that these equal \eqref{eqn:regular_AH}.
\end{remark}

\begin{definition}\cite[Definition p. 284]{ReSi72}
\label{def:norm_resolvent_convergence}
Let $A,A_1,A_2,\dots$ be self-adjoint operators on a Banach space $\fX$. We say that the sequence $(A_n)_{n\in\N}$ converges in
\emph{norm resolvent sense} and write
\begin{align*}
A_n \normresconv_{n\rightarrow \infty} A
\end{align*}
 if
\begin{equation*}
\limn \norm{(\i + A_n)^{-1} - (\i + A)^{-1}}_{\fX \to \fX} = 0.
\end{equation*}
A sequence converges in norm resolvent sense if and only if the above convergence holds with ``$\i$'' replaced by ``$\lambda$''  for any $\lambda \in \C \setminus \R$, see \cite[Theorem VIII.19]{ReSi72}.
\end{definition}

The first result concerns the construction of Anderson Hamiltonians
on bounded domains, with Dirichlet boundary conditions.

\begin{theorem}[Theorem~\ref{thm:convergence_of_Dirichlet_AH}] \label{thm:main_AH}
Assume the construction assumption~\ref{assump:all_three_assump}~\ref{item:construction_assump}. 
Let $U$ be a bounded domain.
There exists a self-adjoint operator $\HAH^{\dir,U}$ on $L^2(U)$ such that
\begin{align*}
\HAHeps^{\dir,U} \normresconv_{\epsilon\downarrow 0}
\HAH^{\dir,U} \mbox{ in probability}.
\end{align*}
Furthermore,  each of the operators has a countable spectrum of eigenvalues and the eigenvalues of $\HAHeps^{\dir,U}$ converge in probability to those of $\HAH^{\dir,U}$. 
Moreover, there exist choices of eigenfunctions of $\HAHeps^{\dir,U}$ and $\HAH^{\dir,U}$ such that one also has convergence of these eigenfunctions in probability. 

The limit $\HAH^{\dir,U}$ is independent of the mollifier $\rho$.
\end{theorem}

The second main result concerns Anderson Hamiltonians on bounded Lipschitz domains with Neumann boundary conditions.
\begin{theorem}[Theorem~\ref{thm:convergence_of_Neumann_AH}] \label{thm:main_Neumann_AH}
We impose the construction and Neumann assumption ~\ref{assump:all_three_assump}~\ref{item:construction_assump} and~\ref{item:neumann_assump}. 
Let $U$ be a bounded Lipschitz domain. There exists   a self-adjoint operator $\HAH^{\neu,U}$ on $L^2(U)$ such that
\begin{align*}
\HAHeps^{\neu,U} \normresconv_{\epsilon\downarrow 0}
\HAH^{\neu,U} \mbox{ in probability}.
\end{align*}
Furthermore, each of the operators has a countable spectrum of eigenvalues and the eigenvalues of $\HAHeps^{\neu,U}$ converge in probability to those of $\HAH^{\neu,U}$.
Moreover, there exist choices of eigenfunctions of $\HAHeps^{\neu,U}$ and $\HAH^{\neu,U}$ such that one also has convergence of these eigenfunctions in probability. 

The limit $\HAH^{\neu,U}$ is independent of the mollifier $\rho$.
\end{theorem}

\begin{remark}
The statement of Theorem~\ref{thm:convergence_of_Dirichlet_AH} is actually slightly more general than that of Theorem~\ref{thm:main_AH}. Convergence in probability implies that there exists a subsequence and a set $\Omega_1 \subseteq \Omega$ of probability one such that the subsequence converges everywhere on $\Omega_1$. For the convergence of the Dirichlet operators, this set $\Omega_1$ can be chosen independently from the choice of bounded domain $U$.
\end{remark}

The last main result concerns the integrated density of states (IDS) of Anderson Hamiltonians. For example, we show that the notion of the IDS for Anderson Hamiltonians with smooth potentials can be extended to irregular potentials.

For a bounded domain $U$ and $L\in [1,\infty)$ we write $|U|$ for the Lebesgue measure of $U$ and
\begin{align*}
      U_L \defby LU = \set{x \in \R^d \given L^{-1} x \in U}.
\end{align*}
We recall that for the Anderson Hamiltonian with a smooth ergodic potential $V$ the integrated density of states $\ecf_V$ is given by the right-continuous and increasing function $\R \rightarrow \R$ with $\lim_{\lambda \rightarrow -\infty} \ecf_V(\lambda) =0$ for which, with $(\lambda_k(U,V))_{k\in\N}$ being the eigenvalues of $-\Delta-V$ with Dirichlet boundary conditions on $U$ (counting multiplicities), for any bounded domain $U$ and continuity point $\lambda$ of $\ecf_V$,
\begin{align*}
      \lim_{L \to \infty} \frac{1}{\abs{U_L}} \sum_{k \in \N} \indic_{\{\lambda_k(U_L,V) \leq \lambda\}}
       = \ecf_V(\lambda).
\end{align*}

\begin{theorem}[Theorem~\ref{thm:IDS_general},  Theorem~\ref{thm:IDS_epsilon} and Theorem~\ref{thm:IDS_tails}]\label{thm:main_IDS}
\

\noindent We impose the construction and ergodic assumption~\ref{assump:all_three_assump}~\ref{item:construction_assump} and~\ref{item:ergodic_assump}. 
There exists a (deterministic) right-continuous and increasing  function $\NAH:\R \to \R$ with
\begin{align*}
\lim_{\lambda \to -\infty} \NAH(\lambda) = 0,
\end{align*}
such that the following holds:
  \begin{enumerate}[leftmargin=*]
  \item
  \label{item:IDS_existence}
 For $(\lambdaAHk^{\dir}(U))_{k \in \N}$ being the eigenvalues of $\HAH^{\dir,U}$ as in Theorem~\ref{thm:main_AH} (counting multiplicities), almost surely, one has for every bounded domain $U$ and every continuity point $\lambda$ of $\NAH$
    \begin{equation*}
      \lim_{L \to \infty} \frac{1}{\abs{U_L}} \sum_{k \in \N} \indic_{\{\lambdaAHk^{\dir}(U_L) \leq \lambda\}}
       = \NAH(\lambda).
    \end{equation*}
 ($\NAH$ is called the integrated density of states of the Anderson Hamiltonian with potential $\xi$.)
    \item \label{item:IDS_approx} Let $\NAHeps$ be the integrated density of states of the Anderson Hamiltonian with potential $ \xi_{\epsilon} - c_{\epsilon}$, where $c_{\epsilon}$ is
    the renormalization constant from Assumption~\ref{assump:all_three_assump}~\ref{item:construction_assump}. \\
    Then, $\NAHeps$ converges vaguely to  $\NAH$ (see Definition~\ref{def:vague_convergence}).
    \item  \label{item:IDS_right_tail}One has $\lim_{\lambda \to \infty} \lambda^{-\frac{d}{2}} \NAH(\lambda) = \frac{\abs{B(0, 1)}}{(2 \pi)^d}$.
    \item \label{item:IDS_left_tail} For any bounded domain $U$ and
    $\alpha \in (0, \infty)$, the following identities hold in $[-\infty, 0]$:
    \begin{align*}
      \limsup_{\lambda \to -\infty} (-\lambda)^{-\alpha} \log \NAH(\lambda) &=
      \limsup_{\lambda \to -\infty} (-\lambda)^{-\alpha} \log \P(\lambdaAHone^{\dir}(U) \leq \lambda), \\
      \liminf_{\lambda \to -\infty} (-\lambda)^{-\alpha} \log \NAH(\lambda) &=
      \liminf_{\lambda \to -\infty} (-\lambda)^{-\alpha} \log \P(\lambdaAHone^{\dir}(U) \leq \lambda).
    \end{align*}
    \item \label{item:IDS_Neumann}Impose furthermore the Neumann assumption, Assumption~\ref{assump:all_three_assump}~\ref{item:neumann_assump}. For  $(\lambdaAHk^{\neu}(U))_{k \in \N}$ being the eigenvalues of $\HAH^{\neu,U}$ as in Theorem~\ref{thm:main_Neumann_AH},
    for every bounded Lipschitz domain $U$ and every continuity point $\lambda$ of $\NAH$,
    \begin{equation*}
      \lim_{\substack{ L \to \infty \\ L \in \N} } \frac{1}{\abs{U_L}} \sum_{k \in \N} \indic_{\{\lambdaAHk^{\neu}(U_L) \leq \lambda\}}
      = \NAH(\lambda), \quad \text{almost surely}.
    \end{equation*}
  \end{enumerate}
\end{theorem}

With the above theorem in combination with the tail asymptotics of the principal eigenvalue proven in  \cite[Theorem 2]{hsu2021asymptotic} we obtain the precise tail behaviour of the IDS for the Anderson Hamiltonian with white noise potential in $d$ dimensions.

\begin{corollary}\label{cor:ids_tails_white_noise}
  Let $d \in \{2, 3\}$ and $\xi$ be the $d$-dimensional white noise. Then,
  \begin{equation*}
    \lim_{\lambda \to -\infty} (-\lambda)^{-\frac{4 - d}{2}} \log \NAH(\lambda)
    = -\frac{8}{d^{d/2} (4 - d)^{2- d/2}} \kappa_{d}^{-4},
  \end{equation*}
  where $\kappa_d$ is the best constant of the Gagliardo-Nirenberg inequality
  \begin{equation*}
    \norm{f}_{L^4(\R^d)} \leq C \norm{\nabla f}_{L^2(\R^d)}^{d/4} \norm{f}_{L^2(\R^d)}^{1 - d/4}.
  \end{equation*}
\end{corollary}
\begin{proof}
  This follows from 
Theorem~\ref{thm:main_IDS}~(d) (see Theorem~\ref{thm:IDS_tails}) and \cite[Theorem 2]{hsu2021asymptotic}.
\end{proof}

\begin{remark}
  The case $d = 1$ is of course known, see \cite{fukushima_spectra_1977}.
  The case $d = 2$ was proved in \cite{matsuda2020integrated}.
  In the physics literature, these tail behaviours have been already expected (e.g., \cite{cardy_electron_1978, Brezin_1980}).
\end{remark}

\subsection{Ideas behind the construction and deriving the stochastic terms}
\label{subsec:strategy_ang_techniques}

In this section we give the general ideas behind the construction of the Anderson Hamiltonians as in Theorem~\ref{thm:main_AH} and Theorem~\ref{thm:main_Neumann_AH}. 
As is common practice when dealing with singular SPDEs, we transform and decompose terms and give a meaning to the appearing stochastic terms by means of a renormalization procedure. 
The heuristics behind how to derive these stochastic terms is described in Remark~\ref{remark:finding_w_eps}. 

\medskip

As mentioned before, we construct the Anderson Hamiltonian by constructing the corresponding symmetric form (for a definition see Definition~\ref{def:symmetric_form}). For a smooth potential $\zeta$, the Anderson Hamiltonian $-\Delta + \zeta$ on a bounded domain $U$ with Dirichlet boundary conditions corresponds to the symmetric form 
\begin{align}
\label{eqn:smooth_pot_sym_form}
(u,v) \mapsto \int_U \nabla u \cdot \nabla v +\int_U \zeta uv,
\end{align}
on the Sobolev space $H_0^1(U)$, i.e., $u,v\in H_0^1(U)$. 
Whereas, if we consider instead Neumann boundary conditions the operator corresponds to \eqref{eqn:smooth_pot_sym_form} on the Sobolev space $H^1(U)$. 
For a discussion on this, we refer to \cite[Section 6 and 7]{davies_1995}. 

Now we want to give a meaning to the symmetric form \eqref{eqn:smooth_pot_sym_form}, basically with ``$\zeta$'' replaced by ``$\xi$''. By the bilinearity of symmetric forms, it suffices to give a meaning on the diagonal, i.e., 
we may consider $u=v$ in the above formula (see also the comment below Definition~\ref{def:symmetric_form}). 
Now we will transform the above formula, in a type of ``partial Cole--Hopf'' transform, by imposing that $u$ is the product of $e^w$ and another function,  for some well-chosen function $w$. We present this in Lemma~\ref{lem:sym_form_transform_Neumann} after introducing some convenient notation in Definition~\ref{def:notation_integral_dot} and recalling Green's formula in Lemma~\ref{lem:integration_by_parts}. For the definition of a Lipschitz domain $U$ and $C^\infty(\overline U)$, see Definition~\ref{def:lipschitz_domain}. 

\begin{definition}
\label{def:notation_integral_dot}
Let $U$ be a bounded Lipschitz domain. 
Let $\nu : \partial U \rightarrow \R^d$ be such that $\nu$ is the outer unit normal on $\partial U$ almost everywhere. 
Let $S$ the $(d-1)$-dimensional Lebesgue measure on $\partial U$. 
Let $g \in C^1(\overline U)$. 
For $x\in \partial U$ we define $\nabla_\nu g(x) \defby \nabla g(x) \cdot \nu(x)$. 
\begin{calc}
Observe that 
  \begin{equation*}
    \nabla_{\nu} g(x) \defby \lim_{h \to 0} \frac{g(x + h \nu(x)) - g(x)}{h}.
  \end{equation*}
\end{calc}
Let $f$ be a measurable function on $\partial U$ such that $f \nabla_\nu g$ is integrable on $\partial U$. 
We write 
  \begin{equation*}
    \int_{\partial U} f \nabla g \cdot \dd \bm{S} \defby \int_{\partial U} f \nabla_{\nu} g \dd S.
  \end{equation*}
\end{definition}

\begin{lemma}[{\cite[Theorem 4.6]{EvansLawrenceCraig2015MTaF}}]\label{lem:integration_by_parts}
\ 
\begin{itemize}
\item  
Let $U$ be a bounded domain. Then, for every 
$f \in H^1(U)$ 
and $g \in C^2_c(U)$,
    \begin{equation*}
        \int_U \nabla f \cdot \nabla g  = -\int_U f \Delta g . 
  \end{equation*}
\item 
  Let $U$ be a bounded Lipschitz domain. Then, for every 
$f\in H^1(U)$ 
  and $g\in C^2(\overline U)$, 
  \begin{equation}\label{eq:integration_by_parts}
    \int_U \nabla f \cdot \nabla g  = -\int_U f \Delta g  + \int_{\partial U} \cT(f) \nabla  g \cdot  \dd \bm{S}, 
  \end{equation}
where $\cT= \cT_{H^1(U)}$ is the trace operator (see Lemma~\ref{lem:extension_and_trace_sobolev}~\ref{item:trace_op_and_inverse}). 
\end{itemize}
\end{lemma}

\begin{lemma}\label{lem:sym_form_transform_Neumann}
  Let $U$ be a bounded Lipschitz domain, $\zeta, w \in C^{\infty}(\overline{U})$ and $u \in H^1(\overline{U})$.
  Set $v \defby e^{-w} u$. Then, one has
  \begin{multline*}
    \int_U (\abs{\nabla u}^2 - \zeta u^2) 
    = \int_U e^{2w} \abs{\nabla v}^2 - \int_U e^{2w} (\zeta + \abs{\nabla w}^2 + \Delta w) v^2
    + \int_{\partial U} \cT(v^2) e^{2w} \nabla w \cdot \dd \bm{S}.
  \end{multline*}
\end{lemma}
\begin{proof}
  One has
  \begin{equation*}
    \int_U e^{2w} \abs{\nabla v}^2 = \int_U e^{2w} \abs{\nabla(e^{-w} u)}^2
    = \int_U \abs{\nabla w}^2 u^2 + \int_U \abs{\nabla u}^2 - \int_U \nabla w \cdot \nabla (u^2)
  \end{equation*}
  and by integration by parts (Lemma~\ref{lem:integration_by_parts}) one has
  \begin{equation*}
    \int_U \nabla w \cdot \nabla(u^2) = \int_{\partial U} \cT(u^2) \nabla w \cdot \dd \bm{S} - \int_U (\Delta w) u^2. \qedhere
  \end{equation*}
\end{proof}

By taking $\zeta = \xi_\epsilon - c_\epsilon$ for some $c_\epsilon \in \R$ and a smooth enough $w^\epsilon$
and by writing 
\begin{align}
y^\epsilon & =  - e^{2 w^{\varepsilon}} ( \xi_\epsilon - c_\epsilon + |\nabla w^\epsilon|^2 + \Delta w^\epsilon ), \label{eq:y_eps_def}\\
u^\flat & = e^{-w^\epsilon} u, \notag
\end{align} 
the Anderson Hamiltonian with potential $\xi_\epsilon - c_\epsilon$, namely, 
\begin{align}
\label{eqn:smooth_anderson_minus_c_eps}
-\Delta - \xi_\epsilon + c_\epsilon,
\end{align}
on $U$ with Dirichlet boundary conditions, corresponds to 
the symmetric form on $H_0^1(U)$ given by, 
\begin{align*}
\int_U |\nabla u|^2 - (\xi_\epsilon - c_\epsilon) u^2
& = \int_U e^{2 w^\epsilon} |\nabla u^\flat|^2
+\int_U y^\epsilon (u^\flat)^2.
\end{align*}
\change{
Similarly, if we instead consider Neumann boundary conditions, it corresponds to the symmetric form on $H^1(U)$  given by, 
\begin{align*}
\int_U |\nabla u|^2 - (\xi_\epsilon - c_\epsilon) u^2
& = \int_U e^{2 w^\epsilon} |\nabla u^\flat|^2
+\int_U y^\epsilon (u^\flat)^2 + \int_{\partial U}  \cT((u^\flat)^2) e^{2w^\epsilon} \nabla w^\epsilon \cdot \dd \bS. 
\end{align*}
As $w^\epsilon$ will be chosen smooth enough such that $H_0^1(U) = e^{w^\epsilon} H_0^1(U)$ and $H^1(U) = e^{w^\epsilon} H^1(U)$, this means that we have to give a meaning to the limit of two, in the case of Dirichlet boundary conditions, and three, in the case of Neumann boundary conditions, symmetric forms, namely 
\begin{align}\label{eq:three_syms}
v \mapsto \int_U e^{2w^\epsilon} |\nabla v|^2, \qquad v\mapsto \int_U y^\epsilon v^2, \qquad 
v\mapsto \int_{\partial U} \cT(v^2) e^{2w^\epsilon} \nabla w^\epsilon \cdot \dd \bS. 
\end{align}
As we will show in more detail in Section~\ref{sec:symmetric_forms}, if there are $w$ and $y$ such that $w^\epsilon \rightarrow w$ in $\cC^\delta$ and $y^\epsilon \rightarrow y$ in $\cC^{-1+\delta}$, then the first two symmetric forms in \eqref{eq:three_syms} converge to the corresponding symmetric forms with $w$ and $y$ in place of $w^{\varepsilon}$ and $y^{\varepsilon}$, respectively. \begin{calc} (Observe that $v\in H^1(U)$ one has $|\nabla v|^2 \in L^1(U)$ and $v^2 \in W^1_1(U)$, where $W^k_p$ is the usual Sobolev space, see Definition~\ref{def:fractional_Sobolev}.) \end{calc} 
This basically reflects the first assumption that we impose and call the \emph{construction assumption}:  Assumption~\ref{assump:all_three_assump}~\ref{item:construction_assump}. 
As $\cT(v^2)$ is in $L^1(\partial U)$ for $v\in H^1(U)$, one would expect the third symmetric form to converge only if $\delta > 1$, as then $\nabla w^\epsilon \rightarrow \nabla w$ in $\cC^{\delta -1}$. 
However, with an additional assumption on a stochastic term, which we call the \emph{Neumann assumption} namely Assumption~\ref{assump:all_three_assump}~\ref{item:neumann_assump},   it suffices to assume that $\delta> \frac12$. 
For this we use a further decomposition of the third symmetric form in \eqref{eq:three_syms}.
We discuss this in more detail in Section~\ref{subsec:stoch_terms_neumann}. }

\bigskip

So as mentioned, the construction and Neumann assumptions impose conditions on the stochastic terms.
Given a noise $\xi$ there is a way on how to find these stochastic terms. We describe the idea in the following remark.

\begin{remark}[Idea behind the derivation of the stochastic term $w^\epsilon$]
\label{remark:finding_w_eps}
Let us present the heuristic idea on how to choose these $w^{\epsilon}$, by forgetting for the moment about the regularization parameter ``$\epsilon$'' and the renormalisation constants ``$c_\epsilon$''. 
Namely, 
we are going to construct a $\w$ such that 
\begin{align*}
\xi + \abs{\nabla \w}^2 + \Delta \w
\end{align*}
is sufficiently regular.
For convenience of conversation, let us introduce a formal notion ``$\deg(\sigma)$'' which more or less reflects the regularity of an object ``$\sigma$'', namely, we set ($\deg$ coincides with $\abs{\cdot}_+$ as  in Definition~\ref{def:degree})
\begin{align*}
\deg(\xi) = -2 + \delta, & & 
\deg(\partial_i \sigma) = \deg(\sigma) - 1, \\
\deg((-\Delta)^{-1} \sigma) = \deg(\sigma) + 2, & & 
\deg(\sigma_1 \cdot \sigma_2) = \deg(\sigma_1) + \deg(\sigma_2). 
\end{align*}
Now we expect to be able to choose $\w$ with positive regularity, so that the term $|\nabla \w|^2$ has a larger regularity than $\Delta \w$, i.e., $\deg(|\nabla \w|^2)>\deg (\Delta \w)$).
Hence, by ignoring $|\nabla \w|^2$ we try to find a $\w$ such that $\Delta \w$ compensates the irregularity of $\xi$.
The most natural choice for this is
$\w = (-\Delta)^{-1} \xi$. In this case,
\begin{equation*}
  \xi + \abs{\nabla \w}^2 + \Delta \w = \abs{ \nabla (-\Delta)^{-1} \xi}^2 =: \tau_1.
\end{equation*}

Observe that $\deg(\tau_1) = -2 + 2 \delta$, which is greater than $\deg(\xi)=-2 + \delta$.
If the degree $-2 + 2 \delta$ is too small to our taste, then we instead set
\begin{equation*}
  \w \defby (-\Delta)^{-1}(\xi + \tau_1).
\end{equation*}
For this $\w$ we obtain
\begin{equation*}
  \xi + \abs{\nabla \w}^2 + \Delta \w =
 \underbrace{2 \nabla (-\Delta)^{-1} \xi \cdot \nabla (-\Delta)^{-1} \tau_1}_{\tau_2}
  +\underbrace{\abs{\nabla (-\Delta)^{-1} \tau_1}^2}_{\tau_3},
\end{equation*}
where $\deg(\tau_2) = -2 + 3 \delta$ and $\deg(\tau_3) = -2 + 4 \delta$ are both greater than $-2 + 2 \delta$.
One can repeat this argument until
one obtains a sum of terms for $\xi + \abs{\nabla \w}^2 + \Delta \w$ such that
each term has sufficiently large degree.
(As Theorem~\ref{theorem:examples_bd_sym_with_weighted_input}~\ref{item:sym_form_on_domain} shows, ``sufficiently large'' means that the degree is greater than $-1$.)

The above arguments are not yet mathematically rigorous, as for instance, the term $\abs{\nabla (-\Delta)^{-1} \xi}^2$, that is the inner product of $\nabla (-\Delta)^{-1} \xi$ with itself, a priori does not make sense since $\nabla (-\Delta)^{-1} \xi$ is not a function in general.
Moreover, it turns out that
$\abs{\nabla (-\Delta)^{-1} \xi_{\epsilon}}^2$ itself does not converge as $\epsilon \downarrow 0$,
but if we take a ``renormalization'' of it, namely
\begin{equation*}
  \abs{\nabla (-\Delta)^{-1} \xi_{\epsilon}}^2 - \expect[\abs{\nabla (-\Delta)^{-1} \xi_{\epsilon}}^2(0)],
\end{equation*}
then it does converge in probability.
Then, we take the limit of it as our definition of $\tau_1$ (instead of the nonrigorous definition $\abs{\nabla (-\Delta)^{-1} \xi}^2$ above).

\medskip

We will discuss the derivation of the stochastic terms in more detail in Section~\ref{sec:assumptions}. 
On the one hand we consider the concrete examples of two and three dimensional white noise, in which case it suffices to take the first and second choice of $w$ as described above, respectively. On the other hand, we mention how to deal with general noises of regularity greater than $-2$ with the use of the theory of regularity structures. 
\end{remark}

\subsection{Outline}
\label{sec:outline}

In Section~\ref{sec:function_spaces}, we introduce  some notation related to the function spaces that we use.
Technical estimates  related to the objects introduced in Section~\ref{sec:function_spaces} are postponed to   Appendix~\ref{subsec:technical_estimates_function_spaces}.
In Section~\ref{sec:assumptions}, we describe the main assumptions and give examples for which these assumptions are satisfied. 
In Section~\ref{sec:symmetric_forms}, we cover  some theory on (deterministic) symmetric forms that will be relevant to our problems.
In Section~\ref{sec:eigenvalues_of_AH}, we give the definition of the Anderson Hamiltonians and prove the main theorems on the construction and the IDS. 
In Appendix~\ref{sec:review_of_reg_str} we review some necessary terminalogies from regularity structures in order to prove in Appendix~\ref{sec:bphz_AH} that under general conditions the construction assumptions are satisfied for subcritical noises. 

\subsection{Notation}
\label{sec:notation}

We set $\N \defby \{1, 2, 3, \ldots\}$ and $\N_0 \defby \{0\} \cup \N$.
We call a subset of $\R^d$ a \emph{domain} if it is
an open subset of $\R^d$.
We denote by $\overline{U}$ the closure of a subset $U$ of $\R^d$.
Given a subset $U$ of $\R^d$, $L \in (0, \infty)$ and $x \in \R^d$, we set
\begin{align*}
U_L \defby  LU =  \set{y \in \R^d \given L^{-1} y \in U},
\end{align*}
$d(x, U) \defby \inf\set{\abs{x-y} \ \ \given y \in U}$,
$B(U, R) \defby \set{y \in \R^d \given d(y, U) \leq R}$ and $B(x, R) \defby B(\{x\}, R)$.
We denote by $\abs{U}$ the Lebesgue measure of a measurable set $U$.

We denote by $\S(\R^d)$ the space of Schwartz functions  equipped with the locally convex topology generated by the Schwartz seminorms,   and, by $\S'(\R^d)$ the space of tempered distributions,
that is, the dual space of $\S(\R^d)$.
We denote by $\supp(f)$ the support of a distribution or a continuous function $f$ in $\R^d$.
Let $k\in \N \cup \{ \infty\}$. For a domain $U$, we write $C^k(U)$ for the $k$ times continuously differentiable functions on $U$ and $C_c^k(U)$ for those functions in $C^k(U)$ with compact support. For a closed set $V\subseteq \R^d$ (we will consider $\overline U$ and $\partial U$ for domains $U$), we define
\begin{align*}
C^k(V) \defby \{ f|_V : f\in C^k(\R^d)\}.
\end{align*}
For a subset $U$ of $\R^d$, either open or closed, we define
\begin{align*}
  \norm{f}_{C^k(U)} &\defby \sup_{x \in U} \sum_{l \in \N_0^d: \abs{l} \leq k} \abs{\partial^k f(x)}
  \quad \text{if } k < \infty
\end{align*}
We denote by $L^p(U)$, $p \in [1, \infty]$, the usual Lebesgue $L^p$-space on $U$.
We denote by $\inp{F}{f}$ the dual pairing of $F \in \S'(\R^d)$ and $f \in \S(\R^d)$ and the dual pairing of
Besov spaces \cite[Theorem 2.17]{sawano2018theory}.
We denote by $f \conv g$ the convolution of $f$ and $g$.
By duality, the convolution $f \conv g$ for $f \in \S(\R^d)$ and $g \in \S'(\R^d)$ is defined and represents a smooth function.

 Let $A,X$ be sets and $f,g: A\times X \rightarrow [0,\infty]$.
We write $f(a,x) \lesssim_a g(a,x)$ if there exists a constant $C\in (0,\infty]$ (possibly) depending on $a$ --for which we also write either $C=C(a)$ or $C=C_a$-- such that $f(a,x) \le C g(a,x)$ for all $x$.
We will not explicitly write
the dependence on the dimension $d$, i.e., we write ``$\lesssim_a$'' instead of ``$\lesssim_{d,a}$''.

\section{Function spaces and Green's functions}\label{sec:function_spaces}
\subsection{Besov spaces on $\R^d$}
Here we describe definitions and important properties of Besov spaces on $\R^d$.
Technical estimates related to Besov spaces will be given in Section~\ref{subsec:estimates_in_Besov}.
\begin{definition}
  The \emph{Fourier transform} of a function $f\in \cS(\R^d)$ is defined by
  \begin{equation*}
    \F f (y) \defby \int_{\R^d} f(x) e^{-2 \pi \i x \cdot y} \dd x. 
  \end{equation*}
We define $\F f$ for $f \in \S'(\R^d)$ by duality:
  $\inp{\F f}{g} \defby \inp{f}{\F g}$ for $g\in \cS(\R^d)$.
\end{definition}
\begin{definition}\label{def:besov_spaces}
  Let $\check{\chi}$, $\chi$  be
  smooth radial functions with values in $[0, 1]$ on $\R^d$ with the following properties:
  \begin{itemize}
    \item $\supp (\check{\chi}) \subseteq B(0, \frac{4}{3})$, $\supp (\chi) \subseteq \set{x \in \R^d \given
    \frac{3}{4} \leq \abs{x} \leq \frac{8}{3}}$.
    \item $\check{\chi}(x) + \sum_{j=0}^{\infty} \chi(2^{-j} x) = 1$ for $x \in \R^d$
    and $\sum_{j  =-\infty}^\infty \chi(2^{-j} x) = 1$ for $x \in \R^d \setminus \{0\}$.
  \end{itemize}
  The existence of such $\check{\chi}$ and $\chi$ is guaranteed by \cite[Proposition 2.10]{Bahouri2011}.
For $f\in \cS'(\R^d)$ we set
\begin{align*}
\Delta_{-1} f = \F^{-1} (\check{\chi} \cF f), \qquad
\Delta_j f = \F^{-1} (\chi(2^{-j} \cdot) \cF f) , \qquad  j\in \N_0.
\end{align*}
  Let $p, q \in [1, \infty]$ and $r \in \R$.
  For $\sigma \in \R$, we set
\begin{align*}
w_{\sigma}(x) \defby (1 + \abs{x}^2)^{-\frac{\sigma}{2}}.
\end{align*}
  The weighted nonhomogeneous Besov space $B_{p, q}^{r, \sigma}(\R^d)$ consists of those distributions $f$ in
    $\S'(\R^d)$ such that $\norm{f}_{B_{p,q}^{r, \sigma}(\R^d)}<\infty$, where
     \begin{align*}
     \norm{f}_{B_{p,q}^{r, \sigma}(\R^d)} \defby
 \Big\| \Big (2^{-rj} \norm{w_{\sigma} \Delta_j f}_{L^p(\R^d)} \Big)_{j=-1}^\infty \Big \|_{\ell^q}
     \end{align*}
Let us mention that the norm actually depends on the choice of $\check \chi$ and $\chi$, though the space does not. See for example \cite[Corollary 2.70]{Bahouri2011}. \cite[Lemma 2.69]{Bahouri2011} implies that different choices of $\check \chi$ and $\chi$ as above give equivalent norms.

  We set $\csp^{r, \sigma}(\R^d) \defby B^{r, \sigma}_{\infty, \infty}(\R^d)$  and write $\csp^{r}(\R^d) \defby \csp^{r,0}(\R^d)$, $B_{p,q}^r(\R^d) = B_{p,q}^{r,0}(\R^d)$.
\end{definition}

\subsection{Sobolev--Slobodeckij spaces on bounded domains}

\change{

\begin{definition}\label{def:lipschitz_domain}
We say that a bounded domain $U$ of $\R^d$ is called a \emph{bounded Lipschitz domain} if its boundary can be
locally approximated by Lipschitz functions in the following sense: 
For each $y \in \partial U$, there exist $r>0$, a Lipschitz function $\gamma : \R^{d-1} \rightarrow \R$ and an bijection (a relabelling) $\sigma: \{1,\dots, d\} \rightarrow \{1,\dots,d\}$ such that 
\begin{align}
\label{eqn:lipschitz_boundary}
U \cap B(y,r) = \{ x\in B(y,r) : x_{\sigma{(d)}} > \gamma( x_{\sigma(1)}, \dots, x_{\sigma(d)} ) \}. 
\end{align}
\begin{calc}
This means in particular 
\begin{align*}
\partial U \cap B(y,r) = \{ x\in B(y,r) : x_{\sigma{(d)}} = \gamma( x_{\sigma(1)}, \dots, x_{\sigma(d)} ) \}. 
\end{align*}
\end{calc}
A function $f: \partial U \rightarrow \R$ is called smooth if for each $y\in \partial U$, additional to the $r$, $\gamma$ and $\sigma$ as above such that \eqref{eqn:lipschitz_boundary} holds, there exists a smooth function $g: \R^{d-1} \rightarrow \R$ such that $f(x) = g(x_{\sigma(1)},\dots, x_{\sigma(d-1)})$ for $x\in \partial U \cap B(y,r)$. 
$C^\infty(\partial U)$ is the space of all smooth functions on $\partial U$. 
\end{definition}

}

\begin{definition}\label{def:fractional_Sobolev}
  Let $U$ be a domain in $\R^d$. Let $p \in [1, \infty]$ and $r \geq 0$.
  \begin{enumerate}
    \item
\label{item:sobolev_space}
    The space $W^r_p(U)$ is the completion of $\set{f|_U \given f \in C^{\infty}(\overline{U}) ,  \norm{f|_U}_{W^r_p(U)}<\infty  }$ with respect to the norm
    \begin{equation*}
      \norm{f}_{W^r_p(U)} \defby \sum_{\alpha \in \N_0^d, \abs{\alpha} \leq r}
      \norm{\partial^{\alpha} f}_{L^p(U)} + \sum_{\alpha \in \N_0^d, \abs{\alpha} = \lfloor r \rfloor}
      [\partial^\alpha f]_{W^{r - \lfloor r \rfloor}_p(U)},
    \end{equation*}
    where $[g]_{W^0_p(U)} \defby 0$ and for $s\in (0,1)$,
    \begin{equation*}
      [g]_{W^s_p(U)} \defby
\begin{cases}
      \Big( \int_{U} \int_U
      \frac{\abs{g(x) - g(y)}^p}{\abs{x - y}^{d + p s}} \dd x\dd y \Big)^{\frac{1}{p}} & p <\infty, \\
 \sup_{x, y \in U, \abs{x - y} \leq 1} \frac{\abs{g(x) - g(y)}}{\abs{x - y}^{s}},  &  p=\infty.
\end{cases}  .
    \end{equation*}
    We set $H^r(U) \defby W_2^r(U)$.
    We denote by $W^r_{p, 0}(U)$ the completion of $C_c^{\infty}(U)$ with respect to the norm
    $\norm{\cdot}_{W^r_p(\R^d)}$ (not $\norm{\cdot}_{W^r_p(U)}$) and we set $H^r_0(U) \defby W^r_{2, 0}(U)$.
    \item
\label{item:sobolev_space_boundary}
    Let $U$ be a bounded Lipschitz domain and $r \in (0, 1)$.
    The space $W^r_p(\partial U)$ is the completion
    of $C^\infty(\partial U)$ with respect to the norm
    \begin{equation*}
      \norm{g \vert_{\partial U}}_{W^r_p(\partial U)} \defby
      \norm{g}_{L^p(\partial U)} +  [g]_{W^r_p(\partial U)}
    \end{equation*}
    where
    \begin{equation*}
      [g]_{W^r_p(\partial U)} \defby
\begin{cases}
      \Big( \int_{\partial U} \int_{\partial U}
      \frac{\abs{g(x) - g(y)}^p}{\abs{x - y}^{d - 1 + p r}} \dd x\dd y \Big)^{\frac{1}{p}} & p < \infty, \\
      \sup_{x, y \in \partial  U, \abs{x - y} \leq 1} \frac{\abs{g(x) - g(y)}}{\abs{x - y}^{ r}} & p = \infty.
      \end{cases}
    \end{equation*}
  \end{enumerate}
\end{definition}

\begin{remark}[Equivalent definitions]
For a bounded domain $U$, let  $\tilde W^r_p(U)$ be the space of $f \in L^p(U)$ such that the distributional derivatives $\partial^\alpha f$ for $|\alpha|\le r$ are in $L^p(U)$ and $\norm{f}_{W^r_p(U)} <\infty$.

Then $W^r_{p,0}(U)$ is the closure of $C_c^\infty(U)$ in $\tilde W^r_p(U)$ and if $U$ is a bounded Lipschitz domain, then $W^r_p(U) = \tilde W^r_p(U)$, see for example \cite[Theorem 1.2]{Mu70} and \cite{MeSe64}.
\end{remark}

\begin{definition}
For $U$ a domain in $\R^d$ and $r\ge 0$ we also write $C^r(U) = W^r_\infty(U)$ and $\norm{f}_{C^r(U)}= \norm{f}_{W^r_\infty(U)}$. 
\end{definition}

The following lemma relates the Sobolev--Slobodeckij spaces $W^r_p$ (and $C^r$) for $U=\R^d$ to the Besov spaces.

\begin{lemma}
\label{lem:equivalence_sob_slobo_and_besov}
Let $s\in (0,\infty) \setminus \N$ and $p \in [1,\infty]$.
Then $W^s_p(\R^d) = B_{p,p}^s(\R^d)$, $C^s(\R^d) = \csp^s(\R^d)$ and the norms $\norm{\cdot}_{W^s_p(\R^d)}$ and $\norm{\cdot}_{B_{p,p}^s(\R^d)}$ are equivalent (hence $\norm{\cdot}_{C^s(\R^d)}$ and $\norm{\cdot}_{\csp^s(\R^d)}$ are equivalent).
\end{lemma}
\begin{proof}
This follows by \cite[p.90]{Triebel1983}:
For $p\in [1,\infty)$ one has $W^s_p(\R^d) = B_{p,p}^s(\R^d)$ with equivalent norms, see \cite[p. 90 and p.113]{Triebel1983} ($W^{s,p}(\R^d)$ is written instead of $W^s_p(\R^d)$ and it is shown that  $W^s_p(\R^d)= \Lambda_{p,p}^s(\R^d) = B_{p,p}^s(\R^d)$), for $C^s(\R^d) = \cC^s(\R^d) = B_{\infty,\infty}^s(\R^d)$ with equivalent norms, see \cite[p.90 (9), (6) and p.113]{Triebel1983} (actually, in \cite{Triebel1983} $\cC^s(\R^d)$ is defined differently but shown to be the same as $B^s_{\infty,\infty}(\R^d)$).
\end{proof}

\begin{lemma}\label{lem:extension_and_trace_sobolev}
  Let $U$ be a bounded Lipschitz domain.
  \begin{enumerate}
  \item
\label{item:extension_op}
 Set $\D \defby \cup_{p \in [1, \infty], r \in [0, \infty)} W_p^r(U)$.
There exists an extension operator $\iota: \D \to \S'(\R^d)$ such that
\begin{itemize}
\item  $\iota(f) = f$  as distributions  on $U$ for $f\in \D$,
\item $\norm{\iota(f)}_{W^r_p(\R^d)} \lesssim_{U, p, r} \norm{f}_{W^r_p(U)}$
  for every $p \in [1, \infty]$, $r \in [0, \infty)$ and $f \in \D$,
\item $\iota(f) \in C^\infty(\R^d)$ for all $f\in C^\infty(\overline U)$.
\end{itemize}

  \item
\label{item:trace_op_and_inverse}
 Let $p \in (1, \infty)$ and $r \in (\frac{1}{p}, 1 + \frac{1}{p})$.
  Then, the map $C^{\infty}(\overline{U}) \rightarrow C^\infty(\partial U),   f \mapsto f\vert_{\partial U}$ extends uniquely to a bounded linear
  operator $\cT = \cT_{W^r_p(U)}: W^r_p(U) \to W^{r - \frac{1}{p}}_p(\partial U)$.
Furthermore,
  there exists a bounded linear operator that is the right inverse of $\cT$.
  \end{enumerate}
\end{lemma}
\begin{proof}
For \ref{item:extension_op} see \cite[Chapter 6]{stein_1970} in combination with \cite[Section 4]{Tr78}, or, for $r\in [0,1)$, \cite[Theorem 5.4]{dinezza2011hitchhikers}.
For \ref{item:trace_op_and_inverse}, see   \textnormal{\cite[Theorem 3]{marschall_1987}}.
\end{proof}


\begin{definition}
\label{def:universal_extension_op_and_trace}
An extension operator $\iota$ as in Lemma~\ref{lem:extension_and_trace_sobolev}~\ref{item:extension_op} is called a \emph{universal extension operator} from $U$ to $\R^d$.  The operator $\cT$ as in Lemma~\ref{lem:extension_and_trace_sobolev}~\ref{item:trace_op_and_inverse} is called the \emph{trace operator}.
\end{definition}

\subsection{Green's functions}
\label{subsec:green_function}

Let $G$ be the Green's function of $-\Delta$ on $\R^d$ (remember that $d\in \N\setminus \{0\}$), which means that $-\Delta G * f = f$ for $f\in \cS(\R^d)$.
That is, $G$ is the distribution which is represented by the function defined for $x\ne 0$ by
\begin{equation*}
  G(x) =
  \begin{cases}
   - \frac{1}{2 \pi} \log \abs{x} & d=2, \\
    \frac{1}{d (d - 2) \omega_d} \abs{x}^{-(d-2)} & d \geq 3,
  \end{cases}
\end{equation*}
where $\omega_d$ is the volume of the unit ball in $\R^d$. 
The identity $-\Delta G * f = f$ for $f\in \cS(\R^d)$ implies the formal identity $G= \cF^{-1}(|2\pi \cdot|^{-2})$. 
As $\abs{\cdot}^{-2}$ is singular at the origin, it is convenient to introduce the following variants of $G$.

\begin{definition}\label{def:G_N_and_H_N}
  Let $\check{\chi}$ be the function introduced in Definition~\ref{def:besov_spaces}.
  For $N \in \N_0$, we set
  \begin{equation*}
    G_N \defby \F^{-1}((1 - \check{\chi}(2^{-N} \cdot)) \abs{2 \pi \cdot}^{-2}).
  \end{equation*}
Even though in general one cannot take the convolution of any two tempered distributions, for a tempered distribution $g\in \cS'(\R^d)$ we write 
\begin{align*}
G_N * g := \cF^{-1}(  (1-\check \chi(2^{-N} \cdot)) |2 \pi \cdot|^{-2} \cF g). 
\end{align*}
Observe that the latter is indeed a tempered distribution as the product of the smooth function $(1-\check \chi(2^{-N} \cdot)) |2 \pi \cdot|^{-2}$, which itself and all its derivatives are of at most polynomial growth, and the tempered distribution $\cF g$, is again a tempered distribution. 

Moreover, observe that for any $g\in \cS'(\R^d)$ there exists a $f\in C_c^\infty(\R^d)$ such that 
\begin{align}
\label{eqn:G_N_is_parametrix}
- \Delta G_N * g = g + f. 
\end{align}
\begin{calc}
Namely $f= - \cF^{-1}(\check \chi(2^{-N} \cdot)) * g$: 
\begin{align*}
-\Delta G_N = \cF^{-1} ( (1 - \check{\chi}(2^{-N} \cdot) )) * g = g + f. 
\end{align*}
\end{calc}
\end{definition}

The parameter $N$ is introduced to control the norm of $G_N \conv f$ by letting $N$ large: 

\begin{lemma}
\label{lemma:estimate_W_n_on_U_L}
Let $U$ be a bounded domain, $\sigma \in (0,\infty)$
and $\delta_- \in (0, \delta)$.
Then for $L\ge 1$ and $N\in\N_0$
  \begin{align*}
  \norm{G_N * g}_{C^{\delta_-}(U_L)} \lesssim_{U,\delta_-,\delta,\sigma} L^\sigma 2^{-(\delta - \delta_-)N} \norm{g}_{\csp^{-2+\delta,\sigma}(\R^d)}.
  \end{align*}
Consequently, for $g\in \csp^{-2+\delta,\sigma}(\R^d)$ we have $\limN G_N *g =0$ in $\cC^{\delta_-}(U)$. 
\end{lemma}
\begin{proof}
This follows by Lemma~\ref{lemma:estimate_C_delta_U_by_weighted} and Corollary~\ref{cor:estimates_of_G_N_and_H_N}. 
\end{proof}

%
%

\begin{remark}\label{rem:G_N_M}
  For $M, N \in \mathbb{N}_0$, we have $G_N - G_M \in \mathcal{S}(\R^d)$. 
  Furthermore, for $f \in \cS'(\R^d)$ and $N \in \N_0$, we set
  \begin{equation}
  \label{eq:Delta_G_N_min_G}
    [\Delta(G_N - G)] \conv f \defby \cF^{-1}[\check{\chi}(2^{-N} \cdot)] \conv f,
  \end{equation}
  which is a smooth function.
  By considering their Fourier transforms, one observes
  \begin{equation}\label{eq:laplacian_on_G_N}
    \Delta (G_N \conv f) = -f + [\Delta(G_N - G)] \conv f . 
  \end{equation}
\end{remark}

\section{Assumptions on the stochastic terms} \label{sec:assumptions}

We motivate and state the main assumptions in Section~\ref{subsec:main_assump}, 
namely, those that are already mentioned in Remark~\ref{rem:assumptions_brief}: the construction assumption~\ref{assump:all_three_assump}~\ref{item:construction_assump}, the Neumann assumption~\ref{assump:all_three_assump}~\ref{item:neumann_assump} and the ergodic assumption~\ref{assump:all_three_assump}~\ref{item:ergodic_assump}. 
In Section~\ref{subsec:examples} we show that the construction assumption is satisfied for the white noise $\xi$ in two and three dimensions and mentioned that it is satisfied for a very general class of subcritical noises within the framework of regularity structures (the arguments are postponed to the Appendix~\ref{sec:bphz_AH}, since it requires the full-fledged theory of regularity structures, see also Appendix~\ref{sec:review_of_reg_str}). 
In Section~\ref{subsec:stoch_terms_neumann} we consider sufficient conditions and examples for which the Neumann assumption~\ref{assump:all_three_assump}~\ref{item:neumann_assump} is satisfied.

\subsection{Assumptions}\label{subsec:main_assump}

  In addition to Assumption~\ref{assump:base}, we will introduce three more assumptions. 
  In this section we go into more depth on the idea behind the construction as explained in Section~\ref{subsec:strategy_ang_techniques} and motivate the main assumptions, which we formulate in  \ref{assump:all_three_assump}.

\begin{remark}[Towards the construction assumptions]
\label{rem:towards_construction_assumptions}
  Recall from Assumption~\ref{assump:base} that $\xi_\epsilon$ is a mollification of the white noise $\xi$.
  As already motivated in Section~\ref{subsec:strategy_ang_techniques}, for the convergence of the smooth Anderson Hamiltonians, it basically suffices to find functions $w^\epsilon$ and scalars $c_\epsilon$ such that 
  \begin{align}
  \label{eqn:y_eps}
  y^\epsilon & =  - e^{2 w^{\varepsilon}} ( \xi_\epsilon - c_\epsilon + |\nabla w^\epsilon|^2 + \Delta w^\epsilon ),
  \end{align}
  converges in $\cC^{-1+\delta}$. 
  Moreover, in Remark~\ref{remark:finding_w_eps} we have described the heuristics of finding such $w^\epsilon$, or actually their limit, $w$, where (formally)
  \begin{align}
  \label{eqn:formal_formulation_of_w}
  w = (-\Delta)^{-1} X, \quad \mbox{ and,} \quad X= \xi + \tau_1 + \cdots + \tau_k. 
  \end{align}
For the above mentioned convergence of $y^\epsilon$, we have not been precise on what $\cC^{-1+\delta}$ means. 
In fact, for the noise and the stochastic terms, we will not consider them as elements on function spaces on $U$, but as elements of weighted Besov spaces on $\R^d$.   
That is, we consider $\xi$ as an element of $\cC^{-2+\delta, \sigma}(\R^d)$. The $\sigma$ determines the ``amount'' of polynomial growth that it allows for. As $\xi$ is of at most polynomial growth, so is the $w$ as in \eqref{eqn:formal_formulation_of_w}. 
But then $e^{2w}$ could be of exponential growth, so that we cannot guarantee a priori that \eqref{eqn:y_eps} is an element of a space that allows only polynomial growth, which we want in order to show the existence of the integrated density of states. 
For this reason we replace $e^{2w^\epsilon}$ by $F(w^\epsilon)$ in the definition of $y^\epsilon$, where $F$ is a bounded smooth function which equals $t\mapsto -e^{2t}$ around the origin. 
  
Now ``$(-\Delta)^{-1}$'' as in \eqref{eqn:formal_formulation_of_w} is a formal notation, which one could replace by ``$G*$'', where $G$ is the Green's function as in Section~\ref{subsec:green_function}. 
However, due to the singularity of $G$, we will not exactly work with the Green's function itself, but rather a version of it in which ``the Fourier terms around zero'' are cut off, i.e., we will work with $G_N$ instead (Definition~\ref{def:G_N_and_H_N}). 
The extra parameter $N$ will be tuned, i.e., it will be replaced by a random variable $M$ that also depends on the domain $U$ so that instead of \eqref{eqn:formal_formulation_of_w}, we consider
\begin{align*}
w = G_M * X, 
\end{align*}
and $M$ is set in such a way that $F(w) = -e^{2w}$.
\end{remark}  

By the above remark we are now  ready to put the assumptions on the stochastic terms for the construction of the Anderson Hamiltonian with Dirichlet boundary conditions, see Assumptions~\ref{assump:all_three_assump}~\ref{item:construction_assump}. 
In these assumptions we will also give a rigorous meaning to the object $X$ that appeared in \eqref{eqn:formal_formulation_of_w}. 

In Remark~\ref{rem:neumann_construction} we comment on the Neumann assumption \ref{assump:all_three_assump}~\ref{item:neumann_assump}, as it is easier to discuss this given that the reader has read the construction assumption~\ref{assump:all_three_assump}~\ref{item:construction_assump} first. 

In Assumptions~\ref{assump:all_three_assump} we display not only both the construction and the Neumann assumptions, but also the ergodic assumption, which is a rather standard assumption that is made to guarantee the existence of the integrated density of states. 



Recall that $\cC^{\alpha,\sigma}$ and $B_{p,q}^{\alpha}$ are defined in Definition~\ref{def:besov_spaces} and $G_N$ is defined in Definition~\ref{def:G_N_and_H_N}.

\begin{assumptions} 
\label{assump:all_three_assump}
\ 
We abbreviate for $\alpha \in \R$
\begin{align*}
\LL^{\alpha} = \bigcap_{\sigma\in (0,\infty)} \bigcap_{p\in [1,\infty)} L^p(\P, \cC^{\alpha,\sigma}(\R^d)). 
\end{align*}
\begin{enumerate}[label={\normalfont(\Roman*)}]
\item 
\label{item:construction_assump}
(Construction assumption)
There exist $(c_\epsilon)_{\epsilon>0}$ in $\R$; $X$ and $(X^\epsilon)_{\epsilon>0}$ in $\LL^{-2+\delta}$
such that $X-\xi \in \LL^{-2+2\delta} $
 and for all $\sigma \in (0,\infty)$ and $p \in [1,\infty)$
\begin{align*}
      \lim_{\varepsilon \downarrow 0} \norm{X^{\varepsilon} - X}_{L^p(\P, \csp^{-2 + \deltamin, \sigma}(\R^d))} & = 0, \\
      \lim_{\epsilon \downarrow 0}  \norm{\XAHeps - \xi_\epsilon - (\XAH - \xi)}_{L^p(\P,\csp^{-2 + 2 \deltamin, \sigma}(\R^d) )} & =0.
    \end{align*}

Let $F : \mathbb{R}^d \to \mathbb{R}$ be a compactly supported smooth function such that 
    $F(x) = - e^{2 x}$ if $\abs{x} \leq 2$.  
    For $N \in \mathbb{N}_0$, we set $W_N^{\varepsilon} \defby G_N \conv X^{\varepsilon}$ and 
    \begin{equation}\label{eq:definition_Y_eps_N}
      Y^{\varepsilon}_N \defby F(W_N^{\varepsilon})\big( \xi_{\varepsilon} - c_{\varepsilon} + 
      \abs{\nabla W^{\varepsilon}_N}^2 + \Delta W_N^{\varepsilon} \big).
    \end{equation}
	There exists a $Y_N \in \LL^{-1+\delta}$ such that for all $\sigma\in (0,\infty)$ and $p\in [1,\infty)$, 
    \begin{equation*}
      \lim_{\varepsilon \downarrow 0} \norm{Y^{\varepsilon}_N - Y_N}_{L^p(\P, \csp^{-1 + \deltamin, \sigma}(\R^d))} = 0.
    \end{equation*}
    Furthermore, 
    there exists an integer $\kAH = \kAH(\deltamin) \in \N$ such that for any $\sigma \in (0, \infty)$ and $p \in [1, \infty)$ we have
    \begin{equation}
    \label{eqn:def_a_AH}
    \aAH \defby \aAH(\deltamin,\sigma) \defby
      \sup_{N \in \N} 2^{-\kAH N} \norm{\YAH_N}_{\csp^{-1+\deltamin, \sigma}(\R^d)} \in L^p(\P).
    \end{equation}    
    Both $X$ and $Y_N$ are independent of the mollifier $\rho$.
    
\item 
\label{item:neumann_assump}
(Neumann assumption)
  We assume that $\delta \in (\frac12,1)$.
  For a bounded Lipschitz domain $U$, let $\tilde Y^{U}_\epsilon$ be the tempered distribution given by (see Definition~\ref{def:notation_integral_dot} for the notation $\int_{\partial U} f \nabla g \cdot \dd \bS$)
  \begin{equation*}
    \varphi \mapsto \int_{\partial U} \varphi \nabla (G_0 \conv \xi_{\epsilon}) \cdot \dd \bm{S}.
  \end{equation*}
  Then there exists a random variable $\tilde Y^{U}$ with values in $\cS'(\R^d)$, independent of the mollifier function $\rho$, such that
  \begin{align}\label{eq:tilde_Y_U_conv}
  \norm{\tilde Y^{U}_\epsilon - \tilde Y^{U} }_{B^{-2 + \delta}_{p, p}(\R^d)} \arroweps 0 \quad \mbox{in probability for all $p \in (2, \infty)$.} 
  \end{align}

  In addition, for every $p \in (8 d, \infty)$ and $q \in [1, \infty)$ we have 
  \begin{equation}\label{eq:Y_U_bound_assump}
    \sup_{L \in \mathbb{N}} L^{-\frac{1}{4}} \norm{\tilde{Y}^{U_L}}_{B_{p, p}^{-2 + \delta}(\mathbb{R}^d)} \in L^q(\P).
  \end{equation}
\item 
\label{item:ergodic_assump}
(Ergodic assumption)
  Recall that our probability space $\Omega$ is the space $\S'(\R^d)$ of tempered distributions.
  Then, we have maps $T_x: \Omega \rightarrow \Omega$ ($x \in \R^d$) of translations $  \omega \to \omega(\cdot - x)$.
  We assume the probability measure $\P$ to be translation invariant ($\P = \P \circ T_x$ for all $x$) and ergodic with respect to $(T_x)_{x \in \R^d}$.  
\end{enumerate}
\end{assumptions}

\begin{remark}\label{rem:gaussian_ergodic}
  It is well-known (e.g. \cite[Proposition~6.1]{rosati21}) that Assumption~\ref{assump:all_three_assump}~\ref{item:ergodic_assump} is satisfied for the Gaussian noise $\xi$ with 
  covariance $\mathbb{E}[\xi(x) \xi(y)] = \gamma(x-y)$ such that  
  \begin{equation*}
    \lim_{\abs{x} \to \infty} \gamma(x) = 0.
  \end{equation*}
  In particular, this assumption is satisfied for the white noise.
\end{remark}
\begin{remark}\label{rem:technical_estimate}
  The technical estimates \eqref{eqn:def_a_AH} and \eqref{eq:Y_U_bound_assump} are necessary only for the construction of the IDS (Section~\ref{subsec:IDS}, in particular Remark~\ref{rem:estimates_on_cZ}).
\end{remark}

\begin{remark}[About the Neumann assumptions]
\label{rem:neumann_construction}
From the discussion in Remark~\ref{rem:towards_construction_assumptions} we know that the $w^\epsilon$ and the $y^\epsilon$ as in Section~\ref{subsec:strategy_ang_techniques} are equal to $W^\epsilon_M$ and $Y^\epsilon_M$ for a well-chosen random variable $M$ (for its definition, see Definition~\ref{def:dirichlet_AH}).  

From the convergences of the stochastic terms $w^\epsilon$ and $y^\epsilon$ that we obtain from the construction assumption, the two first symmetric forms in \eqref{eq:three_syms} converge, as forms on $H^1(U)$. 
We need an additional assumption in order to deal with the third symmetric form in \eqref{eq:three_syms}, which is, 
\begin{align*}
v \mapsto \int_{\partial U} \cT(v^2) e^{2W_M^\epsilon} \nabla W_M^\epsilon \cdot \dd \bS. 
\end{align*}
Set 
\begin{align*}
\hatYAHeps_M &\defby G_M \conv (\XAHeps - \xi_\epsilon) + (G_M - G_0) \conv \xi_\epsilon. 
\end{align*}
We decompose the above symmetric form as the sum of the following symmetric forms 
\begin{align}
\label{eqn:two_neumann_symm_forms}
v\mapsto 
\int_{\partial U} \cT(v^2) e^{2W_M^\epsilon} \nabla (G_0 \conv \xi_{\epsilon}) \cdot \dd \bm{S},
\qquad v \mapsto \int_{\partial U} \cT(v^2) e^{2 W_M^\epsilon} \nabla \widehat Y_M^\epsilon  \cdot \dd \bS. 
\end{align}
Now, by the convergence $\xi_\epsilon \rightarrow \xi$ and the assumed convergences of the $X^\epsilon$ and $X^\epsilon - \xi_\epsilon$ from the construction assumption~\ref{assump:all_three_assump}~\ref{item:construction_assump}, we obtain by estimates on the Green's function, which we have put in Corollary~\ref{cor:estimates_of_G_N_and_H_N}, and estimates of derivatives in Besov spaces, see Lemma~\ref{lem:derivative_and_lifting_in_Besov},
\begin{align}
\label{eqn:W_M_AH_eps_convergence}
& \WAHeps_M \rightarrow \WAH_M
& & \mbox{ in } \cC^{\deltamin,\sigma}(\R^d) , \\
\label{eqn:G_0_xi_eps_convergence}
& G_0 * \xi_\epsilon \rightarrow G_0 * \xi
& &\mbox{ in } \cC^{ \deltamin, \sigma }(\R^d), \\
\label{eqn:G_M_X_AH_eps_min_xi_eps_convergence}
& G_M *(\XAHeps - \xi_\epsilon) \rightarrow G_M *(\XAH - \xi)
& &\mbox{ in } \cC^{2\deltamin,\sigma}(\R^d), \\
\label{eqn:G_M_min_G_0_xi_eps_convergence}
& (G_M - G_0) * \xi_\epsilon \rightarrow (G_M - G_0) * \xi
& &\mbox{ in } \cC^{1+\deltamin, \sigma }(\R^d).
\end{align}
In particular, by combining \eqref{eqn:G_M_X_AH_eps_min_xi_eps_convergence} and \eqref{eqn:G_M_min_G_0_xi_eps_convergence}, for $\delta \in (\frac12,1)$ we have 
\begin{align}
\label{eqn:Y_M_eps_conv}
\widehat Y_M^\epsilon \rightarrow \widehat Y_M \qquad \mbox{in } \cC^{2\delta,\sigma}(\R^d). 
\end{align}
In Section~\ref{sec:symmetric_forms} we will show in more detail that \eqref{eqn:W_M_AH_eps_convergence} and \eqref{eqn:Y_M_eps_conv} imply the convergence of the second symmetric form in \eqref{eqn:two_neumann_symm_forms}. 
This is basically because the function $e^{2W_M^\epsilon} \nabla \widehat Y_M^\epsilon$ converges in a space of positive regularity. 
Now for $\delta > 1$ the function $e^{2W_M^\epsilon} \nabla (G_0 * \xi_\epsilon)$ converges in a space of positive regularity. 
However for $\delta \in (\frac12, 1)$ the convergence in \eqref{eqn:G_0_xi_eps_convergence} does not seem to suffice, due to the integration over the $d-1$ dimensional boundary $\partial U$.
Let us elaborate on this.
We can rewrite the first symmetric form in \eqref{eqn:two_neumann_symm_forms} as  $v\mapsto \langle \tilde Y^U_\epsilon, e^{2\WAHeps_M} v^2 \rangle$, where $\tilde Y^U_\epsilon$ is the distribution 
as in \ref{assump:all_three_assump}~\ref{item:neumann_assump} given by 
  \begin{equation*}
    \varphi \mapsto \int_{\partial U} \varphi \nabla (G_0 \conv \xi_{\epsilon}) \cdot \dd \bm{S}. 
  \end{equation*} 
The distribution $\tilde Y^U_\epsilon$ could be formally interpreted as the product of $ \nabla (G_0 *\xi_\epsilon)$ with the distribution  $\delta_{\partial U}$, given by $\varphi \mapsto \int_{\partial U} \varphi \dd S$.
The distribution  $\delta_{\partial U}$ is of regularity $-1$ (e.g., for $I= [0,1]^{d-1} \times \{0\}$, $\delta_{I} $ is the tensor product of $\1_{[0,1]^{d-1}}$ and $\delta_0$, which are in $\cC^{0}(\R^{d-1})$ and $\cC^{-1}(\R)$, respectively; hence the tensor product is in $\cC^{-1}(\R^d)$, see \cite{SiUl09}).
Therefore  the product of these two, and thus $\tilde Y^U_\epsilon$, converges only  in the space of regularity $-2 + \deltamin$.
As the symmetric form is on $H^1$, by this reasoning, this convergence is sufficient only for $\delta>1$. 

When additionally imposing conditions on the boundary, called the Neumann assumption~\ref{assump:all_three_assump}~\ref{item:neumann_assump} it actually suffices to assume $\delta > \frac12$. 
For the details we refer to  Theorem~\ref{theorem:examples_bd_sym_with_weighted_input}, though we do mention it is basically due to the identity 
\begin{align*}
\int_{\partial U} \cT(v^2) e^{2W_M^\epsilon} \nabla (G_0 \conv \xi_{\epsilon}) \cdot \dd \bm{S} 
= \langle \tilde Y^U_\epsilon , \cR( e^{2\WAHeps_M} \cT(v)^2) \rangle,
\end{align*}
where $\cT$ is the trace operator and $\cR$  is the composition of the right inverse trace operator and the extension operator (see 
 Lemma~\ref{lem:extension_and_trace_sobolev}~\ref{item:trace_op_and_inverse}). 
\end{remark}

\begin{remark}
  To construct the natural Neumann Anderson Hamiltonian without the Neumann condition~\ref{assump:all_three_assump}~\ref{item:neumann_assump}, 
  we conjecture that ``boundary renormalisation'' is necessary.
  
  For instance, if $\xi$ is the 3D white noise, the recent work \cite{GeHa2021boundary} suggests that the
  operators associated to the symmetric forms
  \begin{equation*}
    (u, v) \mapsto \int_{U} \nabla u \cdot \nabla v - (\xi_{\epsilon} - c_{\epsilon}) u v
    + c'_{\epsilon} \int_{\partial U} u v \dd S
  \end{equation*}
  converge, where the constants $c'_{\epsilon}$ diverge logarithmically.
\end{remark}

\subsection{Construction assumption}\label{subsec:examples}
Here we give examples of the noise $\xi$ for which the construction assumption~\ref{assump:all_three_assump}~\ref{item:construction_assump} is satisfied. We recall that the heuristic to find $X$ is outlined in Remark~\ref{remark:finding_w_eps}.
\begin{example}
  \label{ex:2D_white_noise}
  Let $d=2$ and $\xi$ be the white noise on $\mathbb{R}^2$. 
Then $\xi \in L^p(\P, \csp^{-2 + \delta, \sigma}(\mathbb{R}^2) )$ 
\begin{calc}
Let us check this by some calculations: 
First of all, observe that 
\begin{align*}
\expect[ |\Delta_j \xi(x)|^2 ] 
= \|\cF^{-1} [ \chi(2^{-j} \cdot) ] (x-\cdot) \|_{L^2}^2
= \|\chi(2^{-j} \cdot) \|_{L^2}^2 
= 2^{jd}. 
\end{align*}
Now we have 
\begin{align*}
\expect[ \|\xi\|_{B_{p,p}^{r,\sigma}(\R^d)}^p ] 
& = \sum_{j=-1}^\infty 2^{rpj} \expect[\|w_\sigma \Delta_j \xi\|_{L^p}^p ] 
= \sum_{j=-1}^\infty 2^{rpj} 
\int_{\R^d} w_\sigma(x)^p 
\expect[| \Delta_j \xi(x)|^p ] \dd x \\
& \lesssim_p \sum_{j=-1}^\infty 2^{rpj} 
\int_{\R^d} w_\sigma(x)^p 
\expect[| \Delta_j \xi(x)|^2 ]^{\frac{p}{2}} \dd x  = \sum_{j=-1}^\infty 2^{rpj} 2^{jp \frac{d}{2}} \|w_\sigma\|_{L^p}^p, 
\end{align*}
which is finite for $r< - \frac{d}{2}$. With the Besov embedding Lemma~\ref{lem:weighted_besov_p_vs_infty}, \eqref{eqn:estimate_infty_by_p_q_besov} we have for all $\kappa >0$, 
\begin{align*}
\expect[ \|\xi\|_{\cC^{r-\frac{d}{p}-\kappa,\sigma}(\R^d)}^p ] 
\lesssim_{p,q,\kappa,r,\sigma}
\expect[ \|\xi\|_{B_{p,p}^{r,\sigma}(\R^d)}^p ], 
\end{align*}
and so $\xi \in L^p(\P,\cC^{-2+\delta,\sigma}(\R^d))$ for $\delta \in (0,1- \frac{d}{p})$ and all $\sigma \in (0,\infty)$. 
Differently said, for any $\delta \in (0,1)$, $\sigma \in (0,\infty)$ and $p \in (\frac{d}{1-\delta},\infty)$ we have 
\begin{align*}
\xi \in L^p(\P,\cC^{-2+\delta,\sigma}(\R^d)). 
\end{align*}
As by Hölder's inequality we have 
\begin{align*}
\|\xi\|_{L^q(\P,\cC^{-2+\delta,\sigma}(\R^d))} 
\le \|\xi\|_{L^p(\P,\cC^{-2+\delta,\sigma}(\R^d))}, 
\end{align*}
the above holds also for $p\in [1,\frac{d}{1-\delta}]$, i.e., for any $p\in [1,\infty)$. 
\end{calc}
and by the proof of Theorem~\ref{theorem:conv_mollifiers_in_weighted}, $\xi_\epsilon \rightarrow \xi$ in $ L^p(\P, \csp^{-2 + \delta, \sigma}(\mathbb{R}^2) )$ for all $p \in [1,\infty)$, 
  $\sigma \in (0, \infty)$ and $\delta \in (0, 1)$.
\begin{calc}
We check this by a simimilar calculation as above. 
By the above arguments, i.e., the Besov embedding and Hölder's inequality, we may restrict to showing 
\begin{align*}
\E[ \|\xi_\epsilon - \xi\|_{B_{p,p}^{r,\sigma}(\R^d)}^p ] \rightarrow 0. 
\end{align*}
Now 
\begin{align*}
\E[ \|\xi_\epsilon - \xi\|_{B_{p,p}^{r,\sigma}(\R^d)}^p ]
=  \sum_{j=-1}^\infty 2^{rpj} \expect[\|w_\sigma ( \rho_\epsilon *\Delta_j \xi - \Delta_j \xi)\|_{L^p}^p ] ,
\end{align*}
and $\expect[\|w_\sigma ( \rho_\epsilon *\Delta_j \xi - \Delta_j \xi)\|_{L^p}^p ] $ converges by Lebesgue's dominated convergence theorem, by means of the arguments of the proof of Theorem~\ref{theorem:conv_mollifiers_in_weighted}. 
\end{calc}
  We set 
  \begin{equation*}
    c_{\varepsilon} \defby \expect[\abs{\nabla G_0 \conv \xi_{\varepsilon}}^2(0)], 
    \qquad  \tau^{\varepsilon} \defby \abs{\nabla G_0 \conv \xi_{\varepsilon}}^2 - c_{\varepsilon}.
  \end{equation*}
In the following, all convergences hold for all for all $p \in [1,\infty)$, 
  $\sigma \in (0, \infty)$ and $\delta \in (0, 1)$, so we refrain from repeating this. 
  As shown in \cite[Proposition~1.3]{hairer_Labbe_2015} (recall \eqref{eqn:G_N_is_parametrix}), 
  \begin{calc}
  see the proof for the convergence in $L^1$ (which suffices because of the hypercontractivity of Wiener Chaoses), 
  \end{calc}
  we have $\tau^{\varepsilon} \to \tau$ in $L^p(\mathbb{P}, \csp^{-2 + 2 \delta, \sigma}(\mathbb{R}^2))$. 
  We then set 
  \begin{equation*}
    X^{\varepsilon} \defby \xi_{\varepsilon} + \tau^{\varepsilon}.
  \end{equation*}
  We define $W_N^{\varepsilon}$ and $Y_N^{\varepsilon}$ as in Assumption~\ref{assump:all_three_assump}.
  To compute $Y_N^{\varepsilon}$, we observe that 
  \begin{equation*}
    \abs{\nabla W_N^{\varepsilon}}^2 = \abs{\nabla G_0 \conv X^{\varepsilon}}^2 + 2 (\nabla G_0 \conv X^{\varepsilon}) \cdot (\nabla (G_N - G_0) \conv X^{\varepsilon}) + \abs{\nabla (G_N - G_0) \conv X^{\varepsilon}}^2. 
  \end{equation*}
  To simplify notation, we set $r_N^{\varepsilon} \defby (G_N - G_0) \conv X^{\varepsilon}$,
  and recall that the limit $r_N \defby (G_N - G_0) \conv X$ is smooth by Remark~\ref{rem:G_N_M}.
  Recalling \eqref{eq:laplacian_on_G_N}, we obtain 
  \begin{align*}
    Z_N^{\varepsilon} &\defby \xi_{\varepsilon} - c_{\varepsilon} + \abs{\nabla W_N^{\varepsilon}}^2 + \Delta W_N^{\varepsilon} \\
    &\phantom{\vcentcolon}= 
    2 (\nabla G_0 \conv \xi_{\varepsilon}) \cdot (\nabla G_0 \conv \tau^{\varepsilon}) + \abs{\nabla G_0 \conv \tau^{\varepsilon}}^2
    + 2 (\nabla G_0 \conv X^{\varepsilon}) \cdot \nabla r_N^{\varepsilon} + \abs{\nabla r_N^{\varepsilon}}^2 \\
    &\phantom{\defby} + \Delta r_N^{\varepsilon} + [\Delta (G_0 - G)] \conv X^{\varepsilon}. 
  \end{align*}
\begin{calc}
We have 
\begin{align*}
\xi_\epsilon = X^\epsilon - \tau^\epsilon, 
\end{align*}
and 
\begin{align*}
X^\epsilon & = \Delta (G_N - G_0) * X^\epsilon + \Delta (G_0 - G) * X^\epsilon - \Delta G_N * X^\epsilon \\
&= \Delta r_N^\epsilon + \Delta (G_0 - G) * X^\epsilon  - \Delta W_N^\epsilon, \\
   \abs{\nabla W_N^{\varepsilon}}^2 
& = 
   \abs{\nabla G_0 \conv X^{\varepsilon}}^2 + 2 (\nabla G_0 \conv X^{\varepsilon}) \cdot \nabla r_N^\epsilon + \abs{\nabla r_N^\epsilon}^2, 
\end{align*}
and 
\begin{align*}
\abs{\nabla G_0 \conv X^{\varepsilon}}^2 
& = \abs{\nabla G_0 \conv \xi_\epsilon}^2 
+ 2 (\nabla G_0 \conv \xi_\epsilon ) \cdot ( 
\nabla G_0 \conv \tau^\epsilon ) 
+ 
\abs{\nabla G_0 \conv \tau^\epsilon }^2 \\
& = \tau^\epsilon + c_\epsilon + 2 (\nabla G_0 \conv \xi_\epsilon ) \cdot ( 
\nabla G_0 \conv \tau^\epsilon ) 
+ 
\abs{\nabla G_0 \conv \tau^\epsilon }^2, 
\end{align*}
from which we deduce the above identity for $Z_N^\epsilon$: 
\begin{align*}
Z_N^\epsilon & = \xi_{\varepsilon} - c_{\varepsilon} + \abs{\nabla W_N^{\varepsilon}}^2 + \Delta W_N^{\varepsilon} \\
& = X^\epsilon -\tau^\epsilon - c_\epsilon + \abs{\nabla W_N^{\varepsilon}}^2 + \Delta W_N^{\varepsilon} \\
& = \Delta r_N^\epsilon + \Delta (G_0 - G) * X^\epsilon - \tau^\epsilon - c_\epsilon + \abs{\nabla W_N^{\varepsilon}}^2  \\
& = \Delta r_N^\epsilon + \Delta (G_0 - G) * X^\epsilon 
+ 2 (\nabla G_0 \conv \xi_\epsilon ) \cdot ( 
\nabla G_0 \conv \tau^\epsilon ) \\
& \quad + 
\abs{\nabla G_0 \conv \tau^\epsilon }^2
+  2 (\nabla G_0 \conv X^{\varepsilon}) \cdot \nabla r_N^\epsilon + \abs{\nabla r_N^\epsilon}^2 . 
\end{align*}
\end{calc}
  By the properties of $G_0$ (see Corollary~\ref{cor:estimates_of_G_N_and_H_N}),
  \begin{align*}
    \nabla G_0 \conv \xi_{\varepsilon} \to \nabla G_0 \conv \xi \quad &\text{in } L^p(\mathbb{P}, \csp^{-1 + \delta, \sigma}(\mathbb{R}^2)), \\
    \nabla G_0 \conv \tau^{\varepsilon} \to \nabla G_0 \conv \tau \quad &\text{in } L^p(\mathbb{P}, \csp^{-1 + 2 \delta, \sigma}(\mathbb{R}^2)).
  \end{align*}
  Since $\delta$ can be arbitrarily close to $1$, the products in the formula of $Z_N^{\varepsilon}$ have well-defined limits, and therefore
$Z_N^{\varepsilon}$ converges in $L^p(\P,\csp^{-1 + \delta, \sigma}(\mathbb{R}^2))$ as $\epsilon \downarrow 0$. 
Because $W_N^\epsilon = G_N * (\xi_\epsilon + \tau^\epsilon)$, by  Corollary~\ref{cor:estimates_of_G_N_and_H_N}
\begin{calc}
(because $\xi_\epsilon \rightarrow \xi$ in $\cC^{-2+\delta,\sigma}(\R^2)$ and $\tau^\epsilon \rightarrow
\tau $ in $L^p(\P, \cC^{-2+2\delta,\sigma}(\R^2))$)
\end{calc}
\begin{align*}
W_N^\epsilon \rightarrow W_N := G_N * (\xi + \tau) \quad \mbox{in } L^p(\P,\cC^{\delta,\sigma}(\R^2)), 
\end{align*}
and, as $F$ is Lipschitz, also $F(W_N^\epsilon)$ converges in $L^p(\P,\cC^{\delta}(\R^2))$ (using Lemma~\ref{lem:localization} to ``get rid of $\sigma$''). 
Therefore, for $\delta \in (\frac12,1)$, $Y_N^\epsilon =  F(W_N^\epsilon) Z_N^\epsilon$ converges in $L^p(\P,\csp^{-1 + \delta, \sigma}(\mathbb{R}^2))$ as $\epsilon \downarrow 0$ by the Bony type estimates for weighted Besov spaces 
\cite[Lemma 2.14]{gubinelli_hoffmanova_2019}. 

  Finally, the estimate  \eqref{eqn:def_a_AH} on $Y_N$ follows from the estimate 
  \begin{equation*}
    \norm{r_N}_{\csp^{\kappa, \sigma}(\mathbb{R}^2)} \lesssim_{\kappa, \sigma} 2^{(\kappa - \delta) N} 
    \norm{X}_{\csp^{-2 + \delta, \sigma}(\mathbb{R}^2)},
  \end{equation*}
  which is a consequence of Corollary~\ref{cor:estimates_of_G_N_and_H_N}. Therefore, Assumption~\ref{assump:all_three_assump} \ref{item:construction_assump} holds. 
\end{example}

\begin{example}
\label{ex:3D_white_noise}
  Let $d = 3$ and $\xi$ be the white noise on $\mathbb{R}^3$.  
Then $\xi \in L^p(\P, \csp^{-2 + \delta, \sigma}(\mathbb{R}^3) )$ 
and by the proof of Theorem~\ref{theorem:conv_mollifiers_in_weighted}, $\xi_\epsilon \rightarrow \xi$ in $ L^p(\P, \csp^{-2 + \delta, \sigma}(\mathbb{R}^3) )$ for all $p \in [1,\infty)$, 
  $\sigma \in (0, \infty)$ and $\delta \in (0, \frac12)$.
  In the following, all convergences hold for all for all $p \in [1,\infty)$, 
  $\sigma \in (0, \infty)$ and $\delta \in (0, \frac12)$, so we refrain from repeating this. 
From the convergences of the renormalized models associated with the 3D parabolic Anderson model \cite[Theorem~5.3]{hairer_multiplicative_2018} or by considering the convergences in  \cite[Theorem~2.38]{gubinelli_semilinear_2020}, we deduce the existence of $\tau_1,\tau_2,\tau_3,\tau_4$ and the following convergences as $\varepsilon \downarrow 0$,
  \begin{align*}
    \tau_1^{\varepsilon} &\defby \abs{\nabla G_0 \conv \xi_{\varepsilon}}^2 - \expect[\abs{\nabla G_0 \conv \xi_{\varepsilon}}^2(0)] \to \tau_1 
    & & \text{in } L^p(\mathbb{P}, \csp^{-2+2 \delta, \sigma}(\mathbb{R}^3)), \\
    \tau_2^{\varepsilon} &\defby (\nabla G_0 \conv \xi_{\varepsilon}) \cdot (\nabla G_0 \conv \tau_1^{\varepsilon}) \to \tau_2 
    & & \text{in } L^p(\mathbb{P}, \csp^{-2 + 3 \delta, \sigma}(\mathbb{R}^3)), \\
    \tau_3^{\varepsilon} &\defby \abs{\nabla G_0 \conv \tau_1^{\varepsilon}}^2 - \expect[\abs{\nabla G_0 \conv \tau_1^{\varepsilon}}^2(0)] \to \tau_3 
    & & \text{in } L^p(\mathbb{P}, \csp^{-2 + 4 \delta, \sigma}(\mathbb{R}^3)), \\
    \tau_4^{\varepsilon} &\defby (\nabla G_0 \conv \xi_{\varepsilon}) \cdot (\nabla G_0 \conv \tau_2^{\varepsilon}) \to \tau_4 
    & & \text{in } L^p(\mathbb{P}, \csp^{-1 + \delta, \sigma}(\mathbb{R}^3)). 
  \end{align*}
\begin{calc}
For the derivation using \cite{gubinelli_semilinear_2020}, observe that $(-\Delta)^{-1}$, $(1-\Delta)^{-1}$ and $G_0 *$ are basically the same operations, in the sense that the difference of them will converge in a space of besser regularity, therefore morally, we have the same identities for the $X_\epsilon$, $X_\epsilon^\zzone, X_\epsilon^\zztwo,X_\epsilon^\zzthree, X_\epsilon^\zzfour$ as in \cite[Theorem 2.38]{gubinelli_semilinear_2020}
\begin{align*}
X_\epsilon & = (-\Delta)^{-1} \xi_\epsilon = G_0 * \xi_\epsilon, \\
\tau_1^\epsilon & = (1-\Delta) X_\epsilon^\zzone, \qquad G_0 * \tau_1^\epsilon = X_\epsilon^\zzone, \\
\tau_2^\epsilon & = (\nabla X_\epsilon) \cdot ( \nabla X_\epsilon^\zzone) \\
& = (1-\Delta) X_\epsilon^\zztwo, \\
\tau_3^\epsilon 
& = (1-\Delta) X_\epsilon^\zzfour, \\
\tau_4^\epsilon 
& = (1-\Delta) X_\epsilon^\zzthree. 
\end{align*}
\end{calc}  
  We set  
  \begin{equation*}
    c_{\varepsilon} \defby \expect[\abs{\nabla G_0 \conv \xi_{\varepsilon}}^2(0)] +   \expect[\abs{\nabla G_0 \conv \tau_1^{\varepsilon}}^2(0)],
  \end{equation*}
  $X^{\varepsilon} \defby \xi_{\varepsilon} + \tau_1^{\varepsilon} + 2 \tau_2^{\varepsilon} + \tau_3^{\varepsilon}$
  and $r_N^{\varepsilon} \defby (G_N - G_0) \conv X^{\varepsilon}$.
  Similar to the 2D case, we obtain
  \begin{align*}
    Z_N^{\varepsilon} &\defby \xi_{\varepsilon} - c_{\varepsilon} + \abs{\nabla W_N^{\varepsilon}}^2 + \Delta W_N^{\varepsilon} \\
    &\phantom{\vcentcolon}= 4 \tau_4^{\varepsilon} + 2 (\nabla G_0 \conv \xi_{\varepsilon}) \cdot (\nabla G_0 \conv \tau_3^{\varepsilon}) 
    + 2 (\nabla G_0 \conv \tau_1^{\varepsilon}) \cdot (\nabla G_0 \conv (2 \tau_2^{\varepsilon} + \tau_3^{\varepsilon})) \\
    & \qquad + 2 (\nabla G_0 \conv X^{\varepsilon}) \cdot \nabla r_N^{\varepsilon}  
    + \abs{\nabla r_N^{\varepsilon}}^2 + \Delta r_N^\epsilon+  [\Delta (G_0 - G)] \conv X^{\varepsilon} \\
    & \qquad + 
\abs{\nabla G_0 \conv ( 2\tau_2^\epsilon + \tau_3^\epsilon) }^2
    ,
  \end{align*}
\begin{calc}
(Indeed, 
\begin{align*}
|\nabla W_N^\epsilon|^2 
& = |\nabla G_0 * X^\epsilon|^2 + 2( \nabla G_0 \conv X^\epsilon) \cdot  \nabla r_N + |\nabla r_N|^2
\end{align*}
and 
\begin{align*}
\abs{\nabla G_0 \conv X^{\varepsilon}}^2 
& = \abs{\nabla G_0 \conv \xi_\epsilon}^2 
+ 2 (\nabla G_0 \conv \xi_\epsilon ) \cdot ( 
\nabla G_0 \conv (\tau_1^\epsilon + 2\tau_2^\epsilon + \tau_3^\epsilon) ) 
+ 
\abs{\nabla G_0 \conv (\tau_1^\epsilon + 2\tau_2^\epsilon + \tau_3^\epsilon) }^2 \\
& = \abs{\nabla G_0 \conv \xi_\epsilon}^2 
+ \abs{\nabla G_0 \conv \tau_1^\epsilon}^2 
+ 2 (\nabla G_0 \conv \xi_\epsilon ) \cdot ( 
\nabla G_0 \conv (\tau_1^\epsilon + 2\tau_2^\epsilon + \tau_3^\epsilon) ) \\
& \quad + 2 ( \nabla G_0 \conv \tau_1^\epsilon) \cdot ( \nabla G_0 \conv ( 2\tau_2^\epsilon + \tau_3^\epsilon) )  + 
\abs{\nabla G_0 \conv ( 2\tau_2^\epsilon + \tau_3^\epsilon) }^2 \\
& = \tau_1^\epsilon +2 \tau_2^\epsilon + \tau_3^\epsilon + c_\epsilon
+ 4 \tau_4^\epsilon
+ 2 (\nabla G_0 \conv \xi_\epsilon ) \cdot ( 
\nabla G_0 \conv \tau_3^\epsilon ) \\
& \quad  + 2 ( \nabla G_0 \conv \tau_1^\epsilon) \cdot ( \nabla G_0 \conv ( 2\tau_2^\epsilon + \tau_3^\epsilon) )  + 
\abs{\nabla G_0 \conv ( 2\tau_2^\epsilon + \tau_3^\epsilon) }^2, 
\end{align*}
so that by using that 
\begin{align*}
\xi_\epsilon 
&= \Delta r_N^\epsilon + \Delta (G_0 - G) * X^\epsilon  - \Delta W_N^\epsilon  - \tau_1^\epsilon - 2\tau_2^\epsilon - \tau_3^\epsilon, 
\end{align*}
we deduce the identity in a similar way:
\begin{align*}
& \xi_{\varepsilon} - c_{\varepsilon} + \abs{\nabla W_N^{\varepsilon}}^2 + \Delta W_N^{\varepsilon} \\
&= 
\Delta r_N^\epsilon + \Delta (G_0 - G) * X^\epsilon  - \tau_1^\epsilon - 2\tau_2^\epsilon - \tau_3^\epsilon - c_\epsilon \\
& \qquad + |\nabla G_0 * X^\epsilon|^2 + 2( \nabla G_0 \conv X^\epsilon) \cdot  \nabla r_N + |\nabla r_N|^2 \\
&= 
\Delta r_N^\epsilon + \Delta (G_0 - G) * X^\epsilon   
+ 2( \nabla G_0 \conv X^\epsilon) \cdot  \nabla r_N + |\nabla r_N|^2
\\
& \qquad + 4 \tau_4^\epsilon
+ 2 (\nabla G_0 \conv \xi_\epsilon ) \cdot ( 
\nabla G_0 \conv \tau_3^\epsilon ) \\
& \qquad  + 2 ( \nabla G_0 \conv \tau_1^\epsilon) \cdot ( \nabla G_0 \conv ( 2\tau_2^\epsilon + \tau_3^\epsilon) )  + 
\abs{\nabla G_0 \conv ( 2\tau_2^\epsilon + \tau_3^\epsilon) }^2, 
\end{align*}
) 
\end{calc}
  from which we see that $Z_N^\epsilon$ converges in $L^p(\mathbb{P}, \csp^{-1 + \delta, \sigma}(\mathbb{R}^3))$. 
  As for the 2D case, $F(W_N^\epsilon)$ converges in $L^p(\P,\cC^\delta(\R^3))$. 
  However, now $\delta \in (0,\frac12)$ and so $-1+2\delta$ is negative. 
  Hence we cannot guarantee the convergence of $Y_N^\epsilon = F(W_N^\epsilon) Z_N^\epsilon$ in the same way as for the 2D case (recall the Bony type estimates for weighted Besov spaces 
\cite[Lemma 2.14]{gubinelli_hoffmanova_2019}). 
  However, with some techniques from paracontrolled distributions \cite{gubinelli_paracontrolled_2015}, we can guarantee the convergence. 
We describe the method without proving all the details. (The symbols ``$\prec, \succ, \circ$'' are here replaced by ``$\varolessthan, \varogreaterthan, \varodot$'' to align with the recent notation.)

From the above, and using Lemma~\ref{lem:derivative_and_lifting_in_Besov}~\ref{item:derivative_besov} and Corollary~\ref{cor:estimates_of_G_N_and_H_N}, observe that \begin{calc} $X^\epsilon $ converges in $L^p(\P,\cC^{-2+2\delta,\sigma}(\R^3)$ and \end{calc} $\nabla(G_0 *\tau_3^\epsilon +r_N^\epsilon)$ converges in $L^p(\P,\cC^{-1+4\delta,\sigma}(\R^3))$ \begin{calc}
 (by playing with $s$ and $r$ in Corollary~\ref{cor:estimates_of_G_N_and_H_N} for $r_N^\epsilon$)
\end{calc}
and $\nabla (G_0 * \xi_\epsilon)$ converges in $L^p(\P,\cC^{-1+\delta,\sigma}(\R^3))$
so that 
\begin{align}
\label{eqn:def_V_N_eps_3D}
4 \tau_4^{\varepsilon} + 2 (\nabla G_0 \conv \xi_{\varepsilon}) 
    \cdot \nabla (G_0 \conv \tau_3^{\varepsilon} + r_N^{\varepsilon}) 
\end{align}
converges in $L^p(\P,\cC^{-1+\delta,\sigma}(\R^3))$ and 
\begin{align*}
\hat{Z}_N^\epsilon \defby Z_N^\epsilon - \Big( 4 \tau_4^{\varepsilon} + 2 (\nabla G_0 \conv \xi_{\varepsilon}) 
    \cdot \nabla (G_0 \conv \tau_3^{\varepsilon} + r_N^{\varepsilon})  \Big)
\end{align*}
converges in $L^p(\mathbb{P}, \csp^{-1 + 2 \delta, \sigma}(\mathbb{R}^3))$. 
Therefore the product $F(W_N^\epsilon) \hat{Z}_N^{\varepsilon}$ converges in the space $L^p(\mathbb{P}, \csp^{-1 + 2 \delta, \sigma}(\mathbb{R}^3))$ for $\delta \in (\frac13, \frac12)$. 

Hence, we are left to consider the product $F(W_N^{\varepsilon}) (Z_N^\epsilon - \hat Z_N^\epsilon)$. 
Now, we finish by showing 
\begin{enumerate}
\item $F(W_N^\epsilon) \tau_4^\epsilon$ converges in $L^p(\P,\cC^{-1+2\delta,\sigma}(\R^3))$, 
\item $F(W_N^\epsilon) (\nabla G_0 * \xi_\epsilon)$ converges in $L^p(\P,\cC^{-1+2\delta,\sigma}(\R^3,\R^3))$. 
\end{enumerate}
Indeed, if (b) holds, then $F(W_N^\epsilon) (\nabla G_0 * \xi_\epsilon) \cdot \nabla (G_0 \conv \tau_3^{\varepsilon} + r_N^{\varepsilon})  $ converges in $L^p(\P,\cC^{-1+2\delta,\sigma}(\R^3))$ because $-2 +6 \delta>0$ for $\delta\in (\frac13,\frac12)$. 
\begin{calc}
For the previous, recall that $\nabla(G_0 *\tau_3^\epsilon +r_N^\epsilon)$ converges in $L^p(\P,\cC^{-1+4\delta,\sigma}(\R^3))$.
\end{calc}

Let $Q^\epsilon$ be either $\tau_4^\epsilon$ or $\nabla G_0 * \xi_\epsilon$ for the moment. 
As the convergences of the paraproducts $\varolessthan$ and $\varogreaterthan$ are guaranteed, it suffices to show the convergence of the resonant product $F(W_N^\epsilon) \varodot Q^\epsilon$. 
  By the paralinearization lemma \cite[Lemma~2.6]{gubinelli_paracontrolled_2015}, 
  \begin{equation*}
    F(W_N^{\varepsilon}) = 
    F'(W_N^{\varepsilon}) \varolessthan W_0^{\varepsilon} 
    +    F'(W_N^{\varepsilon}) \varolessthan (W_N^{\varepsilon}  - W_0^\epsilon)
    + R_N^{\varepsilon},
  \end{equation*}
  where $\norm{R_N^{\varepsilon}}_{\csp^{2 \delta, \sigma}} \lesssim_{\delta, \sigma} 1 + \norm{W_N^{\varepsilon}}_{\csp^{\delta, \sigma}}$.  
    Since $W_N - W_0$ is smooth, it suffices to consider the resonant product between $F'(W_N^{\varepsilon}) \varolessthan W_0^{\varepsilon}$ and $Q^{\varepsilon}$. 
By the so-called Commutator lemma \cite[Lemma~2.16]{gubinelli_hoffmanova_2019}
\begin{align*}
&\|    
    (F'(W_N^{\varepsilon}) \varolessthan W_0^{\varepsilon}) \varodot Q^\epsilon
    -  F'(W_N^{\varepsilon}) (W_0^{\varepsilon} \varodot Q^\epsilon) 
\|_{\cC^{-1+3\delta,2\sigma}(\R^3)} \\
& \lesssim \| F'(W_N^{\varepsilon}) \|_{\cC^\delta(\R^3)} 
\|W_0^\epsilon\|_{\cC^{\delta,\sigma}(\R^3)} 
\|Q^\epsilon\|_{\cC^{-1+\delta,\sigma}(\R^3)},
\end{align*}
from which can deduce that for the convergence of $F(W_N^\epsilon) \varodot Q^\epsilon$ it suffices to show convergence of $ F'(W_N^{\varepsilon}) (W_0^{\varepsilon} \varodot Q^\epsilon)$. 

\underline{Proof of (b)}
For $Q^\epsilon = \nabla (G_0 * \xi_\epsilon)$ we have $W_0^\epsilon \varodot \nabla (G_0 * \xi_\epsilon)
= \frac12 \nabla ((G_0 *\xi_\epsilon) \varodot (G_0 *\xi_\epsilon)) + G_0 * (X^\epsilon -\xi_\epsilon) \varodot \nabla(G_0 *\xi_\epsilon)$ and so the convergence follows because $(G_0 * \xi_\epsilon) \varodot (G_0 * \xi_\epsilon)$ and $G_0 * (X^\epsilon - \xi_\epsilon)$ converge in $L^p(\P,\cC^{2\delta,\sigma}(\R^3))$. 

\underline{Proof of (a)}
For $Q^\epsilon = \tau_4^\epsilon$, we have the following identity
\begin{align*}
(G_0 \conv \xi_{\varepsilon}) \varodot \tau_4^\epsilon
& = 
(\nabla G_0 \conv \xi_{\varepsilon}) \varodot (\nabla G_0 \conv \tau_4^\epsilon) \\
& \qquad - \nabla \cdot [(G_0 \conv \xi_{\varepsilon}) \varodot \nabla (G_0 \conv \tau_4^\epsilon)] 
+ (G_0 \conv \xi_{\varepsilon}) \varodot [\Delta(G_0 - G) \conv \tau_4^\epsilon]. 
\end{align*}
The convergence of the two terms in the second line follow rather straightforwardly (recall \eqref{eq:Delta_G_N_min_G}), whereas the first term on the right-hand side (on the first line) 
converges as is shown in the proof of  \cite[Theorem~2.38]{gubinelli_semilinear_2020}.

  As in the 2D case, the estimate  \eqref{eqn:def_a_AH}  on $Y_N$ follows from Corollary~\ref{cor:estimates_of_G_N_and_H_N}. 
\end{example}

The above examples are specific cases for which the construction condition~\ref{assump:all_three_assump}~\ref{item:construction_assump} is valid. 
This condition is however valid under very general assumptions in terms of regularity structures, as we show in the next theorem. 
Namely, we consider the regularity structure associated with the generalized parabolic Anderson model 
  \begin{equation}\label{eq:gpam}
    \partial_0 u = \Delta u + \sum_{i, j=1}^d g_{i, j}(u) \partial_i u \partial_j u
    + \sum_{i=1}^d h_i(u) \partial_i u + k(u) + f(u) \xi,
  \end{equation}
or equivalently, the following elliptic equation, as it leads to the same regularity structure
  \begin{equation*}
    \Delta u =  \sum_{i, j=1}^d g_{i, j}(u) \partial_i u \partial_j u
    + \sum_{i=1}^d h_i(u) \partial_i u + k(u) + f(u) \xi,
  \end{equation*}
as described in \cite{bruned_2019} and we also consider the renormalization approach considered therein called the BPHZ renormalization, which is further developed in the subsequent works \cite{chandra2018analytic,hairer2023bphztheoremregularitystructures, bailleul2024randommodelsregularityintegrabilitystructures}.

In the following theorem, the assumption we rely on, namely Assumption~\ref{assump:convergence_of_BPHZ}, is formulated in the language of regularity structures, based on the language discussed preceding it in Appendix~\ref{sec:review_of_reg_str}. 
This assumption basically means that the solution theory of the generalized parabolic Anderson model can be developed.  
Hence, the following theorem states that if this is the case, then the construction condition~\ref{assump:all_three_assump}~\ref{item:construction_assump} is valid as well.

\begin{theorem}\label{thm:convergence_X_Y_BPHZ}
  Impose Assumption~\ref{assump:convergence_of_BPHZ}, which assumes that the BPHZ renormalization models converge and a few probabilistic estimates hold. 
By setting  $X^\epsilon \defby X^{\rmodel^{\BPHZ, \epsilon}}$ and $X \defby X^{\rmodel^{\BPHZ}}$ as in  Definition~\ref{def:def_of_X_W_N_Y_N_for_model}
and letting $c_\epsilon$ be defined by \eqref{eq:def_of_c_epsilon}, 
the construction condition~\ref{assump:all_three_assump}~\ref{item:construction_assump} is valid. 
\end{theorem}

\begin{examples}
\label{examples:cov_gamma}
The work \cite{chandra2018analytic}, see especially Theorem 2.31 and Theorem 2.34 therein, gives
 conditions of the noise $\xi$ under which Assumption~\ref{assump:convergence_of_BPHZ} holds.
It is worth observing that Assumption~\ref{assump:convergence_of_BPHZ} holds for the 2D and the 3D white noise, and
the Gaussian noise $\xi$ whose covariance is formally given by
\begin{equation*}
  \expect[\xi(x) \xi(y)] = \gamma(x-y),
\end{equation*}
where $\gamma: \R^d \setminus \{0\} \to [0, \infty)$ is smooth and bounded away from $0$ and
for some $\delta \in (0, 1)$ we have
\begin{equation*}
  \sup_{\substack{k \in \N_0^d, \\ \abs{k} \leq 6d}} \sup_{x \in B(0,1) \setminus \{0\}}
  \abs{\partial^k \gamma (x)} \abs{x}^{\min\{4, d\} - \delta+|k|} < \infty,
\end{equation*}
see \cite[Theorem 2.15]{chandra2018analytic}. For example one could take $\gamma$ to be given by
\begin{align*}
\gamma(x) = c |x|^{-\alpha}
\end{align*}
for some $c\in (0,\infty)$ and $\alpha \in (0,\min\{d,4\})$.
\begin{calc}
It would be interesting to know, whether the assumptions of \cite{chandra2018analytic} cover, for instance,
the Gaussian noise $\xi$
whose covariance is formally given by
\begin{equation*}
  \expect[\xi(x) \xi(y)] = c \abs{x_1 - y_1}^{-\alpha_1} \cdots \abs{x_d - y_d}^{-\alpha_d},
  \quad c \in (0, \infty),
\end{equation*}
where $\alpha_1, \ldots, \alpha_d \in (0, 1)$ with $\alpha_1 + \cdots + \alpha_d < 4$. (Because the corresponding $\gamma$ does not satisfy the above condition.)
Considering the Gaussian noise with this covariance seems interesting in view of \cite[Theorem 1.3]{chen2014}.
\end{calc}
\end{examples}

\subsection{Neumann assumption}
\label{subsec:stoch_terms_neumann}

In this section we discuss sufficient conditions for the Neumann condition~\ref{assump:all_three_assump}~\ref{item:neumann_assump} (Proposition~\ref{prop:boundary_nabla_xi} and Lemma~\ref{lem:example_of_boundary_xi}) and examples which satisfy the Neumann condition (Example~\ref{example:example_of_gamma} and Example~\ref{example:2D_WN_satisf_Neumann_cond}).

\begin{proposition}\label{prop:boundary_nabla_xi}
Let $\delta \in (1/2, 1)$.
Suppose that for any bounded Lipschitz domain $U$, there exists a $\xi_{U}$ such that 
\begin{align*}
\norm{\indic_{U} \xi_{\varepsilon} - \xi^U}_{B^{-2 + \delta}_{p, p}(\R^d)} \arroweps 0 \quad \mbox{in probability for all $p \in (2, \infty)$.} 
\end{align*}
Then, the convergence \eqref{eq:tilde_Y_U_conv} holds. Furthermore, for any $\sigma \in (\frac{d}{p}, \infty)$ we have the bound 
\begin{equation*}
  \norm{\tilde{Y}^{U_L}}_{B_{p, p}^{-2+\delta}(\R^d)}
  \lesssim_{\delta, p, \sigma, U} L^{2 \sigma} \norm{\xi}_{\csp^{-2 + \delta, \sigma}(\R^d)}
  + \norm{\xi^{U_L}}_{B_{p, p}^{-2 + \delta}(\R^d)} \quad \text{for all }L \geq 1.
\end{equation*}
\end{proposition}
\begin{proof}
  The proof is based on the duality \cite[Theorem~2.17]{sawano2018theory}.
  By the integration by parts formula (Lemma~\ref{lem:integration_by_parts}),
  \begin{equation}\label{eq:integration_by_parts_for_xi}
\inp{\tilde Y^{U_L}_\epsilon}{\varphi } =
    \int_{\partial U_L} \varphi \nabla (G_0 \conv \xi_{\epsilon}) \cdot \dd \bm{S}
    = \int_{U_L} \nabla \varphi \cdot \nabla (G_0 \conv \xi_{\epsilon}) - \int_{U_L} \varphi \Delta (G_0 \conv \xi_{\epsilon}).
  \end{equation}
  We first consider the first term on the right-hand side of \eqref{eq:integration_by_parts_for_xi}.
  Let $\phi$ be a smooth function on $\R^d$ such that $\phi = 1$ on a neighborhood $U$.
The map
  \begin{equation}\label{eq:first_term_of_boundary_xi}
    \cS(\R^d) \rightarrow \R, \quad  \varphi \mapsto \inprd{\indic_{U_L} \nabla \varphi}{\phi(L^{-1} \cdot) \nabla (G_0 \conv \xi)}
  \end{equation}
  is well-defined, is independent of $\phi$ and is an element of $B_{p, p}^{-2 + \delta}(\R^d)$.
  Indeed, if $q \in (1, 2)$ is such that $p^{-1} + q^{-1} = 1$, by the duality
  \cite[Theorem 2.17]{sawano2018theory},
  \begin{equation*}
    \abs{\inprd{\indic_{U_L} \nabla \varphi}{\phi(L^{-1} \cdot) \nabla (G_0 \conv \xi)}}
    \leq \norm{\indic_{U_L} \nabla \varphi}_{B_{q, q}^{1 - \delta}(\R^d)} \norm{\phi(L^{-1} \cdot)
    \nabla (G_0 \conv \xi)}_{B_{p, p}^{-1 + \delta}(\R^d)}.
  \end{equation*}
  By Lemma~\ref{lem:derivative_and_lifting_in_Besov} and Lemma~\ref{lem:scaling_of_iota} (see Definition~\ref{def:functional_inequalities_constants} for $C_{\operatorname{Mult}}$), as $1-\delta <  \frac12 <  \frac1q$,
  \begin{equation*}
    \norm{\indic_{U_L} \nabla \varphi}_{B_{q, q}^{1 - \delta}(\R^d)}
    \lesssim_{q, \delta} C_{\operatorname{Mult}}^{U_L}[W_q^{1-\delta}] \norm{\nabla \varphi}_{B_{q, q}^{1 - \delta}(\R^d)}
    \lesssim_{q, \delta, U} \norm{\varphi}_{B_{q, q}^{2 - \delta}(\R^d)}.
  \end{equation*}
  By Lemma~\ref{lem:weighted_besov_p_vs_infty}  (remember that we have $p \sigma>d$) and Lemma~\ref{lem:localization},
  \begin{equation*}
    \norm{\phi(L^{-1} \cdot) \nabla (G_0 \conv \xi)}_{B_{p, p}^{-1 + \delta}(\R^d)}
    \lesssim_{p, \delta, \sigma} L^{2 \sigma} \norm{\nabla (G_0 \conv \xi)}_{\csp^{-1 + \delta, \sigma}(\R^d)}.
  \end{equation*}
  By Lemma~\ref{lem:derivative_and_lifting_in_Besov} and Corollary~\ref{cor:estimates_of_G_N_and_H_N},
  \begin{equation*}
    \norm{\nabla (G_0 \conv \xi)}_{\csp^{-1 + \delta, \sigma}(\R^d)}
    \lesssim_{\delta, \sigma} \norm{\xi}_{\csp^{-2 + \delta, \sigma}(\R^d)}.
  \end{equation*}
  Therefore, the distribution defined by \eqref{eq:first_term_of_boundary_xi} belongs to
  the dual space of $B_{q, q}^{2 - \delta}(\R^d)$, which is identified with $B_{p, p}^{-2 + \delta}(\R^d)$,
  and its norm in $B_{p, p}^{-2 + \delta}(\R^d)$ is bounded by a multiple of 
  \begin{equation}
  \label{eqn:L_2sigma_xi_norm}
    L^{2 \sigma} \norm{\xi}_{\csp^{-2 + \delta, \sigma}(\R^d)}.
  \end{equation}
  Now it is easy to see that this distribution is the limit of the first term of the right-hand side of \eqref{eq:integration_by_parts_for_xi}  (as $\norm{\xi_\epsilon - \xi}_{\cC^{-2+\delta,\sigma}(\R^d)} \rightarrow 0$).

Let us turn to the second term of the right-hand side of \eqref{eq:integration_by_parts_for_xi}.
  Note that
  \begin{equation*}
    \int_{U_L} \varphi \Delta (G_0 \conv \xi_{\epsilon}) = - \int_{U_L} \varphi \xi_{\epsilon} + \int_{U_L} \varphi [\Delta(G_0 - G)] \conv \xi_{\epsilon}.
  \end{equation*}
  The first term is equal to $\inprd{\indic_{U_L} \xi_{\epsilon}}{\varphi}$, whose convergence is guaranteed by the assumption.
  The second term converges to
  \begin{equation*}
    \int_{U_L} \varphi [\Delta(G_0 - G)] \conv \xi, 
  \end{equation*}
  which equals 
  \begin{align}
  \label{eqn:Delta_G_0_min_G_def}
  \langle \1_{U_L} \varphi, \phi(L^{-1} \cdot) [\Delta (G_0 - G)] * \xi \rangle. 
  \end{align}
  Now, as $[\Delta(G_0 - G)] \conv \xi$ is the smooth function $\cF^{-1}(\check \chi) * \xi = \Delta_{\le -1} \xi$ by \eqref{eq:Delta_G_N_min_G} (for $\Delta_{\le J}$ see Definition~\ref{def:Delta_le_LP_block}) we obtain that the distribution defined by \eqref{eqn:Delta_G_0_min_G_def} has its norm in $B_{p,p}^{-2+\delta}(\R^d)$ also bounded by a multiple of \eqref{eqn:L_2sigma_xi_norm} by using a similar argument as above and Lemma~\ref{lem:fourier_cutoff}
to obtain the estimate 
\begin{align*}
\| [\Delta (G_0 - G) ] * \xi \|_{\cC^{-1+\delta,\sigma}(\R^d)} 
\lesssim_\sigma 
\| \xi \|_{\cC^{-2+\delta,\sigma}(\R^d)}. 
\end{align*}
\end{proof}

\begin{lemma}\label{lem:convergence_of_indic_U_xi_epsilon}
  We assume $\delta \in (\frac{1}{2}, 1)$  and that there exists a
  $\delta' \in (0, 1)$ such that, for each $p \in (1, \infty)$,
  there exist a constant $C^{\bdry}_p \in (0, \infty)$ and a map $\bm{\epsilon}_p^{\bdry}: (0, 1)^3 \times \R^d \to (0, \infty)$
  with the following properties.
  \begin{enumerate}[label={\normalfont(\roman*)}]
    \item 
	\label{item:boundary1}    
    One has
    \begin{equation}\label{eq:bound_of_bm_epsilon_boundary}
      \sup_{\epsilon_1, \epsilon_2, \lambda \in (0, 1), x\in \R^d}\bm{\epsilon}_p^{\bdry}(\epsilon_1, \epsilon_2; \lambda,x) < \infty
    \end{equation}
    and for each fixed $\lambda\in (0,1)$ and $x\in \R^d$, one has $\lim_{\epsilon_1, \epsilon_2 \downarrow 0} \bm{\epsilon}_p^{\bdry}(\epsilon_1, \epsilon_2; \lambda,x) = 0$ (by which we mean $\lim_{\epsilon\downarrow 0} \sup_{\epsilon_1,\epsilon_2 \in (0, \epsilon)} \bm{\epsilon}_p^{\bdry}(\epsilon_1, \epsilon_2; \lambda,x) = 0$). 
    \item 
    \label{item:boundary2} 
    For every $\epsilon_1, \epsilon_2, \lambda \in (0, 1)$, $q \in (1, \infty)$, bounded Lipschitz domain $U$,
    $x \in B(U, 1)$ and $\phi \in C^2(\R^d)$ with
    $\norm{\phi}_{\csp^2(\R^d)} \leq 1$ and $\supp(\phi) \subseteq B(0, 1)$, one has, with $\phi^{\lambda}_x \defby \lambda^{-d} \phi(\lambda^{-1}(\cdot - x))$, 
    \begin{align}\label{eq:bound_xi_epsilon_on_U}
      \expect[ \abs{\inprd{\xi_{\epsilon_1}}{\indic_U \phi^{\lambda}_x}}^q] 
      & \leq C^{\bdry}_q \lambda^{(-2+\delta + \delta')q}, \\
\label{eq:cauchy_bm_epsilon_boundary}
      \expect[ \abs{\inp{\xi_{\epsilon_1} - \xi_{\epsilon_2}}{\indic_U \phi^{\lambda}_x}}^q]
      & \leq \bm{\epsilon}_q^{\bdry}(\epsilon_1, \epsilon_2; \lambda,x) \lambda^{(-2 + \delta + \delta')q}.
    \end{align}
  \end{enumerate}
  Then there exists a $\xi^U$ with values in $\cC^{-2+\delta}(\R^d)$ and $r \in (0, \infty)$ 
  such that for all $p \in (\frac{d}{\delta'}+1,\infty)$ and $q \in [p, \infty)$ the random variables $\indic_U \xi_{\varepsilon}$ 
  converge to $\xi^U$ in $L^{\tymchange{q}}(\P, B_{p,p}^{-2+\delta}(\R^d))$ as $\varepsilon \downarrow 0$ and 
    \begin{equation}\label{eq:bound_of_xi_U}
      \expect[ \norm{\xi^U}_{B_{p,p}^{-2 + \delta}(\R^d)}^{\tymchange{q}}] \lesssim_{p, q, \delta, \delta'}
      \abs{B(U, r)}^{\tymchange{\frac{q}{p}}} C^{\bdry}_{q}.
    \end{equation}
Consequently, by Proposition~\ref{prop:boundary_nabla_xi}, the Neumann condition~\ref{assump:all_three_assump}~\ref{item:neumann_assump} holds. 
\end{lemma}
\begin{proof}
  To simplify notation, we set $\eta \defby \indic_U (\xi_{\epsilon_1} - \xi_{\epsilon_2})$.
  We use the wavelet characterization of Besov spaces given in Propositions~\ref{prop:wavelet_basis} and~\ref{prop:besov_wavelet}.
  Using the notation therein, we have
  \begin{equation*}
    \norm{\eta}_{B_{p,p}^{-2+\delta}(\R^d)}^q 
    \lesssim \Big( \sum_{n \in \mathbb{N}_0} 2^{np(-2 + \delta) - n d} \sum_{G \in \mathfrak{G}^n, m \in \mathbb{Z}^d} 
    \abs{\inp{\eta}{2^{\frac{nd}{2}} \Psi^{n, G}_m}}^p \Big)^{\frac{q}{p}}.
  \end{equation*}
  Since $\scalefcn$ and $\motherfcn$ are compactly supported, by \eqref{eq:wavelet_basis} there exists an $r \in (0, \infty)$ such that
  the sum with respect to $m$ is over $\Z^d \cap 2^{n} B(U, r)$.
  Therefore, by Minkowski's inequality (remember that $\frac{q}{p}\ge 1$)
  \begin{align*}
  \mathbb{E}[ \norm{\eta}_{B_{p,p}^{-2+\delta}(\R^d)}^q ]^{\frac{p}{q}}
 & 
      \lesssim \norm[\Big]{\sum_{n \in \mathbb{N}_0} 2^{np(-2 + \delta) - n d} \sum_{G \in \mathfrak{G}^n, m \in \mathbb{Z}^d \cap 2^n B(U, r)} 
    \abs{\inp{\eta}{2^{\frac{nd}{2}} \Psi^{n, G}_m}}^p}_{L^{\frac{q}{p}}(\mathbb{P})}
\\ 
    & \lesssim \sum_{n \in \mathbb{N}_0} 2^{np(-2 + \delta) - n d} \sum_{G \in \mathfrak{G}^n, m \in \mathbb{Z}^d \cap 2^n B(U, r)} 
    \mathbb{E}[\abs{\inp{\eta}{2^{\frac{nd}{2}} \Psi^{n, G}_m}}^q]^{\frac{p}{q}}. 
  \end{align*}
By \eqref{eq:cauchy_bm_epsilon_boundary}, by \eqref{eq:wavelet_basis} and 
observing that $2^{\frac{nd}{2}} \Psi_m^{n,G} = 2^{\frac{d}{2}} \phi^{\lambda}_{\lambda m}$ for $\phi = \prod_{j=1}^k \psi_{G_j}$ and $\lambda = 2^{-(n-1)_+}$, 
\begin{calc}
\begin{align*}
\mathbb{E}[\abs{\inp{\eta}{2^{\frac{nd}{2}} \Psi^{n, G}_m}}^q]
& = 2^{\frac{dq}{2}} 
\expect[ \abs{\inp{\xi_{\epsilon_1} - \xi_{\epsilon_2}}{\indic_U \phi^{\lambda}_{\lambda m}}}^q]  \\
& \le 2^{\frac{dq}{2}} 
\bm{\epsilon}_q^{\bdry}(\epsilon_1, \epsilon_2; 2^{-(n-1)},2^{-(n-1)_+}m) 2^{-(n-1)_+ (-2 + \delta + \delta')q} , 
\end{align*}
\end{calc}
 and $\sup_{n\in\N_0} \# \fG^n <\infty$, we obtain
  \begin{equation*}
  \sum_{G\in \fG^n} 
\mathbb{E}[\abs{\inp{\eta}{2^{\frac{nd}{2}} \Psi^{n, G}_m}}^q]^{\frac{p}{q}} \lesssim 
    2^{- n p(-2 + \delta + \delta')} 
    \bm{\epsilon}_q^{\bdry}(\epsilon_1, \epsilon_2; 2^{-(n-1)_+}, 2^{-(n-1)_+}m )^{\frac{p}{q}}.
  \end{equation*}
%
  Therefore 
  \begin{equation} 
    \label{eqn:sum_with_boldface_eps}
    \expect[\norm{\eta}_{B_{p,p}^{-2+\delta}(\R^d)}^q ]^{\frac{p}{q}} \\
    \lesssim \sum_{n \in \N_0} 2^{-n \delta' - nd}
    \sum_{m \in \mathbb{Z}^d \cap 2^{n} B(U, r)}  
    \bm{\epsilon}_q^{\bdry}(\epsilon_1, \epsilon_2; 2^{-(n-1)_+}, 2^{-(n-1)_+}m)^{\frac{p}{q}}.
  \end{equation}
  As $\sum_{m\in \mathbb{Z}^d \cap 2^n B(U,r)} 1 \lesssim 2^{nd} |B(U,r)|$, 
  we have 
  \begin{equation*}
    \sum_{n \in \N_0} 2^{-n \delta' - nd}
    \sum_{m \in \mathbb{Z}^d \cap 2^{n} B(U, r)} 1 \lesssim |B(U,r)| \sum_{n \in \N_0} 2^{-n \delta'}. 
  \end{equation*}
  In view of the condition \ref{item:boundary1}, the dominated convergence theorem yields that the right-hand side (and thus left-hand side) of \eqref{eqn:sum_with_boldface_eps} converges to $0$ as $\epsilon_1,\epsilon_2 \downarrow 0$. 
  This proves the claim on the convergence of $\indic_U \xi_{\epsilon}$.
  By using \eqref{eq:bound_xi_epsilon_on_U}, one can similarly prove the estimate \eqref{eq:bound_of_xi_U}.

  Regarding the last claim, to prove \eqref{eq:Y_U_bound_assump}, let $\alpha > 0$. We have the bound 
  \begin{equation*}
    \mathbb{E}\Big[\big( \sup_{L \in \mathbb{N}} L^{-\alpha} \norm{\tilde{Y}^{U_L}}_{B_{p, p}^{-2 + \delta}} \big)^q \Big] 
    \leq \sum_{L \in \mathbb{N}} L^{-\alpha q} \mathbb{E}[ \norm{\tilde{Y}^{U_L}}_{B_{p, p}^{-2 + \delta}}^q ].
  \end{equation*} 
  By Proposition~\ref{prop:boundary_nabla_xi} and \eqref{eq:bound_of_xi_U}, for any $\sigma \in (d/p,\infty)$ the right-hand side is up to constant bounded by
  \begin{equation*}
    \sum_{L \in \mathbb{N}} L^{- \alpha q} ( L^{2 \sigma q} + L^{\frac{q d}{p}}). 
  \end{equation*}
  Therefore, if $\alpha > \frac{2d}{p}$, by choosing $\sigma$ sufficiently close to $\frac{d}{p}$, the above sum is finite. 
  In particular, if $p > 8d$ we may choose $\alpha = \frac{1}{4}$.
\end{proof}


%

\begin{lemma}\label{lem:example_of_boundary_xi}
Suppose that $\gamma, f,g : \R^d \rightarrow \R$ are measurable functions such that 
  \begin{equation}\label{eq:gamma_bounded_by_f_plus_g}
    \abs{\gamma} \leq f + g,
  \end{equation}
  with $g \in L^{\infty}(\R^d)$ and 
  $f(\lambda x) = \lambda^{-\alpha} f(x)$ with $\alpha < 3$ for every $\lambda \in (0, \infty)$ and $x \in \R^d$. 
  Furthermore, suppose that $f$ is locally integrable. 
    Let $\xi$ be a centered Gaussian noise whose covariance is  given by
\begin{align*}
\expect[\langle \xi, \varphi \rangle \langle \xi, \psi \rangle ] = \langle \gamma * \varphi, \psi \rangle_{L^2(\R^d)}, \quad \varphi, \psi \in C_c^\infty(\R^d).
\end{align*}
  Then, $\xi$ satisfies the Neumann condition~\ref{assump:all_three_assump}~\ref{item:neumann_assump}. 
\end{lemma}

\begin{proof}
In view of Lemma~\ref{lem:convergence_of_indic_U_xi_epsilon}, it suffices to show that the conditions of that lemma are satisfied.
As $\xi$ is Gaussian, so is for example $\langle \xi_{\epsilon_1}, \1_U \phi_x^\lambda \rangle_{\R^d}$  for $\phi \in \cS(\R^d)$, $x\in \R^d$, $\lambda \in (0,\infty)$, where $\phi^{\lambda}_x \defby \lambda^{-d} \phi(\lambda^{-1}(\cdot - x))$. 
As for Gaussian random variables $Z$ one has $\expect{[|Z|^p]} = \expect{[|Z|^2]}^{\frac{p}{2}} \expect{[|X|^p]}$, for $X$ a standard normal random variable,
it is sufficient to consider $p=2$.

Let $U$ be a bounded Lipschitz domain and $\phi$ be as in Lemma~\ref{lem:convergence_of_indic_U_xi_epsilon}~\ref{item:neumann_assump}.
Observe that because $\rho$ is symmetric
\begin{align*}
\inp{\xi_{\epsilon_1} - \xi_{\epsilon_2} }{ \psi }
\begin{calc} = \inp{\xi * \rho_{\epsilon_1} - \xi * \rho_{\epsilon_2} }{ \psi } \end{calc}
= \inp{\xi * (\rho_{\epsilon_1} - \rho_{\epsilon_2}) }{ \psi }
= \inp{\xi }{ (\rho_{\epsilon_1} - \rho_{\epsilon_2}) * \psi }.
\end{align*}
Therefore, for any $\lambda \in (0,1)$ and $x\in \R^d$, 
\begin{align*}
 \expect[\inprd{\xi_{\epsilon_1} - \xi_{\epsilon_2}}{\indic_U \phi^{\lambda}_x}^2]
& = \inp { \gamma * (\rho_{\epsilon_1} - \rho_{\epsilon_2}) * (\1_U \phi_x^\lambda) }{(\rho_{\epsilon_1} - \rho_{\epsilon_2}) * (\1_U \phi_x^\lambda) }_{L^2(\R^d)}
.
\end{align*}
  We set
    \begin{equation*}
    \bm{\epsilon}_2^{\bdry}(\epsilon_1, \epsilon_2; \lambda,x)
    \defby \lambda^\alpha
    \abs*{
\inp { \gamma * (\rho_{\epsilon_1} - \rho_{\epsilon_2}) * (\1_U \phi_x^\lambda) }{(\rho_{\epsilon_1} - \rho_{\epsilon_2}) * (\1_U \phi_x^\lambda) }_{L^2(\R^d)} }.
  \end{equation*}
One has
    \begin{equation*}
    \lim_{\epsilon_1, \epsilon_2 \downarrow 0}
   \inp { \gamma * \rho_{\epsilon_1}  * (\1_U \phi_x^\lambda) }{ \rho_{\epsilon_2} * (\1_U \phi_x^\lambda) }_{L^2(\R^d)}
    = \inprd{\gamma \conv (\1_U \phi_x^\lambda)}{(\1_U \phi_x^\lambda)}_{L^2(\R^d)},
  \end{equation*}
  and hence $\lim_{\epsilon_1, \epsilon_2 \downarrow 0} \bm{\epsilon}_2^{\bdry}(\epsilon_1, \epsilon_2; \lambda,x) = 0$.
  To prove the bounds \eqref{eq:bound_of_bm_epsilon_boundary} and \eqref{eq:bound_xi_epsilon_on_U}, it suffices to show
  \begin{equation*}
    \sup_{\epsilon \in (0, 1), \lambda \in (0, 1), x \in \R^d} \lambda^{\alpha}
      \abs*{
\inp { \gamma * \rho_{\epsilon}  * (\1_U \phi_x^\lambda) }{\rho_{\epsilon}  * (\1_U \phi_x^\lambda) }_{L^2(\R^d)} } < \infty.
  \end{equation*}
Let us write
\begin{align}
\label{eqn:rescaled_recentered_set}
U_x^\lambda = \lambda^{-1} (U - x).
\end{align}
  By \eqref{eq:gamma_bounded_by_f_plus_g} (remember that $\phi^{\lambda}_x \defby \lambda^{-d} \phi(\lambda^{-1}(\cdot - x))$)
  \begin{align*}
&   \sup_{x\in \R^d}   \abs*{
\inp { \gamma * \rho_{\epsilon}  * (\1_U \phi_x^\lambda) }{\rho_{\epsilon}  * (\1_U \phi_x^\lambda) }_{L^2(\R^d)} }  \\
   & = \sup_{x\in \R^d} \abs{\inprd{\gamma(\lambda \cdot) \conv  (\1_{U_x^\lambda} \phi)  \conv \rho_{\epsilon/\lambda}}{ (\1_{U_x^\lambda} \phi)  \conv \rho_{\epsilon/\lambda}}} \\
   & \leq \lambda^{-\alpha} \inprd{f \conv \abs{\phi} \conv \abs{\rho_{\epsilon/\lambda}}}
    {\abs{\phi} \conv \abs{\rho_{\epsilon/\lambda}}}
    + \inprd{g(\lambda \cdot) \conv \abs{\phi} \conv \abs{\rho_{\epsilon/\lambda}}}
    {\abs{\phi} \conv \abs{\rho_{\epsilon/\lambda}}}.
  \end{align*}
  \begin{calc}
\begin{align*}
\inp{ f * \phi(\lambda \cdot)}{ g }_{L^2(\R^d)}
& = \int \int f(x-y) \phi(\lambda y) g(x) \dd y \dd x \\
& = \int \int f(x-\frac{z}{\lambda}) \phi(z) g(x) \lambda^{-d} \dd z \dd x \\
& = \lambda^{-2d} \int \int f(\frac{w-z}{\lambda}) \phi(z) g(\lambda w)  \dd z \dd w \\
& = \lambda^{-2d} \inp{ f(\tfrac{1}{\lambda} \cdot) * \phi}{ g(\tfrac{1}{\lambda} \cdot)  }_{L^2(\R^d)}. 
\end{align*}
\end{calc}
Since, using Young's inequality, one can bound the second term by
\begin{calc}
\begin{align*}
\langle f * g , g \rangle
\le \|f*g\|_{L^\infty} \|g\|_{L^1}
\le \|f\|_{L^\infty} \|g\|_{L^1}^2.
\end{align*}
\end{calc}
  \begin{equation*}
    \norm{g}_{L^{\infty}(\R^d)} \norm{\phi}_{L^1(\R^d)}^2 \norm{\rho}_{L^1(\R^d)}^2,
  \end{equation*}
  it comes down to showing
  \begin{equation*}
    \sup_{\mu \in (0, \infty)} \inprd{f \conv \abs{\phi} \conv \abs{\rho_{\mu}}}
    {\abs{\phi} \conv \abs{\rho_{\mu}}} < \infty.
  \end{equation*}
$\bullet$  Suppose $\mu \leq 1$. Let $\sigma > d$. By the weighted Young's inequality, Theorem~\ref{thm:weighted_young}, one has
  \begin{equation*}
    \inprd{f \conv \abs{\phi} \conv \abs{\rho_{\mu}}}
    {\abs{\phi} \conv \abs{\rho_{\mu}}}
    \lesssim_{\sigma} \norm{w_{\sigma} f}_{L^1(\R^d)} \norm{w_{\sigma} \phi}_{L^2(\R^d)}^2
    \norm{w_{-\sigma} \rho_{\mu}}_{L^1(\R^d)}^2.
  \end{equation*}
\begin{calc}
\begin{align*}
    \inprd{f \conv \abs{\phi} \conv \abs{\rho_{\mu}}}
    {\abs{\phi} \conv \abs{\rho_{\mu}}}
&    \lesssim 
        \|w_{\sigma}(f \conv \abs{\phi} \conv \abs{\rho_{\mu}})\|_{L^2(\R^d)}
    \|w_{-\sigma}(\abs{\phi} \conv \abs{\rho_{\mu}})\|_{L^2(\R^d)} \\
& \lesssim 
        \|w_{\sigma} f \|_{L^1(\R^d)} 
    \|w_{-\sigma}(\abs{\phi} \conv \abs{\rho_{\mu}})\|_{L^2(\R^d)}^2  \\
& \lesssim 
        \|w_{\sigma} f \|_{L^1(\R^d)} 
    \|w_{\sigma}\phi \|_{L^2(\R^d)}^2
    \|w_{-\sigma} \rho_{\mu}\|_{L^1(\R^d)}^2 
\end{align*}
\end{calc}
  As $\phi$ is a continuous function with compact support, we have $\norm{w_{\sigma} \phi}_{L^2(\R^d)}<\infty$.
  Since $f$ is locally integrable and satisfies the scaling property,
  \begin{equation*}
    \norm{w_{\sigma} f}_{L^1(\R^d)}
    = \int_{\partial B(0,1)} \abs{f(x)} \dd S(x) \int_0^{\infty} r^{d - 1 -\alpha} (1 + r^2)^{-\frac{\sigma}{2}} \dd r < \infty.
  \end{equation*}
  Then, we observe, as $\mu \leq 1$,
  \begin{equation*}
    \norm{w_{-\sigma} \rho_{\mu}}_{L^1(\R^d)} = \int_{\R^d} (1 + \mu^2 \abs{x}^2)^{\frac{\sigma}{2}} \abs{\rho}(x) \dd x
    \leq \int_{\R^d} (1 + \abs{x}^2)^{\frac{\sigma}{2}} \abs{\rho}(x) \dd x < \infty.
  \end{equation*}
$\bullet$  Now suppose $\mu \geq 1$.
  By change of variables, 
  \begin{align*}
  \inprd{f \conv \abs{\phi} \conv \abs{\rho_{\mu}}}
    {\abs{\phi} \conv \abs{\rho_{\mu}}}
     &= \mu^{-\alpha} \inprd{f \conv \abs{\phi_{\mu^{-1}}} \conv \abs{\rho}}
    {\abs{\phi_{\mu^{-1}}} \conv \abs{\rho}}.
\end{align*}
Therefore, it reduces to the case $\mu \leq 1$.
\end{proof}

\begin{example}
\label{example:example_of_gamma}
An example of a $\gamma$ that satisfies the conditions of Lemma~\ref{lem:example_of_boundary_xi} is the following.
Let $n\in \{1,\dots,d\}$ and $d_1,\dots,d_n \in \N$ be such that $d= d_1+\cdots + d_n$.
Let $\alpha_1, \ldots, \alpha_n \in (0, \infty)$ are such that
  $\alpha_j < d_j$ for all $j$ and $\alpha_1 + \ldots + \alpha_n < 3$.   Then, for $x = (x_1, \ldots, x_n) \in \R^{d_1} \times \cdots \times \R^{d_n}$,
 we set
  \begin{equation*}
    \gamma(x) \defby \abs{x_1}^{-\alpha_1} \cdots \abs{x_n}^{-\alpha_n}.
  \end{equation*}
  For this example, $f = \gamma$ and $g = 0$. Observe that for the local integrability of $f$, the condition $\alpha_j < d_j$ for all $j$ is  necessary.  
\end{example}

\begin{example}
\label{example:2D_WN_satisf_Neumann_cond}
  The 2D white noise $\xi^{2\textnormal{D}}$ does not satisfy the conditions of Lemma~\ref{lem:example_of_boundary_xi}.
  However, one has
  \begin{equation*}
    \expect[\inprd{\xi_{\epsilon_1}^{2\textnormal{D}} - \xi_{\epsilon_2}^{2\textnormal{D}}}{\indic_U \phi^{\lambda}_x}^2]
    = \lambda^{-2} \norm{  (\1_{U_x^\lambda} \phi)  \conv (\rho_{\epsilon_1/ \lambda} - \rho_{\epsilon_2/ \lambda})}_{L^2(\R^d)},
  \end{equation*}
 where $U_x^\lambda$ is as in \eqref{eqn:rescaled_recentered_set}.
 Therefore, the 2D white noise satisfies the condition of Lemma~\ref{lem:convergence_of_indic_U_xi_epsilon} and therefore the Neumann condition~\ref{assump:all_three_assump}~\ref{item:neumann_assump}
 as well.
\end{example}

\section{Analysis of symmetric forms}\label{sec:symmetric_forms}

It is common practice in the theory of rough paths \cite{Lyons1998} to first show the existence of sufficiently many stochastic objects and then apply deterministic analysis to derive results. In this section we consider the (deterministic) analysis of symmetric forms, which we use in Section~\ref{sec:eigenvalues_of_AH} in combination with Assumptions~\ref{assump:all_three_assump} to construct the Anderson Hamiltonian and derive its spectral properties.

First we recall the definition of a symmetric form and some related definitions in Definition~\ref{def:def_symmetric_forms},  then, in Definition~\ref{def:symmetric_form},  we describe the symmetric forms $\cE_{W,\cZ}^U$  that we will study in Sections~\ref{subsec:basic_properties_of_symmetric_forms} and~\ref{subsec:estimates_eigenvalues}.
In Section~\ref{subsec:bounded_sym_forms} we study
examples of bounded symmetric forms.
In Section~\ref{subsec:basic_properties_of_symmetric_forms} we study basic spectral properties of the symmetric forms and their associated self adjoint operators. 
And in Section~\ref{subsec:estimates_eigenvalues} we consider estimates of eigenvalues. 


\begin{definition}
\label{def:def_symmetric_forms}
Let $H$ be a Hilbert space over $\R$.
A bilinear map $\cQ : \cD(\cQ) \times \cD(\cQ) \rightarrow \R$, with $\cD(\cQ)$ a dense subspace of $H$, is called a \emph{symmetric form}  on $H$ if $\cQ(u,v) = \cQ(v,u)$ for all $u,v\in \cD(\cQ)$.
Let $\cQ$ be a symmetric form on $H$. We write
\begin{equation}\label{eq:def_of_form_norm}
  \formnorm{\cQ}_{H} \defby \sup_{u \in \cD(\cQ), \norm{u}_H=1} \abs{\cQ(u, u)}.
\end{equation}
If $\formnorm{\cQ}_H<\infty$, then we call $\cQ$ a \emph{bounded symmetric form}. In that case, without loss of generality we assume $\cD(\cQ) = H$.
The set of bounded symmetric forms a Banach space under the norm $\formnorm{\cdot}_H$.
Then, a sequence $(\cZ_n)_{n\in\N}$ of bounded symmetric forms converges to a bounded symmetric form $\cZ$ if
\begin{align*}
\formnorm{\cZ_n - \cZ}_H \rightarrow 0.
\end{align*}
Let $M>0$. A symmetric form $\cQ$ is called \emph{$M$-bounded from below} if
$\cQ(u, u) + M \norm{u}_{H}^2 \geq 0$ for all $u\in \cD(\cQ)$.
It is called \emph{bounded from below} if it is $M$-bounded from below for some $M>0$.
If $\cQ$ is $M$-bounded from below and
$(\D(\cQ), \cQ + M \inp{\cdot}{\cdot}_H)$
is a Hilbert space for some $M > 0$, then $\cQ$ is said to be \emph{closed}.
If $\cQ$ is a closed symmetric form and $M$ is as above, then a subset of $\D(\cQ)$ is called a \emph{core} for $\cQ$ if it is dense in the Hilbert space $(\D(\cQ), \cQ + M \inp{\cdot}{\cdot}_H)$.

Observe that a symmetric form is determined by its values on the diagonal of $H\times H$, i.e., $\cQ(u,v) = \frac12 [ \cQ(u+v,u+v) - \cQ(u,u) - \cQ(v,v)]$. For this reason we often only define symmetric forms on the diagonal.
\end{definition}

\begin{definition}\label{def:symmetric_form}
  Let $U$ be a bounded domain,
  $W \in L^{\infty}(U)$ and $\cZ$ be a bounded symmetric form on $H^s(U)$ for some $s \in [0, 1)$.
  We define the symmetric form $\E = \E^U_{W, \cZ}$ on $e^W H^1(U)$  as follows: for $u = e^W u^{\flat}$ with
  $u^{\flat} \in H^1(U)$, we set
  \begin{equation*}
    \E(u, u) \defby \E^U_{W, \cZ}(u,u) \defby
     \int_U e^{2W(x)} \abs{\nabla u^{\flat}(x)}^2 \dd x + \cZ(u^{\flat}, u^{\flat}).
  \end{equation*}
\end{definition}

\subsection{Main examples of bounded symmetric forms}
\label{subsec:bounded_sym_forms}

Recall the notation $\formnorm{\cdot}$ from \eqref{eq:def_of_form_norm} and the constants in Definition~\ref{def:functional_inequalities_constants}.

\begin{theorem}
\label{theorem:examples_bd_sym_with_weighted_input}
Let $\delta \in (0,1)$, $\sigma \in (0,\infty)$
and $s\in (1-\delta,1)$.
\begin{enumerate}
\item
\label{item:sym_form_on_domain}
Let $Y\in \cC^{-1+\delta,\sigma}(\R^d)$.
For any bounded domain $U$ and $\phi \in C_c^\infty(\R^d)$ such that $\phi =1$ on a neighborhood of $U$, the formula
\begin{align*}
\cZ_Y^{U}(v,v) = \langle \phi Y , \1_U v^2 \rangle
\end{align*}
defines a bounded symmetric form on $H^s_0(U)$
and if $U$ is moreover Lipschitz, it also defines a bounded symmetric form on $H^s(U)$.
The symmetric form $\cZ_Y^U$ is independent of the choice of $\phi$.
Moreover, for $L\ge 1$
  \begin{align}
  \label{eqn:bound_cZ_Y_U_L_H_0}
    \formnorm{\cZ_Y^{U_L}}_{H^s_0(U_L)} \lesssim_{\delta, \epsilon,U}
L^{2\sigma}  \norm{Y}_{\csp^{-1 + \delta,\sigma}(\R^d)},
  \end{align}
  and if $U$ is a bounded Lipschitz domain, then
  \begin{align}
    \label{eqn:bound_cZ_Y_U_L_H}
    \formnorm{\cZ_Y^{U_L}}_{H^s(U_L)} \lesssim_{\delta, \epsilon, U}
L^{2\sigma}  \norm{Y}_{\csp^{-1 + \delta,\sigma}(\R^d)}.
  \end{align}

  \item
\label{item:sym_form_boundary_tilde}
Let $U$ be a bounded Lipschitz domain. Suppose that $\delta \in (\frac12,1)$ and $s \in (\frac32 - \delta, 1)$.
Let $\epsilon\in (0,\delta - \frac12)$, $p \in (2,\infty)$ and $q\in (1,2)$ be such that $\frac1p+\frac1q=1$ and
\begin{align}
\label{eqn:condition_beta_and_s}
\beta \defby 2-\delta+\epsilon - \frac1q \le \frac12, \qquad
2-\delta + \epsilon -\frac{1}{2q} + \frac{d}{2p} \le s.
\end{align}
Let $\tilde{Y} \in B_{p, p}^{-2+\delta}(\R^d)$ with $\supp(\tilde{Y}) \subseteq \partial U$, $\norm{\tilde Y_\epsilon - \tilde Y}_{B_{p,p}^{-2+\delta}(\R^d)} \arroweps 0$ for some $\tilde Y_\epsilon$ given by $\varphi \mapsto \int_{\partial U} \varphi f_\epsilon \dd S$ for $f_\epsilon \in L^1(\partial U)$,
  and $V \in \csp^{\beta}(U)$.
Then, with $\cT = \cT_{W_{2q}^{\beta + \frac{1}{2q} }(U)}$,  $\cR : W_q^{\beta -\epsilon}(\partial U) \rightarrow W_q^{2-\delta}(U) $ a right inverse of $\cT_{W_q^{2-\delta}(U)}$ 
and $\iota$ a universal extension operator from $U$ to $\R^d$ as in Lemma~\ref{lem:extension_and_trace_sobolev},
  \begin{equation*}
    \tilde{\cZ}(v, v)
    \defby \tilde{\cZ}_{\tilde Y,V}^U (v, v)
    \defby \inprd{\tilde{Y}}{ \iota \circ  \cR[V (\cT v)^2]}
  \end{equation*}
  defines a bounded symmetric form on $H^s(U)$ that is independent of the choice of $\cR$ and $\iota$.

If $\tilde{Y}_L \in B_{p, p}^{-2+\delta}(\R^d)$ with $\supp(\tilde{Y}_L) \subseteq \partial U_L$
  and $V_L \in \csp^{\delta}(U_L)$ for $L\ge 1$, then
  \begin{align}
  \label{eqn:bound_tilde_Z_in_L}
    \formnorm{\tilde{\cZ}^{U_L}_{\tilde Y_L,V_L}}_{H^s(U_L)}
    \lesssim_{\delta, \epsilon, p,U} L^{2\epsilon}
    \norm{V_L}_{C^{\delta}(U_L)}
    \norm{\tilde{Y_L}}_{B_{p, p}^{-2+\delta}(\R^d)}.
  \end{align}

\item
\label{item:sym_form_boundary_hat}
Let $U$ be a bounded Lipschitz domain. Suppose $\delta \in (0,1)$ and $s\in (\frac12,1)$.
Let $\widehat{Y} \in \cC^{1+\delta, \sigma}(\R^d)$ and $V \in \csp^{\delta}(\R^d)$. Then, with $\cT= \cT_{H^s(U)}$ as in Lemma~\ref{lem:extension_and_trace_sobolev} (for the notation see Definition~\ref{def:notation_integral_dot}),
\begin{align*}
\widehat{\cZ}^{U}_{\widehat{Y},V}(v, v) \defby \int_{\partial U} (\cT v)^2 V \nabla \widehat{Y} \cdot \dd \bm{S}
\end{align*}
 defines a
  bounded symmetric form on $H^s(U)$ for every $s \in (\frac{1}{2}, 1)$. Moreover, for $L\ge 1$
  \begin{equation}
    \label{eqn:bound_hat_Z_in_L}
    \formnorm{\widehat{\cZ}^{U_L}_{\widehat{Y},V}}_{H^s(U_L)}
    \lesssim_{\delta, \sigma, U}
 L^\sigma \norm{V}_{C^\delta(U_L)} \norm{\widehat{Y}}_{\cC^{1+\delta,\sigma}(\R^d)}.
  \end{equation}
\end{enumerate}
\end{theorem}
\begin{proof}
\ref{item:sym_form_on_domain}
Let us first consider $U$ to be a bounded Lipschitz domain and $v\in H^s(U_L)$. We comment on how to obtain \eqref{eqn:bound_cZ_Y_U_L_H_0} afterwards.
Let $\phi$ be as in the statement.
It is rather straightforward to check that the definition of $\cZ_Y^U$ does not depend on $\phi$.
Observe that therefore $\cZ_Y^{U_L}(v,v) = \langle \phi(L^{-1} \cdot) Y, \1_{U_L} v^2 \rangle$ for all $L>0$.
Choose a $p\in [1,\infty]$ and an $\epsilon \in (0,\delta)$ such that $1-\delta + \epsilon +\frac{d}{2p} \le s$ and $p \sigma > d$.
By the duality of Besov spaces \cite[Theorem 2.17]{sawano2018theory}, for $q\in [1,\infty]$ such that $\frac1p + \frac1q=1$, we have
  \begin{equation*}
    \abs{\inprd{\phi(L^{-1} \cdot) Y}{\indic_{U_L} v^2}} \lesssim_{\delta, p} \norm{\phi(L^{-1} \cdot) Y}_{B^{-1 + \delta}_{p, p}(\R^d)}
    \norm{\indic_{U_L} v^2}_{B_{q, q}^{1 - \delta}(\R^d)}.
  \end{equation*}
By Lemma~\ref{lem:localization} and Lemma~\ref{lem:weighted_besov_p_vs_infty} (more specifically, \eqref{eqn:estimate_p,q_by_infty_infty_besov} using that $p \sigma >d$) we have
\begin{align*}
\norm{\phi(L^{-1} \cdot) Y}_{B^{-1 + \delta}_{p, p}(\R^d)} \lesssim_{p,\delta,\sigma,\phi}
L^{2\sigma} \norm{Y}_{B^{-1 + \delta,2\sigma}_{p, p}(\R^d)}
\lesssim_{p,\delta, \sigma}
L^{2\sigma} \norm{Y}_{\cC^{-1+\delta,\sigma}(\R^d)}.
\end{align*}

Then by Lemma~\ref{lem:equivalence_sob_slobo_and_besov} (see Definition~\ref{def:functional_inequalities_constants} for $C_{\operatorname{Mult}}$ and $C_{\operatorname{Prod}}$)
  \begin{align*}
    \norm{\indic_{U_L} v^2}_{B_{q, q}^{1 - \delta}(\R^d)} &  \lesssim_{\delta, p}
    \norm{\indic_{U_L} v^2}_{W_{q}^{1 - \delta}(\R^d)} \leq C_{\operatorname{Mult}}^{U_L}[W_q^{1-\delta}] \norm{v^2}_{W_{q}^{1 - \delta}(U_L)} \\
&    \leq C_{\operatorname{Mult}}^{U_L}[W_q^{1-\delta}] C_{\operatorname{Prod}}^{U_L}[W^{1-\delta + \epsilon}_{2q} \to W^{1 - \delta}_q]
    \norm{v}^2_{W^{1 - \delta + \epsilon}_{2q}(U_L)}.
  \end{align*}
  Now we apply the embedding estimate (see Definition~\ref{def:functional_inequalities_constants} for $C_{\operatorname{Embed}}$) and the estimate $\norm{u}_{H^{1-\delta + \epsilon +\frac{d}{2p}}(U_L)} \lesssim_{\delta,\epsilon,p} \norm{u}_{H^{s}(U_L)}$ (as $1-\delta + \epsilon +\frac{d}{2p}\le s$), we have
  \begin{equation*}
    \norm{v}_{W_{2q}^{1 - \delta + \epsilon}(U_L)}
 \lesssim_{\delta,\epsilon,p} C_{\operatorname{Embed}}^{U_L}[H^{1-\delta + \epsilon +\frac{d}{2p}} \to W_{2q}^{1 - \delta + \epsilon}] \norm{v}_{H^{s}(U_L)}.
  \end{equation*}
Hence
  \begin{multline*}
    \formnorm{\cZ}_{H^s(U_L)} \lesssim_{\delta, \epsilon, p,\sigma} L^{2\sigma} C^{U_L}_{\operatorname{Mult}}[W_q^{1 - \delta}]
    C_{\operatorname{Prod}}^{U_L}[H^{1-\delta+ \epsilon} \to W^{1 - \delta}_1] \\
    \times C_{\operatorname{Embed}}^{U_L}[H^s \to W_{2q}^{1 - \delta+ \epsilon}]
    \norm{Y}_{\cC^{-1+\delta,\sigma}(\R^d)}.
  \end{multline*}

Therefore \eqref{eqn:bound_cZ_Y_U_L_H} follows by Lemma~\ref{lem:scaling_of_iota}.
If $U$ is not necessarily Lipschitz but $v \in H^s_0(U)$, in the above estimates, we can replace the constant $C_{\operatorname{Mult}}^U[W_q^{1-\delta}]$ by $1$ and the estimate \eqref{eqn:bound_cZ_Y_U_L_H_0} follows similarly.

\ref{item:sym_form_boundary_tilde}
First, observe that the requirements for the existence of $\cT$ and $\cR$ as in Lemma~\ref{lem:extension_and_trace_sobolev} are satisfied. Indeed, $\beta+\frac{1}{2q}  \in (\frac{1}{2q}, 1+ \frac{1}{2q})$, or equivalently, $\beta \in (0,1)$, and $2-\delta \in (\frac{1}{q},1+\frac{1}{q})$: On the one hand we have  $2-\delta \in (1,\frac32) \subseteq (\frac{1}{q},1+ \frac1q)$ and
because 
$\frac{1}{q} \in (\frac12,1)$, 
on the other hand we have 
$\beta = 2- \delta + \epsilon - \frac{1}{q}  > 2-\delta -   1 = 1- \delta > 0$ and $\beta \le \frac12$ by assumption. 

Again by the duality, we have
  \begin{equation}
  \label{eqn:estimate_pairing_tilde_Y_by_duality}
    \abs{\inprd{\tilde{Y}}{ \iota \circ  \cR[ V (\cT v)^2]}} \lesssim_{\delta, p} \norm{\tilde{Y}}_{B_{p, p}
    ^{-2 + \delta}(\R^d)}
    \norm{ \iota \circ  \cR[ V (\cT v)^2]}_{B_{q, q}^{2-\delta}(\R^d)}.
  \end{equation}
Let us use this estimate first to show that $\tilde \cZ$ is independent of the choice of $\cR$ and $\iota$, as we will use this for our norm estimates. For $\varphi \in C^\infty(\overline U) \cap W_p^r(U)$ and $\cV \in C^\infty(\overline U)$ the function $\iota \circ \cR[\cV(\cT \varphi)^2]$ is equal to $\cV \varphi^2$ on $\partial U$ and thus for $\epsilon>0$
\begin{align*}
\int_{\partial U} f_\epsilon (\iota \circ \cR[\cV(\cT \varphi)^2]) \dd S
= \int_{\partial U} f_\epsilon \cV \varphi^2 \dd S.
\end{align*}
By the above estimates we have already seen that $\tilde \cZ^U_{\tilde Y, V}(v)$ is continuous as a function of $\tilde Y$, $V$ and $v$. As $C^\infty(\overline U) \cap W_p^r(U)$ is dense in $W_p^r(U)$, $V$ is the limit of smooth functions in $C^\infty(\overline U)$ and $\tilde Y$ is the limit of $\tilde Y_\epsilon$, it therefore follows that $\tilde \cZ$ is independent of the choice of $\cR$ and $\iota$.  
  
We continue by estimating the right-hand side of \eqref{eqn:estimate_pairing_tilde_Y_by_duality}. 
Again, by using Lemma~\ref{lem:equivalence_sob_slobo_and_besov} (remember that $\cR$ is a bounded linear operator $W_q^{\beta  -\epsilon}(\partial U) \rightarrow W_q^{2-\delta}(U)$),
  \begin{equation}
  \label{eqn:estimate_iota_circ_cR_V_cT_sq}
    \norm{ \iota \circ  \cR[V (\cT v)^2]}_{B_{q, q}^{2-\delta}(\R^d)}
    \lesssim_{\delta, p}
    \norm{\iota}_{W_q^{2-\delta}(U) \rightarrow W_q^{2-\delta}(\R^d)}
    \norm{\cR} \norm{V (\cT v)^2}_{W_q^{\beta - \epsilon}(\partial U)}.
  \end{equation}
Now we estimate $\norm{V (\cT v)^2}_{W_q^{\beta- \epsilon}(\partial U)}$.
  Recall the notation $[\cdot]$ from Definition~\ref{def:fractional_Sobolev}~\ref{item:sobolev_space_boundary}.
  If we set $\psi \defby (\cT v)^2$, then
  \begin{align*}
&   [V \psi]_{W^{\beta- \epsilon}_q(\partial U)}
\leq \Big( \int_{\partial U } \int_{\partial U }
  \frac{\abs{V(x) - V(y)}^q \abs{\psi(x)}^q}{\abs{x-y}^{d-1 + q(\beta - \epsilon)
  }} \dd S(x) \dd S(y) \Big)^{\frac{1}{q}} \\
  &\hspace{3cm}+ \Big(\int_{\partial U }\int_{\partial U }
  \frac{\abs{\psi(x) - \psi(y)}^q \abs{V(x)}^q}{\abs{x-y}^{d-1 + q(\beta - \epsilon)
  }} \dd S(x)\dd S(y) \Big)^{\frac{1}{q}} \\
  &\lesssim_{\beta} \norm{V}_{W_\infty^{\beta}(U)} \Big(\int_{\partial U}
  \int_{\partial U} \frac{\abs{\psi(x)}^q}{\abs{x-y}^{d - 1 - q \epsilon}} \dd S(x) \dd S(y) \Big)^{\frac{1}{q}}
  + \norm{V}_{L^{\infty}(U)} [\psi]_{W^{\beta - \epsilon}_q(\partial U)} \\
  &\lesssim_{\beta,\delta} \norm{V}_{\cC^{\delta}(U)}
  \Big( 1 + \sup_{x \in \partial U} \int_{\partial U} \frac{\dd S(y)}{\abs{x-y}^{d-1 - q \epsilon}} \Big)^{\frac{1}{q}}
 \norm{\psi}_{W^{\beta - \epsilon}_q(\partial U)},
  \end{align*}
 where we used $\norm{V}_{W_\infty^{\beta}(U)} \vee \norm{V}_{L^\infty(U)} \lesssim_{\beta,\delta} \norm{V}_{W_\infty^{\delta}(U)} =  \norm{V}_{C^\delta(U)}$ which holds because $\beta \le \frac12 <\delta$.

  Therefore, by observing (see Definition~\ref{def:functional_inequalities_constants} for $C_{\operatorname{Prod}}$ and $C_{\operatorname{Embed}}$)
  \begin{align*}
  \norm{\psi}_{W^{\beta - \epsilon}_q(\partial U)} = 
    \norm{(\cT v)^2}_{W^{\beta - \epsilon}_q(\partial U)}
    & \leq C_{\operatorname{Prod}}^{\partial U}[W^{\beta}_{2q} \to W^{\beta - \epsilon}_q]
    \norm{\cT v}_{W^{\beta}_{2q}(\partial U)}^2 \\
    & \leq C_{\operatorname{Prod}}^{\partial U}[W^{\beta}_{2q} \to W^{\beta - \epsilon}_q] \norm{\cT}^2
    \norm{v}_{W^{\beta + \frac{1}{2q}}_{2q}(U)}^2
  \end{align*}
  and $\norm{v}_{W^{\beta + \frac{1}{2q}}_{2q}(U)} \leq
  C_{\operatorname{Embed}}^U[H^s \to W_{2q}^{\beta + \frac{1}{2q}}] \norm{v}_{H^s(U)}$, we obtain (as we may take the infimum over $\iota$ and $\cR$)
    \begin{align*}
    \formnorm{\tilde{\cZ}}_{H^s(U)}
&    \lesssim_{\delta,\epsilon, p} C_{\operatorname{Prod}}^{\partial U}
    [W_{2q}^{\beta} \to W_q^{\beta - \epsilon}]
    C_{\operatorname{Embed}}^U[H^s \to W^{\beta + \frac{1}{2q}}_{2q}]^2 
    C_{\operatorname{R}}^{\partial U}[W_q^{2-\delta}]
    \\
& \qquad   \times \Big(1 + \sup_{x \in \partial U} \int_{\partial U} \frac{\dd S(y)}{\abs{x-y}^{d-1 - q \epsilon}} \Big)^{\frac{1}{q}}
 C^U_{\operatorname{Ext}}[W^{q}_{2-\delta}]  
    \norm{\cT }^2 \\
    & \qquad \times
    \norm{V}_{C^{\delta}(U)}
    \norm{\tilde{Y}}_{B_{p, p}^{-2 + \delta}(\R^d)}.
  \end{align*}

With this \eqref{eqn:bound_tilde_Z_in_L} follows from Lemma~\ref{lem:scaling_of_iota} (for the estimate on the $C_{\operatorname{Embed}}$ we use the second inequality in \eqref{eqn:condition_beta_and_s}) and because
\begin{align*}
\int_{\partial U_L} \frac{\dd S(y)}{\abs{x-y}^{d-1 - q \epsilon}}
=
\int_{\partial U} L^{d-1} \frac{\dd S(y)}{\abs{Lx-Ly}^{d-1 - q \epsilon}}
= L^{q\epsilon}
\int_{\partial U} \frac{\dd S(y)}{\abs{x-y}^{d-1 - q \epsilon}},
\end{align*}
so that
\begin{align*}
\sup_{x \in \partial U_L} \Big(1 +\int_{\partial U_L} \frac{\dd S(y)}{\abs{x-y}^{d-1 - q \epsilon}} \Big)^{\frac{1}{q}}
\le L^{\epsilon}
\sup_{x \in \partial U} \Big(1 +\int_{\partial U} \frac{\dd S(y)}{\abs{x-y}^{d-1 - q \epsilon}} \Big)^{\frac{1}{q}}.
\end{align*}
The latter supremum is finite: $\int_{\partial U} \frac{\dd y}{\abs{x-y}^{d-1 - q \epsilon}}$ is finite for all $x\in \partial U$ due to the fact that $U$ is a Lipschitz domain, so that by the compactness of $\partial U$ and continuity of $x\mapsto \int_{\partial U} \frac{\dd y}{\abs{x-y}^{d-1 - q \epsilon}}$ it follows that the supremum is finite as well.
The other $L^\epsilon$ factor is due to the estimate $  C^{\partial U_L}_{\operatorname{Prod}}[W_{2q}^{r} \to W_q^{r - \epsilon}] \lesssim_{p, \epsilon, U} L^{\epsilon}$, see Lemma~\ref{lem:scaling_of_iota}.

\ref{item:sym_form_boundary_hat}
First, observe that (for $\nu$ being the outer normal on $\partial U$)
  \begin{align*}
  \abs*{\widehat{\cZ}^{U}_{\widehat{Y},V}(v, v)  }
  \cand \begin{calc} =
  \int_{\partial U} (\cT v)^2 V \nabla \widehat{Y} \cdot \dd \bm{S}
  \end{calc} \cnewline
 & =   \abs*{\int_{\partial U} (\cT v)^2 V \nabla_\nu \widehat{Y} \dd S}
 \leq \norm{V}_{L^{\infty}(\partial U)} \norm{\nabla \widehat{Y}}_{L^{\infty}(\partial U)} \norm{\cT v}_{L^2( \partial U)}^2.
  \end{align*}
Secondly, use $\norm{\cT v}_{L^2( \partial U)} \le \norm{\cT v}_{W_2^{s - \frac{1}{2}}(\partial U)} \leq \norm{\cT_{H^s(U)}} \norm{v}_{H^s(U)}$, $ \norm{V}_{L^{\infty}(\partial U)} \le \norm{V}_{C^\delta(U)}$ and
$\norm{\nabla \widehat{Y}}_{L^\infty(\partial U_L)}
\le \norm{\phi(L^{-1} \cdot) \nabla \widehat{Y} }_{L^\infty(\R^d)}
\lesssim \norm{\phi(L^{-1} \cdot) \nabla \widehat{Y} }_{\cC^\delta(\R^d)}
\lesssim_{\delta,\sigma,\phi} L^\sigma \norm{\widehat{Y}}_{\cC^{1+\delta,\sigma}(\R^d)}$, where the last inequality is due to Lemma~\ref{lem:localization}. \eqref{eqn:bound_hat_Z_in_L} then follows by  Lemma~\ref{lem:scaling_of_iota} (observe that we use that  $s>\frac12$ in order to have $\sup_{L\ge 1} \norm{\cT_{H^s(U_L)}} <\infty$).
\end{proof}

\subsection{Basic spectral properties}\label{subsec:basic_properties_of_symmetric_forms}


In this section, we will work in the setting of Definition~\ref{def:symmetric_form}. More precisely:

\begin{secassump}[for this section]
In this section, we fix a bounded domain $U$, a function $W \in L^{\infty}(U)$, and a symmetric form $\cZ$ on $H^s(U)$ for some $s\in (0,1)$. 
\end{secassump}
The goal is to introduce the Dirichlet and Neumann operators (see Definition~\ref{def:self_adjoint_operator_for_cE_W_cZ} below) corresponding to $\cE = \cE^U_{W, \mathcal{Z}}$ (see Definition~\ref{def:symmetric_form}) and study their spectral properties.

\begin{definition}
Let $\cQ$ be a symmetric form on a Hilbert space $H$.
Let $D$ be the set of $u \in \cD(\cQ)$ such that there exists a $\tilde u \in H$ such that $\cQ(u,v) = \langle \tilde u , v\rangle_H$ for all $v\in \cD(\cQ)$. For such $u$ the element $\tilde u$ is unique, and we will write $Au = \tilde u$.
Then $A$ on $D$ forms a linear operator on $H$, called the \emph{operator associated with} $\cQ$.
\end{definition}

\begin{definition}
\label{def:interpolation_constant}
Let $U$ be a bounded Lipschitz domain and $s\in (0,1)$. We define
\begin{align*}
    C^U_{\operatorname{IP}}[H^s] \defby
    \sup_{f \in H^1(U) \setminus \{0\}} \frac{\norm{f}_{H^s(U)}}{\norm{f}_{L^2(U)}^{1-s} \norm{f}_{H^1(U)}^s}.
\end{align*}
If $U$ is a bounded domain that is not necessarily Lipschitz, we define $C^U_{\operatorname{IP}}[H^s]$ similarly as above by replacing ``$H^a(U)$'' by ``$H^a_0(U)$'' for $a$ being either $s$ or $1$.
\end{definition}

\begin{lemma}\label{lem:estimate_of_Z}
Let $s\in (0,1)$.
  Let $\cZ$ be a bounded symmetric form on $H^s_0(U)$.
  Then, for any $\delta \in (0, 1)$ and $v \in H^1_0(U)$, we have
  \begin{equation*}
    \abs{\cZ(v, v)}
    \leq \delta  \int_U |\nabla v|^2  + \Big(\delta+\delta^{-\frac{s}{1 - s}} C^U_{\operatorname{IP}}[H^s_0]^{\frac{2}{1-s}}
    \formnorm{\cZ}_{H_0^s(U)}^{\frac{1}{1 - s}} \Big) \norm{v}_{L^2(U)}^2.
  \end{equation*}
 If $\cZ$ is a bounded symmetric form on $H^s(U)$, then the above statement holds with $H^s_0$  replaced by $H^s$.
\end{lemma}
\begin{proof}
  We only prove the claim for a bounded Lipschitz domain.
  One has $\abs{\cZ(v, v)} \leq \formnorm{\cZ}_{H^s(U)} \norm{v}_{H^s(U)}^2$.
  By interpolation and Young's inequality (using that $a^s b^{1-s} \le a +b$),
  \begin{equation*}
    \norm{v}^2_{H^s(U)} \leq C_{\operatorname{IP}}^U[H^s]^2 \norm{v}^{2s}_{H^1(U)}
    \norm{v}^{2(1-s)}_{L^2(U)}
    \leq  C_{\operatorname{IP}}^U[H^s]^2 (\eta \norm{v}^{2}_{H^1(U)}
    + \eta^{-\frac{s}{1-s}} \norm{v}_{L^2(U)}^2)
  \end{equation*}
  for any $\eta \in (0, \infty)$.
  Therefore,
  \begin{equation*}
    \abs{\cZ(v, v)} \leq \eta \formnorm{\cZ}_{H^s(U)} C_{\operatorname{IP}}^U[H^s]^2 \norm{v}^{2}_{H^1(U)}
    + \eta^{-\frac{s}{1-s}} \formnorm{\cZ}_{H^s(U)} C_{\operatorname{IP}}^U[H^s]^2 \norm{v}^{2}_{L^2(U)}.
  \end{equation*}
  We can choose $\eta$ so that $\delta = \eta \formnorm{\cZ}_{H^s(U)} C_{\operatorname{IP}}^U[H^s]^2$  and use that $\|v\|_{H^1(U)}^2 = \|v\|_{L^2(U)}^2 + \int_U |\nabla v|^2$.
\end{proof}

\begin{definition}\label{def:self_adjoint_operator_for_cE_W_cZ}
If $\cZ$ is a bounded symmetric form on $H_0^1(U)$, then we write $\cE^{\dir,U}_{W,\cZ}$ for $\cE^U_{W,\cZ}$ (recall Definition~\ref{def:symmetric_form}) with $\cD(\cE^{\dir,U}_{W,\cZ}) = e^W H_0^1(U)$ and let $\H^{\dir} = \H^{\dir,U}= \H^{\dir,U}_{W, \cZ}$ be the operator associated with $\cE^{\dir,U}_{W,\cZ}$ on $L^2(U)$.

If $\cZ$ is a bounded symmetric form on $H^1(U)$, then we write $\cE^{\neu,U}_{W,\cZ}$ for $\cE^U_{W,\cZ}$ with $\cD(\cE^{\neu,U}_{W,\cZ}) = e^W H^1(U)$ and let $\H^{\neu} = \H^{\neu,U} =  \H^{\neu,U}_{W, \cZ}$ be the operator associated with $\cE^{\neu,U}_{W,\cZ}$ on $L^2(U)$.
\end{definition}

\begin{proposition}\label{prop:closedness_of_form}
 Let $W \in L^{\infty}(U)$ and $\cZ$ be a bounded symmetric form on $H^s_0(U)$ for some $s \in (0, 1)$.
  Then $\cE^{\dir,U}_{W, \cZ}$ is closed   and $e^W C^{\infty}_c(U)$ is a core. Consequently, $\H^{\dir}$ is self-adjoint.

If $\cZ$ instead is a bounded symmetric form on $H^s(U)$, then  $\cE^{\neu,U}_{W,\cZ}$ is closed and $e^W C^\infty(\overline U)$ is a core.
 Consequently, $\H^{\neu}$ is self-adjoint.
\end{proposition}
\begin{proof}
In view of Lemma~\ref{lem:estimate_of_Z} and the symmetric form version of the Kato-Rellich theorem \cite[Theorem 1.33 in Chapter VI]{Kato1995},
we can assume $\cZ = 0$. Observe that for $u = e^W u^{\flat}$, $u^{\flat} \in H^1(U)$
\begin{equation*}
  e^{-2 \norm{W}_{L^{\infty}(U)}} \cE_{0, 0}(u^{\flat}, u^{\flat}) \leq \cE_{W, 0}(u, u)
  \leq e^{2 \norm{W}_{L^{\infty}(U)}} \cE_{0, 0}(u^{\flat}, u^{\flat}).
\end{equation*}
Therefore the claim follows as $\cE_{0,0}$ is closed and $C_c^\infty(U)$ is a core for $\cE_{0,0}$.

The self-adjointness of the corresponding operators follows as they are closed densely defined and symmetric,  cf. \cite[Section 4.4]{davies_1995}).
\end{proof}

By applying a standard result from the spectral theory, we can easily show that the spectrum of $\H^{\dir}$ on a bounded domain
and that of $\H^{\neu}$ on a bounded Lipschitz domain are discrete and  that the min-max formula (also known as the Courant-Fischer formula) holds for the eigenvalues.
\begin{calc}
\begin{lemma*}[{\cite[Corollary 4.2.3, Theorem 4.5.2, Theorem 4.5.3]{davies_1995}}]
\label{lem:eigenvalues_of_nonnegative_operator}
  Let $(H, \inp{\cdot}{\cdot})$ be a Hilbert space and let $A$ be a self-adjoint operator that is bounded from below.
  Let $B$ be the closed symmetric form associated with $A$ and let $D$ be a core of $B$.
  Then, we have
  \begin{align*}
    \lambda_k \defby& \inf_{L \subspace \D(A), \dim L = k}
    \sup_{f \in L, \norm{f} = 1} \inp{Af}{f}, \\
     =& \inf_{L \subspace \D(B), \dim L = k}
    \sup_{f \in L, \norm{f} = 1} B(f, f), \\
    =& \inf_{L \subspace D, \dim L = k}
    \sup_{f \in L, \norm{f} = 1} B(f, f),
  \end{align*}
  where $L \subspace X$ means $L$ is a subspace of $X$.
  Furthermore, if $\lim_{k \to \infty} \lambda_k = \infty$, then $\Spec(A) = (\lambda_k)_{k\in\N}$
  and the resolvent $(\lambda - A)^{-1}$ is compact for $\lambda \notin \Spec(A)$.
\end{lemma*}
\begin{remark*}\label{rem:compact_embedding_of_form_domain}
  If the embedding $\D(B) \hookrightarrow H$ is compact, then we have $\lim_{k \to \infty} \lambda_k =
  \infty$. See \cite[Corollary 4.2.3]{davies_1995}.
\end{remark*}
\end{calc}
\begin{proposition}\label{prop:eigenvalues_of_AH}
 The spectrum of $\H^{\dir}$ is given by a sequence of eigenvalues $(\lambda^{\dir}_k)_{k\in\N}$ counting multiplicities, such that $\lambda^{\dir}_1 \le \lambda^{\dir}_2 \le \cdots$. Moreover, by using the notation $\subspace$ for ``is a linear subspace of'', 
  \begin{align*}
    \lambda_k^{\dir} \defby \lambda_k^{\dir}(U; W, \cZ) \defby& \inf_{\substack{L \subspace \D(\H^{\dir}) \\ \dim L = k}}
    \sup_{\substack{u \in L \\ \norm{u}_{L^2(U)} = 1}}
    \inp{\H^{\dir} u}{u}_{L^2(U)}, \\
     =& \inf_{\substack{L \subspace e^W H^1_0(U) \\ \dim L = k}}
    \sup_{\substack{u \in L\\ \norm{u}_{L^2(U)} = 1}}
    \cE^U_{W, \mathcal{Z}}(u, u), \\
    =& \inf_{\substack{L \subspace  e^W C^{\infty}_c(U) \\ \dim L = k}}
    \sup_{\substack{u \in L \\ \norm{u}_{L^2(U)} = 1}}  \cE^U_{W, \mathcal{Z}}(u, u)
  \end{align*}
  and $\lim_{k \to \infty} \lambda_k^{\dir} = \infty$.
  In particular, $(\lambda - \H^{\dir})^{-1}$ is a compact operator for  all $\lambda$ that are not in the spectrum of $\H^{\dir}$.
If $U$ is a bounded Lipschitz domain,
an analogous statement for
  $\H^{\neu}$ holds if $H^1_0(U)$ and $C^{\infty}_c(U)$ are replaced by $H^1(U)$ and $C^{\infty}(\overline{U})$.
\end{proposition}
\begin{proof}
\begin{calc}In view of Lemma~\ref{lem:eigenvalues_of_nonnegative_operator} and
Remark~\ref{rem:compact_embedding_of_form_domain}, \end{calc}
By well-known results of spectral theory (see e.g., {\cite[Corollary 4.2.3, Theorem 4.5.2, Theorem 4.5.3]{davies_1995}} in combination with Proposition~\ref{prop:closedness_of_form}), it suffices to show that the form domain is compactly embedded in $L^2(U)$, which follows from the compact embeddings of Sobolev spaces  (see \cite[Theorem 8.11.2]{Bhattacharyya_2012} for the fact that
the embedding $H^1_0(U) \hookrightarrow L^2(U)$ is compact for any bounded domain $U$
\cite[Theorem 8.11.4]{Bhattacharyya_2012} for the fact that
the embedding $H^1(U) \hookrightarrow L^2(U)$ is compact for any bounded Lipschitz domain $U$).
\end{proof}
We show continuous dependence of the spectral structure with respect to $W$ and $\cZ$.
This follows from the result of \cite{kuwae_shioya_2003}.
\begin{definition}\label{def:convergence_of_symmetric_forms}
  Let $H$ be a Hilbert space  and $M>0$.  Let $(\cQ_n)_{n\in\N}$ and $\cQ$ be closed symmetric forms that are
  $M$-bounded from below.
  We use the following convention: if $u \notin \D(\cQ)$, then we set $\cQ(u, u) \defby \infty$.
  \begin{enumerate}
  \item
	\label{item:gamma_convergence}
  \cite[Definition 2.8]{kuwae_shioya_2003} We say the sequence $(\cQ_n)_{n\in\N}$
  \emph{$\Gamma$-converges} to $\cQ$, if the following hold:
    \begin{enumerate}[(i)]
    \item If the sequence $(u_n)_{n\in\N}$ converges to $u$ in $H$, then
    \begin{equation}
    \label{eqn:gamma_liminf}
      \cQ(u, u) \leq \liminf_{n \to \infty} \cQ_n(u_n, u_n).
    \end{equation}
    \item For any $u \in \D(\cQ)$, there exists a sequence $(u_n)_{n\in\N}$ in $H$ such that
    \begin{equation*}
      u_n \to u \text{ in } H \quad \text{and} \quad \lim_{n \to \infty} \cQ_n(u_n, u_n) = \cQ(u, u).
    \end{equation*}
    \end{enumerate}
  \item \cite[Definition 2.12]{kuwae_shioya_2003} The sequence $(\cQ_n)_{n\in\N}$ is said to be \emph{compact} if
  the condition
  \begin{equation*}
    \sup_{n \in \N} \cQ_n(u_n, u_n) + (M+1)\norm{u_n}_H^2 < \infty
  \end{equation*}
  implies $(u_n)_{n\in\N}$ is precompact in $H$, that is the sequence has a converges subsequence in $H$.
  \item \cite[Definition 2.13]{kuwae_shioya_2003} We say the sequence $(\cQ_n)_{n\in\N}$ \emph{converges compactly}
  to $\cQ$ if $(\cQ_n)_{n\in\N}$ $\Gamma$-converges to $\cQ$ and
  if $(\cQ_n)_{n\in\N}$ is compact. In that case, we write
  \begin{align*}
  \cQ_n \xrightarrow{\rm{compact}}_{n\rightarrow\infty} \cQ.
  \end{align*}
  \end{enumerate}
\end{definition}
\begin{lemma}\label{lem:compact_convergence_of_symmetric_forms}
 Let $H$ be a Hilbert space and $M>0$.
Suppose that $(\cQ_n)_{n\in\N}$ is a sequence of closed quadratic forms on $H$  that are $M$-bounded from below and converges compactly to $\cQ$.
  Let $A_n$ (resp. $A$) be the self-adjoint operator associated with $\cQ_n$ (resp. $\cQ$).
  \begin{enumerate}
  \item \textnormal{\cite[Theorem 2.4 and Theorem 2.5]{kuwae_shioya_2003}} For any bounded continuous function $f$ on $\R$, we have $\norm{f(A_n) - f(A)}_{H \to H} \to 0$.  In particular, $A_n \normresconv_{n\rightarrow \infty} A$ (see Definition~\ref{def:norm_resolvent_convergence}).

  \item \textnormal{\cite[Corollary 2.5]{kuwae_shioya_2003}} Let $\lambda_{k, n}$ (resp. $\lambda_k$) be the $k$-th eigenvalue   of $A_n$ (resp. $A$), counting multiplicities.
Then, we have $\lim_{n \to \infty} \lambda_{k, n} = \lambda_k$ for any $k$.
 Moreover, for any $k$ there exist (a choice of the) $k$-th eigenfunctions $\phi_{k, n}$ (resp. $\phi_k$) of $A_n$ (resp. $A$) such that
  $\phi_{k, n}$ converges to $\phi_k$ in $H$.
  \end{enumerate}
\end{lemma}

\begin{remark}
\label{remark:to_convergence_sym_forms}
In the proof of the following theorem we use the following elementary fact.
If $(a_n)_{n\in\N}$ is a sequence in $\R$ and $\liminfn a_n <\infty$, then there exists a strictly increasing function $\varphi : \N \rightarrow \N$ such that $\liminfn a_{\varphi(n)} = \liminfn a_n$ and $\supn a_{\varphi(n)} <\infty$.
\end{remark}
\begin{calc}
Let $(a_n)_{n\in\N}$ be as such and $a= \liminfn a_n$. For all $m\in\N$ there exists an $N_m\in\N$ such that for all $N\ge N_m$,  $\inf_{n \ge N} a_n \in (a - \frac1m, a+ \frac1m)$. Therefore, we can find a strictly increasing sequence $(k_m)_{m\in\N}$ such that $a_{k_m} \in (a - \frac1m , a +\frac1m)$ for all $m$. Set $\varphi(m)=k_m$ for all $m$. Then $\liminfn a_{\varphi(n)} = a$ and $\supn a_{\varphi(n)} \le a+1$.
\end{calc}

In the following theorem we consider the compact convergence for the symmetric forms $\cE_{W_n,\cZ_n}^{\dir,U}$ on $H_0^s(U)$ and the symmetric forms $\cE_{W_n,\cZ_n}^{\neu,U}$ on $H^s(U)$ as defined in Definition~\ref{def:self_adjoint_operator_for_cE_W_cZ}.

\begin{theorem}\label{thm:convergence_of_symmetric_forms}
  Let $s \in [0, 1)$.
  Suppose that $W_n \rightarrow W$ in $C(\overline U)$.
\begin{itemize}
\item  If $\formnorm{\cZ_n - \cZ}_{H^s_0(U)} \xrightarrow{n\rightarrow \infty} 0$,
then
\begin{align*}
\cE_{W_n, \cZ_n}^{\dir,U} \xrightarrow{\rm{compact}}_{n\rightarrow\infty} \cE_{W, \cZ}^{\dir,U}.
\end{align*}
\item  If $U$ is a bounded Lipschitz domain and $\formnorm{\cZ_n - \cZ}_{H^s(U)} \xrightarrow{n\rightarrow \infty} 0$,
then
\begin{align*}
\cE_{W_n, \cZ_n}^{\neu,U} \xrightarrow{\rm{compact}}_{n\rightarrow\infty} \cE_{W, \cZ}^{\neu,U}.
\end{align*}
\end{itemize}
\end{theorem}
\begin{proof}
We only prove the second statement.
  We first show that $(\E^{\neu}_{W_n, \cZ_n})_{n\in\N}$ is compact. Suppose $\sup_{n\in\N} \E^{\neu}_{W_n, \cZ_n}(u_n, u_n)
  + (M+1) \norm{u_n}_{L^2(U)}^2 < \infty$. We set $u_n^{\flat} \defby e^{-W_n} u_n$.
  By Lemma~\ref{lem:estimate_of_Z}, we have $\sup_{n\in\N} \norm{u^{\flat}_n}_{H^1(U)} < \infty$
 because for any $\delta \in (0,1)$, 
    \begin{align*}
  & \cE^{\neu}_{W_n,\cZ_n}(u_n,u_n)
   = \int_U e^{2W_n} \abs{\nabla u^{\flat}_n}^2 + \cZ_n (u_n^\flat,u_n^\flat) \\
  & \ge e^{-2\|W_n\|_{L^\infty}} (1-\delta) \int_U \abs{\nabla u^{\flat}_n}^2  - \Big(\delta + C \delta^{-\frac{s}{1-s}} \formnorm{\cZ_n}_{H^s(U)}^{\frac{1}{1-s}}\Big) \|u_n^\flat \|_{L^2(U)}^2,
  \end{align*}
  \tymchange{
and thus $\sup_{n\in\N} \norm{u_n^\flat }_{H^1(U)} < \infty$ as $W_n \rightarrow W$ in $C(\overline U)$.
Since the embedding $H^1(U) \hookrightarrow L^2(U)$ is compact, the sequence $(u_n^\flat)_{n\in\N}$ is also precompact in $L^2(U)$.
As $u_n = e^{W_n} u_n^\flat$ and $\|W_n - W\|_{L^\infty(\overline U)} \rightarrow 0$, also $(u_n)_{n\in\N}$ is precompact in $L^2(U)$. 
  }

Next we show that $(\E^{\neu}_{W_n, \cZ_n})_{n\in\N}$ $\Gamma$-converges to $\E^{\neu}_{W, \cZ}$.
Since (ii) of Definition~\ref{def:convergence_of_symmetric_forms}~\ref{item:gamma_convergence} is trivial,   we focus on showing (i) Definition~\ref{def:convergence_of_symmetric_forms}~\ref{item:gamma_convergence}.

  Suppose that $(u_n)_{n\in\N}$ converges to $u$ in $L^2(U)$.
As we want to show \eqref{eqn:gamma_liminf}, we may assume $\liminfn \cE_{W_n,\cZ_n}^{\neu,U}(u_n,u_n) <\infty$ and by Remark~\ref{remark:to_convergence_sym_forms} we may as well assume $\supn \cE_{W_n,\cZ_n}^{\neu,U}(u_n,u_n)<\infty$ (by possibly considering a subsequence), so that by the above $\supn \norm{u_n^\flat}_{H^1(U)}<\infty$.
It suffices to show
\begin{align}
\label{eq:fatoru_sobolev}
  \int_U e^{2W} \abs{\nabla u^{\flat}}^2 \leq
\liminf_{n \to \infty} \int_U e^{2W_n} \abs{\nabla u^{\flat}_n}^2, \\
\label{eqn:convergence_cZ}
    \lim_{n \to \infty} \cZ_n(u^{\flat}_n, u^{\flat}_n) = \cZ(u^{\flat}, u^{\flat}).
\end{align}
Since the sequence $(u_n^{\flat})_{n\in\N}$ is  bounded in $H^1(U)$ and converges to $u^{\flat}$ in $L^2(U)$, by interpolation it converges to $u^\flat$ in $H^s(U)$ from which \eqref{eqn:convergence_cZ} follows.
\begin{calc}
Indeed, we have
\begin{align*}
&\cZ(u,u) - \cZ_n(u_n,u_n)
 = \cZ(u-u_n,u-u_n) - \cZ(u_n,u_n) + 2\cZ(u,u_n) - \cZ_n(u_n,u_n)\\
& = \cZ(u-u_n,u-u_n) - 2\cZ(u_n,u_n) + 2\cZ(u,u_n) + [\cZ- \cZ_n](u_n,u_n)\\
& = \cZ(u-u_n,u-u_n) + 2\cZ(u-u_n,u_n) + [\cZ- \cZ_n](u_n,u_n).
\end{align*}
\end{calc}
  For $v \in C_c^{\infty}(U)$,
  \begin{equation*}
    \inp{\nabla u^{\flat}_n}{v}_{L^2(U)} = - \inp{u^{\flat}_n}{\nabla v}_{L^2(U)}
    \to -\inp{u^{\flat}}{\nabla v}_{L^2(U)} = \inp{\nabla u^{\flat}}{v}_{L^2(U)},
  \end{equation*}
  which implies that $(\nabla u_n^{\flat})_{n\in\N}$ converges weakly to $\nabla u^{\flat}$ in $L^2(U)$.
  Therefore, $(e^{W_n} \nabla u_n^{\flat})_{n\in\N}$ converges weakly to $e^W \nabla u^{\flat}$ in $L^2(U)$ and
  this implies \eqref{eq:fatoru_sobolev} (this follows for example by the dual representation of the norm on $L^2(U)$).
\end{proof}

\subsection{Estimates of eigenvalues}
\label{subsec:estimates_eigenvalues}

\tymchange{This section serves as preparation for Section~\ref{subsec:IDS} on the integrated density of states and its results are not used for the construction of Anderson Hamiltonians (Theorem~\ref{thm:main_AH} and Thereom~\ref{thm:main_Neumann_AH}).}
\begin{secassump}[for this section]
We assume the following throughout this section:
\begin{itemize}
  \item $W$ is a continuous function defined on $\R^d$.
  \item $s \in [0, 1)$, $\cU$ is the  collection of bounded Lipschitz domains.
  \item $\cZ^U$ is a bounded symmetric form on $H^s(U)$ for all $U\in \cU$.
\end{itemize}
 For each $U\in \cU$, we let
 $\cE^U = \cE^U_{W, \cZ^U}$ be the symmetric form defined in Definition~\ref{def:symmetric_form}.
Recall the notations $\H^{\neu,U}$ and $\H^{\dir,U}$ from
Definition~\ref{def:self_adjoint_operator_for_cE_W_cZ} and $\lambda^{\neu}_k(U)$, $\lambda^{\dir}_k(U)$ from  Proposition~\ref{prop:eigenvalues_of_AH}.
\end{secassump}

\begin{remark}
For Dirichlet boundary conditions we do not necessarily need to consider Lipschitz domains.
Indeed, if $\cU$ would instead be the collection of all bounded domains and $\cZ^U$ a bounded symmetric form on $H^s_0(U)$ for all $U\in \cU$, then the statements of Lemma~\ref{lem:monotonicity_of_IDS}, Lemma~\ref{lem:box_decomposition_neumann}~\ref{item:dirichlet_ev_estimates_decomp} and Lemma~\ref{lem:upper_estimates_N}~\ref{item:box_decomposition_dirichlet_reversed} remain valid.
\end{remark}

\begin{definition}\label{def:eigenvalue_counting_function}
For $\# \in \{\neu, \dir\}$  and $U\in\cU$, we define the \emph{eigenvalue counting functions}  $\ecf^{\#}(U, \lambda)$  for $\lambda \in \R$  by
  \begin{equation*}
    \ecf^{\#}(U, \lambda) \defby \ecf^{\#}_{W, \cZ}(U, \lambda) \defby \sum_{k = 1}^{\infty}
    \indic_{\{\lambda_k^{\#}(U; W, \cZ) \leq \lambda \}}.
  \end{equation*}
  We set $\ecf^{\#}_0(U, \lambda) \defby \ecf^{\#}_{0, 0}(U, \lambda)$, which is the eigenvalue counting function of the Neumann or the Dirichlet Laplacian on $U$.

For $L>R>0$ we set
\begin{align}
\label{eqn:def_U_L_R_and_C_partial_U_L_R}
U_L^R \defby U_L \cap B(\partial U_L, R),
\quad C(\partial U_L, R) \defby \set{x \in U_L \given d(x, \partial U_L) = R} .
\end{align}
(Observe $C(\partial U_L, R) = \partial U_L^R  \setminus \partial U_L$.)
We denote by $H^1_{\fm,R}(U_L^R)$
the closure in $H^1(U_L^R)$  with respect to $H^1$-norm of the set
    \begin{equation*}
      \set{\phi \in C^{\infty}(\overline{U_L^R}) \given \phi = 0 \mbox{ on a neighborhood of }  C(\partial U_L, R)}.
    \end{equation*}
    Let $\ecf^{\fm}_0(U_L^R, \lambda)$ be the eigenvalue counting function of the operator associated with the symmetric form
    $(u,u) \mapsto \int_{U_L^R} \abs{\nabla u}^2$ with the domain $H^1_{\fm,R}(U^R_L)$.
\end{definition}

\begin{lemma}\label{lem:monotonicity_of_IDS}
 Let $U, U_1,U_2\in \cU$, $U_1\subseteq U_2$ and $\lambda \in \R$. Then
  \begin{align*}
      \ecf^{\dir}(U, \lambda) &\leq \ecf^{\neu}(U, \lambda), \\
    \ecf^{\dir}(U_1, \lambda) &\leq \ecf^{\dir}(U_2, \lambda).
  \end{align*}
\end{lemma}
\begin{proof}
  Since $H^1_0(U) \subseteq H^1(U)$, the min-max formula (Lemma~\ref{prop:eigenvalues_of_AH}) implies $\lambda_k^{\dir}(U) \ge  \lambda_k^{\neu}(U)$ for all $k$, and thus the first inequality. The second also follows by the min-max formula, as  $H^1_0(U_1) \subseteq H^1_0(U_2)$.
\end{proof}

\begin{lemma}\label{lem:estimate_of_eigenvalue_counting_functions_of_AH}
  Let $U \in \cU$,  $s\in (0,1)$,  $\theta \in (0, \infty)$ and $\lambda\in \R$.
  We set
  \begin{align*}
  \Lambda^\pm_{\lambda,\theta}( W, \cZ) \defby (1 \pm \theta) e^{\pm 4 \norm{W}_{L^{\infty}(U)}} (\lambda \pm A_\pm),
  \end{align*}
  where
  \begin{equation*}
    A_{\pm}
	 \defby A_{\pm,\theta}^{W,\cZ}
    \defby  \theta + \Big(\frac{\theta}{1 \pm \theta} \Big)^{-\frac{s}{1-s}} C_{\operatorname{IP}}^U[H^s]^{\frac{2}{1-s}}
    e^{(2 \pm \frac{2s}{1-s}) \norm{W}_{L^{\infty}(U)}} \formnorm{\cZ}_{H^s(U)}^{\frac{1}{1-s}}.
  \end{equation*}
  Then, one has
  \begin{align*}
   \ecf^{\dir}_0(U, \Lambda^-_{\lambda, \theta}(  W, \cZ)) \leq \ecf^{\dir}(U, \lambda) \leq \ecf^{\neu}(U, \lambda)
  \leq \ecf^{\neu}_0(U, \Lambda^+_{\lambda,\theta}(  W, \cZ)).
\end{align*}
\end{lemma}
\begin{proof}
  We only prove $\ecf^{\neu}(U, \lambda) \leq \ecf^{\neu}_0(U, \Lambda^+_{\lambda,\theta}( U, W, \cZ))$; the other inequality follows similarly.
  By setting $\delta \defby \frac{\theta}{1 + \theta} e^{-2 \norm{W}_{L^{\infty}(U)}}$, Lemma~\ref{lem:estimate_of_Z} yields
  \begin{equation*}
    \abs{\cZ(u^{\flat}, u^{\flat})} \leq \delta \int_U \abs{\nabla u^{\flat}}^2 + A_+ e^{-2 \norm{W}_{L^{\infty}(U)}} \int_U (u^{\flat})^2.
  \end{equation*}
  One has $\int_U e^{2 W} \abs{\nabla u^{\flat}}^2 \geq e^{-2 \norm{W}_{L^{\infty}(U)}} \int_U \abs{\nabla u^{\flat}}^2$.
  Therefore, by Proposition~\ref{prop:eigenvalues_of_AH},
  \begin{align*}
    \lambda_k^{\neu}(U) &= \inf_{ \substack{L \subspace H^1(U) , \\ \dim L = k}} \sup_{\substack{u^{\flat} \in L , \\ \int e^{2 W} (u^{\flat})^2 = 1}}
    \int_U e^{2 W} \abs{\nabla u^{\flat}}^2 + \cZ(u^{\flat}, u^{\flat})  \\
    &\geq \inf_{ \substack{L \subspace H^1(U) , \\ \dim L = k}} \sup_{\substack{u^{\flat} \in L , \\ \int e^{2 W} (u^{\flat})^2 = 1}}
    \frac{e^{-2 \norm{W}_{L^{\infty}(U)}}}{1 + \theta} \int_U \abs{\nabla u^{\flat}}^2 - A_+ e^{-2 \norm{W}_{L^{\infty}(U)}} \int_U (u^{\flat})^2  \\
    &\geq \frac{e^{-2 \norm{W}_{L^{\infty}(U)}}}{1 + \theta}
    \Big\{ \inf_{ \substack{L \subspace H^1(U) , \\ \dim L = k}} \sup_{\substack{u^{\flat} \in L , \\ \int e^{2 W} (u^{\flat})^2 = 1}}
    \int_U \abs{\nabla u^{\flat}}^2 \Big\} - A_+.
  \end{align*}
  We compute
  \begin{align}
\notag
    \inf_{ \substack{L \subspace H^1(U) , \\ \dim L = k}} \sup_{\substack{u^{\flat} \in L , \\ \int e^{2 W} (u^{\flat})^2 = 1}}
    \int_U \abs{\nabla u^{\flat}}^2
    &= \inf_{ \substack{L \subspace H^1(U) , \\ \dim L = k}} \sup_{\substack{u^{\flat} \in L,\\ \int e^{2 W} (u^{\flat})^2 \leq 1}} \int_U \abs{\nabla u^{\flat}}^2 \\
\notag
    &\geq \inf_{ \substack{L \subspace H^1(U) , \\ \dim L = k}} \sup_{\substack{u^{\flat} \in L,\\ \int (u^{\flat})^2 \leq e^{-2 \norm{W}_{L^{\infty}(U)}}}}
    \int_U \abs{\nabla u^{\flat}}^2 \\
    \label{eqn:estimate_eigenvalues_with_W}
    &= e^{-2 \norm{W}_{L^{\infty}(U)}} \lambda_k^{\neu}(U; 0, 0).
  \end{align}
  Therefore,
  \begin{equation*}
    \lambda^{\neu}_k(U) \geq \frac{e^{-4 \norm{W}_{L^{\infty}(U)}}}{1 + \theta} \lambda^{\neu}_k(U; 0, 0) - A_+
  \end{equation*}
  and the claimed inequality follows.
\end{proof}
As Lemma~\ref{lem:estimate_of_eigenvalue_counting_functions_of_AH} suggests, we need estimates of $\ecf^{\#}_0(U, \lambda)$.
The following lemma is sufficient for our purpose.
\begin{lemma}\label{lem:estimate_of_N_0}
  Let $U$ be a bounded Lipschitz domain.
  \begin{enumerate}
    \item
    \label{item:estimate_N_0_m}
    Then, there exist $C_U, R_U>0$ such that
    \begin{equation*}
      \ecf^{\fm}_0(U_L^R, \lambda) \leq C_U R^d (1+\lambda)^{\frac{d}{2}} L^{d-1}
    \end{equation*}
    for every $L \geq 1$, $\lambda \geq 0$ and $R \geq R_U$.
    \item
    \label{item:estimates_N_0_d_and_N_0_n}
    \textnormal{\cite[Theorem 3.1 and Theorem 3.2]{safarov_filonov_2010}}
     There exists a $C'_U>0$ such that
    \begin{align*}
      \frac{\abs{B(0, 1)}}{(2 \pi)^d} \abs{U} \lambda^{\frac{d}{2}} - C'_U \lambda^{\frac{d - 1}{2}} \log \lambda
  &    \leq \ecf^{\dir}_0(U, \lambda) \\
  &  \leq \ecf^{\neu}_0(U, \lambda)
      \leq \frac{\abs{B(0, 1)}}{(2 \pi)^d} \abs{U} \lambda^{\frac{d}{2}} + C'_U \lambda^{\frac{d - 1}{2}} \log \lambda
    \end{align*}
    for every $\lambda \geq 2$.
  \end{enumerate}
\end{lemma}
\begin{proof}
  The claim (a) follows from the proof of \cite[Theorem 6.2]{doi_iwatsuka_mine_2001}.
  Indeed, we can combine the estimates (6.20), (6.23), (6.24) and (6.25) therein.
\end{proof}

\begin{definition}
\label{def:ordering_symm_forms}
Let $\cQ_1$ and $\cQ_2$ be closed symmetric forms on Hilbert spaces $H_1$ and $H_2$ that are bounded from below.
We write  $\cQ_1 \prec \cQ_2$  if there exists an isometry $\Phi: H_2 \to H_1$ such that $\Phi(\D(\cQ_2)) \subseteq
  \D(\cQ_1)$ and
  $\cQ_1(\Phi(f), \Phi(f)) \leq \cQ_2(f, f)$ for every $f \in \D(\cQ_2)$.
\end{definition}

\begin{lemma}\label{lemma:eigenvalue_comparison}
Let $\cQ_1,\cQ_2$ be as in Definition~\ref{def:ordering_symm_forms} and let $A_1$ and $A_2$ be the associated self-adjoint operators.
Suppose that the spectrum of $A_1$ and that of $A_2$ are discrete. We denote them by
  $(\mu_k(A_1))_{k\in\N}$ and $(\mu_k(A_2))_{k\in\N}$, respectively. Then, $\cQ_1 \prec \cQ_2$ implies $\mu_k(A_1) \leq \mu_k(A_2)$ for every $k$.
\end{lemma}
\begin{proof}
This follows from the min-max formula.
  \begin{calc}
  (Lemma~\ref{lem:eigenvalues_of_nonnegative_operator}),
  \begin{align*}
    \mu_k(A) &= \inf_{K \subspace \D(\cQ_A), \dim K = k} \sup_{g \in K, \norm{g} = 1}
    \cQ_A(g, g) \\
    &\leq \inf_{L \subspace \D(\cQ_B), \dim L = k} \sup_{f \in L, \norm{f} = 1}
    \cQ_A(\Phi(f), \Phi(f)) \\
    &\leq \inf_{L \subspace \D(\cQ_B), \dim L = k} \sup_{f \in L, \norm{f} = 1}
    \cQ_B(f, f) = \mu_k(B).
  \end{align*}
  \end{calc}
\end{proof}

In order to compare the eigenvalue counting functions on different domains, it will be convenient to introduce the following symmetric forms.

\begin{definition}
\label{def:sum_sym_forms}
Let $J\in \N$.
Let $\cE_j$ be a symmetric form on a Hilbert space $H_j$ for $j\in \{1,\dots,J\}$.
We define the symmetric form $\bigoplus_{j=1}^J \cE_j$ on the Hilbert space $\bigoplus H_j$ by $\D(\bigoplus_{j=1}^J \cE_j) = \bigoplus_{j=1}^J \D(\cE_j)$ and for $v = \bigoplus_{j=1}^J v_j$ with $v_j \in \D(\cE_j)$, $(\bigoplus_{j=1}^J \cE_j) (v,v) := \sum_{j=1}^J \cE_j(v_j,v_j)$.

Observe that if $A_j$ is the operator associated with $\cE_j$ for all $j$, then the operator $\bigoplus_{j=1}^J A_j$ defined by $\bigoplus_{j=1}^J A_j v = \bigoplus_{j=1}^J A_j v_j$ for $v= \bigoplus_{j=1}^J v_j \in \bigoplus_{j=1}^J \D(A_j) = : \D(\bigoplus_{j=1}^J A_j) $ is the operator associated with $\bigoplus_{j=1}^J \cE_j$.
In particular, the principal eigenvalue of $\bigoplus_{j=1}^J A_j$ is given by $\minjJ \lambda_1(A_j)$, where $\lambda_1(A_j)$ is the principal eigenvalue of $A_j$ for all $j$.

Moreover, if $A_j$ has a countable spectrum for all $j$, then one has $\ecf_{\bigoplus_{j=1}^J \cE_j} = \sum_{j=1}^J \ecf_{\cE_j}$, where $\ecf_\cQ$ is the eigenvalue counting function corresponding to the operator associated with $\cQ$.

Observe that using this notation, one also has $\ecf_{a\cQ+b \fI}(\lambda) = \ecf_{\cQ}(\frac{\lambda-b}{a})$, where $\fI$ is the symmetric form $\fI(v,v) = \norm{v}^2$.
Moreover, $\cQ_1 \prec \cQ_2$ implies $\ecf_{\cQ_1} \ge \ecf_{\cQ_2}$.
\end{definition}

\begin{lemma}\label{lem:box_decomposition_neumann}
Let $U,U_1,\dots,U_J\in \cU$,
$U = \cup_{j=1}^J U_j$ with $\overline U_j \cap \overline U_k = \partial U_j \cap \partial U_k$ for $j \neq k$.
\begin{enumerate}
\item
\label{item:dirichlet_ev_estimates_decomp}
If
\begin{align}
\label{eqn:cZ_restriction_consistence}
\cZ^{U_j}(v,v) = \cZ^{U}(v,v), \qquad v\in H_0^1(U_j), j\in \{1,\dots, J\},
\end{align}
 then
\begin{equation}\label{eq:box_decomposition_dirichlet}
    \lambda^{\dir}_1(U) \leq \minjJ \lambda_1^{\dir}(U_j)
    \quad \mbox{and} \quad
    \ecf^{\dir}(U, \lambda) \geq \sum_{j=1}^J \ecf^{\dir}(U_j, \lambda).
  \end{equation}
\item
\label{item:neumann_ev_estimates_decomp}
If
\begin{align}
\label{eq:cZ_decomposable}
\cZ^U(v, v) = \sum_{j=1}^J \cZ^{U_j}(v\vert_{U_j}, v\vert_{U_j}),
\quad v\in H^1(U),
\end{align}
 then
  \begin{equation*}
    \lambda^{\neu}_1(U) \geq \minjJ  \lambda_1^{\neu}(U_j)
    \quad \mbox{and} \quad
    \ecf^{\neu}(U, \lambda) \leq \sum_{j=1}^J \ecf^{\neu}(U_j, \lambda).
  \end{equation*}
\end{enumerate}
\end{lemma}
\begin{proof}
\ref{item:dirichlet_ev_estimates_decomp} follows from the fact that $\oplus_{j=1}^J H^1_0(U_j) \subseteq H^1_0(U)$.

\ref{item:neumann_ev_estimates_decomp}
As $L^2(U)$ and $\bigoplus_{j=1}^J L^2(U_j)$ are isometric,  $H^1(U) \subseteq \oplus_{j=1}^J H^1(U_j)$ and
  $\cE^U(u, u) = \sum_{j=1}^J \cE^{U_j}(u\vert_{U_j}, u\vert_{U_j})$ for $u \in e^{W} H^1(U)$, we have  $\oplus_{j=1}^J \cE^{\neu,U_j} \prec \cE^{\neu,U}$.
Now both inequalities follow from Lemma~\ref{lemma:eigenvalue_comparison} (see also the comments in Definition~\ref{def:sum_sym_forms}).
\end{proof}

\begin{remark}
\label{rem:example_decomposable}
Observe that \eqref{eqn:cZ_restriction_consistence} and \eqref{eq:cZ_decomposable} hold for $U,U_1,\dots,U_J$ as in Lemma~\ref{lem:upper_estimates_N}~\ref{item:box_decomposition_dirichlet_reversed} if $\delta\in (0,1)$, $\sigma \in [0,\infty)$,
$Y \in \csp^{-1 + \delta}(\R^d)$ and $\cZ^U$ is given by $\cZ^U_Y$ for $U \in \cU$ as in Theorem~\ref{theorem:examples_bd_sym_with_weighted_input}~\ref{item:sym_form_on_domain}; or if  $Y \in C^1(\R^d)$ and
 $\cZ^U$ is given for $U\in \cU$ by
  \begin{equation*}
    \cZ^U(v, v) \defby \int_{\partial U} v^2 \nabla Y \cdot \dd \bm{S},
  \end{equation*}
or if it is a linear combination of the above examples.
\end{remark}

We can give a ``reversed'' inequality of \eqref{eq:box_decomposition_dirichlet}.
 First we present an auxiliary lemma which is based on the IMS formula, see \cite{Si83}.

\begin{lemma}
\label{lem:IMS_trick}
Let $J\in \N$ and $U, U_1,\dots, U_J\in \cU$.
Let $\eta_1,\dots,\eta_J$ be smooth functions $\R^d \rightarrow [0,1]$ such that there exists an $A>0$ such that
\begin{align*}
\Big\| \sum_{j=1}^J |\nabla \eta_j|^2 \Big\|_{L^\infty(\R^d)} \le A, \quad j\in \{1,\dots, J\},
\qquad \sum_{j=1}^J \eta_j^2 = 1 \mbox{ on } U.
\end{align*}
Then
\begin{align*}
\cE^U_{W,0}(u,u) \ge \sum_{j=1}^J \cE_{W,0}^U(\eta_j u, \eta_j u) - A\|\eta_j u\|_{L^2}^2, \quad u\in e^W H^1(U).
\end{align*}
\end{lemma}
\begin{proof}
Observe that $\sum_{j=1}^J \nabla (\eta_j^2) =0$.
Let $u = e^{W} u^{\flat}$ with $u^{\flat} \in H^1(U)$.
Then
\begin{align*}
\eta_j^2 |\nabla u^\flat|^2 =
\abs{\nabla (\eta_j u^{\flat})}^2
    - \abs{\nabla \eta_j}^2 (u^{\flat})^2 - \nabla(\eta_j^2) \cdot u^{\flat} \nabla u^{\flat},
\end{align*}
and therefore 
  \begin{align*}
   \MoveEqLeft[1]
    \int_{U} e^{2 W} \abs{\nabla u^{\flat}}^2
    = \sum_{j=1}^J \int_U
    e^{2 W} \eta_j^2 \abs{\nabla u^{\flat}}^2 \\
    &\geq \sum_{j=1}^J \Big\{ \int_U
    e^{2 W}  \abs{\nabla (\eta_j u^{\flat})}^2 - A \norm{\eta_j u^\flat }_{
    L^2}^2 \Big\}.
  \end{align*}
\end{proof}

\begin{remark}
So far we have only considered the Anderson Hamiltonians on bounded domains, which means bounded open subsets of $\R^d$.
However, whether one considers $U$ or $\overline U$, does not intrinsically make a difference. In the following lemma and further on we will consider the Anderson Hamiltonian on closed boxes of the form $[0,L]^d$ for example. One may read $(0,L)^d$ instead in order to align with the rest of the text, though we write $[0,L]^d$ as this is more common in the literature.
\end{remark}

\begin{lemma}\label{lem:upper_estimates_N}
  Let $Z \in \csp^{-1 + \delta}(\R^d)$ with $\delta \in (0, 1)$ and
  suppose  $\cZ^U= \cZ^U_Z$ as in Theorem~\ref{theorem:examples_bd_sym_with_weighted_input}~\ref{item:sym_form_on_domain} for every Lipschitz domain $U$.
\begin{enumerate}
\item
\label{item:box_decomposition_dirichlet_reversed}
  There exists a  $K > 0$ (which depends only on $d$) such that for all $U\in \cU$, all $l,L>0$ with $L>2l $ and $n\in\N$,
  \begin{align*}
    \lambda_1^{\dir}([0, nL]^d)
    &\geq \min_{k \in \Z^d \cap [-1, n+1]^d} \lambda_1^{\dir}
    (kL + [-l, L + l]^d) - \frac{K}{l^2}, \\
    \ecf^{\dir}([0, nL]^d, \lambda) &\leq \sum_{k \in \Z^d \cap [-1, n+1]^d}
    \ecf^{\dir}\Big(kL + [-l, L + l]^d, \lambda + \frac{K}{l^2} \Big).
  \end{align*}
\item
\label{item:upper_bound_of_neumann_ids_by_dirichlet_ids}
There exists a $K>0$ (depending only on $d$) such that for all $U\in \cU$  and $s\in (1-\delta,1)$  there exist $C_{s, U}, R_U>0$ such that for all $L \ge 1$, $\lambda \in \R$ and $R \ge R_U$,
  \begin{align*}
    \ecf^{\neu}(U_L, \lambda) & \leq \ecf^{\dir}(U_L, \lambda + K R^{-2}) \\
&  \qquad   + C_{s, U} R^{d} L^{d - 1}  e^{d \frac{3 - 2s}{1 - s}\norm{W}_{L^{\infty}(U_L)}}
     (1 + \max\{\lambda, 0\} + \formnorm{\cZ}_{H^s(U)}^{\frac{1}{1-s}} )^{\frac{d}{2}} .
  \end{align*}
\end{enumerate}
\end{lemma}
\begin{proof}
\ref{item:box_decomposition_dirichlet_reversed}
  According to \cite[Lemma 8.2]{chouk2020asymptotics}, there exists a smooth function $\eta: \R^d
  \to [0, 1]$ and a $K>0$, such that $\eta = 1$ on $[0, L - 2l]^d$,
  $\supp(\eta) \subseteq [-2l, L]^d$,
  $\norm{\nabla \eta}_{L^{\infty}(\R^d)}^2 \leq \frac{K}{2^d l^2}$ and
  \begin{equation*}
    \sum_{k \in \Z^d} \eta(x + k L)^2 = 1 \quad \mbox{ for } x \in \R^d.
  \end{equation*}
  We set $\eta_k \defby \eta(\,\cdot + (l, l, \ldots, l) + Lk)$ for $k\in \Z^d$.
  Observe that $\supp(\eta_k) \subseteq kL + [-l, L + l]^d$, and,
  \begin{align*}
  \sum_{k \in [-1, n+1]^d \cap \Z^d} \eta_k^2 = 1 \mbox{ on } [0,nL]^d, 
  \qquad \Big\| \sum_{k\in [-1, n+1]^d \cap \Z^d} |\nabla \eta_k|^2 \Big\|_{L^\infty(\R^d)} \le \frac{K}{l^2}. 
  \end{align*}
  Therefore, the map
  \begin{equation*}
    \Phi: L^2([0, nL]^d) \rightarrow \bigoplus_{k \in [-1, n+1]^d \cap \Z^d} L^2
    (kL + [-l, L + l ]^d), \quad  u \mapsto (\eta_k u)_{k \in [-1, n+1]^d \cap \Z^d}.
  \end{equation*}
  is an isometry and $\Phi(e^{W} H^1_0([0, nL]^d))
  \subseteq \oplus_{k \in [-1, n+1]^d \cap \Z^d} e^{W} H^1_0(kL + [-l, L + l]^d)$.

Observe that for $u\in H^1_0([0, nL]^d)$, for  $\phi \in C_c^\infty(\R^d)$ such that $\phi =1$ on a neighborhood of $[0,nL]^d$,
\begin{align*}
\cZ_Z^{[0,nL]^d} (u,u)
= \inp{ \phi Z}{ \1_{[0,nL]^d} u^2}
& = \sum_{k \in [-1, n+1]^d \cap \Z^d} \inp{ \phi Z}{ \1_{[0,nL]^d} \eta_k^2 u^2 } \\
& = \sum_{k \in [-1, n+1]^d \cap \Z^d} \cZ_Z^{kL + [-l, L + l]^d}  (\eta_k u, \eta_k u).
\end{align*}
  Therefore, by Lemma~\ref{lem:IMS_trick}, 
  \begin{equation*}
    \E_{W,\cZ}^{[0, nL]^d}(u, u)
    \geq \sum_{k \in [-1, n+1]^d \cap \Z^d} \Big\{ \E_{W,\cZ}^{kL + [-l, L + l]^d}(\eta_k u, \eta_k u)
    - \frac{K}{l^2} \norm{\eta_k u}_{L^2(kL + [-l, L + l]^d)}^2 \Big\}.
  \end{equation*}
and thus $\E_{W,\cZ}^{[0, nL]^d} \succ \bigoplus_{k \in [-1, n+1]^d \cap \Z^d} [ \E_{W,\cZ}^{kL + [-l, L + l]^d} - \frac{K}{l^2} \fI ]$ (where $\fI$ is as in Definition~\ref{def:sum_sym_forms}), from which we conclude the estimates (use the discussion in Definition~\ref{def:sum_sym_forms}).

\ref{item:upper_bound_of_neumann_ids_by_dirichlet_ids}
  As given in \cite[Proposition 4.3]{doi_iwatsuka_mine_2001}, there exist smooth functions $\alpha_1$ and $\alpha_2$ on $\R^d$ and a $K>0$ (only depending on $d$) such that
  \begin{align*}
&     \supp(\alpha_1) \subseteq U_L \setminus B(\partial U_L, \tfrac{R}{2}), \quad \supp(\alpha_2) \subseteq B(\partial U_L,R),  \\
& \alpha_1^2 + \alpha_2^2 =1 \mbox{ on a neighborhood of } U_L ,
\qquad    \Big\|  \sum_{j=1}^2 \abs{\nabla \alpha_j}^2 \Big\|_{L^\infty(\R^d)} \leq K R^{-2}.
  \end{align*}
Recall the definitions of $U_L^R$ and $H^1_{\fm,R}(U^R_L)$ from Definition~\ref{def:eigenvalue_counting_function}.
  The map
  \begin{equation*}
    \Phi: L^2(U_L) \rightarrow  L^2(U_L) \oplus L^2(U^R_L), \quad u \mapsto \alpha_1 u \oplus \alpha_2 u
  \end{equation*}
  is an isometry and $\Phi(e^W H^1(U_L)) \subseteq e^W H^1_0(U_L) \oplus e^W H^1_{\fm,R}(U^R_L)$.

Observe that $\cZ^{U_L}(\alpha_2 u,  \alpha_2 u) = \cZ^{U_L^R}( \alpha_2 u, \alpha_2 u)$ as $\supp \alpha _2 \cap U_L \subseteq U_L^R$.
Therefore, by Lemma~\ref{lem:IMS_trick}, for $u\in H^1(U_L)$, 
  \begin{equation*}
    \cE^{U_L}(u, u) \geq
 \cE^{U_L}(\alpha_1 u,\alpha_1 u)
+ \cE^{U_L^{R}}(\alpha_2 u,\alpha_2 u)
   -  \sum_{j=1}^2 K R^{-2} \norm{\alpha_j u}_{L^2(U_L)}^2.
  \end{equation*}
 By applying Lemma~\ref{lem:estimate_of_Z} with $\delta= \frac{e^{-2\|W\|_{L^\infty(U_L)}}}{2}$, we obtain
  \begin{align*}
    \cE^{U_L^{R}}(\alpha_2 u, \alpha_2 u)
    & \geq \frac{e^{-2 \norm{W}_{L^{\infty}( U_L )}}}{2} \int_{U^R_L} \abs{\nabla (\alpha_2 u^{\flat})}^2
    - A \norm{\alpha_2 u}^2_{L^2(U^R_L)} ,
  \end{align*}
  where
  \begin{equation*}
    A \defby \frac{e^{-2\|W\|_{L^\infty(U_L)}}}{2} + 2^{\frac{s}{1-s}}
    e^{\frac{2 s }{1-s} \norm{W}_{L^{\infty}( U_L )}}
    C_{\operatorname{IP}}^{U_L}[H^s]^{\frac{2}{1-s}}
    \formnorm{\cZ}_{H^s(U)}^{\frac{1}{1-s}}.
  \end{equation*}
 Therefore,
$\cE^{\neu,U_L} \succ (\cE^{\dir,U_L} - KR^{-2} \fI) + (\cE^{\fm,R,U_L^R} - (KR^{-2}+A) \fI )$, where $\cE^{\fm,R,U_L^R}$ is the restriction of $\cE^{\neu,U_L^R}$ to $H_{\fm,R}(U_L^R)$, and thus, by Lemma~\ref{lemma:eigenvalue_comparison} (the additional factor $e^{2\norm{W}_{L^\infty(U_L)}}$ is explained similarly as in \eqref{eqn:estimate_eigenvalues_with_W}),
  \begin{equation*}
    \ecf^{\neu}(U_L, \lambda) \leq \ecf^{\dir}(U_L, \lambda + KR^{-2})
    + \ecf^{\fm}_0(U^R_L, 2 e^{4 \norm{W}_{L^{\infty}( U_L )}} (\lambda + KR^{-2} + A)).
  \end{equation*}
  By Lemma~\ref{lem:estimate_of_N_0}~\ref{item:estimate_N_0_m}, for $R \geq R_{U}$,
  \begin{align*}
&     \ecf^{\fm}_0(U^R_L, 2 e^{4 \norm{W}_{L^{\infty}(U)}} (\lambda + KR^{-2} + A)) \\
& \qquad     \lesssim_U R^{d} L^{d-1}
    \big\{ e^{4 \norm{W}_{L^{\infty}(U)}} (1+\max\{\lambda, 0\}  + K R^{-2} + A) \big\}^{\frac{d}{2}} .
  \end{align*}
  It remains to apply Lemma~\ref{lem:scaling_of_iota}, more specifically \eqref{eqn:estimate_scaling_Embed}: $C^{U_L}_{\operatorname{IP}}[H^s] \lesssim_{s,U} 1$.
\end{proof}

\section{The Anderson Hamiltonian with Dirichlet and Neumann boundary conditions}\label{sec:eigenvalues_of_AH}

Based on the results obtained Section~\ref{sec:symmetric_forms} we can give the definition of the Anderson Hamiltonian
$-\Delta - \xi$ with Dirichlet and with Neumann boundary conditions, and show that it is the limit of the operators $-\Delta - \xi_{\epsilon} + c_{\epsilon}$, where
the $c_{\epsilon}$'s are as in Assumption~\ref{assump:all_three_assump}~\ref{item:construction_assump}.

We discuss the construction of the Anderson Hamiltonian with Dirichlet boundary conditions in Section~\ref{subsec:dirichlet_anderson} and with Neumann boundary conditions in Section~\ref{subsec:Neumann}. 
In Section~\ref{subsec:IDS} we consider the integrated density of states associated to the Anderson Hamiltonian.


\subsection{The Dirichlet Anderson Hamiltonian}
\label{subsec:dirichlet_anderson}

\begin{secassump}[for this section]
In this section we impose the construction assumption~\ref{assump:all_three_assump}~\ref{item:construction_assump}. 
\end{secassump}

As discussed in Remark~\ref{rem:towards_construction_assumptions}, 
we will choose $M$ to be a random variable with values in $\N_0$ such that $F(W_M^\epsilon) =e^{2W_M^\epsilon}$, where $F$ is as in Assumption~\ref{assump:all_three_assump}~\ref{item:construction_assump}:

\begin{definition}
\label{def:dirichlet_AH}
 Let $U$ be a bounded domain and let $r\in (1-\delta,1)$. Using the notation of Theorem~\ref{theorem:examples_bd_sym_with_weighted_input}~\ref{item:sym_form_on_domain}, for $N\in\N_0$ we define the following symmetric forms on $H_0^r(U)$:
  \begin{equation}\label{eq:def_of_cZ_AH}
    \cZAH_N[U] \defby
	\cZ^U_{\YAH_N}, \quad
    \cZAHeps_N[U] \defby
    \cZ^U_{\YAHeps_N}.
  \end{equation}
For $\delta_- \in (0, \delta]$ and $\gamma \in (0, \infty)$, we set
  \begin{equation}
  \label{eq:M_AH}
    \MAH(U, \delta_-; \gamma) \defby \inf \set{N \in \N \given  \norm{\WAH_n}_{\csp^{\delta_-}(U)} \leq \gamma \mbox{ for all } n \ge N },
  \end{equation}
  which is finite by Lemma~\ref{lemma:estimate_W_n_on_U_L}.
  Recalling the notation from Definition~\ref{def:self_adjoint_operator_for_cE_W_cZ}, for
  $M =  \MAH(U, \deltamin; 1)$ (which attains its values in $\N_0$ by Lemma~\ref{lemma:estimate_W_n_on_U_L}), we define
  \begin{equation*}
    \HAH^{\dir} \defby \HAH^{\dir,U} \defby \cH_{\WAH_M, \cZ_\MAH[U]}^{\dir,U}, \quad
    \HAHeps^{\dir} \defby \HAHeps^{\dir,U} \defby \cH_{\WAHeps_M, \cZAHeps_M[U]}^{\dir,U}.
  \end{equation*}
  Recalling Proposition~\ref{prop:eigenvalues_of_AH}, we set
  \begin{equation*}
    \lambdaAHk^{\dir}(U) \defby \lambda_k^{\dir}(U; \WAH_M, \cZAH_M[U]), \quad
    \lambdaAHkeps^{\dir}(U) \defby \lambda_k^{\dir}(U; \WAHeps_M, \cZAHeps_M[U]).
  \end{equation*}
\end{definition}

\begin{theorem}\label{thm:convergence_of_Dirichlet_AH}
  For $\epsilon \in (0,1)$, we have 
\begin{align}
\label{eqn:H_eps_fd_equals_smooth_AH}
\HAHeps^{\dir} = -\Delta -\xi_{\epsilon} + c_{\epsilon}.
\end{align}
  Let $(\epsilon_n)_{n\in\N}$ be a sequence in $(0, 1)$ such that $\epsilon_n \to 0$.
  Then, there exist a subsequence $(\epsilon_{n_m})_{m\in\N}$ and a subset
  $\Omega_1 \subseteq \Omega$ of $\P$-probability $1$ such that on $\Omega_1$ the following holds:
  for any bounded domain $U$, one has
\begin{align}
  \label{eqn:limit_resolvent_sense_dirichlet}
  \HAHepsnm^{\dir,U} \normresconv_{m\rightarrow \infty}
  \HAH^{\dir,U},
\end{align}
for all $k\in\N$
  \begin{equation*}
    \lim_{m \to \infty} \lambdaAHkepsnm^{\dir}(U) = \lambdaAHk^{\dir}(U),
  \end{equation*}
and there exist an eigenfunction $\phi_k$ of $\HAH^{\dir,U}$ corresponding to $\lambdaAHk^{\dir}$ and eigenfunctions $\phi_{k,m}$ of $\HAH^{\dir,U}_{\epsilon_{n_m}}$ corresponding to $\lambdaAHkepsnm^{\dir}(U)$ for all $m\in\N$ such that 
\begin{align*}
\phi_{k,m} \xrightarrow{m\to \infty} \phi_k \quad \mbox{ in } L^2(U). 
\end{align*} 
\end{theorem}
\begin{proof}
\eqref{eqn:H_eps_fd_equals_smooth_AH} follows by our choice of $M$ (see \eqref{eq:M_AH}) and by Lemma~\ref{lem:sym_form_transform_Neumann} and 
the construction assumption~\ref{assump:all_three_assump}~\ref{item:construction_assump} (see also the discussion in Section~\ref{subsec:strategy_ang_techniques} and Remark~\ref{rem:towards_construction_assumptions}). 

Let  $\sigma \in (0, 1)$.
By~\ref{assump:all_three_assump}~\ref{item:construction_assump},
there exist a subsequence $(\epsilon_{n_m})_{m\in\N}$ and a subset
  $\Omega_1 \subseteq \Omega$ of $\P$-probability $1$ such that on $\Omega_1$, for every $N \in \N_0$,
  \begin{align*}
\lim_{m\rightarrow \infty}  \norm{ \XAHepsnm - \XAH }_{\csp^{\deltamin,\sigma}(\R^d)} =0,
\qquad
\lim_{m\rightarrow \infty} \norm{\YAHepsnm_N - \YAH_N}_{\csp^{-1 + \deltamin, \sigma}(\R^d)}.
  \end{align*}
Observe that by Lemma~\ref{lemma:estimate_C_delta_U_by_weighted} and Corollary~\ref{cor:estimates_of_G_N_and_H_N}, because $W_M^\epsilon - W_M = G_M * (X^\epsilon - X)$, for all $\epsilon \in (0,1)$,
\begin{align*}
\norm{ \WAHeps_M - \WAH_M}_{\csp^{\deltamin}(U)}
\lesssim_{U,  \delta,\sigma}
\norm{ \XAHeps - \XAH }_{\csp^{-2+\delta,\sigma}(\R^d)},
\end{align*}
and that by Theorem~\ref{theorem:examples_bd_sym_with_weighted_input}~\ref{item:sym_form_on_domain}, for $r\in (1-\deltamin,1)$,
\begin{align*}
\formnorm{ \cZAHeps_M - \cZ_\MAH }_{H^r_0(U)}
\lesssim_{\delta,p,U} \norm{\YAHeps_M - Y_\MAH}_{\csp^{-1 + \deltamin, \sigma}(\R^d)}.
\end{align*}
Therefore, the claim follows from  Theorem~\ref{thm:convergence_of_symmetric_forms} and Lemma~\ref{lem:compact_convergence_of_symmetric_forms}. 
\end{proof}

\begin{remark}
\label{remark:larger_M_dirichlet}
Let $\fM$ be a random variable with values in $\N_0$ such that $\fM \ge \MAH(U, \deltamin; 1)$.
Then, almost surely, $\HAH^{\dir}(U) = \H^{\dir}_{\WAH_\fM,\cZAH_\fM[U]}(U)$, because $ \H^{\dir}_{\WAHeps_\fM,\cZAHeps_\fM[U]}(U) = -\Delta- \xi_\epsilon + c_\epsilon = \HAHeps^{\dir}(U)$ and similarly as in Theorem~\ref{thm:convergence_of_Dirichlet_AH}, $\H^{\dir}_{\WAH_\fM,\cZAH_\fM[U]}(U)$ is the limit (in the sense of \eqref{eqn:limit_resolvent_sense_dirichlet}) of $ \H^{\dir}_{\WAHeps_\fM,\cZAHeps_\fM[U]}(U)$.
We will apply this in Section~\ref{subsec:IDS} with $\fM = \MAH(U, \deltamin; \gamma)$ for $\gamma \in (0,1]$.
\end{remark}

\subsection{The Neumann Anderson Hamiltonian}\label{subsec:Neumann}

As described in the beginning of Section~\ref{sec:symmetric_forms} (below Lemma~\ref{lem:sym_form_transform_Neumann}), the boundary term will be dealt with by the decomposition into symmetric forms $\tilde \cZ$ and $\widehat \cZ$. Let us first consider the ingredients for the latter symmetric form.

\begin{definition}
\label{def:hat_Y_AH}
  Let $U$ be a bounded Lipschitz domain.
For $N\in\N_0$ we define
\begin{align*}
\hatYAH_N &\defby G_N \conv (\XAH - \xi) + (G_N - G_0) \conv \xi, \\
\hatYAHeps_N &\defby G_N \conv (\XAHeps - \xi_\epsilon) + (G_N - G_0) \conv \xi_\epsilon, \qquad \epsilon \in (0,1).
\end{align*}
\end{definition}

\begin{lemma}
\label{lemma:estimate_hat_Y_N_AH}
Let $U$ be a bounded Lipschitz domain. Then for $\delta \in (\frac12,1)$ and $N\in\N_0$
\begin{align*}
\norm{\hatYAH_N}_{\cC^{1+\delta,\sigma}}
& \lesssim_{U, \delta} \norm{\XAH - \xi}_{\csp^{-2 + 2 \delta, \sigma}(\R^d)}
    + 2^N \norm{\xi}_{\csp^{-2 + \delta, \sigma}(\R^d)} , \\
    \norm{\hatYAH_N- \hatYAHeps_N}_{\cC^{1+\delta,\sigma}}
& \lesssim_{U, \delta} \norm{\XAH - \xi - (\XAHeps - \xi_\epsilon)}_{\csp^{-2 + 2 \delta, \sigma}(\R^d)}
    + 2^N \norm{\xi - \xi_\epsilon}_{\csp^{-2 + \delta, \sigma}(\R^d)} , 
\end{align*}
Moreover, $\norm{\hatYAH_N - \hatYAHeps_N}_{\cC^{1+\delta,\sigma}}$ converges in probability to $0$.
\end{lemma}
\begin{proof}
The estimates follow by 
the construction assumption~\ref{assump:all_three_assump}~\ref{item:construction_assump}
and Corollary~\ref{cor:estimates_of_G_N_and_H_N}.
\end{proof}

\begin{definition}\label{def:cZ_Neumann}
Impose the construction and Neumann assumptions, \ref{assump:all_three_assump}~\ref{item:construction_assump} and \ref{item:neumann_assump}, 
  let $U$ be a bounded Lipschitz domain.
Let $r\in (1-\deltamin,1)$.
Using the notation of Theorem~\ref{theorem:examples_bd_sym_with_weighted_input},
we define the following symmetric forms on $H^r(U)$, for $N\in\N_0$
  \begin{align*}
 \tildecZAH_N[U] & \defby \tilde \cZ_{\tilde Y^{U}, e^{2 \WAH_N}} ^{U}, \qquad
  \tildecZAHeps_N[U] \defby \tilde \cZ_{\tilde Y^{U}_\epsilon, e^{2 \WAHeps_N}} ^{U},  \\
     \hatcZAH_N[U]
    & \defby \widehat \cZ^{U}_{\hatYAH_N, e^{2 \WAH_N} } ,
    \qquad    \hatcZAHeps_N[U]
    \defby \widehat \cZ^{U}_{\hatYAHeps_N, e^{2 \WAHeps_N} } .
  \end{align*}
We furthermore make abuse of notation (compared to the symmetric forms on $H_0^r(U)$ as in \eqref{eq:def_of_cZ_AH}) and define the following symmetric forms on $H^r(U)$,
  \begin{equation}\label{eq:def_of_cZ_AH_neumann}
    \cZAH_N[U] \defby
	\cZ^U_{\YAH_N}, \quad
    \cZAHeps_N[U] \defby
    \cZ^U_{\YAHeps_N}.
  \end{equation}
  Then we define
  \begin{align*}
    \cZnAH_N[U]  & \defby \cZAH_N[U] +  \tildecZAH_N[U] + \hatcZAH_N[U], \\
    \cZnAHeps_N[U] & \defby \cZAHeps_N[U] +   \tildecZAHeps_N[U] + \hatcZAHeps_N[U], \qquad \epsilon \in (0,1).
  \end{align*}
  Recalling the notations from Definition~\ref{def:self_adjoint_operator_for_cE_W_cZ} and
  Proposition~\ref{prop:eigenvalues_of_AH}, for $M = \MAH(U, \deltamin; 1)$ (see \eqref{eq:M_AH}) we set
  \begin{align*}
    \HAH^{\neu}
    & \defby \HAH^{\neu,U} \defby \cH^{\neu,U}_{\WAH_M, \cZnAH_M[U]},
    &  \lambdaAHk^{\neu}(U)
    &\defby \lambda_k^{\neu}(U; \WAH_M, \cZnAH_M[U]), \\
        \HAHeps^{\neu}
        & \defby \HAHeps^{\neu,U} \defby \cH^{\neu}_{\WAHeps_M, \cZnAHeps_M[U]},
     &  \lambdaAHkeps^{\neu}(U)
     & \defby \lambda_k^{\neu}(U; \WAHeps_M, \cZnAHeps_M[U]).
  \end{align*}
\end{definition}
\begin{theorem}\label{thm:convergence_of_Neumann_AH}
Impose the construction and Neumann assumptions, \ref{assump:all_three_assump}~\ref{item:construction_assump} and \ref{item:neumann_assump}. 
Let $U$ be a bounded Lipschitz domain.
For $\epsilon \in (0,1)$,
\begin{align}
\label{eqn:H_eps_fn_equals_smooth_AH}
\HAHeps^{\neu} = -\Delta -\xi_{\epsilon} + c_{\epsilon}
\end{align}
  Then, one has (see Definition~\ref{def:norm_resolvent_convergence} for ``$\normresconv$'')
\begin{align*}
\HAHeps^{\neu,U}
	& \normresconv_{\epsilon \downarrow 0}
\HAH^{\neu,U} \mbox{ in probability,} \\
 \lambdaAHkeps^{\neu}(U)
	 & \arroweps  \lambdaAHk^{\neu}(U)   \mbox{ in probability,}  \qquad  k \in \N  .
\end{align*}
and for all $k\in\N$ there exist an eigenfunction $\phi_k$ of $\HAH^{\neu,U}$ corresponding to $\lambdaAHk^{\neu}(U)$ and eigenfunctions $\phi_{k,m}$ of $\HAH^{\neu,U}_{\epsilon}$ corresponding to $\lambdaAHkeps^{\neu}(U)$ for all $\epsilon>0$ such that 
\begin{align*}
\phi_{k,m} \xrightarrow{m\to \infty} \phi_k \quad \mbox{ in } L^2(U) \mbox{ in probability}. 
\end{align*} 
\end{theorem}
\begin{proof}
\eqref{eqn:H_eps_fn_equals_smooth_AH} follows by 
the construction assumption~\ref{assump:all_three_assump}~\ref{item:construction_assump}
and Lemma~\ref{lem:sym_form_transform_Neumann} (see also the discussion in Section~\ref{subsec:strategy_ang_techniques} and Remark~\ref{rem:neumann_construction}). 

The convergences follows in a similar as in Theorem~\ref{thm:convergence_of_Dirichlet_AH}, that is we again apply Theorem~\ref{thm:convergence_of_symmetric_forms} and Lemma~\ref{lem:compact_convergence_of_symmetric_forms}. 
The convergence of $W_n$ and $\cZ_N^\epsilon[U]$ follow in the same way as in Theorem~\ref{thm:convergence_of_Dirichlet_AH}. 
The convergence of $\tilde \cZ_N^\epsilon[U]$ is guaranteed by the Neumann assumption~\ref{assump:all_three_assump}~\ref{item:neumann_assump} and Theorem~\ref{theorem:examples_bd_sym_with_weighted_input}~\ref{item:sym_form_boundary_tilde} and the convergence of $\widehat \cZ_N^\epsilon[U]$ by 
Theorem~\ref{theorem:examples_bd_sym_with_weighted_input}~\ref{item:sym_form_boundary_hat} and 
 Lemma~\ref{lemma:estimate_hat_Y_N_AH}.
\end{proof}

Without the Neumann assumption~\ref{assump:all_three_assump}~\ref{item:neumann_assump}, 
we can still construct an artificial Neumann Anderson Hamiltonian,
which will be used in Section~\ref{subsec:IDS} as a technical tool.
\begin{definition}\label{def:artificial_Neumann_AH}
Impose the construction assumption~\ref{assump:all_three_assump}~\ref{item:construction_assump}
  and
  let $U$ be a bounded Lipschitz domain and $r\in (1-\deltamin,1)$.
  For $M = M(U, \deltamin; 1)$, we set
  \begin{equation*}
    \artAHU \defby \cH_{\WAH_M, \cZAH_M[U]}^{\neu,U},
    \quad
    \artevk(U) \defby \lambda_k^{\neu}(U; \WAH_M, \cZAH_M[U]).
  \end{equation*}
\end{definition}

\begin{remark}
\label{remark:larger_M_neumann}
Similar to Remark~\ref{remark:larger_M_dirichlet}, for $\fM$ being a random variable with values in $\N_0$ such that $\fM \ge M(U,\deltamin,1)$, one has, almost surely,
$\artAHU = \overline {\cH}^{\neu,U}_{\WAH_\fM, \cZAH_\fM[U]}$
and, if one imposes the Neumann condition~\ref{assump:all_three_assump}~\ref{item:neumann_assump}, then 
$ \HAH^{\neu,U} = {\cH}^{\neu,U}_{\WAH_\fM, \cZAH_\fM[U]}$.
\end{remark}

\subsection{Integrated density of states}\label{subsec:IDS}
\tymchange{
The aim of this section is to study the integrated density of states (IDS) associated to
the Anderson Hamiltonian with potential $\xi$. Namely, we prove Theorem~\ref{thm:main_IDS}: 
\begin{itemize}
  \item Theorem~\ref{thm:main_IDS}~\ref{item:IDS_existence} and~\ref{item:IDS_Neumann} are proven by Theorem~\ref{thm:IDS_general}.
  \item Theorem~\ref{thm:main_IDS}~\ref{item:IDS_approx} is proven by Theorem~\ref{thm:IDS_epsilon}.
  \item Theorem~\ref{thm:main_IDS}~\ref{item:IDS_right_tail} and~\ref{item:IDS_left_tail} are proven by Theorem~\ref{thm:IDS_tails}.
\end{itemize}
Compared to the previous two sections, this section is rather technical involving quantitative estimates obtained before, especially in Section~\ref{subsec:estimates_eigenvalues}. 
Before stating the assumptions for this section, we give some estimates on $M$ and the symmetric forms $\cZ$ and $\cZnAH$ in the following lemma which will be used later on (they will be reformulated in Remark~\ref{rem:estimates_on_cZ} under Assumption~\ref{ass:section_IDS}). 
}



\begin{lemma}\label{lem:estimate_of_cZ_AH_Neumann}
Impose the construction assumption~\ref{assump:all_three_assump}~\ref{item:construction_assump}. 
Let $\delta_- \in (0,\delta)$ and $\sigma \in (0,\infty)$. 
Let $U$ be a bounded domain. 
We will use the following abbreviation for the $M$ defined in \eqref{eq:M_AH}: 
$ M_{L,\delta_-,\gamma}  \defby \MAH(U_L, \deltaone; \gamma)$.
Recall also the definition $\aAH$ and $\kAH$ from \eqref{eqn:def_a_AH}. 
\begin{enumerate}
\item
\label{item:estimate_M_gamma}
There exists a $C = C(U, \deltaone,  \sigma) \in (0, \infty)$   such that for all $L \geq 1$ and $\gamma \in (0, \infty)$ one has
  \begin{equation}
  \label{eqn:estimate_M_gamma}
    \MAH(U_L, \deltaone; \gamma) \leq 1 + (\deltatwo - \deltaone)^{-1} \log_2
    \big(C \gamma^{-1} L^{\sigma} \norm{\XAH}_{\csp^{-2 + \deltatwo, \sigma}(\R^d)} \big).
  \end{equation}
 \item 
 \label{item:estimate_first_part_symm_f_neumann}
Suppose $U$ is a bounded Lipschitz domain. 
For every $r\in (1-\delta_-,1)$,  $L \geq 1$ and $\gamma \in (0, \infty)$,
  \begin{equation*}
    \formnorm{\cZAH_{M_{L,\delta_-,\gamma}}[U_L]}_{H^r(U_L)} \lesssim_{U, \deltaone, \deltatwo, \sigma}
    \gamma^{-(\deltatwo - \deltaone)^{-1} \kAH}
    L^{(2+  \kAH (\deltatwo - \deltaone)^{-1}) \sigma}
    \norm{\XAH}_{\csp^{-2 + \deltatwo, \sigma}(\R^d)}^{ (\deltatwo - \deltaone)^{-1} \kAH } \aAH.
  \end{equation*}

\item
\label{item:estimate_complete_part_symm_f_neumann}
Additionally, impose the Neumann assumption~\ref{assump:all_three_assump}~\ref{item:neumann_assump} and that $1/2 < \deltaone$. 
Suppose that $r\in (\frac32 -\deltaone, 1)$, $\epsilon \in (0,\deltaone-\frac12)$, $p \in (2, \infty)$ and $q\in (1,2)$  are such that $\frac1p + \frac1q=1$ and  \eqref{eqn:condition_beta_and_s} holds for $s=r$.
  Then, for every $\sigma \in (0, \infty)$, $L \geq 1$ and $\gamma \in (0, 1]$,
\end{enumerate}
  \begin{multline*}
    \formnorm{\cZnAH_{M_{L,\delta_-,\gamma}}[U_L]}_{H^r(U_L)} \\
    \lesssim_{U, \deltaone, \deltatwo, \sigma, \epsilon }
    \gamma^{-(\deltatwo - \deltaone)^{-1} \kAH}
    L^{(2 + (\deltatwo - \deltaone)^{-1} \kAH)\sigma}
    \norm{\XAH}_{\csp^{-2 + \deltatwo, \sigma}(\R^d)}^{(\deltatwo - \deltaone)^{-1} \kAH} (\aAH +
    \norm{\xi}_{\csp^{-2+\delta, \sigma}(\R^d)} ) \\
    + L^{2\epsilon}
     \norm{\tilde Y^{U_L}}_{B_{p, p}^{-2 + \delta}(\R^d)}
    + L^{\sigma} \norm{\XAH - \xi}_{\csp^{-2 + 2 \deltatwo, \sigma}(\R^d)}.
  \end{multline*}
\end{lemma}
\begin{proof}
\ref{item:estimate_M_gamma}
\eqref{eqn:estimate_M_gamma} is a direct consequence of Lemma~\ref{lemma:estimate_W_n_on_U_L}.

\ref{item:estimate_first_part_symm_f_neumann} 
 follows by
Theorem~\ref{theorem:examples_bd_sym_with_weighted_input}~\ref{item:sym_form_on_domain} (see \eqref{eqn:bound_cZ_Y_U_L_H_0}) since by definition of $\kAH$ and $\aAH$ (see \eqref{eqn:def_a_AH}),
  \begin{equation}
  \label{eqn:estimate_Y_N_AH}
    \norm{\YAH_N}_{\csp^{-1 + \deltaone, \sigma}(\R^d)}
    \leq 2^{\kAH N} \aAH,
  \end{equation}
  and by using \eqref{eqn:estimate_M_gamma}.

For~\ref{item:estimate_complete_part_symm_f_neumann} we use~\ref{item:estimate_first_part_symm_f_neumann} and estimate $\tildecZAH_{M_{L,\gamma}}[U_L]$ and $ \hatcZAH_{M_{L,\gamma}}[U_L]$. By Theorem~\ref{theorem:examples_bd_sym_with_weighted_input}~\ref{item:sym_form_boundary_tilde} and~\ref{item:sym_form_boundary_hat} we have
\begin{align*}
 \tildecZAH_{M_{L,\gamma}}[U_L]
& \lesssim_{\delta,\epsilon,p,U} L^{2\epsilon} \norm{e^{2\WAH_{M_{L,\gamma}}}}_{C^{\delta}(U_L)} \norm{\tilde Y^{U_L}}_{B_{p,p}^{-2+\delta}(\R^d)}, \\
\hatcZAH_{M_{L,\gamma}}[U_L]
& \lesssim_{\delta,\sigma, U}
 L^\sigma \norm{e^{2\WAH_{M_{L,\gamma}}}}_{C^{\delta}(U_L)} \norm{\hatYAH_{M_{L,\gamma}}}_{\cC^{1+{\delta},\sigma}(\R^d)}.
\end{align*}
As for any $x,y\in \R$,
  \begin{align*}
  |e^x - e^y| = e^x |1 - e^{y-x}|
  \le C e^x |y-x|
  \le C e^{x\vee y} |y-x|,
  \end{align*}
by definition of $M_{L,\gamma}$ we have
$\norm{e^{2\WAH_{M_{L,\gamma}}}}_{\csp^{\delta}(U)} \lesssim
2\gamma  e^{2\gamma} \le 2e^2 $.

Therefore, we obtain the desired inequality by the estimate of $\norm{\hatYAH_{M_{L,\gamma}}}_{\cC^{1+\delta,\sigma}(\R^d)}$ from Lemma~\ref{lemma:estimate_hat_Y_N_AH} and \eqref{eqn:estimate_M_gamma}.
\end{proof}

\begin{assumption} \textbf{(Assumptions for this section)}
\label{ass:section_IDS}
  Throughout the rest of this section, we impose the construction assumption~\ref{assump:all_three_assump}~\ref{item:construction_assump} and the ergodic assumption~\ref{assump:all_three_assump}~\ref{item:ergodic_assump}. 
  We fix $\deltaone, \deltatwo, r \in (0, 1)$, $\sigma \in (0, 1/4)$ and $p \in (\tymchange{8d}, \infty)$ such that
  $\deltaone < \deltatwo $
  and
  \begin{equation}\label{eq:condition_of_sigma}
    (2 + (\deltatwo - \deltaone)^{-1} \kAH(\deltaone)) d \sigma < 1 - r.
  \end{equation}
  We set
  \begin{align*}
&    \vertiii{\bm{\xi}}
    \defby 1 + \norm{\XAH}_{\csp^{-2 + \deltatwo, \sigma}(\R^d)}
    + \aAH + \norm{\XAH - \xi}_{\csp^{-2 + 2 \deltatwo, \sigma}(\R^d)}, \\
&     \aAHeps \defby \aAHeps(\deltamin,\sigma) \defby
      \sup_{N \in \N} 2^{-\kAH N} \norm{\YAHeps_N}_{\csp^{-1+\deltamin, \sigma}(\R^d)} \in L^p(\P), \\
&    \vertiii{\bm{\xi_\epsilon}}
    \defby 1 + \norm{\XAHeps}_{\csp^{-2 + \deltatwo, \sigma}(\R^d)}
    + \aAHeps + \norm{\XAHeps - \xi_\epsilon}_{\csp^{-2 + 2 \deltatwo, \sigma}(\R^d)}, \qquad \epsilon>0.
  \end{align*}
Whenever we impose the Neumann assumption \ref{assump:all_three_assump}~\ref{item:neumann_assump}, we implicitly also assume that $\deltaone > \frac{1}{2}$; 
the condition \eqref{eqn:condition_beta_and_s} is satisfied for $q\in (1,2)$ such that $\frac1p + \frac1q=1$ and some $\epsilon\in (0,\deltaone - \frac12)$;
and for any bounded Lipshitz domain $U$, we set (with $\tilde Y^U$ as in Assumption~\ref{assump:all_three_assump}~\ref{item:neumann_assump})
  \begin{align*}
&    \vertiii{\xi}_{\partial U} \defby \sup_{L \in \N} L^{-\frac{1}{4}}
    \norm{ \tilde Y^{U_L} }_{B_{p, p}^{-1 + \delta}(\R^d)} ,
\quad 
      \vertiii{\xi_\epsilon}_{\partial U} \defby \sup_{L \in \N} L^{-\frac{1}{4}}
    \norm{ \tilde Y^{U_L}_\epsilon }_{B_{p, p}^{-1 + \delta}(\R^d)}, \quad \epsilon>0.
  \end{align*}
\end{assumption}

\begin{remark}
\label{rem:estimates_on_cZ}
  By the construction assumption \ref{assump:all_three_assump}~\ref{item:construction_assump}, 
   $\vertiii{\bm{\xi}} \in L^q(\P)$ for every $q \in [1, \infty)$ and under the Neumann assumption~\ref{assump:all_three_assump}~\ref{item:neumann_assump}, in particular \eqref{eq:tilde_Y_U_conv}, $\vertiii{\xi}_{\partial U} \in L^q(\P)$ for every $q \in [1, \infty)$.
By Lemma~\ref{lem:estimate_of_cZ_AH_Neumann} and the condition on $\sigma$, \eqref{eq:condition_of_sigma}, there exists an $m\in\N$ such that for all bounded Lipschitz domains $U$, for all $L\ge 1$ and $\gamma \in (0,1]$, for $M_{L,\deltaone,\gamma} = M(U_L,\deltaone,\gamma)$,
  \begin{align}
  \label{eq:estimate_of_cZ_AH_artificial_Neumann}
    \formnorm{\cZAH_{M_{L,\deltaone,\gamma}}[U_L]}_{H^r(U_L)}
    & \lesssim_U \gamma^{-(\deltatwo - \deltaone)^{-1} \kAH} L^{\frac{1-r}{d}} \vertiii{\bm{\xi}}^m, \\
      \label{eq:estimate_of_cZ_AH_n_Neumann}
     \formnorm{\cZnAH_{M_{L,\deltaone,\gamma}}[U_L]}_{H^r(U_L)}
    & \lesssim_U \gamma^{-(\deltatwo - \deltaone)^{-1} \kAH} L^{\frac{1-r}{d}} (\vertiii{\bm{\xi}}^m + \vertiii{\xi}_{\partial U}) \quad \mbox{ under~\ref{assump:all_three_assump}~\ref{item:neumann_assump}}.
  \end{align}
For \eqref{eq:estimate_of_cZ_AH_n_Neumann} observe that $\sigma \le \frac{1-r}{d}$ and that we may choose $\epsilon>0$ as in Lemma~\ref{lem:estimate_of_cZ_AH_Neumann}~\ref{item:estimate_complete_part_symm_f_neumann} such that $2\epsilon<\frac{1-r}{d}$.
In \eqref{eq:estimate_of_cZ_AH_artificial_Neumann} one may replace ``$\cZAH$'' and ``$\boldsymbol{\xi}$'' by ``$\cZAHeps$'' and ``$\boldsymbol{\xi_\epsilon}$'' and in \eqref{eq:estimate_of_cZ_AH_n_Neumann}  one may replace ``$\cZnAH$'', ``$\boldsymbol{\xi}$'' and ``$\xi$'' by ``$\cZnAHeps$'', ``$\boldsymbol{\xi_\epsilon}$'' and ``$\xi_\epsilon$''.
\end{remark}

\tymchange{Our first goal is to construct the IDS for the Anderson Hamiltonian, see Definition~\ref{def:IDS} We begin by introducing some notation related to eigenvalue counting functions.}
\begin{definition}
Recall the notation from Definition~\ref{def:eigenvalue_counting_function}.
For a bounded domain $U$, $\lambda \in \R$ and $\epsilon>0$,  we set
$\NAH^{\dir} \defby \ecf^{\dir}_{\WAH_M, \cZAH_M}$ and $\NAHeps^{\dir} \defby \ecf^{\dir}_{\WAHeps_M,\cZAHeps_M}$, i.e.,
  \begin{equation*}
    \NAH^{\dir}(U, \lambda) \defby \sum_{k \in \N} \indic_{\{\lambdaAHk^{\dir}(U) \leq \lambda \}}, \qquad
 \NAHeps^{\dir}(U, \lambda) \defby \sum_{k \in \N} \indic_{\{\lambdaAHkeps^{\dir}(U) \leq \lambda \}}.
  \end{equation*}
  If $U$ is a bounded Lipshitz domain, we set
   $\artNAH \defby \ecf^{\neu}_{\WAH_M, \cZAH_M}$ and $\artNAHeps \defby \ecf^{\neu}_{\WAHeps_M,\cZAHeps_M}$, i.e.,
  \begin{equation*}
    \artNAH(U, \lambda) \defby \sum_{k \in \N} \indic_{\{\artevk(U) \leq \lambda \}}, \qquad
    \artNAH(U, \lambda) \defby \sum_{k \in \N} \indic_{\{\artevkeps(U) \leq \lambda \}},
  \end{equation*}
  and under the Neumann assumption~\ref{assump:all_three_assump}~\ref{item:neumann_assump}
  we set
  $\NAH^{\neu} \defby \ecf^{\neu}_{\WAH_M,\cZnAH_M}$ and $\NAHeps^{\neu} \defby \ecf^{\neu}_{\WAHeps_M,\cZnAHeps_M}$, i.e.,
  \begin{align*}
    \NAH^{\neu}(U, \lambda) \defby \sum_{k \in \N}
    \indic_{\{\lambdaAHk^{\neu}(U) \leq \lambda \}},
    \qquad
    \NAHeps^{\neu}(U, \lambda) \defby \sum_{k \in \N}
    \indic_{\{\lambdaAHkeps^{\neu}(U) \leq \lambda \}}.
  \end{align*}
\end{definition}

\begin{remark}
\label{remark:restricting_to_eps_0}
In most of the following we restrict our statements to
$\NAH^{\dir}$,
$\artNAH$ and
$\NAH^{\neu}$.
However, by `adding some $\epsilon$'s' the statements are also valid by replacing the occurrences of ``$\NAH^{\dir}$'',
``$\artNAH$''
``$\NAH^{\neu}$'', ``$\boldsymbol{\xi}$'' and ``$\xi$'',
and by ``$\NAHeps^{\dir}$'',
``$\artNAHeps$''
``$\NAHeps^{\neu}$'', ``$\boldsymbol{\xi}_{\epsilon}$'' and ``$\xi_{\epsilon}$''.
\end{remark}

\begin{lemma}\label{lem:upper_bound_of_neumann_N}
  Let $U$ be a bounded Lipschitz domain.
  Then, for any $\theta \in (0, 1)$, there exist
  $\lambda_{U,\theta}$, $C_{U, \theta,r} \in (0, \infty)$ and an integer $l \in \N$ such that  for every $\lambda \geq \lambda_{\theta, U}$
  \begin{align}
  \label{eqn:estimate_N_fd_below_ball}
    \NAH^{\dir}(U, \lambda)
    &  \ge (1 - \theta) \frac{\abs{B(0, 1)}}{(2 \pi)^d} \abs{U}
    \{\lambda +  \theta + C_{U, \theta,r}  \vertiii{\bm{\xi}}^l \}^{\frac{d}{2}},  \\
  \label{eqn:estimate_artN_fn_above_ball}
    \artNAH(U, \lambda)
    & \leq (1 + \theta) \frac{\abs{B(0, 1)}}{(2 \pi)^d} \abs{U}
    \{\lambda +  \theta + C_{U, \theta,r}  \vertiii{\bm{\xi}}^l \}^{\frac{d}{2}}.
  \end{align}
  In particular, $\expect[\artNAH(U, \lambda)^m] < \infty$ for every $m \in (0, \infty)$ and $\lambda \in \R$.

  If we furthermore impose the Neumann assumption~\ref{assump:all_three_assump}~\ref{item:neumann_assump},
  then
  \begin{equation}
  \label{eqn:estimate_N_fn_above_ball}
    \NAH^{\neu}(U, \lambda) \leq (1 + \theta) \frac{\abs{B(0, 1)}}{(2 \pi)^d} \abs{U}
    \{\lambda +  \theta +  C_{U, \theta,r} (\vertiii{\bm{\xi}} + \vertiii{\xi}_{\bdry U})^l  \}^{\frac{d}{2}}.
  \end{equation}
\end{lemma}
\begin{proof}
 The proof of \eqref{eqn:estimate_N_fd_below_ball} and \eqref{eqn:estimate_N_fn_above_ball} are similar to \eqref{eqn:estimate_artN_fn_above_ball}, hence we only give the proof of the latter.
Remember that $\ecf_0^{\dir}$ and $\ecf_0^{\neu}$ are the eigenvalue counting functions of $-\Delta$ with Dirichlet and Neumann boundary conditions, respectively.
By Lemma~\ref{lem:estimate_of_N_0}~\ref{item:estimates_N_0_d_and_N_0_n}, there exists a $\lambda_{U,\theta}>0$ such that
  for $\lambda \geq \lambda_{U,\theta}$ we have
  \begin{equation*}
    \ecf^{\neu}_0(U, \lambda) \leq (1 + \theta)^{\frac{1}{2}} \frac{\abs{B(0,1)}}{(2 \pi)^d} \abs{U} \lambda^{\frac{d}{2}}.
  \end{equation*}
\begin{calc}
For the Dirichlet estimate: By Lemma~\ref{lem:estimate_of_N_0}~\ref{item:estimates_N_0_d_and_N_0_n},  there exists a $\lambda_{U,\theta}>0$ such that
  for $\lambda \geq \lambda_{U,\theta}$ we have
  \begin{equation*}
    \ecf^{\dir}_0(U, \lambda) \ge (1 - \theta)^{\frac{1}{2}} \frac{\abs{B(0,1)}}{(2 \pi)^d} \abs{U} \lambda^{\frac{d}{2}}.
  \end{equation*}
\end{calc}
Let $\theta' \in (0,\infty),\gamma \in (0,1]$.
  By Lemma~\ref{lem:estimate_of_eigenvalue_counting_functions_of_AH},  and Remark~\ref{remark:larger_M_neumann},  with $\Lambda^+_{\lambda,\theta'} \defby \Lambda^+_{\lambda,\theta'}(\WAH_{M_\gamma}, \cZAH_{M_\gamma})$  where $M_\gamma$ is the random variable $\MAH(U, \deltamin; \gamma)$ (see \eqref{eq:M_AH}), one has
  \begin{equation*}
    \artNAH(U, \lambda)
    \leq (1 + \theta)^{\frac{1}{2}} \frac{\abs{B(0, 1)}}{(2 \pi)^d} \abs{U}
    (\Lambda^+_{\lambda,\theta'})^{\frac{d}{2}}.
  \end{equation*}
\begin{calc}
For the Dirichlet estimate: with instead Remark~\ref{remark:larger_M_dirichlet} and $\Lambda^-_{\lambda,\theta'}$,
\begin{align*}
  \NAH^{\dir}(U, \lambda)
    \ge (1 - \theta)^{\frac{1}{2}} \frac{\abs{B(0, 1)}}{(2 \pi)^d} \abs{U}
    (\Lambda^-_{\lambda,\theta'})^{\frac{d}{2}}.
\end{align*}
\end{calc}
  Recalling the definition of $\Lambda^+_{\lambda,\theta'}$, one observes that there exists a constant  $C'_{U,\theta',r}$  such that
  \begin{equation*}
    \Lambda^+_{\lambda,\theta'}
    \leq (1 + \theta') e^{ (2+\frac{2r}{1-r})  \gamma} (\lambda +  \theta' +  C'_{U, \theta', r}  \formnorm{\cZAH_{M_\gamma}}_{H^r(U)}^{\frac{1}{1-r}}).
  \end{equation*}
\begin{calc}
For the Dirichlet estimate:
\begin{align*}
    \Lambda^-_{\lambda,\theta'}
    \geq (1 - \theta') e^{ - (2+\frac{2r}{1-r})  \gamma} (\lambda -  \theta' -  C'_{U, \theta', r}  \formnorm{\cZAH_{M_\gamma}}_{H^r(U)}^{\frac{1}{1-r}}).
\end{align*}
\end{calc}
  Therefore, if $\gamma \defby  (2+\frac{2r}{1-r})^{-1}  \log (1 + \theta')$ and
  $\theta' \defby  (1 + \theta)^{\frac{1}{2d}} -1 \in (0,\theta) $, and $C_{U,\theta,r} = C'_{U,\theta',r}  \gamma^{\frac{1-r}{d\sigma} -2} $ (see \eqref{eq:condition_of_sigma}),  using
  \eqref{eq:estimate_of_cZ_AH_artificial_Neumann}, one has
  \begin{calc}
  due to the fact that
  \begin{align*}
    (1 + \theta') e^{ (2+\frac{2r}{1-r})  \gamma}
    = (1 + \theta')^2 = (1+\theta)^{\frac{1}{d}},
  \end{align*}
  \end{calc}
  \begin{align*}
    \Lambda^+_{\lambda,\theta'}
    & \leq (1 + \theta)^{\frac{1}{d}} (\lambda +  \theta +  C_{U,\theta,r}  \vertiii{\bm{\xi}}^l),
  \end{align*}
\begin{calc}
For the Dirichlet estimate:
if $\gamma \defby - (2+\frac{2r}{1-r})^{-1}  \log (1 - \theta')$ and
  $\theta' \defby  1- (1 - \theta)^{\frac{1}{2d}} \in (0,\theta) $, and $C_{U,\theta,r} = C'_{U,\theta',r}  \gamma^{\frac{1-r}{d\sigma} -2}$
  \begin{align*}
    \Lambda^-_{\lambda,\theta'}
    & \ge (1 - \theta)^{\frac{1}{d}} (\lambda - \theta - C_{U,\theta,r}  \vertiii{\bm{\xi}}^l),
  \end{align*}
  \begin{calc}
  due to the fact that
  \begin{align*}
    (1 - \theta') e^{-(2+\frac{2r}{1-r})  \gamma}
    = (1 - \theta')^2 = (1+\theta)^{\frac{1}{d}},
  \end{align*}
  \end{calc}
\end{calc}
  which yields \eqref{eqn:estimate_artN_fn_above_ball}.
\end{proof}

By Lemma~\ref{lem:estimate_of_eigenvalue_counting_functions_of_AH} we derive $\NAH^{\dir}(U,\lambda) \le \artNAH(U,\lambda)$ and similarly - by using that $\cZ_N^{\neu}[U] = \cZ_N[U]$ on $H_0^r$ (which follows due to the fact that $\cT$ equals zero on $H_0^r(U)$, see Lemma~\ref{lem:extension_and_trace_sobolev}) - $\NAH^{\dir}(U,\lambda) \le \NAH^{\neu}(U,\lambda)$.
Therefore, as a direct consequence of Lemma~\ref{lem:upper_bound_of_neumann_N} 
we obtain the following asymptotics.
These asymptotics agree with the asymptotics of the eigenvalue counting function for the Laplacian operator, as proven by Weyl (also called Weyl's law) and later generalised for a class of Schrödinger operators by Kirsch and Martinelli \cite[Proposition 2.3]{Kirsch_1982} (observe that our results agree with the work of Mouzard on two dimensional manifolds, see \cite{mouzard2020weyl}).

\begin{proposition}
\label{prop:weyl}
Let $U$ be a bounded Lipschitz domain, then
\begin{align*}
\lim_{\lambda \rightarrow \infty} \lambda^{-\frac{d}{2}} \NAH^{\dir} (U,\lambda)
=
\lim_{\lambda \rightarrow \infty} \lambda^{-\frac{d}{2}}  \artNAH (U,\lambda)
= \frac{|B(0,1)|}{(2\pi)^d} |U|,
\end{align*}
and under the Neumann assumption~\ref{assump:all_three_assump}~\ref{item:neumann_assump},
\begin{align*}
\lim_{\lambda \rightarrow \infty} \lambda^{-\frac{d}{2}}  \NAH^{\neu}
(U,\lambda) = \frac{|B(0,1)|}{(2\pi)^d} |U|.
\end{align*}
\end{proposition}

\begin{proposition}\label{prop:existence_of_IDS}
There exist (deterministic) functions $\R \rightarrow [0,\infty)$, $\cNAH^{\dir}$ and $\artcNAH$, such that for all $\lambda \in \R$ and $y\in \R^d$, $\P$-almost surely and in $L^1(\P)$, with $Q= y + [-\frac12,\frac12]^d$,
  \begin{align*}
    \cNAH^{\dir}(\lambda) & =  \lim_{L \in \Q, L \to \infty} \frac{1}{   L^d  }
    \NAH^{\dir}(Q_L , \lambda), \\
    \artcNAH (\lambda) &=  \lim_{L \in \Q, L \to \infty} \frac{1}{L^d}
    \artNAH( Q_L , \lambda)
  \end{align*}
  exist. Moreover,
  \begin{align}
\notag 
     \sup_{L > 0} \frac{1}{L^d} \expect[\NAH^{\dir}(Q_L, \lambda)]
   &  = \cNAH^{\dir}(\lambda)  \\
   \label{eqn:overline_sN_n_equals_inf}
   & \leq \artcNAH (\lambda)
    = \inf_{L > 0} \frac{1}{L^d} \expect[\artNAH(Q_L, \lambda)].
  \end{align}
  Under the Neumann assumption~\ref{assump:all_three_assump}~\ref{item:neumann_assump},
 one may simultaneously replace $\artNAH$ by $\NAH^{\neu}$ and $\artcNAH $ by $\cNAH^{\neu}$ in the above definition and inequality.
\end{proposition}
\begin{proof}
The existence of the mentioned limits in the $\P$-almost sure sense and the equalities in \eqref{eqn:overline_sN_n_equals_inf} follow by an application of the ergodic theorem by Akcoglu and Krengel \cite{AkcogluM.A1981Etfs}, see also \cite[Section 3]{Kirsch_1982} for a similar construction (in case the potential is a function). 
The conditions of this ergodic theorem are satisfied on the one hand by the ergodic assumptions~\ref{assump:all_three_assump}~\ref{item:ergodic_assump} and on the other hand by the fact that $ Q \mapsto  \NAH^{\dir}(Q, \lambda)$ is superadditive and $ Q \mapsto  \artNAH(Q, \lambda)$ is subadditive, which follows by Lemma~\ref{lem:box_decomposition_neumann} and Remark~\ref{rem:example_decomposable}. 
That the convergences hold also in $L^1(\P)$ and the inequality in \eqref{eqn:overline_sN_n_equals_inf} holds, follows by Lemma~\ref{lem:monotonicity_of_IDS} as for $Q \subseteq [-1, 1]^d$, it implies
  \begin{equation*}
    \NAH^{\dir}(Q, \lambda) \leq \NAH^{\dir}([-1, 1]^d, \lambda)
    \leq \artNAH([-1, 1]^d, \lambda),
  \end{equation*}
  and $\artNAH([-1, 1]^d, \lambda)$ is in  $L^1(\P)$  by Lemma~\ref{lem:upper_bound_of_neumann_N}.
\end{proof}
\begin{remark}
  The cube $Q$ in Proposition~\ref{prop:existence_of_IDS} does not need to be centered at the origin. This is important in the proof of
  Theorem~\ref{thm:IDS_general}.
\end{remark}
\begin{definition}
We define (deterministic) functions $\NAH^{\dir}, \artNAH : \R \rightarrow \R$ by 
  \begin{equation*}
    \NAH^{\dir}(\lambda) \defby \inf_{\lambda' > \lambda} \cNAH^{\dir}(\lambda') 
    \quad \mbox{ and }
    \quad \artNAH(\lambda)  \defby \inf_{\lambda' > \lambda}
    \artcNAH(\lambda'). 
  \end{equation*}
     Note that they are right-continuous functions that satisfy $\lim_{\lambda \rightarrow -\infty} \NAH^\#(\lambda)=0$ for $\#$ denoting either $\dir$ or $\neu$.
\end{definition}

\begin{definition}
\label{def:vague_convergence}
A sequence $(f_n)_{n\in\N}$ of increasing functions $\R \rightarrow [0,\infty)$, is said to \emph{converge vaguely} to some function $f : \R \rightarrow [0,\infty)$ if $f_n(\lambda) \rightarrow f(\lambda)$ for all $\lambda \in \R$ that are continuity points of $f$.
\end{definition}

\begin{remark}\label{rem:convergence_of_ids_cube}
  If $\lambda$ is a continuity point of $\NAH^{\dir}(\lambda)$, then
  $\NAH^{\dir}(\lambda) = \cNAH^{\dir}(\lambda)$.
\end{remark}

\begin{remark}
\label{remark:upper_estimate_with_norm_xis}
Observe that by Lemma~\ref{lem:upper_estimates_N}~\ref{item:upper_bound_of_neumann_ids_by_dirichlet_ids} and \eqref{eq:estimate_of_cZ_AH_artificial_Neumann} for $\gamma=1$, for all $\mu >0$ and bounded Lipschitz domains $U$, there exists a $C_{U,\mu}>0$ such that
  \begin{align*}
    \NAH^{\dir}(U_L, \lambda)
    & \leq \artNAH(U_L, \lambda) \\
    \cand \begin{calc}
    \leq \NAH^{\dir}(U_L, \lambda + \mu)
    + C_{U, \mu} L^{d-1}
    [1 + \max\{\lambda, 0\} + \formnorm{\cZAH_{M_{L,1}}}_{H^r( U_L )}^{\frac{1}{1-r}}]^{\frac{d}{2}} \end{calc} \cnewline
    \cand \begin{calc}
    \leq \NAH^{\dir}(U_L, \lambda + \mu)
    + C_{U, \mu} L^{d-1}
    [1 + \max\{\lambda, 0\} + L^{\frac{1}{d}} \vertiii{\bm{\xi}}^{\frac{m}{1-r}} ]^{\frac{d}{2}} \end{calc} \cnewline
    & \leq \NAH^{\dir}(U_L, \lambda + \mu)
    + C_{U, \mu} L^{d-\frac12}
    [1 + \max\{\lambda, 0\} + \vertiii{\bm{\xi}}^{\frac{m}{1-r}} ]^{\frac{d}{2}},
  \end{align*}
  and under 
  the Neumann assumption~\ref{assump:all_three_assump}~\ref{item:neumann_assump}
\begin{align*}
   \NAH^{\dir}(U_L, \lambda)
    & \leq \NAH(U_L, \lambda) \\
    & \leq \NAH^{\dir}(U_L, \lambda + \mu)
    + C_{U, \mu} L^{d-\frac12}
    [1 + \max\{\lambda, 0\} + (\vertiii{\bm{\xi}}^{m}+\vertiii{\xi}_{\partial U})^\frac{1}{1-r} ]^{\frac{d}{2}}.
\end{align*}
\end{remark}

\begin{proposition}\label{prop:ids_boundary_condition}
$\NAH^{\dir} = \artNAH$ and, under  the Neumann assumption~\ref{assump:all_three_assump}~\ref{item:neumann_assump},
  $\NAH^{\dir} = \NAH^{\neu}$.
\end{proposition}
\begin{proof}
Let $\lambda$ be a continuity point of both $\NAH^{\dir}$ and $\overline \NAH^{\neu}$. By Remark~\ref{remark:upper_estimate_with_norm_xis} applied to $U=[-\frac12,\frac12]^d$, by Proposition~\ref{prop:existence_of_IDS}, because $\cNAH^{\dir}(\lambda) = \NAH^{\dir}(\lambda)$ and $\overline \cNAH^{\neu}(\lambda) = \overline \NAH^{\neu}(\lambda)$, see Remark~\ref{rem:convergence_of_ids_cube}, we have for all $\mu>0$,
\begin{align*}
\NAH^{\dir}(\lambda) \le \overline \NAH^{\neu}(\lambda) \le \NAH^{\dir}(\lambda+ \mu),
\end{align*}
so that the equality follows as both $\NAH^{\dir}$ and $\overline \NAH^{\neu}$ are right-continuous.
Under the Neumann assumption~\ref{assump:all_three_assump}~\ref{item:neumann_assump},
we can argue similarly with ``$\NAH^{\neu}$'' instead of ``$\overline \NAH^{\neu}$''.
\end{proof}

\begin{calc}
\begin{proof}[Alternative proof of Proposition~\ref{prop:ids_boundary_condition}]
We could invoke Lemma~\ref{lem:upper_estimates_N}~\ref{item:upper_bound_of_neumann_ids_by_dirichlet_ids}, but instead give a more elementary proof.
We follow the start of the proof of Lemma~\ref{lem:upper_estimates_N}~\ref{item:upper_bound_of_neumann_ids_by_dirichlet_ids}.

Let $l>0$, $n\in\N$, $n\ge 3$.
\begin{align*}
U_1 =  [-(n-1)l, (n-1)l]^d , \quad
 U_2 = [-nl, nl]^d \setminus [-(n-2)l, (n-2)l]^d .
\end{align*}
Then, for $L=nl$, $R=2l$ and $Q=[-1,1]^d$ we have
$U_1 =  Q_L \setminus B(\partial Q_L, R)$ and $U_2 = Q_L^R \subseteq B(\partial U_L, R)$.
Because $H_{\fm,R}^1(U_2)$ is embedded in $H^1(U_2)$, we have $\overline{\cE}^{\neu,U_2} \succ \overline{\cE}^{\fm,R,U_2}$ (where $\overline \cE= \cE_{\WAH_M,\overline Z_M^{\neu}}$) and thus by the proof of Lemma~\ref{lem:upper_estimates_N}~\ref{item:upper_bound_of_neumann_ids_by_dirichlet_ids},
\begin{align*}
\artNAH([-nl, nl]^d, \lambda) \leq \NAH^{\dir}([-(n-1)l, (n-1)l]^d, \lambda + \frac{K}{l})
  + \artNAH(U_2, \lambda + \frac{K}{l}).
\end{align*}

  Since $U_2$ can be decomposed into $\{n^d - (n-2)^d\}$ boxes of size $l$,
  the subadditivity of $\artNAH$ and the translation invariance yield
  \begin{equation*}
    \expect[\artNAH(U_2, \lambda)] \leq (n^d - (n-2)^d) \expect[ \artNAH([0, l]^d, \lambda)].
  \end{equation*}
  Therefore,
  \begin{multline*}
    \frac{1}{(nl)^d} \expect[ \artNAH([-nl, nl]^d, \lambda)]
    \leq \frac{1}{(nl)^d} \expect\Big[ \NAH^{\dir}\Big([-(n-1)l, (n-1)l]^d, \lambda+ \frac{K}{l}\Big)\Big] \\
    + \frac{n^d - (n-2)^d}{(nl)^d} \expect\Big[\artNAH\Big([0, l]^d, \lambda + \frac{K}{l}\Big)\Big].
  \end{multline*}
  By letting $n \to \infty$, if $\lambda$ is a continuity point of $\artNAH$, one has
  \begin{equation*}
  \artNAH(\lambda) \leq \cNAH^{\dir}(\lambda + Kl^{-1}) \leq \NAH^{\dir}(\lambda + 2Kl^{-1}),
  \qquad l > 0.
  \end{equation*}
  Since $\artNAH$ and $\NAH^{\dir}$ are right-continuous, we complete the proof.
\end{proof}
\end{calc}

\begin{definition}\label{def:IDS}
  Thanks to Proposition~\ref{prop:ids_boundary_condition}, we may simply write
  \begin{equation*}
    \NAH \defby \NAH^{\dir} = \artNAH (= \NAH^{\neu} \quad \mbox{under
    the Neumann assumption~\ref{assump:all_three_assump}~\ref{item:neumann_assump}}).
  \end{equation*}
We call $\NAH$ the \emph{integrated density of states} (IDS) for the Anderson Hamiltonian with potential $\xi$.
\end{definition}
\tymchange{
One can see the validity of the above definition from the following two perspectives: 
(i) $\NAH$ is the limit of eigenvalue counting functions on growing domains (Theorem~\ref{thm:IDS_general}) and 
(ii) $\NAH$ is the limit of the IDS associated with mollified noise (Theorem~\ref{thm:IDS_epsilon}).
}
\begin{theorem}\label{thm:IDS_general}
  Let $U$ be a bounded domain. Then, almost surely,
  \begin{equation}
  \label{eqn:convergence_IDS_dirichlet}
    \lim_{L \to \infty} \frac{1}{\abs{U_L}} \NAH^{\dir}(U_L, \cdot) = \NAH \quad \mbox{vaguely}.
  \end{equation}
  If $U$ is a bounded Lipschitz domain, then
  \begin{equation}
  \label{eqn:convergence_IDS_hat_neumann}
     \lim_{L \in \N, L \to \infty} \frac{1}{\abs{U_L}} \artNAH(U_L, \cdot) = \NAH \quad \mbox{vaguely}.
  \end{equation}
  Under the Neumann assumption~\ref{assump:all_three_assump}~\ref{item:neumann_assump}, one can replace $\artNAH$ by $\NAH^{\neu}$  in \eqref{eqn:convergence_IDS_hat_neumann}.
\end{theorem}
\begin{proof}
Firstly, observe that we may assume $U$ to be a bounded Lipschitz domain due to the monotonicity of
  $\NAH^{\dir}(U, \lambda)$ as a function of $U$.
\begin{calc}
If $U$ is a bounded domain but not Lipschitz, then for all $\epsilon\in (0,|U|)$ there exist bounded Lipschitz domains $U^-$, $U^+$ such that $U^- \subseteq U \subseteq U^+$ and $|U\setminus U^-|, |U^+\setminus U|<\epsilon$. Then
\begin{align*}
\frac{\NAH^\partial(U_L^-,\cdot)}{|U_L^-| + \epsilon L}
\le \frac{\NAH^\partial(U_L,\cdot)}{|U_L| }
\frac{\NAH^\partial(U_L^+,\cdot)}{|U_L^+| - \epsilon L}.
\end{align*}
Then use moreover, e.g., $\frac{|U_L^-|}{|U_L^-| + \epsilon L} = \frac{|U^-|}{|U^-|+\epsilon}$.
\end{calc}
By Remark~\ref{remark:upper_estimate_with_norm_xis} it suffices to prove \eqref{eqn:convergence_IDS_dirichlet} (also for $\NAH^{\neu}$ under the Neumann assumption~\ref{assump:all_three_assump}~\ref{item:neumann_assump}.

Let $\lambda \in \R$ be a continuity point of $\NAH$.
  We set
  \begin{equation*}
    I_n \defby \set{k \in \Z^d \given k + [0, 1]^d \subseteq 2^n U},
    \quad J_n \defby \set{k \in \Z^d \given (k + [0, 1]^d) \cap (2^n U) \neq \emptyset}.
  \end{equation*}
  By Lemma~\ref{lem:monotonicity_of_IDS} and  Lemma~\ref{lem:box_decomposition_neumann},
  \begin{equation*}
    \sum_{k \in I_n} \NAH^{\dir}(2^{-n} L(k + [0, 1]^d), \lambda)
    \leq \NAH^{\dir}(U_L, \lambda)
    \leq \sum_{k \in J_n} \artNAH(2^{-n} L(k + [0, 1]^d), \lambda).
  \end{equation*}
  Therefore, by Proposition~\ref{prop:existence_of_IDS} and Remark~\ref{rem:convergence_of_ids_cube},
  \begin{equation*}
    \frac{\# I_n}{2^{dn} \abs{U}} \NAH^{\dir}(\lambda) \leq
    \liminf_{L \to \infty} \frac{1}{\abs{U_L}} \NAH^{\dir}(U_L, \lambda)
    \leq \limsup_{L \to \infty} \frac{1}{\abs{U_L}} \NAH^{\dir}(U_L, \lambda)
    \leq \frac{\# J_n}{2^{dn} \abs{U}} \artNAH(\lambda).
  \end{equation*}
  By Proposition~\ref{prop:ids_boundary_condition}, one has $\NAH = \NAH^{\dir} = \artNAH$.
  Thus, the proof is complete by letting $n \to \infty$.
\end{proof}

\begin{remark}
Recall Remark~\ref{remark:restricting_to_eps_0}. Let $\epsilon\in (0,1)$.
There exist functions $\cNAHeps^{\dir}, \artcNAHeps$ and $\cNAHeps^{\neu}$ such that analogues statements as in Proposition~\ref{prop:existence_of_IDS} hold.
Then we define $\NAHeps^{\dir}(\lambda) := \inf_{\lambda'>\lambda} \cNAHeps^{\dir}(\lambda')$ and similarly $\artNAHeps $ and $\NAHeps^{\neu}$.
By analogous arguments as in Theorem~\ref{thm:IDS_general} we also have $\NAHeps^{\dir} = \artNAHeps$ ($= \NAHeps^{\neu} $ under the Neumann assumption~\ref{assump:all_three_assump}~\ref{item:neumann_assump}).
In this case $\NAHeps := \NAHeps^{\dir}$ is called the integrated density of states for the Anderson Hamiltonian with potential $\xi_\epsilon - c_\epsilon$.
\end{remark}

For the convergence of $\NAHeps$ to $\NAH$ that we show in Theorem~\ref{thm:IDS_epsilon}, we introduce the following auxiliary lemma.

\begin{lemma}
\label{lemma:liminf_expect_IDS_eps}
Let $\#$ denote either $\dir$ or $\neu$.
For all $L>0$, $\lambda \in \R$ and $\mu>0$,
\begin{align*}
\liminf_{\epsilon\downarrow 0}  \expect[ \NAHeps^\#(Q_L,\lambda )]  \ge  \expect[ \NAH^\#(Q_L,\lambda-\mu)].
\end{align*}
\end{lemma}
\begin{proof}
First we observe that as $\lambdaAHkeps^{\#}(Q_L) \to \lambdaAHk^{\#}(Q_L)$ in probability for all $k$ (by Theorem~\ref{thm:convergence_of_Dirichlet_AH} and Theorem~\ref{thm:convergence_of_Neumann_AH}), the following holds:
For all $(\epsilon_n)_{n\in\N}$ in $(0,\infty)$ with $\epsilon_n \downarrow 0$ as $n\rightarrow \infty$, there exists a subsequence $(\epsilon_{n_m})_{m\in\N}$  and a $\Omega_1 \subseteq \Omega$  of $\P$-probability $1$, such that on $\Omega_1$, $\lambdaAHkepsnm^{\#}(Q_L) \to \lambdaAHk^{\#}(Q_L)$ for all $k$,  and therefore for all $\mu>0$
\begin{align*}
\liminf_{\epsilon\downarrow 0}  \NAHeps^\#(Q_L,\lambda )  \ge   \NAH^\#(Q_L,\lambda-\mu).
\end{align*}
Indeed, if $\mu_{1,m} \le \mu_{2,m}\le \cdots$ and $\mu_{k,m} \rightarrow \mu_k$ in $\R$ as $m\rightarrow \infty$, for all $k\in\N$, then $\liminfm \sum_{k=1}^\infty \1_{\{ \mu_{k,m} \le \mu\}} \ge \sum_{k=1}^\infty \1_{\{ \mu_{k} \le \mu-\epsilon\}}$ for all $\mu \in \R$ and $\epsilon>0$.
 \begin{calc}
 Let $\mu \in \R$, $\epsilon>0$. Let $k\in \N$ be such that $\mu_k \le \mu -\epsilon < \mu_{k+1}$, i.e., $\sum_{k=1}^\infty \1_{\{ \mu_{k} \le \mu-\epsilon\}}=k$. Then there exists an $M\in\N$ such that for all $m\ge M$, $\mu_{k,m} \le \mu$, i.e., $\liminfm \sum_{k=1}^\infty \1_{\{ \mu_{k,m} \le \mu\}} \ge k$.
 \end{calc}
Therefore, for all $(\epsilon_n)_{n\in\N}$ in $(0,\infty)$ with $\epsilon_n \downarrow 0$ as $n\rightarrow \infty$, there exists a subsequence $(\epsilon_{n_m})_{m\in\N}$ such that by Fatou's lemma
\begin{align*}
\liminfm \expect[ \NAHepsnm^\#(Q_L,\lambda )]  \ge
 \expect[ \liminfm \NAHepsnm^\#(Q_L,\lambda )]
 \ge
\expect{[\NAH^\#(Q_L,\lambda-\mu)]}.
\end{align*}
From this the inequality follows.
\end{proof}

\begin{theorem}\label{thm:IDS_epsilon}
  $\NAHeps \to \NAH$ vaguely.
\end{theorem}
\begin{proof}
  Let $Q \defby [-1/2, 1/2]^d$, $L \in \N$ and $\mu \in (0, 1)$. Let $\lambda \in \R$ be a continuity point of
  $\NAH$. By  \eqref{eqn:overline_sN_n_equals_inf}  (see also Remark~\ref{rem:convergence_of_ids_cube})
  \begin{equation*}
    \NAH(\lambda) - \NAHeps(\lambda)
    \leq \frac{1}{L^d} \expect[\artNAH(Q_L, \lambda) - \NAHeps^{\dir}(Q_L, \lambda)].
  \end{equation*}
 By Lemma~\ref{lemma:liminf_expect_IDS_eps}, for all $L\ge 1$ and $\mu>0$,
  \begin{equation*}
    A \defby \limsup_{\epsilon \downarrow 0} \{\NAH(\lambda) - \NAHeps(\lambda) \}
    \leq \frac{1}{L^d} \expect[ \artNAH(Q_L, \lambda) - \NAH^{\dir}(Q_L, \lambda - \mu)].
  \end{equation*}
Therefore, by \eqref{eqn:overline_sN_n_equals_inf}, taking the infimum over $L\ge1$ in the above inequality, we obtain $A\le \NAH(\lambda) - \NAH(\lambda - \mu)$ for all $\mu>0$. As $\lambda$ is a continuity point of $\NAH$, it follows that $A\le 0$.
 Similarly,
\begin{calc}
as by \eqref{eqn:overline_sN_n_equals_inf} one also has
  \begin{equation*}
    \NAH(\lambda) - \NAHeps(\lambda)
    \geq \frac{1}{L^d} \expect[\NAH^{\dir}(Q_L, \lambda) - \artNAH(Q_L, \lambda)].
  \end{equation*}
\end{calc}
one can show \begin{calc}
using Lemma~\ref{lemma:liminf_expect_IDS_eps} for $\#$ denoting $\neu$,
\end{calc}
  \begin{equation*}
    \liminf_{\epsilon \downarrow 0} \{ \NAH(\lambda) - \NAHeps(\lambda) \} \geq 0. \qedhere
  \end{equation*}
\end{proof}
\tymchange{Finally, we prove the tail behavior of the IDS.}
\begin{theorem}\label{thm:IDS_tails}
  One has the following tail estimates of the IDS.
  \begin{enumerate}
    \item One has $\lim_{\lambda \to \infty} \lambda^{-\frac{d}{2}} \NAH(\lambda) = \frac{\abs{B(0, 1)}}{(2 \pi)^d}$.
    \item
    For  every bounded domain $U$ and every $\alpha \in (0, \infty)$, one has
    \begin{align}
    \label{eqn:limsup_lambda_min_infty}
      \limsup_{\lambda \to -\infty} (-\lambda)^{-\alpha} \log \NAH(\lambda)
      &= \limsup_{\lambda \to -\infty} (-\lambda)^{-\alpha} \log \P(\lambdaAHone^{\dir}(U) \leq \lambda), \\
    \label{eqn:liminf_lambda_min_infty}
      \liminf_{\lambda \to -\infty} (-\lambda)^{-\alpha} \log \NAH(\lambda)
      &= \liminf_{\lambda \to -\infty} (-\lambda)^{-\alpha} \log \P(\lambdaAHone^{\dir}(U) \leq \lambda).
    \end{align}
  \end{enumerate}
\end{theorem}
\begin{proof}
  (a) Let $Q \defby [0, 1]^d$.
By applying Fatou's lemma 
and the first inequality of \eqref{eqn:overline_sN_n_equals_inf}, we obtain
  \begin{equation*}
    \expect[\liminf_{\lambda \to \infty} \lambda^{-\frac{d}{2}}\NAH^{\dir}(Q, \lambda)]
    \leq \liminf_{\lambda \to \infty} \lambda^{-\frac{d}{2}} \NAH(\lambda)
    \leq \limsup_{\lambda \to \infty} \lambda^{-\frac{d}{2}} \NAH(\lambda).
  \end{equation*}
  \begin{calc}
  \begin{align*}
    \expect[\liminf_{\lambda \to \infty} \lambda^{-\frac{d}{2}}\NAH^{\dir}(Q, \lambda)]
  & \leq  \liminf_{\lambda \to \infty}   \expect[\lambda^{-\frac{d}{2}}\NAH^{\dir}(Q, \lambda)] \\
  & \leq  \liminf_{\lambda \to \infty} \lambda^{-\frac{d}{2}} \sup_{L>0} \frac{1}{L^d} \expect[\NAH^{\dir}(Q_L, \lambda)] \\
  &    \leq \liminf_{\lambda \to \infty} \lambda^{-\frac{d}{2}} \cNAH^{\dir}(\lambda)
      \leq \liminf_{\lambda \to \infty} \lambda^{-\frac{d}{2}} \NAH^{\dir}(\lambda).
  \end{align*}
  \end{calc}
  Since  $\lim_{\lambda \to \infty} \lambda^{-\frac{d}{2}}\NAH^{\dir}(Q, \lambda) = \frac{\abs{B(0, 1)}}{(2 \pi)^d}$
  by  Proposition~\ref{prop:weyl}, the lower bound is obtained. To obtain the upper bound,
  by the last inequality of \eqref{eqn:overline_sN_n_equals_inf} and 
  the estimate \eqref{eqn:estimate_artN_fn_above_ball} from Lemma~\ref{lem:upper_bound_of_neumann_N},  for any $\theta \in (0, 1)$ we have
  \begin{align*}
    \limsup_{\lambda \to \infty} \lambda^{-\frac{d}{2}} \NAH(\lambda)
    &\leq 
    \limsup_{\lambda \to \infty} \lambda^{-\frac{d}{2}} \mathbb{E}[\overline{\NAH}^{\neu}(Q, \lambda)] \\
  & \le 
  \limsup_{\lambda \to \infty} \lambda^{-\frac{d}{2}}
  (1+\theta) \frac{|B(0,1)|}{(2\pi)^d}  \mathbb{E}[ \{\lambda +  \theta + C_{Q, \theta,r}  \vertiii{\bm{\xi}}^l \}^{\frac{d}{2}}] \\
  & \le   (1+\theta) \frac{|B(0,1)|}{(2\pi)^d}.
  \end{align*}
  Now the identity in (a) follows.

  (b) Let $\lambda < 0$.
  Thanks to the monotonicity of  $\lambdaAHone^{\dir}$ (see Proposition~\ref{prop:eigenvalues_of_AH}), we may and do assume
  $U = [0, L]^d$ for some $L \in (0, \infty)$.
  By \eqref{eqn:overline_sN_n_equals_inf} (see also Remark~\ref{rem:convergence_of_ids_cube}),
  \begin{equation*}
    \P(\lambdaAHone^{\dir}(U) \leq \lambda) \leq \expect[ \NAH^{\dir}(U, \lambda)]
    \leq L^d \NAH(\lambda).
  \end{equation*}
 Therefore, we establish that the left-hand sides are greater or equal to the right-hand sides of \eqref{eqn:limsup_lambda_min_infty} and \eqref{eqn:liminf_lambda_min_infty}.
  By Lemma~\ref{lem:upper_estimates_N}~\ref{item:box_decomposition_dirichlet_reversed}, for $l \in (0, L/2)$ and $n \in \N$, one has
  \begin{equation*}
    \NAH^{\dir}(U_n, \lambda) \leq
    \sum_{k \in \Z^d \cap [-1, n+1]^d} \NAH^{\dir}(k + [-l, L + l]^d, \lambda + K l^{-2})
  \end{equation*}
  and hence
  \begin{equation*}
    \frac{1}{n^d} \expect[\NAH^{\dir}(U_n, \lambda)]
    \leq \frac{(n+2)^d}{n^d} \expect[\NAH^{\dir}([0, L + 2l]^d, \lambda + Kl^{-2})].
  \end{equation*}
  Letting $n \to \infty$, for $p, q \in (1, \infty)$ with $p^{-1} + q^{-1} = 1$, one obtains
  \begin{align*}
    \NAH(\lambda) &\leq \expect[\NAH^{\dir}([0, L + 2l]^d, \lambda + Kl^{-2})] \\
    &= \expect[\NAH^{\dir}([0, L + 2l]^d, \lambda + Kl^{-2}) \indic_{\{ \lambdaAHone^{\dir}([0, L + 2l]^d) \leq
    \lambda + K l^{-2}\}}] \\
    &\leq \expect[\NAH^{\dir}([0, L + 2l]^d, K l^{-2})^q]^{\frac{1}{q}} \P(\lambdaAHone^{\dir}([0, L + 2l]^d) \leq
    \lambda + K l^{-2})^{\frac{1}{p}}.
  \end{align*}
  where we applied H\"older's inequality in the second inequality.
  Note  that
  \begin{equation*}
    \expect[\NAH^{\dir}([0, L + 2l]^d, K l^{-2})^q] < \infty
  \end{equation*}
   by Lemma~\ref{lem:estimate_of_eigenvalue_counting_functions_of_AH}
   and Lemma~\ref{lem:estimate_of_N_0}~\ref{item:estimates_N_0_d_and_N_0_n}.
Therefore, for $U = [0,L+2l]^d$, the left-hand sides are less or equal to the right-hand sides of \eqref{eqn:limsup_lambda_min_infty} and \eqref{eqn:liminf_lambda_min_infty}.
As $L$ and $l \in (0,L/2)$ can be chosen arbitrarily, the equalities follow.
\end{proof}



\appendix
\addtocontents{toc}{\protect\setcounter{tocdepth}{1}}
\section{Estimates related to function spaces}\label{subsec:technical_estimates_function_spaces}
\subsection{Estimates in Besov spaces}
\label{subsec:estimates_in_Besov}
This subsection gives estimates in weighted Besov spaces (see Definition~\ref{def:besov_spaces}).

\begin{lemma}\label{lem:weighted_besov_p_vs_infty}
  Let $p, q \in [1, \infty]$, $r \in \R$ and $\sigma_1, \sigma_2 \in [0, \infty)$. Then
\begin{align}
\label{eqn:estimate_infty_by_p_besov}
\norm{f}_{\csp^{r, \sigma_1}(\R^d)}
    & \lesssim_{p, r, \sigma_1}
    \norm{f}_{B_{p, \infty}^{r+\frac{d}{p}, \sigma_1 }(\R^d)}, \\
\label{eqn:estimate_infty_by_p_q_besov}
\norm{f}_{\csp^{r, \sigma_1}(\R^d)}
    & \lesssim_{p,q, r,\kappa, \sigma_1}
    \norm{f}_{B_{p, q}^{r+\frac{d}{p}+\kappa, \sigma_1 }(\R^d)},  \qquad \kappa >0.
\end{align}
If $p \sigma_2 > d$, then
  \begin{equation}
\label{eqn:estimate_p,q_by_infty_infty_besov}
    \norm{f}_{B_{p, q}^{r, \sigma_1 + \sigma_2}(\R^d)}
    \lesssim_{p, q, r, \sigma_1, \sigma_2} \norm{f}_{\csp^{r, \sigma_1}(\R^d)}.
  \end{equation}
\end{lemma}
\begin{proof}
By \cite[Theorem 6.5]{Triebel2006}, one has $\norm{f}_{B_{p, q}^{r, \sigma}(\R^d)} \sim_{p, q, r, \sigma} \norm{w_{\sigma} f}_{B_{p, q}^r(\R^d)}$.
Therefore \eqref{eqn:estimate_infty_by_p_besov} and \eqref{eqn:estimate_infty_by_p_q_besov} follow by the (unweighted) Besov embedding, see \cite[Section 2.7.1]{Triebel1983}.
For \eqref{eqn:estimate_p,q_by_infty_infty_besov}, the product estimate in the Besov space (see \cite[Corollary 2.1.35]{Ma18}, which follows also from \cite[Lemma 2.1]{PrTr16})
yields
  \begin{equation*}
    \norm{f}_{B_{p, q}^{r, \sigma_1 + \sigma_2 }(\R^d)}
    \lesssim_{p, q, r, \sigma_1, \sigma_2} \norm{w_{\sigma_2}}_{B^{\abs{r} + \frac12}_{p, q}(\R^d)}
    \norm{f}_{\csp^{r, \sigma_1}(\R^d)}.
  \end{equation*}
Now $\norm{w_{\sigma_2}}_{B^{\abs{r} + \frac12}_{p, q}(\R^d)} \lesssim \norm{w_{\sigma_2}}_{B^{\abs{r} + \frac34}_{p, p}(\R^d)}
\lesssim \norm{w_{\sigma_2}}_{W^{\abs{r} + \frac34}_{p}(\R^d)}
\lesssim \norm{w_{\sigma_2}}_{W^{\abs{r} + 1}_{p}(\R^d)}$ (by Lemma~\ref{lem:equivalence_sob_slobo_and_besov}).
Since $\abs{\partial^m w_{\sigma_2}} \lesssim_{m, \sigma_2} w_{\sigma_2 - \abs{m}}$,
  we have $ \norm{w_{\sigma_2}}_{W^{\abs{r} + 1}_{p}(\R^d)} < \infty$ if $p \sigma_2 > d$.
\end{proof}

\begin{theorem}[Weighted Young's inequality]
\label{thm:weighted_young}
Let $p,q,r\in [1,\infty]$, $\frac{1}{r}+1= \frac{1}{p}+\frac{1}{q}$ and $\sigma \in [0,\infty)$. Then
\begin{align*}
\|w_\sigma (f * g) \|_{L^r} \lesssim_\sigma \|w_{-\sigma} f\|_{L^p} \|w_{\sigma} g\|_{L^q}.
\end{align*}
\end{theorem}
\begin{proof}
\begin{calc}
\begin{align*}
(1+|x|^2)^{-\frac{\sigma}{2}} (1+|y|^2)^{-\frac{\sigma}{2}} \lesssim (1+|x-y|^2)^{-\frac{\sigma}{2}},
\end{align*}
because $1+|x-y|^2 \le (1+|x|^2)(1+|y|^2)$..
\end{calc}
Using that $w_\sigma(x) \lesssim_\sigma w_\sigma(x-y) w_{-\sigma}(y)$, one can estimate $w_\sigma(x) |f*g|(x)\le [(|f|w_{-\sigma}) * (|g|w_\sigma)](x)$, see also \cite[Theorem 2.4]{mourrat_weber_17}. The rest follows by Young's inequality, \cite[Theorem 1.4]{Bahouri2011}.
\end{proof}

\begin{theorem}
\label{theorem:conv_mollifiers_in_weighted}
Let $r\in \R$ and $\sigma \in [0,\infty]$.
Let $\varphi \in \cS(\R^d)$ and $\int \varphi =1$ and $\varphi_\epsilon(x) = \epsilon^{-d} \varphi(\epsilon^{-1} x)$.
Then, for all $\eta \in \csp^{r,\sigma}(\R^d), \delta>0$,
\begin{align*}
\norm{ \varphi_\epsilon * \eta - \eta }_{\csp^{r-\delta,\sigma+\delta }(\R^d) } \arroweps 0.
\end{align*}
\end{theorem}
\begin{proof}
In this proof we refrain from writing ``$(\R^d)$''.
Let $\delta>0$ and $p \in (\frac{d}{\delta},\infty)$.
By Lemma~\ref{lem:weighted_besov_p_vs_infty}, \eqref{eqn:estimate_p,q_by_infty_infty_besov}, $\eta$ is an element of $B_{p,1}^{r,\sigma+\delta}$.
As by Lemma~\ref{lem:weighted_besov_p_vs_infty}, \eqref{eqn:estimate_infty_by_p_besov},
\begin{align*}
\norm{ \varphi_\epsilon * \eta - \eta }_{\csp^{r-\delta,\sigma+\delta }}
\lesssim
\norm{ \varphi_\epsilon * \eta - \eta }_{B_{p,1}^{r,\sigma+\delta }}
=  \sum_{j=-1}^\infty  2^{rj} \| w_{\sigma+\delta}( \varphi_\epsilon * \Delta_j \eta - \Delta_j \eta) \|_{L^p}.
\end{align*}
It suffices to show for all $j$ that $\| w_{\sigma+\delta}( \varphi_\epsilon * \Delta_j \eta - \Delta_j \eta) \|_{L^p} \arroweps 0$ and
\begin{align}
\label{eqn:estimate_difference_weighted_conv}
\| w_{\sigma+\delta}( \varphi_\epsilon * \Delta_j \eta - \Delta_j \eta) \|_{L^p} \lesssim \| w_{\sigma+\delta} \Delta_j \eta\|_{L^p}
\end{align}
As $\varphi_\epsilon * \Delta_j \eta$ converges to $\Delta_j \eta$ almost everywhere (as it does at every Lebesgue point, see  \cite[Proposition 2.3.8]{HyvNVeWe16}),
the converges follows from \eqref{eqn:estimate_difference_weighted_conv} by Lebesgue's dominated convergence theorem.
By the weighted Young inequality, we have $\| w_{\sigma+\delta} (\varphi_\epsilon * \Delta_j \eta) \|_{L^p}
\lesssim_\sigma \| w_{-(\sigma+\delta)} \varphi_\epsilon\|_{L^1} \|w_{\sigma+\delta} \Delta_j \eta\|_{L^p}$. As $w_{-(\sigma+\delta)}(\epsilon x) \le w_{-(\sigma+\delta)}( x)$ for $\epsilon\in (0,1)$, we have $\| w_{-(\sigma+\delta)} \varphi_\epsilon\|_{L^1} \le \| w_{-(\sigma+\delta)} \varphi\|_{L^1}$ which is finite because $\varphi \in \cS(\R^d)$. This proves \eqref{eqn:estimate_difference_weighted_conv}.
\end{proof}

\begin{lemma}\label{lem:localization}
  Let $p, q \in [1, \infty]$,
  $r \in \R$ and $\sigma \in [0, \infty)$. Let $Z \in  B_{p,q}^{r,\sigma}(\R^d)  $ and
  $\phi \in \S(\R^d)$. Then, one has
  \begin{equation*}
    \norm{\phi(L^{-1} \cdot) Z}_{B^r_{p, q}(\R^d)}
    \lesssim_{p, q, r, \sigma, \phi} L^{\sigma} \norm{Z}_{B_{p, q}^{r, \sigma}(\R^d)}, \qquad L \ge 1.
  \end{equation*}
\end{lemma}
\begin{proof}
  By the product estimate for Besov spaces \cite[Theorem 4.37]{sawano2018theory}, we have
  \begin{equation*}
    \norm{\phi(L^{-1} \cdot) Z}_{B_{p, q}^{r}(\R^d)}
    \lesssim_{p, q, r} \norm{\phi(L^{-1} \cdot) w_{-\sigma}}_{\csp^{\abs{r} + \frac12}(\R^d)}
    \norm{w_{\sigma} Z}_{B_{p, q}^{r}(\R^d)}.
  \end{equation*}
  Since $\norm{w_{\sigma} Z}_{B_{p, q}^r(\R^d)} \lesssim_{p, q, r, \sigma}
  \norm{Z}_{B_{p, q}^{r, \sigma}(\R^d)}$
  by \cite[Theorem 6.5]{Triebel2006}  and $\norm{\cdot}_{\cC^{|r|+\frac12}}\lesssim \norm{\cdot}_{W_\infty^{|r|+\frac12}}\lesssim \norm{\cdot}_{W_\infty^{|r|+1}}$ by Lemma~\ref{lem:equivalence_sob_slobo_and_besov}, it suffices to show
  \begin{equation*}
    \norm{\phi(L^{-1} \cdot) w_{-\sigma}}_{ W_\infty^{\abs{r} + 1}(\R^d) } \lesssim_{r, \sigma, \phi} L^{\sigma}.
  \end{equation*}
  For this, it suffices to show
  \begin{equation*}
    \norm{\partial^m [w_{-\sigma} \phi(L^{-1} \cdot)]}_{L^{\infty}(\R^d)} \lesssim_{\sigma, m, \phi} L^{\sigma}
  \end{equation*}
  for every $m \in \N_0^d$.
  By the Leibniz rule,
  \begin{equation*}
    \partial^m[ w_{-\sigma} \phi(L^{-1} \cdot) ]
    = \sum_{l=0}^m \binom{m}{l} L^{-\abs{m-l}} \partial^l w_{-\sigma} \partial^{m-l} \phi (L^{-1} \cdot).
  \end{equation*}
  Since $\abs{\partial^l w_{-\sigma}(x)} \lesssim_{\sigma, l} (1 + \abs{x}^2)^{\frac{\sigma - \abs{l}}{2}}$,
  we obtain
  \begin{align*}
    \sup_{x \in \R^d} \abs{\partial^m[ w_{-\sigma} \phi(L^{-1} \cdot)](x)}
    &\lesssim_{\sigma, m}
    \sum_{l=0}^m \binom{m}{l} L^{-\abs{m-l}}
    \sup_{x \in \R^d} (1 + \abs{x}^2)^{\frac{\sigma- \abs{l}}{2} } \abs{\partial^{m-l} \phi (L^{-1} x)} \\
    &=
    \sum_{l=0}^m \binom{m}{l} L^{-\abs{m-l}}
    \sup_{x \in \R^d} (1 + \abs{L x}^2)^{\frac{\sigma - \abs{l}}{2}} \abs{\partial^{m-l} \phi (x)} \\
    &\lesssim_{m, \sigma, \phi} L^{\sigma}. \qedhere
  \end{align*}
\end{proof}

\begin{lemma}
\label{lemma:estimate_C_delta_U_by_weighted}
Let $U$ be a bounded domain, $r\in \R$ and $\sigma \in (0,\infty)$. Then for for $Z\in \cC^{r,\sigma}(\R^d)$
  \begin{align*}
  \norm{Z}_{C^r(U_L)} \lesssim_{U,\sigma} L^\sigma \norm{Z}_{\csp^{r,\sigma}(\R^d)}, \qquad L \ge 1.
  \end{align*}
\end{lemma}
\begin{proof}
Let $\phi \in C_c^\infty(\R^d)$ be $1$ on a neighborhood of $U$. Then $  \norm{Z}_{C^r(U_L)} =   \norm{\phi (L \cdot) Z}_{C^r(U)} \le   \norm{\phi(L \cdot) Z}_{C^r(\R^d)}$.
By Lemma~\ref{lem:equivalence_sob_slobo_and_besov},
\begin{align*}
\norm{\phi (L\cdot) Z}_{C^r(\R^d)}
\lesssim_r \norm{\phi (L\cdot) Z}_{\csp^r(\R^d)}.
\end{align*}
Therefore, we obtain the desired estimate by an application of Lemma~\ref{lem:localization}.
\end{proof}

\begin{lemma}\label{lem:derivative_and_lifting_in_Besov}
  Let $p, q \in [1, \infty]$, $r \in \R$, $\sigma \in [0, \infty)$, $m \in \N_0^d$ and $a \in \R$.
  \begin{enumerate}
    \item
    \label{item:derivative_besov}
    One has $\norm{\partial^m f}_{B^{r - \abs{m}, \sigma}_{p, q}(\R^d)} \lesssim_{p, q, r, \sigma, m}
    \norm{f}_{B_{p, q}^{r, \sigma}(\R^d)}$.
    \item
	\label{item:lifting_besov}
    Let $\tilde{\chi}$ be a smooth function on $\R^d$ such that
      $\tilde{\chi} = 0$ in a neighborhood of $0$ and all the derivatives are bounded.
     Then, one has $\norm{\F^{-1}[\abs{2 \pi \cdot}^{2 a} \tilde{\chi} \F f]}_{B_{p, q}^{r - 2 a, \sigma}(\R^d)}
    \lesssim_{p, q, r, \sigma, a, \tilde{\chi}} \norm{f}_{B_{p, q}^{r, \sigma}(\R^d)}$.
  \end{enumerate}
\end{lemma}
\begin{proof}
  We only prove (a), as the proof of (b) is similar.
  We use the notations from Definition~\ref{def:besov_spaces}.
  Let $\psi$ be a compactly supported smooth function on $\R^d$ such that $\psi = 1$ on a neighborhood of
  $\supp(\chi)$ and set $\psi_j \defby \psi(2^{-j} \cdot)$.
Because $\Delta_j \partial^m f = \cF^{-1}[\chi(2^{-j} \cdot)\psi_j] * \Delta_j f$ for $j\in \N_0$, by Theorem~\ref{thm:weighted_young} we have
  \begin{equation}
  \label{eqn:derivative_LP_block}
    \norm{w_{\sigma} ( \Delta_j \partial^m f ) }_{L^p(\R^d)}
    \leq \norm{w_{-\sigma} \F^{-1}[(-2 \pi i \cdot)^m \psi_j]}_{L^1(\R^d)}
    \norm{w_{\sigma} \Delta_j f}_{L^{p}(\R^d)}.
  \end{equation}
  It remains to observe that  for all $j\in \N_0$
  \begin{align*}
    2^{-j \abs{m}} \norm{w_{-\sigma} \F^{-1}[(-2 \pi i \cdot)^m \psi_j]}_{L^1(\R^d)}
&    = \norm{w_{-\sigma} \F^{-1}[(-2^{-j + 1} \pi i \cdot)^m \psi(2^{-j} \cdot)]}_{L^1(\R^d)} \\
&     = \norm{ 2^{jd} \Big[  w_{-\sigma}(2^{-j} \cdot) [\F^{-1}[(-2 \pi i \cdot)^m \psi \Big] (2^{j} \cdot) }_{L^1(\R^d)} \\
&     \le  \norm{   w_{-\sigma} [\F^{-1}[(-2 \pi i \cdot)^m \psi ] }_{L^1(\R^d)},
  \end{align*}
  as $w_{-\sigma}(2^{-j} \cdot) \le w_{-\sigma}$.
  For $j=-1$ a similar estimate as \eqref{eqn:derivative_LP_block} holds for a $\tilde \psi \in C_c^\infty(\R^d)$ with $\tilde \psi =1 $ on $\supp (\tilde \chi)$.
\end{proof}

\begin{definition}
\label{def:Delta_le_LP_block}
For $J\in \N_0$ or $J=-1$ we write
\begin{align*}
\Delta_{\le J} f = \sum_{j=-1}^J \Delta_j f,
\qquad \Delta_{\ge J} f = \sum_{j=J}^\infty \Delta_j f.
\end{align*}
\end{definition}

\begin{remark}
\label{rem:sum_large_LP_blocks}
Observe that by definition of $\check \chi$ and $\chi$ (Definition~\ref{def:besov_spaces}), for $N\in\N$
\begin{align*}
(1-\check \chi)(2^{-N}x) = \sum_{j=N}^\infty \chi(2^{-j}x), \qquad x\in \R^d,
\end{align*}
and therefore
\begin{align*}
\Delta_{\ge N} f & = \cF^{-1} \Big( (1-\check \chi)(2^{-N}\cdot ) \cF f \Big) , \quad
\Delta_{\le N} f  = \cF^{-1} \Big( \check \chi(2^{-N}\cdot ) \cF f \Big) .
\end{align*}
\end{remark}

\begin{lemma}\label{lem:fourier_cutoff}
  Let $p, q \in [1, \infty]$, $r, s \in \R$ with $r \leq s$, $\sigma \in [0, \infty)$ and $N \in \N_0$.
  Then, one has  (observe the difference of the positions of $r$ and $s$)
  \begin{align*}
    \norm{\Delta_{\ge N} f}_{B_{p, q}^{r, \sigma}(\R^d)} &\lesssim_{s-r, \sigma} 2^{-(s-r)N}
    \norm{f}_{B_{p, q}^{s, \sigma}(\R^d)}, \\
    \norm{ \Delta_{\le N} f}_{B_{p, q}^{s, \sigma}(\R^d)}
    &\lesssim_{s-r, \sigma} 2^{(s-r)N} \norm{f}_{B_{p, q}^{r, \sigma}(\R^d)}.
  \end{align*}
\end{lemma}
\begin{proof}
  We first observe by Remark~\ref{rem:sum_large_LP_blocks} that $\Delta_{\ge N} \Delta_j f = [2^{Nd} \cF^{-1} (1 - \check{\chi}) (2^N \cdot) * f $.
  Thus, by \cite[Theorem 2.4 and Lemma 2.6]{mourrat_weber_17}, one has
  \begin{equation*}
    \norm{w_{\sigma}(\Delta_j \Delta_{\ge N} f)}_{L^p(\R^d)}
    \lesssim \norm{w_{\sigma}(\Delta_j f)}_{L^p(\R^d)}
  \end{equation*}
  Therefore,
  \begin{align*}
    \norm{\Delta_{\ge N} f}_{B_{p, q}^{r, \sigma}(\R^d)} &= \Big(
    \sum_{j=N-1}^{\infty} 2^{jqr} \norm{w_{\sigma} (\Delta_j \Delta_{\ge N} f)}_{L^p(\R^d)}^q
    \Big)^{\frac{1}{q}}\\
    &\leq 2^{-(N-1)(s-r)} \Big(
    \sum_{j=N-1}^{\infty} 2^{jqs} \norm{w_{\sigma} (\Delta_j \Delta_{\ge N} f)}_{L^p(\R^d)}^q
    \Big)^{\frac{1}{q}} \\
    &\lesssim_{s-r, \sigma} 2^{-N(s-r)} \Big(
    \sum_{j=N-1}^{\infty} 2^{jqs} \norm{w_{\sigma} (\Delta_j f)}_{L^p(\R^d)}^q
    \Big)^{\frac{1}{q}} \\
    &\leq 2^{-N(s-r)} \norm{f}_{B_{p, q}^{s, \sigma}(\R^d)}.
  \end{align*}
  The second inequality can be proven similarly.
\end{proof}

Recall the definition of an admissible kernel and of $G_N$, see Definition~\ref{def:admissible_kernel} and Definition~\ref{def:G_N_and_H_N}.

\begin{corollary}\label{cor:estimates_of_G_N_and_H_N}
  Let $p, q \in [1, \infty]$, $r, s \in \R$ with $r \leq s$, $\sigma \in [0, \infty)$ and $N \in \N_0$.
  Let $K$ be an admissible kernel. Set $H_N \defby G_N - K$.
  Then, one has
  \begin{align*}
    \norm{G_N \conv f}_{B_{p, q}^{r + 2, \sigma}(\R^d)} &\lesssim_{p, q, r, s, \sigma} 2^{-(s-r)N}
    \norm{f}_{B_{p, q}^{s, \sigma}(\R^d)}, \\
    \norm{H_N \conv f}_{B_{p, q}^{s + 2, \sigma}(\R^d)}
    &\lesssim_{p, q, r, s, \sigma} 2^{(s-r)N} \norm{f}_{B_{p, q}^{r, \sigma}(\R^d)},
   \\
       \norm{(G_N - G_0) \conv f}_{B_{p, q}^{s+2, \sigma}(\R^d)}
       & \lesssim_{p, q, r, s, \sigma} 2^{(s-r) N} \norm{f}_{B_{p, q}^{r, \sigma}(\R^d)}.
  \end{align*}
\end{corollary}
\begin{proof}
  Suppose $\psi \in C_c^{\infty}(\R^d)$ is $0$ in a neighborhood of $0$ and is
  equal to $1$ on $\supp(1 - \check{\chi})$.
  If we set $g \defby \F^{-1}[\abs{2 \pi \cdot}^{-2} \psi \F f]$, then
  $G_N \conv f = \Delta_{\ge N} g$. Therefore, the first claimed inequality
  follows from Lemma~\ref{lem:derivative_and_lifting_in_Besov}~\ref{item:lifting_besov} and Lemma~\ref{lem:fourier_cutoff}.

  To prove the second claimed inequality, recall that one has $H_N = (G_N - G_0) + (G_0 - K)$.
  By Lemma~\ref{lem:H_N_is_Schwarz} below, $G_0 - K$ belongs to $\S(\R^d)$. Therefore,
  \begin{equation*}
    \norm{(G_0 - K) \conv f}_{B_{p, q}^{s+2, \sigma}(\R^d)} \lesssim_{p, q, r, s, \sigma} \norm{f}_{B_{p, q}^{r, \sigma}(\R^d)}.
  \end{equation*}
  On the other hand, one has $(G_N - G_0) \conv f = (\Delta_{\ge N} - \Delta_{\ge 0}) g = (\Delta_{\le N-1} - \Delta_{-1})g$.
  Hence, by Lemma~\ref{lem:derivative_and_lifting_in_Besov}~\ref{item:lifting_besov} and by Lemma~\ref{lem:fourier_cutoff}, the third inequality and thus the second follow.
\end{proof}

Finally, we recall a wavelet characterization of weighted Besov spaces.
\begin{theorem}[{\cite{meyer_1993}, \cite[Theorem 1.61]{Triebel2006}}]\label{thm:scale_function_and_wavelet}
  For any $k \in \N$, there exist $\scalefcn, \motherfcn \in C_c^k(\R)$  with the following properties.
  \begin{itemize}
    \item For $n \in \N_0$, if we denote by $V_n$ the subspace of $L^2(\R)$ spanned by
    \begin{equation*}
      \set{\scalefcn(2^n \cdot - m) \given m \in \Z},
    \end{equation*}
    then the inclusions $V_0 \subseteq V_1 \subseteq \cdots \subseteq V_n \subseteq V_{n+1} \subseteq \cdots$ hold and
    $L^2(\R)$ is the closure of $\cup_{n \in \N_0} V_n$.
    \item The set
    \begin{equation*}
      \set{\scalefcn(\cdot - m) \given m \in \Z} \cup \set{\motherfcn(\cdot - m) \given m \in \Z}
    \end{equation*}
    forms an orthonormal basis of $V_1$.
    Therefore, the set
    \begin{equation*}
      \set{\scalefcn(\cdot - m) \given m \in \Z} \cup
      \set{2^{\frac{n}{2}} \motherfcn(2^n \cdot - m) \given n \in \N_0, m \in \Z}
    \end{equation*}
    forms an orthonormal basis of $L^2(\R)$.
    \item One has $\int_{\R} x^l \motherfcn(x) \dd x = 0$ for every $l \in \{1, 2, \ldots, k\}$.
  \end{itemize}
\end{theorem}
One can build an orthonormal basis of $L^2(\R^d)$ as follows.
\begin{proposition}[{\cite[Proposition 1.53]{Triebel2006}}]\label{prop:wavelet_basis}
  Let $k \in \N$ and let $\scalefcn, \motherfcn \in C_c^k(\R^d)$ be as in Theorem~\ref{thm:scale_function_and_wavelet}.
  For $n \in \N_0$, we define the sets of $d$-tuples by
  \begin{equation*}
    \fG^n \defby
    \begin{cases}
      \{ (\ff, \ldots, \ff) \} & \text{if } n = 0, \\
      \set{(G_1, \ldots, G_d) \in \{\ff, \fm\}^d \given \exists j \text{ s.t. } G_j = \fm} & \text{if } n \geq 1.
    \end{cases}
  \end{equation*}
  For $n \in \N_0$, $G \in \fG^n$, $m \in \Z^d$ and $x \in \R^d$, we set $(n-1)_+ = \max \{n-1,0\}$ and 
  \begin{equation}\label{eq:wavelet_basis}
    \wavelet_m^{n, G}(x)
    \defby 2^{\frac{d (n-1)_+}{2}} \prod_{j=1}^d \psi_{G_j}(2^{(n-1)_+}x_j - m_j).
  \end{equation}
  The set $\set{\wavelet_m^{n, G} \given n \in \N_0, G \in \fG^n, m \in \Z^d}$ forms an orthonormal
  basis of $L^2(\R^d)$.
\end{proposition}

With the expansion by the basis $\set{\wavelet_m^{n, G} \given n \in \N_0, G \in \fG^n, m \in \Z^d}$,
one can give a wavelet characterization of weighted Besov spaces.
\begin{proposition}[{\cite[Theorem 6.15]{Triebel2006}}]\label{prop:besov_wavelet}
  Let $p, q \in [1, \infty]$, $r \in \R$ and $\sigma \in (0, \infty)$.
  Suppose
  \begin{equation*}
    k > \max \Big\{r, \frac{2d}{p} + \frac{d}{2} - r\Big\}
  \end{equation*}
  and let $\set{\wavelet_m^{n, G} \given n \in \N_0, G \in \fG^n, m \in \Z^d}$ be  as in Proposition~\ref{prop:wavelet_basis}.
  Then, there exists a constant $C \in (0, \infty)$ such that for every $f \in B_{p, q}^{r, \sigma}(\R^d)$ one has
  \begin{multline*}
    C^{-1}\norm{f}_{B_{p,q}^{r, \sigma}(\R^d)} \\
    \leq
    \Big\lVert \Big(2^{n(r-d/p)} \Big(\sum_{G \in \fG^n, m \in \Z^d} w_{\sigma}(2^{-n} m)^p
    \abs{2^{nd/2} \inp{f}{\wavelet^{n, G}_m}}^p
    \Big)^{1/p} \Big)_{n \in \N_0} \Big\rVert_{l^q(\N_0)} \\
    \leq C \norm{f}_{B_{p,q}^{r, \sigma}(\R^d)}.
  \end{multline*}
\end{proposition}

\subsection{Estimates of constants of functional inequalities on bounded domains}
\label{subsec:estimates_of_constants}

In Definition~\ref{def:interpolation_constant} we have introduced the smallest constant that appears in interpolation inequalities. In this section we introduce also other constants that appear in functional inequalities and study their behaviour (also under scaling of the underlying domain).

\begin{definition}\label{def:functional_inequalities_constants}
  Let $U$ be a bounded domain and $p,p_1,p_2\in [1,\infty]$,  $r_1,r_2,s\in [0,\infty)$, $r \in (0,\infty)$ and $\delta \in (0,r)$.
  We set
  \begin{align*}
    &C^U_{\operatorname{Embed}}[W^{r_1}_{p_1} \to W^{r_2}_{p_2}] \defby
    \sup_{f \in W^{r_1}_{p_1}(U) \setminus \{0\}} \frac{\norm{f}_{W^{r_2}_{p_2}(U)}}
    {\norm{f}_{W^{r_1}_{p_1}(U)}}, \\
    &C^U_{\operatorname{Prod}}[W_{2p}^{r} \to W_p^{r - \delta}]
    \defby \sup_{f \in W^r_{2p}(U) \setminus \{0\}}\frac{\norm{f^2}_{W^{r-\delta}_{p}(U)}}
    {\norm{f}_{W^r_{2p}(U)}^2}.
  \end{align*}
Similarly, we set $C^U_{\operatorname{Embed}}[W^{r_1}_{p_1, 0} \to W^{r_2}_{p_2, 0}], \ldots$ by
  replacing the function spaces  ``$W_p^r$'' to those with zero boundary conditions ``$W_{p,0}^r$''.
  If $U$ is a bounded Lipschitz domain, for a universal extension operator $\iota$ from $U$ to $\R^d$ as in Lemma~\ref{lem:extension_and_trace_sobolev}, we set
  \begin{align*}
    &C^U_{\operatorname{Ext}}[W^{r_1}_{p_1}, W^{r_2}_{p_2}] \defby
    \inf \{ \norm{\iota}_{W^{r_1}_{p_1}(U) \rightarrow W^{r_1}_{p_1}(\R^d) }
    + \norm{\iota}_{W^{r_2}_{p_2}(U)\rightarrow W^{r_2}_{p_2}(\R^d) } \\
    &\hspace{5cm}\vert \iota \mbox{ is an universal extension operator}\}, \\
    & C^{\partial U}_{\operatorname{R}}(W^r_p) \defby \inf \{ \norm{\cR}_{W^{r }_p(\partial U) \to W^{r+\frac{1}{p}}_p(U)} \ \vert \
    \cR \mbox{ is a right inverse of } \cT_{W_p^{r+\frac1p}(U)}  \},  \\
    &C^{\partial U}_{\operatorname{Prod}}[W_{2p}^{r} \to W_p^{r - \delta}]
    \defby \sup_{f \in W^r_{2p}(\partial U) \setminus \{0\}}\frac{\norm{f^2}_{W^{r-\delta}_{p}(\partial U)}}
    {\norm{f}_{W^r_{2p}(\partial U)}^2}, \\
    &C^U_{\operatorname{Mult}}[W_p^r] \defby
    \sup_{f \in W_p^r(U) \setminus \{0\}} \frac{\norm{\indic_U f}_{W_p^r(\R^d)}}{\norm{f}_{W^r_p(U)}},
  \end{align*}
  and $C^U_{\operatorname{Ext}}[W^{r}_{p}] \defby
  C^U_{\operatorname{Ext}}[W^{r}_{p}, W^{r}_{p}]$.
\end{definition}

\begin{lemma}\label{lem:scaling_of_embed_const}
  Let $U$ be a domain.
  \begin{enumerate}
  \item Let $p_1, p_2 \in (1, \infty)$ with $p_1 \leq p_2$ and
  $r_1, r_2 \in [0, \infty)$ with $r_2 = r_1 - d(\frac{1}{p_1} - \frac{1}{p_2})$. Then, one has
  $C^U_{\operatorname{Embed}}[W^{r_1}_{p_1, 0} \to W^{r_2}_{p_2, 0}] \lesssim_{p_1, p_2, r_1} 1$.
  If $U$ is a bounded Lipschitz domain, one has
  $ C^U_{\operatorname{Embed}}[W^{r_1}_{p_1} \to W^{r_2}_{p_2}] \lesssim_{p_1, p_2, r_1}
    C^U_{\operatorname{Ext}}[W^{r_1}_{p_1}]$.
  \item Let $s \in (0, 1)$. Then, one has
  $C^U_{\operatorname{IP}}[H^s_0] \lesssim_s 1$.
  If $U$ is a bounded Lipschitz domain, one has
  $C^U_{\operatorname{IP}}[H^s] \lesssim_s C^U_{\operatorname{Ext}}[L^2, H^1]$.
  \end{enumerate}
\end{lemma}
\begin{proof}
  We only prove the claim for a bounded Lipschitz domain $U$.

  (a) Let $\iota$ be a universal extension operator from $U$ to $\R^d$. Then, by using
  the Sobolev embedding in $\R^d$ \cite[Theorem 8.12.6]{Bhattacharyya_2012} for the second inequality,
  \begin{equation*}
    \norm{f}_{W^{r_2}_{p_2}(U)} \leq \norm{\iota(f)}_{W^{r_2}_{p_2}(\R^d)}
    \lesssim_{p_1, p_2, r_1} \norm{\iota(f)}_{W^{r_1}_{p_1}(\R^d)}
    \leq \norm{\iota}_{W_{p_1}^{r_1}(U) \rightarrow W_{p_1}^{r_1}(\R^d)} \norm{f}_{W^{r_1}_{p_1}(U)},
  \end{equation*}
and thus $\norm{f}_{W^{r_2}_{p_2}(U)} \lesssim_{p_1, p_2, r_1}  C^U_{\operatorname{Extend}}[W^{r_1}_{p_1}] \norm{f}_{W^{r_1}_{p_1}(U)}$.

  (b) We can prove the claim similarly by using the inequality
  \cite[Proposition 2.22]{Bahouri2011}
  \begin{equation*}
    \norm{f}_{H^s(\R^d)} \lesssim_s \norm{f}_{L^2(\R^d)}^{1-s} \norm{f}_{H^1(\R^d)}^s. \qedhere
  \end{equation*}
\end{proof}

\begin{lemma}\label{lem:scaling_of_prod_const}
  Let $U$ be a bounded domain, $p \in [1, \infty)$, $r \in (0, 1)$ and $\epsilon \in (0, r)$.
  Then we have
  \begin{equation*}
    C^U_{\operatorname{Prod}}[W_{2p, 0}^{r} \to W_{p, 0}^{r- \epsilon}]
    \lesssim_{p, \epsilon} 1,
  \end{equation*}
  and if $U$ is a bounded Lipschitz domain
  \begin{equation*}
    C^U_{\operatorname{Prod}}[W_{2p}^{r} \to W_p^{r - \epsilon}]
    \lesssim_{p, \epsilon} 1,
  \end{equation*}
  \begin{equation*}
    C^{\partial U}_{\operatorname{Prod}}[W_{2p}^{r} \to W_p^{r - \epsilon}]
    \lesssim_{p, \epsilon} 1 + \sup_{x \in \partial U}
    \Big(\int_{\partial U} \frac{\dd y}{\abs{x-y}^{d - 1 - 2 p \epsilon}}\Big)^{\frac{1}{2p}}.
  \end{equation*}
\end{lemma}
\begin{proof}
  We only prove the first inequality.
  Since
  \begin{equation*}
    \int_{\abs{x} \geq 1} \frac{1}{\abs{x}^{d + p (r-\epsilon)}} \dd x < \infty,
  \end{equation*}
  one has
  \begin{equation*}
    \norm{f^2}_{W_{p}^{r- \epsilon}(\R^d)} \lesssim_{p, r} \norm{f^2}_{L^{p}(\R^d)} +
    \int_{\abs{x-y} \leq 1} \frac{\abs{f(x)^2 - f(y)^2}^p}{\abs{x - y}^{d + p(r - \epsilon)}} \dd x \dd y.
  \end{equation*}
Observe that $\norm{f^2}_{L^{p}(\R^d)} \leq \norm{f}^2_{L^{2p}(\R^d)}$.
Furthermore, observe that
  \begin{equation*}
    \frac{\abs{f(x)^2 - f(y)^2}^p}{\abs{x-y}^{d + p(r - \epsilon)}}
    = \Big( \frac{\abs{f(x) + f(y)}}{\abs{x-y}^{\frac{d}{2p}  - \epsilon}} \Big)^p
     \Big( \frac{\abs{f(x) - f(y)}}{\abs{x-y}^{\frac{d}{2p} + r}} \Big)^p,
  \end{equation*}
so that, by  H\"older's inequality
  \begin{equation*}
    \int_{\abs{x-y} \leq 1} \frac{\abs{f(x)^2 - f(y)^2}^p}{\abs{x - y}^{d + p (r - \epsilon)}} \dd x \dd y
    \leq \norm{f}_{W^r_{2p}(\R^d)} \Big(
    \int_{\abs{x - y} \leq 1} \frac{\abs{f(x) + f(y)}^{2p}}{\abs{x-y}^{d - 2p \epsilon}} \dd x \dd y \Big)^{\frac{1}{2p}}.
  \end{equation*}
  \begin{calc}
  \begin{align*}
   \int_{\abs{x - y} \leq 1} \frac{\abs{f(x) + f(y)}^{2p}}{\abs{x-y}^{d - 2p \epsilon}} \dd x \dd y
&    \le   2 \int_{\abs{x - y} \leq 1} \frac{\abs{f(x) }^{2p}}{\abs{x-y}^{d - 2p \epsilon}} \dd x \dd y \\
&    \le    2 \|f\|_{L^{2p}}^{2p} \int_{\abs{z} \leq 1} \frac{1}{\abs{z}^{d - 2p \epsilon}} \dd z.
  \end{align*}
  \end{calc}
Now the latter integral can be estimated by $\|f\|_{L^{2p}}$ times the following integral over the unit ball that can be estimated as follows
  \begin{equation*}
    \int_{\abs{x} \leq 1} \frac{\dd x}{\abs{x}^{d - 2p \epsilon}}
    \lesssim \int_0^1 \frac{dr}{r^{1 - 2p \epsilon}} = \frac{1}{2 p \epsilon}. \qedhere
  \end{equation*}
\end{proof}

\begin{lemma}[{\cite[Proposition 5.3]{Triebel2002}}]\label{lem:mult_by_Lipschitz_domain}
  Let $p \in (1, \infty)$, $r \in (0, \frac{1}{p})$ and let $U$ be a bounded Lipschitz domain.
  Then, the map
  \begin{equation*}
    B_{p, p}^r(\R^d) \rightarrow B_{p, p}^r(\R^d), \quad f \mapsto f \indic_U
  \end{equation*}
  is a bounded linear operator.
\end{lemma}

\begin{lemma}\label{lem:scaling_of_iota}
  Let $U$ be a bounded Lipschitz domain.
  Then, we have
  \begin{align}
\label{eqn:estimate_scaling_Ext}
    \sup_{L \geq 1} C^{U_L}_{\operatorname{Ext}}[W^{r_1}_{p_1}, W^{r_2}_{p_2}] < \infty,
    & \quad p_1, p_2 \in [1, \infty], r_1, r_2 \in [0, \infty), \\
\label{eqn:estimate_scaling_cT}
    \sup_{L \geq 1} \norm{\cT_{W^{r}_{p}(U_L)}} < \infty,
    & \quad p \in (1, \infty), r \in (\tfrac1p, 1 + \tfrac1p), \\
\label{eqn:estimate_scaling_cR}
    \sup_{L \geq 1}  C^{\partial U_L}_{\operatorname{R}}[W^{r}_{p}]  < \infty,
    & \quad p \in (1, \infty), r \in (0, 1), \\
\label{eqn:estimate_scaling_Embed}
     \sup_{L \geq 1} C^{U_L}_{\operatorname{Embed}}[W^{r_1}_{p_1} \to W^{r_2}_{p_2}]
 < \infty,
    & \quad
    \begin{cases}
    p_1, p_2 \in (1, \infty), r_1, r_2 \in [0, \infty), & \\   p_1 \leq p_2,  r_2 = r_1 - d(\frac{1}{p_1} - \frac{1}{p_2}),
    \end{cases} \\
    \label{eqn:estimate_scaling_Interpolation}
     \sup_{L \geq 1}  C^{U_L}_{\operatorname{IP}}[H^s] < \infty
     & \quad s\in (0,1), \\
\label{eqn:estimate_scaling_Mult}
\sup_{L \geq 1} C^{U_L}_{\operatorname{Mult}}[W_p^r] < \infty,
& \quad p \in (1,\infty), r\in (0,\tfrac1p), \\
\label{eqn:estimate_scaling_Prod_dom}
\sup_{L \ge 1} C^{U_L}_{\operatorname{Prod}}[W_{2p}^{r} \to W_p^{r - \epsilon}] <\infty
 & \quad p \in [1,\infty), r\in (0,1), \epsilon\in (0,r), \\
\label{eqn:estimate_scaling_Prod_boundary_dom}
\sup_{L\ge 1} L^{-\epsilon} C^{\partial U_L}_{\operatorname{Prod}}[W_{2p}^{r} \to W_p^{r - \epsilon}] <\infty,
 & \quad p \in [1,\infty), r\in (0,1), \epsilon\in (0,r).
  \end{align}
If $U$ is a bounded domain (that is not necessarily Lipschitz), then \eqref{eqn:estimate_scaling_Embed} and \eqref{eqn:estimate_scaling_Prod_dom} hold by replacing the occurrences of the form ``$W^a_b$'' by ``$W^a_{b,0}$''.
\end{lemma}
\begin{proof}
  Let $\iota$ be a universal extension operator from $U$ to $\R^d$ (Definition~\ref{def:universal_extension_op_and_trace}).
  For $L \geq 1$, we define a universal extension operator $\iota_L$ from $U_L$ to $\R^d$ by
  $\iota_L(f) \defby \iota(f(L \cdot))(L^{-1} \cdot)$.
  By change of variables, and using that $\partial^\alpha \iota(f) = \iota( \partial^\alpha f)$ it is straightforward to check that  $\norm{\iota_L}_{W^r_p(U_L) \rightarrow W^r_p(\R^d)} \leq \norm{\iota}_{W^r_p(U) \rightarrow W^r_p(\R^d)}$.
  This implies \eqref{eqn:estimate_scaling_Ext}.
  The two \eqref{eqn:estimate_scaling_cT} and \eqref{eqn:estimate_scaling_cR} estimates can be  proven similarly.

\eqref{eqn:estimate_scaling_Embed} and \eqref{eqn:estimate_scaling_Interpolation}  follow by Lemma~\ref{lem:scaling_of_embed_const} and by \eqref{eqn:estimate_scaling_Ext}.

\eqref{eqn:estimate_scaling_Mult}
First observe $r\in (0,1)$. Set $F \defby \iota_L(f)$ for  $f \in W_p^r(U_L)$. Then, $\1_{U_L}f = \1_{U_L} F$ and thus $\norm{\indic_{U_L} f}_{W_p^r(\R^d)} \le \norm{F}_{L^p} +    [\indic_{U_L} F]_{W^{r}_p(\R^d)} $.
  By change of variables, one has $  [g(L^{-1} \cdot) ]_{W_p^{s}(\R^d)} = L^{\frac{d}{p} -s} [g]_{W_p^{s}(\R^d)}$ and thus
  \begin{calc}
  \begin{align*}
  [g(L^{-1} \cdot) ]_{W_p^{s}(\R^d)}
& =    \Big( \int_{\R^d \times \R^d}
      \frac{\abs{g(L^{-1} x) - g(L^{-1} y)}^p}{\abs{x - y}^{d + p s}} \dd x\dd y \Big)^{\frac{1}{p}} \\
&  =  \Big( \int_{\R^d \times \R^d}
      \frac{\abs{g(w) - g(z)}^p}{L^{d+ps} \abs{w - z}^{d + p s}} L^{2d} \dd w\dd z \Big)^{\frac{1}{p}} \\
& =  L^{\frac{d}{p} -s} \Big( \int_{\R^d \times \R^d}
      \frac{\abs{g(w) - g(z)}^p}{L^{d+ps} \abs{w - z}^{d + p s}} L^{2d} \dd w\dd z \Big)^{\frac{1}{p}} =  L^{\frac{d}{p} -s} [g]_{W_p^{s}(\R^d)} .
  \end{align*}
  \end{calc}
  \begin{equation*}
    [\indic_{U_L} F]_{W^{r}_p(\R^d)} = L^{\frac{d}{p} - r} [\indic_{U} F(L \cdot)]_{W^{r}_p(\R^d)}.
  \end{equation*}
By Lemma~\ref{lem:mult_by_Lipschitz_domain}
 and Lemma~\ref{lem:equivalence_sob_slobo_and_besov},  one has
  \begin{align*}
    [\indic_{U} F(L \cdot)]_{W^{r}_p(\R^d)} \le \norm{\indic_{U} F(L \cdot)}_{W^{r}_p(\R^d)}
 &   \lesssim_{U, p, r} \norm{F(L \cdot)}_{L^p(\R^d)} + [F(L \cdot)]_{W^{r}_p(\R^d)} \\
&    = L^{-\frac{d}{p}} \norm{F}_{L^p(\R^d)} + L^{-\frac{d}{p} - r} [F]_{W^{r}_p(\R^d)}.
  \end{align*}
 The claim follows because $\norm{F}_{W^p_r(\R^d)} \le  \norm{\iota_L}_{W_p^r(U_L) } \norm{f}_{W^p_r(U_L)}  \le  \norm{\iota}_{W_p^r(U) } \norm{f}_{W^p_r(U_L)}$.

\eqref{eqn:estimate_scaling_Prod_dom} and \eqref{eqn:estimate_scaling_Prod_boundary_dom} follow from Lemma~\ref{lem:scaling_of_prod_const}.
\end{proof}

\section{A regularity structure for the gPAM}\label{sec:review_of_reg_str}

In this appendix, we list the necessary definitions and results from the paper of Bruned, Hairer and Zambotti~\cite{bruned_2019} regarding the regularity structure for the generalized Parabolic Anderson model
\begin{equation}
  \partial_0 u = \Delta u + \sum_{i, j=1}^d g_{i, j}(u) \partial_i u \partial_j u
  + \sum_{i=1}^d h_i(u) \partial_i u + k(u) + f(u) \xi, \label{eq:gpam_app}
\end{equation}
that we use in Appendix~\ref{sec:bphz_AH}. 

\subsection{Terminologies}\label{subsec:terminologies_reg_str}
Here we review some terminologies from \cite{bruned_2019}.
\begin{definition}
  We fix a $\emph{type set}$ $\gls{typeset} \defby \{\Xi, \rint\}$.
  The symbol $\gls{Xi}$ represents the noise $\xi$ and the symbol $\gls{integration}$
  represents an abstract integration operator.
\end{definition}
\begin{definition}\label{def:terminology_tree}
  We define the following notions regarding graphs.
  \begin{enumerate}[leftmargin=*]
    \item A \emph{rooted tree} is a finite connected simple graph without cycles, with a distinguished vertex called
    the \emph{root}. We do not allow for an empty tree but we allow for a \emph{trivial tree} $\trivialtree$ which consists of only
    one vertex. Vertices will be called \emph{nodes}.
    Given a rooted tree $T$, the set of nodes and that of edges are denoted by $N = N_T$ and by $E = E_T$ respectively.
    We denote by $\gls{root}$ the root of $T$.
    Nodes of a rooted tree are endowed with a partial order $\leq$ by their distances
    from the root. We orient edges $(x, y) \in E$
    so that $x \leq y$.
    \item A \emph{forest} is a finite simple graph without cycles. We say a forest is \emph{rooted} if every
    component of the forest is a rooted tree.
    We allow for an empty rooted forest.
    Given a rooted forest $F$, the set of nodes and that of edges are denoted by $N = \gls{node}$ and by $E = \gls{edge}$ respectively.
    \item A tree or a forest is called $\emph{typed}$ if it is endowed with a map $\gls{typemap}:E \to \typeset$ where
    $E$ is the set of edges.
    \item We say $A$ is a \emph{subforest} of a forest $F$, and write $A \subseteq F$, if
    $N_A \subseteq N_F$ and $E_A \subseteq E_F$ and if $(x, y) \in E_A$ implies $\{x, y\} \subseteq N_A$.
    We note that a \emph{subtree} of a rooted tree is again a rooted tree whose root is the unique vertex which is
    closest to the root of the original rooted tree. Therefore, a subforest of a rooted forest is again
    a rooted forest. If a forest is typed, a subforest inherits types by restriction.
    If $A$ and $B$ are (rooted, typed) forests, we denote by $A \sqcup B$ the disjoint union of $A$ and $B$ with
    types naturally inherited.
  \end{enumerate}
\end{definition}
\begin{definition}\label{def:forests}
  In this paper, a typed forest $F$ is often equipped with a \emph{colouring} $\hat{F}$ and \emph{decorations}
  $\fn, \fo, \fe$ as follows.
  \begin{enumerate}[leftmargin=*]
    \item A pair $\gls{colourful_forest}$ is called a \emph{colourful forest} if the following hold:
    \begin{itemize}
      \item $F = (E_F, N_F, \ft)$ is a typed rooted forest.
      \item One has $\hat{F}:E_F \sqcup N_F \to \{0,1, 2\}$ such that if $\hat{F}((x,y)) = i > 0$ for $(x, y) \in E_F$
      then $\hat{F}(x) = \hat{F}(y) = i$.
    \end{itemize}
    \item  If $(F, \hat{F})$ is a colourful forest and
    \begin{itemize}
      \item $\fn:N_F \to \N_0^d$,
      \item $\fo:N_F \to \Z \oplus \Z[\typeset]$ with $\supp(\fo) \subseteq \cup_{i > 0}\hat{F}^{-1}(i)$,
      \item $\fe:E_F \to \N_0^d$ and $\supp(\fe) \subseteq E_F \setminus \supp(\hat{F})$,
    \end{itemize}
    then the $5$-tuple $\gls{decorated_forest}$, also written $(F, \hat{F})^{\fn, \fo}_{\fe}$, is called a \emph{decorated forest}.
      We denote by $\rforest$ the set of decorated forests.
    \item For $x, y \in N_F$, we write $x \sim y$ if they are connected in $\cup_{i > 0} \hat{F}^{-1}(i)$.
    \item Given a decorated forest $(F, \hat{F}, \fn, \fo, \fe)$, we view a subforest $A \subseteq F$ as
    a decorated forest by restricting the associated maps $(\hat{F}, \fn, \fo, \fe)$.
    \item We write $\gls{trivial_tree_with_decoration}$ for the decorated tree
      $(\bullet, 2, m, 0, 0)$.
  \end{enumerate}
\end{definition}
Many examples of colourful forests can be found in \cite{bruned_2019}.
\begin{definition}
  Two notions of product for forests are defined as follows.
  \begin{enumerate}[leftmargin=*]
    \item For decorated forests $\tau_i = (F_i, \hat{F}_i, \fn_i, \fo_i, \fe_i)$ ($i=1,2$),
    we define the \emph{forest product} by
    \begin{equation*}
      \gls{forest_product} \defby (F_1 \sqcup F_2, \hat{F}_1 + \hat{F}_2, \fn_1 + \fn_2, \fo_1 + \fo_2, \fe_1 + \fe_2)
    \end{equation*}
    where, for $i \neq j$, $(\hat{F}_i, \fn_i, \fo_i, \fe_i)$ are set to $0$ on $F_j$.
    \item For a a decorated forest $\tau = (F, \hat{F}, \fn, \fo, \fe)$, we denote by $\gls{join}(\tau)$ the decorated tree
    \begin{equation*}
      (\rjoin(F), [\hat{F}], [\fn], [\fo], \fe),
    \end{equation*}
    where $\rjoin(F)$ is the tree obtained by gluing all the roots of $F$,
    \begin{equation*}
      [\hat{F}](\rho_{\rjoin(F)}) \defby \max_{y \text{ is a root of $F$}} \hat{F}(y), \quad
      [\hat{F}](x) = \hat F(x) \,\, \mbox{for } x \neq \rho_{\rjoin(F)}
    \end{equation*}
    and $[\fn]$ and $[\fo]$
    are defined at the new root by summing the values at the roots of $F$,  and are equal to $\fn$ and $\fo$ elsewhere, respectively.
    The \emph{tree product} is defined by
    \begin{equation*}
      \gls{tree_product} \defby \rjoin (\tau_1 \cdot \tau_2).
    \end{equation*}
  \end{enumerate}
\end{definition}
\begin{definition}[{\cite[Section~4.3]{bruned_2019}}]\label{def:tree_join}
    For a decorated tree $\tau = (T, \hat{T})^{\fn, \fo}_{\fe}$ and $k \in \N_0^d$,
    we write
    $\rint_k(\tau)$ for a decorated tree $\sigma = (S, \hat{S})^{\tilde{\fn}, \tilde{\fo}}_{\tilde{\fe}}$
    obtained by connecting the old root $\rho_{\tau}$ to a new root $\rho_{\sigma}$ with a new edge $e= (\rho_0,\rho_\tau)$
    and by defining $\hat{S}$, $\tilde{\fn}$, $\tilde{\fo}$ and
    $\tilde{\fe}$ as an extension of $\hat{T}$, $\fn$, $\fo$ and $\fe$ such that
    \begin{equation*}
      \hat{S}(\rho_{\sigma}) = \tilde{\fn}(\rho_{\sigma}) = \tilde{\fo}(\rho_{\sigma}) = 0, \quad \ft(e) = \rint, \quad
      \hat{S}(e) = 0, \quad \tilde{\fe}(e) = k,
    \end{equation*}
    For $i \in \{1, \ldots, d\}$, we write $\rint_i(\tau) \defby \rint_{\bm{e_i}}(\tau)$,
    where $\bm{e}_i$ is the $i$th standard basis of $\R^d$.
\end{definition}
\begin{definition}[{\cite[Definition~5.3]{bruned_2019}}]\label{def:degree}
  We assign the \emph{degree} $\gls{degree}$ to the types by
  \begin{equation}\label{eq:xi_degree}
    \abs{\Xi}\defby -2 + \delta, 
     \quad \abs{\rint} \defby 2.
  \end{equation}
  We extend the degree to $(k, v) \in \Z^d \oplus \Z[\typeset]$ by
  \begin{equation*}
    \abs{(k, v)} \defby \sum_{i=1}^d k_i + a \abs{\Xi} + b \abs{\rint}
  \end{equation*}
  where $v = a \Xi + b \rint$.
  For a decorated tree $\tau = (F, \hat{F}, \fn, \fo, \fe)$,
  we set
  \begin{equation*}
    \hat{E}_i \defby \hat{F}^{-1}(i) \cap E_F, \quad \hat{E} \defby \hat{E}_1 \cup \hat{E}_2, \quad
    \hat{N}_i \defby \hat{F}^{-1}(i) \cap N_F
  \end{equation*}
  and we define two notions of degrees $\abs{\cdot}_-$ and
  $\abs{\cdot}_+$ by
  \begin{align*}
    \abs{\tau}_- &\defby \sum_{e \in E_F \setminus \hat{E}} (\abs{\ft(e)} - \fe(e))
    + \sum_{x \in N_F} \fn(x) \\
    \abs{\tau}_+ &\defby \sum_{e \in E_F \setminus \hat{E}_2} (\abs{\ft(e)} - \fe(e))
    + \sum_{x \in N_F} \fn(x) + \sum_{x \in N_F \setminus \hat{N}_2} \abs{\fo(x)}.
  \end{align*}
\end{definition}
\subsection{Hopf algebras on forests and trees}
In this section, we introduce Hopf algebra structures on some spaces of forests and those of trees.
For this purpose, we begin with introducing contraction operators.
\begin{definition}[{\cite[Definition~3.18]{bruned_2019}}]
  We set
  \begin{equation*}
    \gls{contraction} (F, \hat{F})^{\fn, \fo}_{\fe} \defby (\rcont_{\hat{F}} F, \hat{F})^{[\fn], [\fo]}_{[\fe]},
  \end{equation*}
  where
  \begin{itemize}
    \item $\rcont_{\hat{F}} F$ is the quotient forest $F / \sim$, where the equivalent relation $\sim$ is
    in the sense of Definition~\ref{def:forests}-(c);
    \item $\hat{F}$ and $[\fe]$ are natural  ``restrictions'';
    \item one has $[\fn](x) \defby \sum_{y \sim x} \fn(y)$;
    \item one has
    \begin{equation*}
      [\fo](x) \defby \sum_{y \sim x} \fo(y) + \sum_{e \in E(x)} \ft(e), \quad E(x) \defby
      \set{(y, z) \in E \given y \sim z \sim x}.
    \end{equation*}
  \end{itemize}
\end{definition}

 For a decorated forest $\tau$ and $i \in \N$, one has a unique decomposition $\tau = \mu \cdot \nu$ such that
  on $\nu$ the map $\hat{F}$ is equal to $i$ and on each component of $\mu$ the map $\hat{F}$ is not equal to $i$ everywhere.

\begin{definition}
  Then, we set
  \begin{equation*}
    k_i(\nu) \defby
    \begin{cases}
      (\trivialtree, i, \sum_{x \in N_{\nu}} \fn(x), 0, 0) &\text{if } \sum_{x \in N_{\nu}} \fn(x) > 0 \\
      \emptyset &\text{otherwise}
    \end{cases}
  \end{equation*}
  and
  \begin{equation*}
    \rcont_i (\tau) \defby \rcont(\mu) \cdot k_i(\nu).
  \end{equation*}
  In addition, we denote by $\hat{\rcont}_i (\tau)$ the decorated forest that is obtained from $\rcont_i(\tau)$ by setting
  $\fo$ to $0$ on $\hat{F}^{-1}(i)$.
\end{definition}
\begin{remark}
  $\nu$ is allowed to be an empty forest $\emptyset$ and $k_i(\emptyset) = \emptyset$.
\end{remark}
With these operators, one can write Hopf algebras associated to regularity structures and renormalization structures.
\begin{definition}
  We define vector spaces $H_1, H_{\circ}$ as follows.
  \begin{enumerate}[leftmargin=*]
    \item We denote by $\gls{H_1}$ the free vector space generated by
    \begin{equation*}
      \gls{basis}(H_1) \defby \set{(F, \hat{F})^{\fn, \fo}_{\fe} \given \hat{F} \leq 1, \,\, \rcont_1(F) = F}.
    \end{equation*}
    \item We denote by $\gls{H_circ}$ the free vector space generated by $\rbasis(H_{\circ})$, where $\tau \in \rbasis(H_{\circ})$
    if and only if
    \begin{multicols}{2}
    \begin{itemize}
      \item $\tau$ is a tree and $\hat{F} \leq 1$;
      \item $\rcont(\tau) = \tau$.
    \end{itemize}
    \end{multicols}
  \end{enumerate}
\end{definition}
\begin{definition}[{\cite[Definition~3.3]{bruned_2019}}]\label{def:coproduct_for_decorated_forests}
  Given a decorated forest $\tau = (F, \hat{F})^{\fn, \fo}_{\fe}$, we denote
  by $\fU_1(\tau)$ the set of all subforests of $F$ which contains $\hat{F}^{-1}(1)$ and  subforests of $F$ that are disjoint from
  $\hat{F}^{-1}(2)$.
  We set
  \begin{equation}\label{eq:coproduct_formula}
    \begin{multlined}
      \Delta_1 \tau \defby \sum_{A \in \fU_1(\tau)}  \sum_{\fn_A : \fn_A \leq \fn} \sum_{\epsilon^F_A}  \frac{1}{\epsilon^F_A !}
      \binom{\fn}{\fn_A} (A, F|_A, \fn_A + \pi \epsilon^F_A, \fo\vert_{N_A}, \fe \vert_{E_A}) \\
      \otimes (F, \hat{F} \cup_1 A, \fn - \fn_A, \fo + \fn_A + \pi(\epsilon^F_A - \fe \indic_A),
      \fe \indic_{E_F \setminus E_A} + \epsilon^F_A),
    \end{multlined}
  \end{equation}
  where
  \begin{itemize}
    \item $\epsilon^F_A$ runs over all maps $E_F \to \N_0^d$ supported on the (outgoing) boundary
    \begin{equation*}
      \partial (A, F) \defby \set{(e_+, e_-) \in E_F \setminus E_A \given e_+ \in  N_A };
    \end{equation*}
    \item for $\epsilon: E_F \to  \N_0^d $ one defines $\pi \epsilon: N_F \to \Z^d$ by
    \begin{equation*}
      \pi \epsilon(x) \defby \sum_{ \substack{e  \in E_F\\ e= (x, y) \text{ for some } y}} \epsilon(x);
    \end{equation*}
    \item $\hat{F} \cup_1 A$ is the map defined by
    \begin{equation*}
      \hat{F} \cup_1 A (x) \defby
      \begin{cases}
        1 & \text{if } x \in A \\
        \hat{F}(x) & \text{otherwise}.
      \end{cases}
    \end{equation*}
  \end{itemize}
\end{definition}
Some of the main results from \cite{bruned_2019} are the following.
\begin{proposition}[{\cite[Proposition~4.11]{bruned_2019}}]\label{prop:forest_hopf_algebra}
    The vector space $H_1$ is a Hopf algebra with multiplication
    \begin{equation*}
      \cM(\tau_1 \otimes \tau_2) \defby \rcont_1 (\tau_1 \cdot \tau_2),
    \end{equation*}
    with unit $\emptyset$, with coproduct $(\rcont_1 \otimes \rcont_1) \Delta_1$ and with
    counit
    \begin{equation*}
      \indic'_{H_1}((F, \hat{F})^{\fn, \fo}_{\fe}) \defby
      \indic_{\{\emptyset\}}((F, \hat{F})^{\fn, \fo}_{\fe}), \quad
    \end{equation*}
    The Hopf algebra $H_1$ is graded with respect to $\abs{\cdot}_-$.
\end{proposition}
\begin{proposition}[{\cite[Proposition 4.14]{bruned_2019}}]\label{prop:forest_coaction}
     The vector space $H_{\circ}$ is a left comodule over the Hopf algebra $H_1$ with coaction
    \begin{equation*}
      (\rcont_1 \otimes \rcont) \Delta_1: H_{\circ} \to H_1 \otimes H_{\circ}.
    \end{equation*}
    \begin{calc}
    The coaction is compatible with the gradings $(H_{\circ}, \abs{\cdot}_-)$ and $(H_1, \abs{\cdot}_-)$ in that
    if $\tau \in \rbasis(H^C_{\circ})$ and if one write
    \begin{equation*}
      (\rcont_1 \otimes \rcont) \Delta_1 \tau =: \sum \tau^{(1)} \otimes \tau^{(2)},
    \end{equation*}
    then one has $\abs{\tau}_- = \abs{\tau^{(1)}}_- + \abs{\tau^{(2)}}_-$.
    \end{calc}
\end{proposition}
\subsection{Rule}
\begin{calc}
The spaces $H_1$, $H_2$ and $H_{\circ}$ are too big for our applications. For instance,
$H_{\circ}$ contains an element
\begin{equation*}
  \begin{drawtree}
    \node[black node]{} [grow'=up]
      child {node[black node]{} edge from parent node[left] {$\Xi$}}
      child {node[black node]{} edge from parent node[right] {$\Xi$}};
  \end{drawtree}\,\,\,,
\end{equation*}
which should not be an element of our regularity structure.
To solve this problem, we introduce the notion of \emph{rule}.
\end{calc}

We set
\begin{equation*}
  \gls{tree} \defby \set{(F, \hat{F}, \fn, \fo, \fe) \in \rforest \given
  F \mbox{ is a tree, } \hat{F} \equiv 0, \,\, \fo \equiv 0}.
\end{equation*}
The set $\rtree$ is a monoid with the tree product and with the trivial tree as unit.
We simply write $T^{\fn}_{\fe}$ for $(T, 0, \fn, 0, \fe) \in \rtree$.
\begin{definition}
  Given a decorated tree $T^{\fn}_{\fe} \in \rtree$, we associate to each $x \in N_T$ a node type
  \begin{equation*}
    \gls{node_type} \defby \rnode(x) \defby \big((\ft(e_1), \fe(e_1)), \ldots, (\ft(e_n), \fe(e_n))\big),
  \end{equation*}
  where $(e_1, \ldots, e_n)$ are the edges leaving the node $x$, namely, for each $j$ one can find a $y_j \in N_T$ such that
  $e_j = (x, y_j)$.
\end{definition}
\begin{definition}
  Let $(\typeset \times \N_0)^n/ \sim_n$ be the set of unordered $n$-tuples valued in $\typeset \times \N_0$ and
  let $\cP \cN$ be the power set of $\cup_{n \in \N_0} (\typeset \times \N_0)^n/ \sim_n$.
  We define the \emph{rule} $\gls{rule}: \typeset \to \cP \cN$ by
  \begin{align*}
    R(\Xi) &\defby \{()\},  \\
    R( \rint  ) &\defby \{([\rint]_n), ([\rint]_n, \rint_i), ([\rint]_n, \rint_i, \rint_j), ([\rint]_n, \Xi);
    n \in \N_0, i, j \in \{1, \ldots, d\} \},
  \end{align*}
  where we write $[\rint]_n$ for the $n$-tuple of $(\rint, 0)$ and write $\rint_i$ for $(\rint, \bm{e}_i)$,
  where $\bm{e}_i$ is the $i$th unit vector in $\R^d$.
\end{definition}
It is not difficult to show that the rule $R$ is subcritical in the sense of \cite[Definition 5.14]{bruned_2019}
and complete in the sense of \cite[Definition 5.20]{bruned_2019}.
\begin{calc}
\begin{lemma*}
  The rule $R$ is subcritical in the sense of \cite[Definition 5.14]{bruned_2019}. That is, one can find
  a map $\regmap: \typeset \to \R$, which defines a map on the power set of $\cup_n (\typeset \times \N_0)^n/ \sim_n$ by
  \begin{equation*}
    \regmap(N) \defby \sum_{(\ft, k) \in N} \big(\regmap(\ft) - k \big),
  \end{equation*}
  such that
  \begin{Cequation}\label{eq:subcritical}
    \regmap(\ft) < \abs{\ft} + \inf_{N \in R(\ft)} \regmap(N) \quad \mbox{for every } \ft \in \typeset.
  \end{Cequation}
\end{lemma*}
\begin{proof}
  For $\ft = \Xi$ to satisfy \eqref{eq:subcritical}, one needs
  \begin{equation*}
    \regmap(\Xi) < \abs{\Xi} = -2 + \delta,
  \end{equation*}
  and thus we set $\regmap(\Xi) \defby -2 + \frac{\delta}{2}$.
  For $\ft = \rint$ to satisfy \eqref{eq:subcritical}, one needs
  \begin{equation*}
    \regmap(\rint) < \abs{\rint} + \min\{0, \regmap(\rint) -1, 2 \regmap(\rint) -2, \regmap(\Xi)\}
  \end{equation*}
  or $0 < \regmap(\rint) < \frac{\delta}{2}$. Therefore, we set $\regmap(\rint) \defby  \frac{\delta}{3}$.
\end{proof}
\begin{remark*}
  By \cite[Proposition 5.15]{bruned_2019}, the set $\set{\tau \in \rtree_{\circ} \given \abs{\tau} \leq \gamma}$
  is finite for every $\gamma \in \R$.
  One can observe this from Lemma~\ref{lem:structure_of_rtree_minus}.
\end{remark*}
\end{calc}
\begin{definition}[{\cite[Definition~5.8]{bruned_2019}}]
  Let $\tau = T^{\fn}_{\fe} \in \rtree$.
  \begin{enumerate}[leftmargin=*]
    \item We say $\tau$ \emph{conforms to the rule $R$ at the node $x$} if the following hold:
    \begin{itemize}
      \item if $x$ is the root, then  $\rnode(x) \in R(\Xi)$ or $\rnode(x) \in R(\rint)$;
      \item otherwise, one has $\rnode(x) \in R(\ft(e))$, where $e$ is the edge such that
      $e = (y, x)$ for some node $y$.
    \end{itemize}
    \item We say $\tau$ \emph{conforms to the rule $R$} if $\tau$ comforms to $R$ at every node, except possibly the root.
    \item We say $\tau$ \emph{strongly conforms to the rule $R$} if $\tau$ comforms to $R$ at every node.
  \end{enumerate}
\end{definition}
\begin{definition}[{\cite[Definition~5.13]{bruned_2019}}]
  We define sets $\rtree_{\diamond}$ ($\diamond \in \{ \circ, 1,  -\}$) as follows.
  \begin{enumerate}[leftmargin=*]
    \item We denote by $\gls{tree_circ} \subseteq \rtree$ the set of trees which strongly conform to $R$.
    \item We denote by $\gls{tree_1} \subseteq \rforest$ the smallest submonoid under the forest product which contains $\rtree_{\circ}$.
    \item We denote by $\gls{tree_-} \subseteq \rtree_{\circ}$ the set of trees $T^{\fn}_{\fe}$ with the following properties:
    \begin{itemize}
      \item one has $\abs{\tau}_- < 0$ and $\fn(\rho_T) = 0$;
      \item if there exists only one edge containing $\rho_T$, then
      \begin{equation}\label{eq:tree_xi}
        T^{\fn}_{\fe} =
        \begin{drawtree}
            \node[black node, label=below:{$\rho_T$}] {} [grow'=up]
              child {node[black node] {}
              edge from parent
                node[right] {$e$}
              };
        \end{drawtree},
        \quad \fn(\rho_T) = \fe(e) = 0, \,\, \ft(e) = \Xi.
      \end{equation}
    \end{itemize}
  \end{enumerate}
\end{definition}
\begin{calc}
One can describe $\rtree_-$ more concretely in the following recursive way.
Let
\begin{align*}
  \rtree^{(1)} &\defby
  \set{\tau \in \rtree \given \forall e \in E_{\tau}, \rho_{\tau} \in e; \rnode(\rho_{\tau}) \in R(\rint)}, \\
  \rtree^{(1)}_- &\defby \set{\tau \in \rtree^{(1)} \given \abs{\tau} < 0}.
\end{align*}
Note $\rtree^{(1)}_-$ consists of two elements, both of which are of the form \eqref{eq:tree_xi}.
Then, recursively, we set
\begin{align*}
  \rtree^{(n)} &\defby \set{\rjoin^{\ft_1}_{k_1}(\tau_1) \cdots \rjoin^{\ft_l}_{k_l}(\tau_l) \given
  \big((\ft_j, k_j)\big)_{j=1}^l \in R(\rint),  \tau_j \in \rtree^{(n-1)}}, \\
  \rtree^{(n)}_- &\defby \rtree^{(n)} \cap \rtree_-.
\end{align*}
Now one has
\begin{equation*}
  \rtree_{\circ} = \bigcup_{n \in \N} \rtree^{(n)}, \quad
  \rtree_- = \bigcup_{n \in \N} \rtree^{(n)}_-.
\end{equation*}
\begin{lemma*}\label{lem:structure_of_rtree_minus}
  Regarding $\rtree^{(n)}$ and $\rtree^{(n)}_-$, one has the following claims.
  \begin{enumerate}[leftmargin=*]
    \item One has
    \begin{equation*}
      \min \set{\abs{\tau} \given \tau \in \rtree^{(n)} \setminus \rtree^{(n-1)}} = -2 + n \delta.
    \end{equation*}
    In particular, $\rtree^{(n)}_-  = \rtree^{(n-1)}_-$ for large $n$.
    \item If $\tau = \rjoin^{\ft_1}_{k_1}(\tau_1) \cdots \rjoin^{\ft_l}_{k_l}(\tau_l) \in \rtree_-^{(n)}$, then
    $\tau_j \in \rtree_-^{(n-1)}$ for every $j$ such that $\ft_j = \rint$. In particular, if $y$ is a leaf of $\tau$, namely if
    $\rnode(y) = \{()\}$, then for the edge $e = (x, y)$ one has $\ft(e) = \Xi$.
  \end{enumerate}
\end{lemma*}
\begin{proof}
  The claims will be proved by induction.
  Indeed, for $n=1$, the claims are true. Suppose that the claims are verified for $n-1$.
  We will check the claim (ii).
  Let
  \begin{equation*}
    \tau = \rjoin^{\ft_1}_{k_1}(\tau_1)\cdots \rjoin^{\ft_l}_{k_l}(\tau_l) \in \rtree_-^{(n)} \setminus \rtree^{(n-1)}_-.
  \end{equation*}
  \begin{itemize}[leftmargin=*]
    \item Suppose $((\ft_j, k_j))_{j=1}^l = ([\rint]_l)$. Then,
    \begin{equation*}
      \abs{\tau} = \sum_{j=1}^l \big(\abs{\tau} + 2\big) + \sum_{x \in N_{\tau}} \fn(x) \geq l \delta > 0,
    \end{equation*}
    contradiction.
    \item Suppose $((\ft_j, k_j))_{j=1}^l = (\rint_p, [\rint]_{l-1})$. Then,
    \begin{equation*}
      \abs{\tau} = \abs{\tau_1} + 1 + \sum_{j=2}^l \big(\abs{\tau_j} + 2\big) + \sum_{x \in N_{\tau}} \fn(x),
    \end{equation*}
    which can be less than $0$ only if $\tau_1, \ldots, \tau_l \in \rint^{(n-1)}$.
    \item Suppose $((\ft_j, k_j))_{j=1}^l = (\rint_p, \rint_q, [\rint]_{l-2})$. Then,
    \begin{equation*}
      \abs{\tau} = \abs{\tau_1} + \abs{\tau_2} + 2 + \sum_{j=2}^l \big(\abs{\tau_j} + 2 \big) + \sum_{x \in N_{\tau}} \fn(x),
    \end{equation*}
    which can be less than $0$ only if $\tau_1, \ldots, \tau_l \in \rint^{(n-1)}$.
    \item Suppose $((\ft_j, k_j))_{j=1}^l = (\Xi, [\rint]_{l-1})$. Then, $\tau_1$ is the trivial tree and
    \begin{equation*}
      \abs{\tau} = -2 + \delta + \sum_{j=2}^{l} \big(\abs{\tau_j} + 2 \big) + \sum_{x \in N_{\tau}} \fn(x),
    \end{equation*}
    which can be less than $0$ only if $\tau_2, \ldots, \tau_l \in \rint^{(n-1)}$.
  \end{itemize}
  To see the claim (i), it suffices to check $\abs{\tau}$ in the above four cases, by knowing that at least one of
  $\tau_j$'s satisfies $\abs{\tau_j} \geq -2 + (n-1) \delta$.
\end{proof}

The notion of completeness is more technical.
\begin{definition*}
  Let $N = \big((\ft_1, k_1), \ldots, (\ft_n, k_n)\big) \in (\typeset \times \N_0^d)^n / \sim_n$.
  \begin{enumerate}[leftmargin=*]
    \item Given  $m \in \N_0^d$, we denote by
      $\partial^m N$ the set of $n$-tuples
      \begin{equation*}
        \big((\ft_1, k_1 + m_1), \ldots, (\ft_n, k_n + m_n) \big), \quad \sum_{j=1}^n m_j = m.
      \end{equation*}
    \item We define a \emph{substitution} operation as follows. Suppose we are given $M \subseteq N$ and
    $\tilde{M} \in \cP \cN$, where $\cP \cN$ is the power set of $\cup_n (\typeset \times \N_0^d)^n / \sim_n$.
    Suppose $\tilde{M}$ is finite and
    write $N = M \sqcup \bar{N}$.
    Then, we set
    \begin{equation*}
      \rreplace_M^{\tilde{M}} N \defby \bar{N} \sqcup \tilde{M}_1 \sqcup \cdots \sqcup \tilde{M}_l,
      \quad \tilde{M} = \{ \tilde{M}_1, \ldots, \tilde{M}_l \}.
    \end{equation*}
  \end{enumerate}
\end{definition*}
\begin{definition*}
  Given a tree $T^{\fn}_{\fe} \in \rtree_{\circ}(R)$, we associate $\bar{\rnode}_T(e) = \bar{\rnode}(e)$
  to each $e = (x, y) \in E_T$ in the
  following recursive way. Suppose $y$ is a leaf, namely $\rnode(y) = \{()\}$. Then, we set
  \begin{equation*}
    \bar{\rnode}(e) \defby R(\ft(e)).
  \end{equation*}
  Otherwise, if $e_1, \ldots, e_l$ are all the edges of the form $e_j = (y, v)$ for some node $v$, we set
  \begin{equation*}
    \bar{\rnode}(e) \defby
    \set{\rreplace_{\rnode(y)}^{\tilde{M}} N \given
    \rnode(y) \subseteq N \in R(\ft(e)), \,\, \tilde{M} = \{M_1, \ldots, M_l\}, \,\,  M_j \in \bar{\rnode}(e_j)}.
  \end{equation*}
\end{definition*}
Let us compute $\bar{\rnode}$ for the following decorated tree:
\begin{equation*}
  \begin{bigdrawtree}
      \node[black node, label=below:{$x_1$}] {} [grow'=up]
        child {node[black node, label=left:{$x_2$}] {}
          child {node[black node, label=left:{$x_4$}] {}
            child {node[black node, label=above:{$x_8$}] {}
            edge from parent node[left] {$\Xi$}
            }
          edge from parent node[left] {$\rint$}
          }
          child {node[black node, label=left:{$x_5$}] {}
            child {node[black node, label=above:{$x_9$}] {}
            edge from parent node[left] {$\Xi$}
            }
          edge from parent node[right] {$\rint_1$}
          }
        edge from parent node[left] {$\rint$}
        }
        child {node[black node, label=right:{$x_3$}] {}
          child {node[black node, label=left:{$x_6$}] {}
            child {node[black node, label=above:{$x_{10}$}] {}
            edge from parent node[left] {$\Xi$}
            }
          edge from parent node[left] {$\rint_1$}
          }
          child {node[black node, label=left:{$x_7$}] {}
            child {node[black node, label=above:{$x_{11}$}] {}
            edge from parent node[left] {$\Xi$}
            }
          edge from parent node[right] {$\rint_1$}
          }
        edge from parent node[right] {$\rint_1$}
        };
  \end{bigdrawtree}.
\end{equation*}
One first sees
\begin{equation*}
  \bar{\rnode}((x_4, x_8)) = \bar{\rnode}((x_5, x_9)) =
  \bar{\rnode}((x_6, x_{10})) = \bar{\rnode}((x_7, x_{11})) = \{()\}.
\end{equation*}
Next, one observes
\begin{equation*}
  \bar{\rnode}((x_2, x_4)) =
  \set{\rreplace_{(\Xi)}^{\{()\}} ([\rint]_n, \Xi) \given n \in \N_0} = \set{([\rint]_n) \given n \in \N_0}.
\end{equation*}
Similarly,
\begin{equation*}
  \bar{\rnode}((x_2, x_5)) =
  \bar{\rnode}((x_3, x_6)) =\bar{\rnode}((x_3, x_6)) =
  \set{([\rint]_n) \given n \in \N_0}.
\end{equation*}
Finally, one observes
\begin{align*}
  \MoveEqLeft
  \bar{\rnode}((x_1, x_2)) \\
  &= \{\rreplace_{(\rint, \rint_1)}^{\{([\rint]_k), ([\rint]_l)\}}(\rint, \rint_1, [\rint]_n),
  \rreplace_{(\rint, \rint_1)}^{\{([\rint]_k), ([\rint]_l)\}}(\rint, \rint_1, \rint_j, [\rint]_n) \\
  &\hspace{8cm} \vert
  k,l,n \in \N_0, j \in \{1, \ldots, d\}\} \\
  &= \set{([\rint]_n), ([\rint]_n, \rint_j) \given n \in \N_0, j \in \{1, \ldots, d\}}
\end{align*}
and
\begin{align*}
  \bar{\rnode}((x_1, x_3))
  &= \set{\rreplace_{(\rint_1, \rint_1)}^{\{([\rint]_k), ([\rint]_l)\}} ([\rint]_n, \rint_1, \rint_1)
  \given k,l, n \in \N_0, j \in \{1, \ldots, d\}} \\
  &= \set{([\rint]_n) \given n \in \N_0}.
\end{align*}
\begin{lemma*}\label{lem:contraction_and_node_type}
  Suppose $F^{\fn}_{\fe} \in \rtree$ strongly conforms to the rule $R$ and $G$ is a subtree of $F$ containing the root of $F$.
  Write $e_1, \ldots, e_l$ for all the edges in $G$ containing the root $\rho_G = \rho_F$ and write
  $\rho_{\rcont_G F}$ for the root of the contracted tree
  \begin{equation*}
    \rcont_G F \defby \rcont (F, \indic_G, \fn, 0, \fe).
  \end{equation*}
  Then, one finds $(M_1, \ldots, M_l) \in \bar{\rnode}_G(e_1) \times \cdots \times \bar{\rnode}_G(e_l)$ and
  $\rnode_G(\rho_G) \subseteq N \in R(\rint)$ such that
  \begin{equation*}
     \rnode_{\rcont_G F}(\rho_{\rcont_{G} F}) = \rreplace_{\rnode_G(\rho_{G})}^{\{M_1, \ldots, M_l\}} N.
  \end{equation*}
\end{lemma*}
\begin{proof}
  We denote by $d(\cdot, \cdot)$ the graph distance on the nodes. Set
  \begin{equation*}
    n \defby \max \set{d(\rho_G, y) \given y \in G }.
  \end{equation*}
  The proof is based on the induction on $n$. When $n=1$, the proof of the claim is obvious.
  Suppose $n \geq 2$ and write $e_j = (\rho_G, y_j)$ for $j=1, \ldots, l$.
  We denote by $G_j$ and by $F_j$ the subtree of $G$ and $F$ respectively such that
  \begin{itemize}
    \item the root of $G_j$ and $F_j$ is $y_j$ and
    \item one has $N_G = \cup_{j=1}^l N_{G_j} \cup \{\rho_G\}$ and $E_G = \cup_{j=1}^l E_{G_j} \cup \{e_1, \ldots, e_l\}$
    and similarly for $F_j$.
  \end{itemize}
  \begin{equation*}
    \begin{bigdrawtree}
      \node[black node, label=below:{$\rho_G$}] {} [grow'=up]
        child {node[black node, label=left:{$y_1$}]{}
          child {node {$G_1$}
            edge from parent[dashed]
          }
          edge from parent node[below left] {$e_1$}
            coordinate (edge_1)
        }
        child {node[black node, label=right:{$y_l$}]{}
          child {node {$G_l$}
            edge from parent[dashed]
          }
          edge from parent node[below right] {$e_l$}
            coordinate (edge_l)
        };
        \draw[dashed] (edge_1) to [bend left=45] (edge_l);
    \end{bigdrawtree}
  \end{equation*}
  By the hypothesis of the induction, for each $j \in \{1, \ldots, l\}$, writing $e^{(j)}_1, \ldots, e^{(j)}_{l_j}$ for
  all the edges in $G_j$ containing $\rho_{G_j} = y_j$, there exist
  \begin{equation*}
    (M^{(j)}_1, \ldots, M^{(j)}_{l_j}) \in \bar{\rnode}_G(e_1^{(j)}) \times \cdots \times \bar{\rnode}_G(e_{l_j}^{(j)})
  \end{equation*}
  and $\rnode_G(y_j) \subseteq N^{(j)} \in R(\rint)$ such that
  \begin{equation*}
    M_j \defby \rnode_{\rcont_{G_j} F_j}(y_j) =
    \rreplace_{\rnode_{G_j}(y_j)}^{\{M_1^{(j)}, \ldots, M_{l_j}^{(j)}\}} N^{(j)}.
  \end{equation*}
  Therefore, one has $M_j \in \bar{\rnode}_G(e_j)$ and
  \begin{equation*}
     \rnode_{\rcont_G F}(\rho_{\rcont_{G} F}) = \rreplace_{\rnode_G(\rho_{G})}^{\{M_1, \ldots, M_l\}} \rnode_F(\rho_F). \qedhere
  \end{equation*}
\end{proof}
If $T^{\fn}_{\fe}$ strongly conforms to the rule $R$, it does not necessarily hold true that a contracted tree
\begin{equation*}
  \rcont (F, \indic_G, \fn, 0, \fe), \quad \text{for } G \subseteq F \text{ containg the root } \rho_F
\end{equation*}
conforms to the rule $R$ at its new root.
However, the following lemma shows that it holds true in a certain interesting situation
and it will be a key ingredient for Lemma~\ref{prop:def_of_renormalization_structures}.
\begin{lemma*}\label{lem:rule_complete}
  The subcritical rule $R$ is complete in the sense of \cite[Definition 5.20]{bruned_2019}. That is,
  if $\tau \in \rtree_-$ and $\rnode(\rho_{\tau}) \subseteq N \in R(\rint)$, then
  writing $e_1, \ldots, e_l$ for the edges containing the root $\rho_{\tau}$ one has
  \begin{equation*}
    \partial^m \rreplace_{\rnode(\rho_{\tau})}^{\{M_1, \ldots, M_l\}} N \in R(\rint)
  \end{equation*}
  for every $(M_1, \ldots, M_l) \in \bar{\rnode}(e_1) \times \cdots \times \bar{\rnode}(e_l)$ and for
  every $m \in \N_0^d$ with $\abs{m} + \abs{\tau} < 0$.
\end{lemma*}
\begin{proof}
  The proof is based on induction. Let
  \begin{equation*}
    \tau = \rjoin^{\ft_1}_{e_1}(\tau_1) \cdots \rjoin^{\ft_l}_{e_l}(\tau_l) \in \rtree_-^{(n)} \setminus \rtree_-^{(n-1)}.
  \end{equation*}
  \begin{itemize}[leftmargin=*]
    \item Suppose $((\ft_j, k_j))_{j=1}^l = (\Xi, [\rint]_{l-1})$.
    Then $\tau_1$ is the trivial tree. By Lemma~\ref{lem:structure_of_rtree_minus}, $\tau_2, \ldots, \tau_l \in
    \rtree_-^{(n-1)}$. By the hypothesis of the induction, one has
    \begin{equation*}
      \bar{\rnode}(e_j) \subseteq
      \begin{cases}
        \set{([\rint]_n) \given n \in \N_0} & \abs{\tau} < -1 \\
        R(\rint) & \abs{\tau} \geq -1.
      \end{cases}
    \end{equation*}
    Since
    \begin{equation*}
      \abs{\tau} = -2 + \delta + \sum_{j=2}^l \big(\abs{\tau_j} + 2 \big) + \fn(\rho_{\tau}) < 0,
    \end{equation*}
    there is at most one $\tau_j$ such that $\abs{\tau_j} \geq -1$.
    \begin{itemize}
      \item If $\abs{\tau_j} < -1$ for every $j$, then
      \begin{equation*}
       \set{\rreplace_{\rnode(\rho_{\tau})}^{\{M_1, \ldots, M_l\}} N \given \rnode(\rho_{\tau}) \subseteq N \in
       R(\ft(\rint)), M_j \in \bar{\rnode}(e_j)}
       \subseteq \set{([\rint]_n) \given n \in \N_0}.
      \end{equation*}
      \item If there exists exactly one $\tau_j$ such that $\abs{\tau_j} \geq -1$, then $\abs{\tau} \geq -1$ and
      \begin{equation*}
       \set{\rreplace_{\rnode(\rho_{\tau})}^{\{M_1, \ldots, M_l\}} N \given \rnode(\rho_{\tau}) \subseteq N \in
       R(\ft(\rint)), M_j \in \bar{\rnode}(e_j)}
       \subseteq R(\ft(e))
      \end{equation*}
    \end{itemize}
    Therefore, in this case, $\tau$ satisfies the claim.
    \item Suppose $((\ft_j, k_j))_{j=1}^l = (\rint_p, \rint_q, [\rint]_{l-2})$.
    Then, by Lemma~\ref{lem:structure_of_rtree_minus}, one has $\tau_1, \ldots, \tau_l \in \rtree_-^{(n-1)}$.
    One has
    \begin{equation*}
      \abs{\tau} = 2(l-1) + \sum_{j=1}^l \abs{\tau_j} + \fn(\rho_{\tau}) < 0
    \end{equation*}
    and therefore there is at most one $j$ such that $\abs{\tau_j} \geq -1$. Then, one can prove the claim as in
    the first case. The case $((\ft_j, k_j))_{j=1}^l = (\rint_p, [\rint]_{l-1})$ can be handled similarly.  \qedhere
  \end{itemize}
\end{proof}
\end{calc}

\subsection{Definition of the regularity structure}
The content of this section is parallel to \cite[Section 5.5]{bruned_2019}.
Our goal here is to construct subspaces of $H_{\diamond}$ ($\diamond \in \{ 1, \circ\}$) which provide a correct
framework for the theory of regularity structures.
Since we desire that elements of those spaces conform to the rule $R$, one might want to consider a subspace spanned by
$\rtree_{\diamond}$. However, $\rtree_{\diamond}$ is not closed under the coproduct of $H_{\diamond}$.
Therefore, we introduce the following definition.
\begin{definition}\label{def:spaces_obeying_rule}
  Recall the notation introduced in Definition~\ref{def:coproduct_for_decorated_forests}.
  For $\diamond \in \{1, \circ\}$,
  we denote by $\rbasis(H^C_{\diamond}) \subseteq \rbasis(H_{\diamond})$ the set consisting of
    \begin{equation}\label{eq:basis_of_H_diamond_C}
      \rcont_1 (F, \hat{F} \cup_1 A, \fn - \fn_A, \fn_A + \pi(\epsilon^F_A - \fe \indic_A), \fe \indic_{F \setminus A} + \epsilon^F_A)
    \end{equation}
    for $\tau = (F, \hat{F})^{\fn, 0}_{\fe} \in \rtree_{\diamond}$, $A \in \fU_1(\tau)$, $\fn_A \leq \fn$ with $\supp(\fn_A) \subseteq N_A$ and
    $\epsilon^F_A:E_F \to \N_0^d$ with $\supp(\epsilon^F_A) \subseteq \partial(A, F)$.
  We denote by $\gls{H_C_diamond}$ the free vector space generated by $\rbasis(H^C_{\diamond})$.
\end{definition}
\begin{remark}
  By choosing $A = \emptyset$
  one observes $\rtree_{\diamond} \subseteq \rbasis(H^C_{\diamond})$ for $\diamond \in \{1, \circ\}$.
  In fact, as Lemma~\ref{lem:H_circ_subbialgebras} below shows,
  $H^C_1$ is
  the smallest subbialgebra of $H_{\diamond}$ both containing $\rtree_{\diamond}$ and closed under the coactions.
\end{remark}
\begin{lemma}\label{lem:H_circ_subbialgebras}
  The subspace $H_1^C$ is a subbialgebra of $H_1$.
  Furthermore, the statements of Proposition~\ref{prop:forest_coaction}
  \begin{calc} and of Proposition~\ref{prop:forest_hopf_algebra}-(iii) \end{calc} remain valid
  if one replaces $(H_1, H_{\circ})$ by $(H^C_1, H^C_{\circ})$.
\end{lemma}
\begin{proof}
  This is essentially proven in \cite[Lemma 5.25 and Lemma 5.28]{bruned_2019}.
  \begin{calc} Here we give a self-contained proof.
  We start to prove $H^C_1$ is a subbialgebra of $H_1$. Recalling that $\rtree_1$ is closed under multiplication,
  it is easy to see $H^C_1$ is closed under multiplication as well. To prove $H^C_1$ is closed under coproduct,
  one first notes that if $\tau \in \rtree_1$, then the first components and the second components appearing in
  the coproduct formula \eqref{eq:coproduct_formula} of $(\rcont_1 \otimes \rcont_1) \Delta_1 \tau$ belong to
  (up to multiplication by a nonzero constant) $\rtree_1$ and $\rbasis(H_1^C)$ respectively.
  Now suppose $\tau_2 = (F_2, \hat{F}_2)^{\fn_2, \fo_2}_{\fe_2} \in \rbasis(H^C_1)$ and write
  \begin{Cequation}\label{eq:tau_2_contraction}
    \tau_2 = \rcont_1
    (F_1, \hat{F}_1 \cup_1 A_1,
    \fn_1 - \fn_{A_1}, \fn_{A_1} + \pi(\epsilon^{F_1}_{A_1} - \fe \indic_{A_1}),
    \fe \indic_{F_1 \setminus A_1} + \epsilon^{F_1}_{A_1})
  \end{Cequation}
  as in \eqref{eq:basis_of_H_diamond_C} and set $\tau_1 \defby (F_1, \hat{F}_1)^{\fn_1, \fo_1}_{\fe_1}$.
  To show $H^C_1$ is closed under coproduct, it suffices to show the first and the second components appearing in
  $(\rcont_1 \otimes \rcont_1) \Delta_1 \tau_2$ belong to $\rbasis(H_1^C)$. This follows from the identity
  \begin{Cequation}\label{eq:coproduct_tau_1}
    (\Id_{H_1} \otimes (\rcont_1 \otimes \rcont_1) \Delta_1) (\rcont_1 \otimes \rcont_1) \Delta_1 \tau_1
    = ( (\rcont_1 \otimes \rcont_1) \Delta_1\otimes\Id_{H_1} ) (\rcont_1 \otimes \rcont_1) \Delta_1 \tau_1.
  \end{Cequation}
  Indeed, suppose $\sigma$ appears in the first components of $(\rcont_1 \otimes \rcont_1) \Delta_1 \tau_2$.
  Then, $\sigma$ is among
  the second components appearing in
  \begin{equation*}
    (\Id_{H_1} \otimes (\rcont_1 \otimes \rcont_1) \Delta_1) (\rcont_1 \otimes \rcont_1) \Delta_1 \tau_1
  \end{equation*}
  and hence, in view of \eqref{eq:coproduct_tau_1},
  among the second components appearing in
  \begin{equation*}
    ( (\rcont_1 \otimes \rcont_1) \Delta_1\otimes\Id_{H_1} ) (\rcont_1 \otimes \rcont_1) \Delta_1 \tau_1.
  \end{equation*}
  However, we know that the first components of $(\rcont_1 \otimes \rcont_1) \Delta_1 \tau_1$ belongs to $\rtree_1$
  and that for $\tau' \in \rtree_1$ the second components in $(\rcont_1 \otimes \rcont_1) \Delta_1 \tau'$ belong to
  $\rbasis(H^C_1)$.
  This implies $\sigma \in \rbasis(H^C_1)$. One can similarly argue that the second components
  appearing in $(\rcont_1 \otimes \rcont_1) \Delta_1 \tau_2$ belong to $\rbasis(H^C_1)$.

  We move to prove $H^C_2$ is a subbialgebra of $H_2$.
  It suffices to show that both the first components and the second components appearing in
  $(\hat{\rcont_2} \otimes \hat{\rcont_2}) \Delta_2 \tau$ for $\tau \in \rbasis(H^C_2)$ belong to
  $\rbasis(H^C_2)$.
  Indeed, suppose $\tau = \tau_2 \in \rbasis(H^C_2)$ is represented as in \eqref{eq:tau_2_contraction}
   and define $\tau_1 \in \rtree_2$ in the same way.
  The claim is a consequence of the coaction identity
  \begin{equation*}
    (\Id_{H_1} \otimes (\hat{\rcont}_2 \otimes \hat{\rcont}_2) \Delta_2) (\rcont_1 \otimes \rcont)\Delta_1 \tau_1
    = \rmult^{(13)(2)(4)} \rcont_1^{\otimes 4} (\Delta_1 \otimes \Delta_1) (\Id_{H_2} \otimes \rcont)\Delta_2 \tau_1.
  \end{equation*}
  In fact, the first components appearing in $(\hat{\rcont_2} \otimes \hat{\rcont_2}) \Delta_2 \tau_2$
  are among the second components of
  \begin{equation*}
    \rcont_1^{\otimes 4} (\Delta_1 \otimes \Delta_1) (\Id_{H_2} \otimes \rcont)\Delta_2 \tau_1,
  \end{equation*}
  which belong to $\rbasis(H^C_2)$ since the first components appearing in $\Delta_2 \tau_1$ belong to
  $\rtree_2$ again. One can similarly argue for the second components
  appearing in $(\hat{\rcont_2} \otimes \hat{\rcont_2}) \Delta_2 \tau_2$, by noting that
  the second components appearing in $\Delta_2 \tau_1$ belong to $\rtree_2$ again.

  Finally, we prove the claims in the lemma after ``Furthermore".
  As for Proposition~\ref{prop:forest_hopf_algebra}-(iii), it suffices to show that the first and the second components
  appearing in $(\rcont_1 \otimes \rcont) \Delta_1 \tau$ for $\tau \in \rbasis(H^C_2)$ belong to
  $\rbasis(H^C_1)$ and $\rbasis(H^C_2)$ respectively.
  This is a consequence of the coproduct identity \cite[Proposition 3.11]{bruned_2019}
  \begin{Cequation}\label{eq:coprod_1_for_tau_1}
    (\rcont_1 \otimes \rcont_1 \otimes \rcont)(\Id_{\rforest} \otimes \Delta_1) \Delta_1 \tau_1
    =(\rcont_1 \otimes \rcont_1 \otimes \rcont) (\Delta_1 \otimes \Id_{\rforest}) \Delta_1 \tau_1
    , \quad \tau_1 \in \rtree_2.
  \end{Cequation}
  Indeed, suppose $\sigma$ appears in the first components
  appearing in $(\rcont_1 \otimes \rcont) \Delta_1 \tau$ for some $\tau \in \rbasis(H^C_2)$.
  Then, $\sigma$ is among the second components appearing in \eqref{eq:coprod_1_for_tau_1} for some $\tau_1 \in \rtree_2$.
  Since the first components of $\Delta_1 \tau_1$ belong to $\rtree_1$, we conclude $\sigma$ belongs to $H^C_1$.
  One can similarly argue for the second components
  appearing in $(\rcont_1 \otimes \rcont) \Delta_1 \tau$.
  The claims for Proposition~\ref{prop:forest_coaction} can be proved similarly.
  \end{calc}
\end{proof}
\begin{definition}\label{def:basis_of_H_R}
  For $\diamond \in \{1, \circ \}$  we denote by $\gls{H_R_diamond}$ the free vector space generated by
  \begin{equation*}
    \rbasis(H^R_{\diamond}) \defby \set{(F, \hat{F})^{\fn, \fo}_{\fe} \in \rbasis(H^C_{\diamond}) \given (F, \hat{F}
    \indic_{\hat{F} \neq 1})^{\fn, 0}_{\fe} \in
    \rtree_{\diamond}}.
  \end{equation*}
\end{definition}
\begin{definition}\label{rem:space_diamond_is_a_subset_of_hat_space_diamond}
  We denote by $\rspace_1$ the free vector space generated by
  \begin{equation*}
    \rbasis(\rspace_1) \defby \set{\tau \in \rbasis(H^C_1) \given (F, 0)^{\fn, 0}_{\fe} \in \rtree_- \text{ for every connected component }
    (F, \hat{F})^{\fn, \fo}_{\fe} \text{ of } \tau}
  \end{equation*}
  and by $\rspace_2$ the free vector space generated by $\rbasis(\rspace_2)$, where $\tau \in \rbasis(\rspace_2)$
  if and only if
  \begin{equation*}
    \hat{\rcont}_2 \big[\rint_{k_1}(\tau_1) \cdots \rint_{k_n}(\tau_n) {\color{blue} \bullet}^{m} \big]
  \end{equation*}
  for $n \in \N_0$, $k_1, \ldots, k_n, m \in \N_0^d$ and $\tau_1, \ldots, \tau_n \in H^R_{\circ}$ such that
  $\abs{\rint_{k_j}(\tau_j)}_+ > 0$ for every $j=1, \ldots, n$.
  We denote by $\rproj_1^C:H_1^C \to \rspace_1$ the natural projection.
  We note that $\rspace_1$ is an algebra under the forest product and $\rspace_2$ is an algebra under the tree product.
\end{definition}
\begin{proposition}[{\cite[Proposition~5.35]{bruned_2019}}]\label{prop:def_of_renormalization_structures}
  The linear map
  \begin{equation*}
    \Delta_1 \defby (\rproj_1^C \otimes \rproj_1^C) (\rcont_1 \otimes \rcont_1) \Delta_1: \rspace_1
    \to \rspace_1 \otimes \rspace_1
  \end{equation*}
  defines a coproduct over the algebra $\rspace_1$ (with the forest product as multiplication).
  With this coproduct and the counit as in Proposition~\ref{prop:forest_hopf_algebra}, $\rspace_1$ is a Hopf algebra.
  Furthermore, the vector space $H^R_{\circ}$ is a right comodule over $\rspace_1$ with coaction
  \begin{equation*}
    \gls{coaction_circ_-} \defby (\rproj_1^C \otimes \Id) (\rcont_1 \otimes \rcont) \Delta_1:
    H^R_{\circ} \to \rspace_1 \otimes H^R_{\circ}.
  \end{equation*}
\end{proposition}

\begin{remark}

  As shown in \cite[Proposition~5.34]{bruned_2019}, one can view $\rspace_2$ as a Hopf algebra with grade $\abs{\cdot}_+$
  and one can define a coaction $\Delta^{\circ}_+: H^R_{\circ} \to H^R_{\circ} \otimes \rspace_2$, of which we do not need the precise definition here but will only use the recursive formula~\cite[Proposition~4.17]{bruned_2019}.

\end{remark}
\begin{definition}\label{def:def_of_gPAM_reg_str}
  We set
  \begin{equation*}
    \rspace \defby H^R_{\circ}, \quad \rspace_+ \defby \rspace_2, \quad \rspace_- \defby \rspace_1,
  \end{equation*}
  Then, in the language of \cite{bailleul_hoshino_20},
  the pair $\gls{regularity_structure}$ is a concrete regularity structure and the
  pair $\gls{renormalization_structure}$ is a renormalization structure for the generalized PAM.
  For $\diamond \in \{-, +\}$, we denote
  \begin{multicols}{2}
  \begin{itemize}
    \item by $\gls{coproduct_diamond}$ the coproduct of $\rspace_{\diamond}$,
    \item by $\gls{unit_diamond}$ the unit of $\rspace_{\diamond}$,
    \item by $\gls{counit_diamond}$ the counit of $\rspace_{\diamond}$ and
    \item by $\gls{antipode_diamond}$ the antipode of $\rspace_{\diamond}$.
  \end{itemize}
  \end{multicols}
  Recall that the product $\rmult_-$ of $\rspace_-$ is the forest product
  while the product $\rmult_+$ of $\rspace_+$ is the tree product.
  We write $\unit \in \rspace$ for $\bullet^0$.
  For $\diamond \in \{-, +\}$, the Hopf algebra $\rspace_{\diamond}$ is graded with $\abs{\cdot}_{\diamond}$.
  The vector space $\rspace$ is graded both with $\abs{\cdot}_-$ and with $\abs{\cdot}_+$.
\end{definition}
\begin{definition}
  As shown in \cite[Proposition~5.39]{bruned_2019}, if
  \begin{equation*}
    A \defby \set{\abs{\tau}_+ \given \tau \in \rbasis(\rspace)},
  \end{equation*}
  and we denote by $G$ the character group of $\rspace_+$,
  the triplet $(A, \rspace, G)$
  is a regularity structure in the sense of \cite[Definition 2.1]{hairer_theory_2014}.
  We have the graded decomposition
  \begin{equation*}
    \rspace = \oplus_{\gamma \in A} \rspace_{\gamma}, \quad \rspace_{\gamma} \defby \spn \set{\tau \in \rbasis(\rspace) \given
    \abs{\tau}_+ = \gamma}.
  \end{equation*}
  We write $\gls{reg-projection}$ for the natural projection from $\rspace$ to $\rspace_{< \beta} \defby \oplus_{\gamma < \beta} \rspace_{\gamma}$.
\end{definition}
\begin{definition}\label{def:product_in_reg_sp}
  If $\tau, \sigma \in \rbasis(\rspace)$ are such that $\tau \sigma \in \rbasis(\rspace)$, we write
  $\tau \rprod \sigma \defby \tau \sigma$. We extend the product $\rprod$ bilineary.
  Note that the product $\rprod$ is not defined for all pairs $(\tau, \sigma)$.
\end{definition}
The following lemma essentially states that the product $\rprod$ is regular in the sense of
\cite[Definition 4.6]{hairer_theory_2014}.
\begin{lemma}[{\cite[Proposition 3.11]{bruned_2019}}]\label{lem:product_is_regular}
  If $\tau, \sigma \in \rbasis(\rspace)$ are such that $\tau \sigma \in \rbasis(\rspace)$, then
  $\Delta^{\circ}_+(\tau \rprod \sigma) = \Delta^{\circ}_+(\tau) \Delta^{\circ}_+(\sigma)$.
\end{lemma}
\begin{definition}\label{def:derivative_in_reg_sp}
  Let $V$ be the subspace of $\rspace$ generated by
  \begin{equation*}
    \rint(\tau_1) \cdots \rint(\tau_n), \qquad \tau_1, \ldots, \tau_n \in \rspace.
  \end{equation*}
  For $i \in \{1, \ldots, d\}$, we define the linear map $\gls{derivative}: V \to \rspace$, called a \emph{derivative}, by
  \begin{equation*}
    \rderivative_i [\rint(\tau_1) \cdots \rint(\tau_n)]
    = \sum_{j=1}^d \rint_{i}(\tau_j) \prod_{k \neq j} \rint(\tau_k).
  \end{equation*}
\end{definition}
\begin{calc}
\begin{definition*}
  We set
  \begin{align*}
    \hat{\Delta}_- &\defby (\rproj_1^C \otimes \Id_{\hat{\rspace}_-}) (\rcont_1 \otimes \rcont_1) \Delta_1: \rspace_- \to \rspace_- \otimes \hat{\rspace}_-, \\
    \hat{\Delta}_+ &\defby (\Id_{\hat{\rspace}_+} \otimes \rproj_2 \hat{\rcont}_2)
    \Delta_2: \hat{\rspace}_+ \to \hat{\rspace}_+ \otimes \rspace_+.
  \end{align*}
  In view of Remark~\ref{rem:space_diamond_is_a_subset_of_hat_space_diamond}, one can regard $\rspace_{\diamond}$
  as a subspace of
  $\hat{\rspace}_{\diamond}$. We write $\rinjection_{\diamond}:\rspace_{\diamond} \to \hat{\rspace}_{\diamond}$
  for the injection.
\end{definition*}
\end{calc}
\begin{proposition}[{\cite[Section 6.1]{bruned_2019}}]
  Let $\rinjection_-: \rspace_- \to H^R_1$ be the natural projection.
    Then, there exists a unique algebra morphism $\hat{\rantipode}_-: \rspace_- \to H^R_1$ such that
    \begin{equation*}
      \rmult_{H_1} (\hat{\rantipode}_- \otimes \Id_{\hat{\rspace}_-}) (\rcont_1 \otimes \rcont_1) \Delta_1 \rinjection_-
      = \counit_-(\cdot) \unit_{H_1} \quad \text{on } \rspace_-,
    \end{equation*}
    where $\rmult_{H_1}$ is the product in $H_1$.
\end{proposition}
\begin{definition}\label{def:twisted_antipode}
  We call $\gls{antipode_twisted}$  a \emph{negative  twisted antipode}.
\end{definition}
As for $\hat{\rantipode}_-$, we only use the following property.
As shown in \cite[Proposition 6.6]{bruned_2019}, one has the following recursive formula:
\begin{equation}\label{eq:antipode_minus_recursive}
  \hat{\rantipode}_- \tau = -\rmult_{H_1} (\hat{\rantipode}_- \otimes \Id_{H^R_1})
  (\hat{\Delta}_- \tau - \tau \otimes \unit_-),
\end{equation}
where $\hat{\Delta}_- \defby (\rproj_1^C \otimes \Id_{H^R_1}) (\rcont_1 \otimes \rcont_1) \Delta_1$.

\subsection{Models and modelled distributions}
Recall the notion of \emph{models} $\gls{model}  = (\Pi, \Gamma)$ from \cite[Definition~2.17]{hairer_theory_2014}. In our situation (generalized parabolic Anderson model), the scaling $\fs$ is uniform:
$\fs = (1, 1, \ldots, 1)$.
We also need the functional $\vertiii{\cdot}_{\gamma; \cptset}$ and the psuedometric $\vertiii{\cdot; \cdot}_{\gamma; \cptset}$
from \cite[(2.16) and (2.17)]{hairer_theory_2014}.
\begin{calc}
With the notion of models, one can associate to an abstract element of $\rspace$ a concrete distribution on $\R^d$.
We list some definitions related to models from \cite{hairer_theory_2014} and \cite{bruned_2019}.
\begin{definition*}
  We write
  $\rspace_{< \gamma} \defby \oplus_{\beta < \gamma} \rspace_{\beta}$. We write $\rproj_{\beta}:\rspace \to \rspace_{\beta}$
  and $\rproj_{<\gamma}:\rspace \to \rspace_{< \gamma}$ for the natural projections.
  Since $\rspace_{\beta}$ is identified with some Euclidean space, we can equip $\rspace_{\beta}$ with the corresponding Euclidean norm for which we write
  $\norm{\cdot}_{\beta}$.
\end{definition*}
\begin{definition*}[{\cite[Definition 2.17]{hairer_theory_2014}}]
  A \emph{model} for the regularity structure $\rspace$ is a pair $\rmodel = (\Pi, \Gamma)$ with the following properties.
  \begin{itemize}
    \item For each $x \in \R^d$, one has a linear map $\Pi_x: \rspace \to \cS'(\R^d)$ and $\Pi = (\Pi_x)_{x \in \R^d}$.
    \item For each $x, y \in \R^d$, one has $\Gamma_{xy} \in G$, where $G$ is the structure group of $\rspace$ and
    $\Gamma = (\Gamma_{xy})_{x, y \in \R^d}$.
    \item One has the identities $\Gamma_{xy} \Gamma_{yz} = \Gamma_{xz}$ and $\Pi_x \Gamma_{xy} = \Pi_y$ for $x, y , z \in \R^d$.
    \item One has the following analytic estimates.
    For every $\gamma \in \R$ and a compact set $\cptset$ of $\R^d$,
    there exists a constant $C_{\gamma, \cptset}$ such that
    \begin{Cequation}\label{eq:model_analytic_estimates}
      \abs{\inprd{\Pi_x \tau}{\phi_x^{\lambda}}} \leq C_{\gamma, \cptset}  \lambda^{\beta},
      \quad \norm{\rproj_{\alpha} \Gamma_{xy} \tau}_{\alpha} \leq C_{\gamma, \cptset} \abs{x-y}^{\beta - \alpha}
    \end{Cequation}
    for every $\alpha, \beta$ such that $\alpha < \beta < \gamma$,
    every $\tau \in \rbasis(\rspace) \cap \rspace_{\beta}$, every $x,y \in \cptset$ with $\abs{x-y} \leq 1$,
    every $\lambda \in (0, 1)$ and every
    $\phi \in \csp^2(\R^d)$ with $\supp(\phi) \subseteq B(0, 1)$ and
    $\norm{\phi}_{\csp^2(\R^d)} \leq 1$ and where $\phi_x^{\lambda}(y) \defby \lambda^{-d} \phi(\lambda^{-1}(y-x))$.
  \end{itemize}
  We denote by $\vertiii{\rmodel}_{\gamma; \cptset}$ the infinimum of the constants $C_{\gamma, \cptset}$ such that the
  inequalities \eqref{eq:model_analytic_estimates} hold.
  Given another model $\tilde{\rmodel} = (\tilde{\Pi}, \tilde{\Gamma})$, we denote by
  $\vertiii{\rmodel; \tilde{\rmodel}}_{\gamma; \cptset}$ the infinimum of those $C>0$ such that
  \begin{equation*}
      \abs{\inp{(\Pi_x - \tilde{\Pi}_x) \tau}{\phi_x^{\lambda}}} \leq C \lambda^{\beta},
      \quad \norm{\rproj_{\alpha} (\Gamma_{xy} - \tilde{\Gamma}_{xy}) \tau}_{\alpha} \leq C
      \abs{x-y}^{\beta - \alpha},
  \end{equation*}
  where $\alpha$, $\beta$, $x$, $y$, $\tau$, $\lambda$ and $\phi$ range as before.
\end{definition*}
\end{calc}
An important concept in the theory of regularity structures is the \emph{modelled distribution} (\cite[Definition~3.1]{hairer_theory_2014}).
We denote by $\gls{modelled_dist} = \mdspace^{\gamma}(\rmodel)$
the space of modelled distributions with respect to the model $\rmodel$ whose
images are in $\rspace_{<\gamma} = \oplus_{\beta < \gamma} \rspace_{\beta}$.
We set
  \begin{equation*}
    \mdspace_{\alpha}^{\gamma}(\rspace, \rmodel) \defby
    \set{f \in \mdspace^{\gamma}(\rspace, \rmodel)
    \given f \mbox{ is }\oplus_{\alpha \leq \beta < \gamma} \rspace_{\beta}\mbox{-valued}}.
\end{equation*}
We will use the norm $\vertiii{\cdot}_{\gamma; \cptset}$ given by \cite[(3.1)]{hairer_theory_2014}.
\begin{calc}
\begin{definition*}[{\cite[Definition 3.1]{hairer_theory_2014}}]
  Let $\gamma \in \R$ and $\rmodel = (\Pi, \Gamma)$ be a model for $\rspace$.
  We define the space $\mdspace^{\gamma}(\rspace, \rmodel)$
  of \emph{modelled distributions}
  as the space of maps $f: \R^d \to \rspace_{< \gamma} $ such that
  for every compact set $\cptset \subseteq \R^d$ one has
  \begin{equation*}
    \vertiii{f}_{\gamma; \cptset} \defby
    \max_{\beta < \gamma} \sup_{x \in \cptset} \norm{\rproj_{\beta} f(x)}_{\beta}
    + \max_{\beta < \gamma} \sup_{\substack{x,y \in \cptset, \\ \abs{x-y} \leq 1}}
    \frac{\norm{\rproj_{\beta}[f(y) - \Gamma_{yx} f(x)]}_{\beta}}{\abs{x-y}^{\gamma - \beta}} < \infty.
  \end{equation*}
  We set
  \begin{equation*}
    \mdspace_{\alpha}^{\gamma}(\rspace, \rmodel) \defby
    \set{f \in \mdspace^{\gamma}(\rspace, \rmodel)
    \given f \mbox{ is }\oplus_{\alpha \leq \beta < \gamma} \rspace_{\beta}\mbox{-valued}}.
  \end{equation*}
\end{definition*}
\end{calc}
\begin{definition}
  Let $\rmodel$ be a model over $\sT$ and let $\gamma, l > 0$.  By \cite[Theorem 3.10]{hairer_theory_2014}, there exists
  a unique continuous linear operator $\gls{reconstruction} = \reconst^{\rmodel}: \mdspace^{\gamma}(\rspace, \rmodel)
  \to \csp^{\min A}_{\operatorname{loc}}(\R^d)$ with the following property: there exists a $C= C(\gamma,l,\rspace)>0$ such that
  for every compact set $\cptset \subseteq \R^d$
  \begin{equation}\label{eq:modelled_dist_approximation}
    \abs{(\reconst f - \Pi_x f(x))(\phi_x^\lambda)} \le C \lambda^{\gamma}
    \norm{\rmodel}_{\gamma; B(\cptset, l)} \vertiii{f}_{\gamma; B(\cptset, l)}, \quad
    \text{where } \phi^{\lambda}_x \defby \lambda^{-d} \phi(\lambda^{-1}(\cdot - x)),
  \end{equation}
  uniformly over $\phi \in \csp^2(B(0,l))$ with $\norm{\phi}_{\csp^r(\R^d)} \leq 1$,
  $\lambda \in (0, 1)$, $f \in \mdspace^{\gamma}(\rspace, \rmodel)$ and $x \in \cptset$.
  The operator $\reconst$ is called the \emph{reconstruction operator}.
\end{definition}
\begin{proposition}[{\cite[Theorem 4.7]{hairer_theory_2014}}]\label{prop:product_modelled_distributions}
  Let $V_1$ and $V_2$ be subspaces of $\rspace$ closed under the action of the structure group.
  Suppose the product $\tau_1 \rprod \tau_2$ is well-defined for every $\tau_1 \in V_1$ and $\tau_2 \in V_2$.
  Let $\rmodel$ be a model for $\rspace$ and
  let $f_i \in \mdspace_{\alpha_i}^{\gamma_i}(V_i, \rmodel)$ for $i = 1, 2$.
  Then, if we set $\gamma \defby \min\{\gamma_1 + \alpha_2, \gamma_2 + \alpha_1\}$, one has
  $\rproj_{< \gamma} (f_1 \rprod f_2) \in \mdspace_{\alpha_1 + \alpha_2}^{\gamma}(\rspace, \rmodel)$.
  Moreover, there exists a constant $C \in (0, \infty)$ which depends only on $\rspace$ such that
  \begin{equation*}
    \vertiii{\rproj_{< \gamma} (f_1 \rprod f_2)}_{\gamma;\cptset} \leq
    C (1 + \vertiii{\rmodel}_{\gamma_1 + \gamma_2;\cptset})^2 \vertiii{f_1}_{\gamma_1; \cptset}
    \vertiii{f_2}_{\gamma_2; \cptset} \quad \text{for every compact set } \cptset \subseteq \R^d.
  \end{equation*}
\end{proposition}
\begin{definition}\label{def:function_and_modelled_distritbusion}
    Let $F \in C^{\infty}_b(\R^d)$ and let $V$ be a subspace of $\set{\tau \in \rspace \given \rproj_{<0} \tau = 0}$
   that is closed under the product $\rprod$ and under the action of the structure group.
    We define the map $ F^{\rprod}:V \to V$ by
    \begin{equation*}
      F^{\rprod}(\tau) \defby \sum_{k \in \N_0} \frac{D^k F(\bar{\tau})}{k!} (\tau - \bar{\tau})^{\rprod k},
      \quad \bar{\tau} \defby \rproj_0 \tau.
    \end{equation*}
    According to \cite[Theorem 4.16]{hairer_theory_2014}, if $\gamma > 0$ and $f \in \mdspace^{\gamma}(V, \rmodel)$, then
    one has
    \begin{equation*}
      F^{\rprod}_{\gamma}(f)(x) \defby \rproj_{< \gamma} F^{\rprod}(f(x)) \in \mdspace^{\gamma}(V, \rmodel).
    \end{equation*}
    Furthermore, there exist a constant $C \in (0, \infty)$ and an integer $k \in \N$,
    which depend only on $\rspace$, $F$ and $\gamma$, such that
    \begin{equation}\label{eq:composition_modelled_distribution}
      \vertiii{\rproj_{< \gamma} F(f)}_{\gamma; \cptset} \leq C (1 + \vertiii{\rmodel}_{\gamma; \cptset} +
      \vertiii{f}_{\gamma: \cptset})^k
      \quad \text{for every compact set } \cptset \subseteq \R^d.
    \end{equation}
\end{definition}
\subsection{Operations with kernels}

\begin{definition}\label{def:admissible_kernel}
  A smooth map $K: \R^d \setminus \{0\} \to \R$  with $\supp K\subseteq B(0,1)$ is called an \emph{admissible kernel} if it satisfies
  \cite[Assumption 5.1]{hairer_theory_2014} with $K(x, y) \defby K(x - y)$ 
  \tymchange{
  and if the following are satisfied: 
  \begin{itemize}
    \item with the notation of \cite[Assumption 5.1]{hairer_theory_2014} one has the scaling 
    $\mathfrak{s} = (1, 1, \ldots, 1)$ and the regularization $\beta = 2$;
    \item one has $\int_{\R^d} x^k K(x) dx = 0$ for $\abs{k} \leq 1$;
    \item the difference $K - G$ is smooth on $\mathbb{R}^d$, where $G$ is the Green's function on $\mathbb{R}^d$ (see Section~\ref{subsec:green_function}). 
  \end{itemize}
  The existence of such a kernel is guaranteed by \cite[Lemma~5.5]{hairer_theory_2014}.
  }
\end{definition}
\begin{definition}[{\cite[Definition 5.9]{hairer_theory_2014}}]
  Given an admissible kernel $K$, a model $(\Pi, \Gamma)$ for $\rspace$ is said to \emph{realize $K$} if
  one has
  \begin{equation*}
    \Pi_x \rint_k \tau = \partial^k K \conv \Pi_x \tau - \sum_{j\in \N_0^d: \abs{\tau}_+ + 2 - \abs{j} - \abs{k} > 0}
    \frac{(\cdot - x)^j}{j!} [\partial^{k+j} K \conv \Pi_x \tau] (x)
  \end{equation*}
  for every $\tau \in \rbasis(\rspace)$, $k \in \N_0^d$ with $\abs{k} \leq 1$ and $x \in \R^d$.
  The space $\bar{\rmodelsp}(\rspace, K)$ of all $K$-admissible models is endowed with the topology induced by
  the collection of pseudometrics $(\vertiii{\cdot \, ; \cdot}_{\gamma, \cptset} )_{\gamma, \cptset}$.
  In fact,
  the space $\bar{\rmodelsp}(\rspace, K)$ is a complete metric space.
\end{definition}
We recall operations of kernels on modelled distributions from \cite[Section 5]{hairer_theory_2014}.
\begin{definition}\label{def:kernel}
  Let $\rmodel = (\Pi, \Gamma)$ be a model realizing $K$ in the sense of \cite[Definition 5.9]{hairer_theory_2014}.
  \begin{enumerate}[leftmargin=*]
    \item  We set
    \begin{equation*}
      \cJ(x)\tau \defby \cJ^{\rmodel}(x) \tau \defby \sum_{\abs{k} < \abs{\tau}_+ + 2}
      \frac{X^k}{k!} \big[D^k K \conv \Pi_x \tau(x)\big], \quad x \in \R^d,
    \end{equation*}
    for $\tau \in \rbasis(\rspace)$ and extend it linearly for $\tau \in \rspace$.
    \item Let $\gamma \in (0, \infty) \setminus \N$ and $f \in \mdspace^{\gamma}(\rspace, \rmodel)$. We set
    \begin{equation}\label{eq:def_of_cN}
      \cN f(x) \defby \cN^{\rmodel}_{\gamma} f(x) \defby \sum_{\abs{k} < \gamma + 2} \frac{X^k}{k!} D^k K \conv
      (\reconst^{\rmodel} f - \Pi_x f(x))(x)
    \end{equation}
    and
    \begin{equation}\label{eq:def_of_curl_K}
      \cK f(x) \defby \cK^{\rmodel}_{\gamma} f(x) \defby (\rint + \cJ^{\rmodel}(x)) f(x) + \cN^{\rmodel}_{\gamma} f(x).
    \end{equation}
    By \cite[Theorem 5.12]{hairer_theory_2014}, $\cK$ maps $\mdspace^{\gamma}(\rspace, \rmodel)$
    to $\mdspace^{\gamma + 2}(\rspace, \rmodel)$ and
    one has $\reconst \cK f = K \conv \reconst f$.
    More precisely, one has
    \begin{equation}\label{eq:schauder_estimate_modelled_distribution}
      \vertiii{\rkernel f}_{\gamma + 2; \cptset} \lesssim_{\rspace, \gamma} (1 + \vertiii{\rmodel}_{\gamma+2; B(\cptset,1)})^2
      \vertiii{f}_{\gamma; B(\cptset, 1)}.
    \end{equation}
    uniformly over $\rmodel \in \rmodelsp(\rspace, K)$, $f \in \mdspace(\rspace, \rmodel)$ and compact sets
    $\cptset \subseteq \R^d$. See \cite[Theorem 5.1]{HAIRER20172578}.
    \item For a smooth function $F$ on $\R^d$ and $\beta \in (0, \infty)$, we set
    \begin{equation*}
      R_{\beta} F(x) \defby \sum_{\abs{k} < \beta} \frac{X^k}{k!} D^k F(x), \quad x \in \R^d.
    \end{equation*}
    Then \cite[Lemma 2.12]{hairer_theory_2014} implies $R_{\beta} F \in \mdspace^{\beta}(\rspace, \rmodel)$.
  \end{enumerate}
\end{definition}

\subsection{BPHZ renormalization}
\label{subsubsection:BPHZ}
Interesting models can be derived from the realization defined below.
\begin{definition}[{\cite[Definition 6.9]{bruned_2019}}]
  We call a linear map $\gls{realization}:\rspace \to \cS'(\R^d)$ a ($\zeta$-)\emph{realization} if
  \begin{equation*}
    \rreal \unit = 1, \quad \rreal \Xi
        =  \zeta, \quad  \rreal(X^k \tau) = x^k \rreal \tau \quad \text{for every } \tau \in \rbasis(\rspace).
  \end{equation*}
  A realization is called \emph{smooth} if its image is a subset of $C^{\infty}(\R^d)$.
  Given an admissible kernel $K$, a realization $\rreal$ is called \emph{$K$-admissible} if it additionally satisfies
  \begin{equation*}
    \bPi \rint_k(\tau) = \partial^k K \conv \tau \quad \text{for every } \tau \in \rbasis(\rspace) \text{ and }
    k \in \N_0^d \text{ with } \abs{k} \leq 1.
  \end{equation*}
\end{definition}
\begin{definition}[{\cite[Definition 6.8]{bruned_2019}}]\label{def:realization_to_model}
  Let $K$ be an admissible kernel.
  To a smooth $K$-admissible realization $\rreal$, one can associate a model
  \begin{equation*}
    \gls{model_from_real} \defby (\Pi, \Gamma)
  \end{equation*}
   realizing $K$ as in \cite[Definition~6.8]{bruned_2019}.
   \begin{calc} More precisely, as follows.
  \begin{itemize}
    \item For $z \in \R^d$, one defines a character $g_z^+(\rreal):\hat{\rspace}_+ \to \R$ by
    \begin{equation*}
      g_z^+(\rreal)(X_i) \defby - z_i, \quad
      g_z^+(\rreal)(\hat{\rint}_k (\tau)) \defby \partial^k K \conv \tau (z), \quad g_z^+(\rreal) (\Xi) = 0.
    \end{equation*}
    \item One defines the character $f_z:\rspace_+ \to \R$ by $f_z \defby g_z^+(\rreal) \hat{\rantipode}_+$.
    \item One sets
    \begin{equation*}
      \Pi_x \defby (\rreal \otimes f_x) \Delta^{\circ}_+,
      \quad \Gamma_{x y} \defby f_x \cdot f_y^{-1} = f_x \cdot (f_y \circ \rantipode_+),
    \end{equation*}
    where $\cdot$ represents the product of the character group $\rcharacter_+$. 
  \end{itemize}
  \end{calc}
  We denote by $\gls{model_sp}$ the closure in $\bar{\rmodelsp}(\rspace, K)$ of
  \begin{equation*}
    \set{\rmodel(\rreal) \given \rreal \text{ is a smooth $K$-admissible realization}}.
  \end{equation*}
\end{definition}
\begin{definition}[{\cite[Proposition 6.12]{bruned_2019}}]\label{def:canonical_model}
  A \emph{($K$-)canonical realization} $\ecanonicalmodel$ for $\xi_{\epsilon}$ is the smooth $K$-admissible
  $\xi_{\epsilon}$-realization characterized by the identities
  \begin{equation*}
    \ecanonicalmodel(\tau \sigma) = \ecanonicalmodel(\tau) \ecanonicalmodel(\sigma),
    \quad \ecanonicalmodel(\rrootred_{\alpha} \tau) = \ecanonicalmodel \tau,
  \end{equation*}
  where $\rrootred_{\alpha} \tau$ is obtained from $\tau = (F, \hat{F})^{\fn, \fo}_{\fe}$ by
  resetting $\hat{F}(\rho_{\rrootred \tau}) = 1$ and $\fo(\tau) = \alpha$.
  We set
  \begin{equation*}
    \rmodel^{\can, \epsilon} \defby (\Pi^{\can, \epsilon}, \Gamma^{\can})
    \defby \rmodel(\ecanonicalmodel).
  \end{equation*}
\end{definition}

In the situation of our interest, the model $\rmodel^{\can, \epsilon}$ does not converge as $\epsilon \downarrow 0$.
To obtain a limit, one has to ``twist'' the realization $\ecanonicalmodel$.
This operation of twisting is called \emph{renormalization}.
The most natural renormalization is called
the \emph{BPHZ renormalization}, as introduced in \cite{bruned_2019}.
\begin{definition}[{\cite[Theorem~6.16]{bruned_2019}}]
  The \emph{BPHZ realization} $\ebphzmodel$
  is a unique $\xi_{\epsilon}$-realization
  characterized by the following properties: 
  \begin{itemize}
    \item $\ebphzmodel = (g \otimes \ecanonicalmodel) \Delta^{\circ}_-$ for some algebraic map $g: \rspace_- \to \R$;
    \item For every $\tau \in \rspace$ with $\abs{\tau}_+ < 0$, one has $\expect[\ebphzmodel \tau (0)] = 0$.
  \end{itemize}
  We set
  \begin{equation*}
    \rmodel^{\BPHZ, \epsilon} \defby (\Pi^{\BPHZ, \epsilon}, \Gamma^{\BPHZ, \epsilon})
    \defby \rmodel(\bPi^{\BPHZ, \epsilon}).
  \end{equation*}
\end{definition}
To solve the generalized parabolic Anderson model \eqref{eq:gpam_app}, we need the convergence of $\rmodel^{\BPHZ, \varepsilon}$. More precisely, we assume the following; for a condition under which this assumption is satisfied, we refer to \cite[Theorem~2.31]{chandra2018analytic}.

\begin{assumption}\label{assump:convergence_of_BPHZ}
  As $\epsilon \downarrow 0$, the family of models $(\rmodel^{\BPHZ, \epsilon})_{\epsilon \in (0, 1)}$ converges to
  some model $\rmodel^{\BPHZ} = (\Pi^{\BPHZ}, \Gamma^{\BPHZ})$, independent of the mollifier $\rho$, in
  $\rmodelsp(\rspace, K)$ in probability.
  Furthermore, there exists a $\delta' \in (0, 1)$ with the following property.
  For every $p \in 2\N$, there exist constants $C_p^{\BPHZ} \in (0, \infty)$ and a map
  $\bm{\epsilon}^{\BPHZ}_p: (0, 1) \to (0, \infty)$ such that $\lim_{\epsilon \downarrow 0} \bm{\epsilon}^{\BPHZ}_p(\epsilon) = 0$
  and the estimates
  \begin{align*}
    \expect[\abs{\inprd{\Pi_x^{\BPHZ} \tau}{\phi_x^{\lambda}}}^{p}]
    &\leq C_p^{\BPHZ} \lambda^{p(\abs{\tau}_+ + \delta')}, \\
    \expect[\abs{\inprd{\Pi_x^{\BPHZ} \tau - \Pi_x^{\BPHZ, \epsilon} \tau}{\phi_x^{\lambda}}}^{p}]
    &\leq \bm{\epsilon}^{\BPHZ}_p(\epsilon) \lambda^{p(\abs{\tau}_+ + \delta')}
  \end{align*}
  hold for all $x \in \R^d$, $\lambda \in (0, 1)$, $\phi \in C^2(\R^d)$ with $\norm{\phi}_{C^2(\R^d)} \leq 1$
  and with $\supp(\phi) \subseteq B(0, 1)$ and $\tau = (T, 0)^{\fn, 0}_{\fe} \in \cT$ with $\abs{\tau}_+ < 0$.
  Here we write $\phi^{\lambda}_x \defby \lambda^{-d} \phi(\lambda^{-1}(\cdot - x))$.
\end{assumption}

\renewcommand{\glossarysection}[2][]{} 

\subsubsection*{List of symbols from Appendix~\ref{sec:review_of_reg_str}}
\printunsrtglossary[type=symbols,style=index]


\section{Proof of Theorem~\ref{thm:convergence_X_Y_BPHZ}}
\label{sec:bphz_AH}
Based on the framework discussed in Appendix~\ref{sec:review_of_reg_str}, here we provide the details to prove Theorem~\ref{thm:convergence_X_Y_BPHZ}.
First, we prepare some results on the kernel $G_N$ from Definition~\ref{def:G_N_and_H_N}.
In the rest of this section, we fix an admissible kernel $K$ as in Definition~\ref{def:admissible_kernel}
and we set
\begin{equation}\label{eq:def_H_N}
  H_N \defby G_N - K.
\end{equation}
\begin{lemma}\label{lem:H_N_is_Schwarz}
  For every $N \in \N_0$, the function $H_N$ belongs to $\S(\R^d)$.
\end{lemma}
\begin{proof}
  Since the Fourier transform of $G_N - G$ has a compact support, we observe that
  $H_N = (G_N - G) + (G - K)$ is smooth. Thus, it comes down to showing that
  $H_N$ decays rapidly or equivalently, as $K$ is supported on $B(0, 1)$, to showing that $G_N$ decays rapidly.
  \tymchange{
    For $m, n \in \N_0^d$, one has 
    \begin{equation*}
      \mathcal{F}[x^n \partial^m G_N](\xi) = (- 2 \pi \i)^{m-n} \partial_{\xi}^n[ \xi^m \abs{\xi}^{-2} (1- \check{\chi})(2^{-N} \xi)].   
    \end{equation*}
    If $\abs{k} > 0$, the function $\partial^k \check{\chi}$ is smooth and supported on $B(0, 2^{N+2}/3) \setminus B(0, 3/4)$.  
    Therefore, with some compactly supported smooth function $R$ one has
    \begin{equation*}
      \partial_{\xi}^n[ \xi^m \abs{\xi}^{-2} (1- \check{\chi})(2^{-N} \xi)] 
      = \partial_{\xi}^n[ \xi^m \abs{\xi}^{-2}] \times (1- \check{\chi})(2^{-N}\xi) + R(\xi).
    \end{equation*} 
    If $n_1, \ldots, n_d$ are sufficiently large, then $\partial^n [\xi^m \abs{\xi}^{-2}]$ is bounded for $\abs{\xi} \geq 3/4$. 
    This means that $\mathcal{F}[x^n \partial^m G_N](\xi)$ is integrable, or equivalently, $x^n \partial^m G_N$ is bounded.
  }
\end{proof}
\begin{remark}
  Thanks to Lemma \ref{lem:H_N_is_Schwarz}, the convolution $H_N \conv f$ is well-defined for $f \in \S'$
  and the distribution $H_N \conv f$ represents a smooth function.
\end{remark}

\begin{definition}\label{def:kernel_cG_N}
  Suppose that the model $\rmodel$ realizes $K$.
     For $\gamma \in (0, \infty) \setminus \N$,
     $f \in \mdspace^{\gamma}(\rspace, \rmodel)$ and $N \in \N_0$, we set
    \begin{equation*}
      \cG_{N} f(x) \defby \cG_{N, \gamma}^{\rmodel} f(x) \defby \cK^{\rmodel}_{\gamma} f(x)
      + R_{\gamma + 2}[H_N \conv \reconst^{\rmodel} f](x).
    \end{equation*}
    Recalling that $G_N = K + H_N$ from \eqref{eq:def_H_N}, one has $\reconst^{\rmodel} \cG^{\rmodel}_{N} f = G_N \conv \reconst^{\rmodel} f$.
    For the motivation behind the parameter $N$, see Remark~\ref{rem:towards_construction_assumptions}. 
\end{definition}
\subsection{Definition of modelled distributions}
Here we rigorously implement the strategy outlined in Remark~\ref{remark:finding_w_eps}, in the framework of Appendix~\ref{sec:review_of_reg_str}.
\begin{definition}\label{def:set_cT}
  We define $\cT, \cT_- \subseteq \rspace$ and $\rbasis(\cT), \rbasis(\cT_-) \subseteq \rbasis(\rspace)$
  as follows.
  \begin{enumerate}
    \item For $\tau_1,\tau_2\in \sT$ we write ``$\nabla \rint(\tau_1)\cdot \nabla \rint(\tau_2)$'' for ``$\sum_{j=1}^d \rint_j(\tau_1) \rint_j(\tau_2)$''.
    \item We denote by $\cT$ the smallest subset of $\rspace$ with the following properties:
    \begin{itemize}
      \item $\Xi \in \cT$ and
      \item if $\tau_1, \tau_2 \in \cT$, then $\nabla \rint(\tau_1) \cdot \nabla \rint(\tau_2) \in \cT$.
    \end{itemize}
    Furthermore, we associate $c(\tau) \in \N$ to each $\tau \in \cT$ by setting $c(\Xi) \defby 1$ and by inductively setting for $\tau_1,\tau_2\in \cT$
    \begin{equation*}
      c(\nabla \rint (\tau_1) \cdot \nabla \rint (\tau_2)) \defby
      \begin{cases}
        2 c(\tau_1) c(\tau_2) & \text{if } \tau_1 \neq \tau_2, \\
        c(\tau_1) c(\tau_2) & \text{if } \tau_1 = \tau_2.
      \end{cases}
    \end{equation*}
    \item One defines $\rbasis(\cT) \subseteq \rbasis(\rspace)$ as the minimal subset with the following properties:
    \begin{itemize}
      \item $\Xi \in \rbasis(\cT)$ and
      \item if $\tau_1, \tau_2 \in \rbasis(\cT)$ and $i \in \{1, \ldots, d\}$, then $\rint_i(\tau_1) \rint_i(\tau_2)
      \in \rbasis(\cT)$.
    \end{itemize}
    \item We set $\cT_- \defby \set{\tau \in \cT \given \abs{\tau}_+ < 0}$ and
    $\rbasis(\cT_-) \defby \set{\tau \in \rbasis(\cT) \given \abs{\tau}_+ < 0}$.
  \end{enumerate}
\end{definition}
\begin{definition}\label{def:kernel_for_cT}
  Given a model $\rmodel$ realizing $K$, we associate $ \tau^{\rkernel} =  \tau^{\rkernel, \rmodel}
  \in \mdspace^{\gamma_{\tau}}(\rspace, \rmodel)$ to each
  $\tau \in \cT_-$ by setting $\Xi^{\rkernel } \defby \Xi$ and by inductively setting
  \begin{equation*}
    \gamma_{\tau} \defby \min\{\gamma_{\tau_1} + 1 + \abs{\tau_2}_+, \gamma_{\tau_2} + 1 + \abs{\tau_1}_+\}, \quad
     \tau^{\rkernel} \defby \sum_{i=1}^d \rderivative_i
     [\rkernel \tau_1^{\rkernel}] \rprod \rderivative_i [\rkernel \tau_2^{\rkernel }]
  \end{equation*}
  for $\tau = \nabla \rint (\tau_1) \cdot \nabla \rint (\tau_2)$.
  The exponent $\gamma_{\Xi}$ is chosen so that $\gamma_{\tau} > 2$ for every $\tau \in \cT_-$.
\end{definition}
\begin{remark}\label{rem:tau_K_and_tau}
  Thanks to Proposition \ref{prop:product_modelled_distributions} and \cite[Theorem 5.12]{hairer_theory_2014}, indeed one has
  $\tau^{\rkernel, \rmodel} \in \mdspace^{\gamma_{\tau}}$.
  Furthermore, for $\tau = \nabla \rint (\tau_1) \cdot \nabla \rint (\tau_2)$ and a compact set $\cptset$, one has
  \begin{equation*}
    \vertiii{\tau^{\rkernel, \rmodel}}_{\gamma_{\tau}; \cptset}
    \lesssim_{\rspace} (1 + \vertiii{\rmodel}_{\gamma_{\tau_1} + \gamma_{\tau_2}+2; B(\cptset, 1)})^6
    \vertiii{\tau_1^{\rkernel, \rmodel}}_{\gamma_{\tau_1}; B(\cptset, 1)}
    \vertiii{\tau_2^{\rkernel, \rmodel}}_{\gamma_{\tau_2}; B(\cptset, 1)}.
  \end{equation*}
  Therefore, there exist a constant $\gamma, C \in (0, \infty)$ and integers $k, l \in \N$,
  which depend only on $\rspace$, such that
  \begin{equation}\label{eq:estimate_of_tau_K}
    \vertiii{\tau^{\rkernel, \rmodel}}_{\gamma; \cptset}
    \leq C (1 + \vertiii{\rmodel}_{\gamma; B(\cptset, l)})^k
  \end{equation}
  uniformly over $\tau \in \cT_-$, $\rmodel \in \rmodelsp(\rspace, \rmodel)$ and compact sets $\cptset \subseteq \R^d$.
\end{remark}
\begin{definition}\label{def:def_of_modelled_distributions_for_AH}
  Let $F \in C_c^{\infty}(\R)$ be such that $F(x) = -e^{2x}$ if $\abs{x} \leq 2$.
  Given $N \in \N$ and a model $\rmodel$ realizing $K$, we set
  \begin{equation*}
    \regX \defby \regX^{\rmodel} \defby \sum_{\tau \in \cT_-} c(\tau) \tau^{\rkernel, \rmodel},
  \end{equation*}
  \begin{equation*}
    \regW_{N} \defby \regW^{\rmodel}_{N}
    \defby \rproj_{< 2} \cG_{N,  2}^{\rmodel} \regX^{\rmodel},
  \end{equation*}
  where $\cG_{N, 2}^{\rmodel}$ is as in Definition \ref{def:kernel_cG_N},
  and
  \begin{align*}
    \regY_{N} &\defby \regY_{N}^{\rmodel} \\
    &\defby  \rproj_{<\delta} \Big[F^{\rprod}(\regW_{N}^{\rmodel})
    \rprod \Big \{  \sum_{\substack{\tau_1, \tau_2 \in \cT_-, \\
    \abs{\tau_1}_+ + \abs{\tau_2}_+ > -2}}
    \sum_{i=1}^d c(\tau_1) c(\tau_2) \rderivative_i [\rkernel^{\rmodel} \tau^{\rkernel, \rmodel}_1]\rprod  \rderivative_i
    [\rkernel^{\rmodel} \tau^{\rkernel, \rmodel}_2] \\
    &\phantom{\defby} +2 \sum_{i=1}^d  \rderivative_i [\rkernel^{\rmodel} \regX^{\rmodel}]
    \star R_2[\partial_i \{H_N \conv ( \reconst^{\rmodel} \regX^{\rmodel})\} ] \Big\}\Big].
  \end{align*}
\end{definition}
\begin{proposition}\label{prop:modelled_distributions_for_AH}
  Suppose that a model $\rmodel$ realizes $K$. Let $N \in \N$.
  Then, one has $\regW_{N}^{\rmodel} \in \mdspace_{0}^2(\rspace, \rmodel)$ and $\regY_{N}^{\rmodel}
  \in \mdspace^{\delta}_{-1+\delta}(\rspace, \rmodel)$.
  More precisely, there exist constants $\gamma, C \in (0, \infty)$ and integers $k, l \in \N$ such that
  the following estimates hold uniformly over
  $N \in \N$,
  $\rmodel, \overline{\rmodel} \in \rmodelsp(\rspace, K)$ and convex compact sets $\cptset \subseteq \R^d$:
  \begin{align*}
    \vertiii{\regX^{\rmodel}}_{2; \cptset}
    & \leq C (1 + \vertiii{\rmodel}_{\gamma; B(\cptset, l)})^k, \\
    \vertiii{\regW^{\rmodel}_{N}}_{2; \cptset}
    & \leq C \{(1 + \vertiii{\rmodel}_{\gamma; B(\cptset, l)})^k
    +  \norm{H_N \conv (\reconst^{\rmodel} X^{\rmodel})}_{C^2(\cptset)} \}, \\
    \vertiii{\regY_{N}^{\rmodel}}_{\delta; \cptset}
    & \leq C (1 + \vertiii{\rmodel}_{\gamma; B(\cptset, l)}
    + \norm{H_N \conv (\reconst^{\rmodel} X^{\rmodel})}_{C^2(\cptset)})^k,
  \end{align*}
  and furthermore
  \begin{align*}
    \vertiii{\regX^{\rmodel}; \regX^{\overline{\rmodel}}}_{2; \cptset} & \leq C (1 + \vertiii{\rmodel}_{\gamma; B(\cptset, l)}
    \vertiii{\overline{\rmodel}}_{\gamma; B(\cptset, l)})^k
    \vertiii{\rmodel; \overline{\rmodel}}_{\gamma; B(\cptset, l)}, \\
    \vertiii{\regW^{\rmodel}_{N};\regW^{\overline{\rmodel}}_{N}}_{2; \cptset}
&    \leq C (1 + \vertiii{\rmodel}_{\gamma; B(\cptset, l)}
    +\vertiii{\overline{\rmodel}}_{\gamma; B(\cptset, l)} )^k
    \vertiii{\rmodel; \overline{\rmodel}}_{\gamma; B(\cptset, l)} \\
& \qquad    +  \norm{H_N \conv (\reconst^{\rmodel} X^{\rmodel}
    - \reconst^{\overline{\rmodel}} X^{\overline{\rmodel}})}_{C^2(\cptset)} \\
    \vertiii{\regY_{N}^{\rmodel}; \regY_{N}^{\overline{\rmodel}}}_{\delta; \cptset}
    & \leq C \Big(1 + \vertiii{\rmodel}_{\gamma; B(\cptset, l)} + \vertiii{\overline{\rmodel}}_{\gamma; B(\cptset, l)} \\
& \qquad     + \norm{H_N \conv (\reconst^{\rmodel} X^{\rmodel})}_{C^2(\cptset)}
    + \norm{H_N \conv (\reconst^{\overline{\rmodel}} X^{\overline{\rmodel}})}_{C^2(\cptset)}
    \Big)^k \\
& \qquad \qquad    \times (
    \vertiii{\rmodel; \overline{\rmodel}}_{\gamma; B(\cptset, l)} +
    \norm{H_N \conv (\reconst^{\rmodel} X^{\rmodel}
    - \reconst^{\overline{\rmodel}} X^{\overline{\rmodel}})}_{C^2(\cptset)}
    ).
  \end{align*}
\end{proposition}
\begin{proof}
  The estimate for $\regX^{\rmodel}$ follows from \eqref{eq:estimate_of_tau_K}.
  As for the estimate of $\regW_{N}^{\rmodel}$, the Schauder estimate \eqref{eq:schauder_estimate_modelled_distribution}
  gives the estimate for $\rkernel \regX^{\rmodel}$. The estimate for $R_2[H_N \conv (\reconst^{\rmodel} \regX^{\rmodel})]$ follows from
  the estimate
  \begin{equation*}
    \vertiii{R_2[H_N \conv ( \reconst^{\rmodel} \regX^{\rmodel})]}_{2; \cptset}
    \leq \sum_{m: \abs{m} \leq 2} \norm{\partial^m[H_N \conv ( \reconst^{\rmodel} \regX^{\rmodel})]}_{L^{\infty}(\cptset)},
  \end{equation*}
  where the convexity of $\cptset$ is used.
  The estimate for $\regY_{N}^{\rmodel}$ follows from Proposition \ref{prop:product_modelled_distributions},
  the estimate \eqref{eq:composition_modelled_distribution} and the Schauder estimate \eqref{eq:schauder_estimate_modelled_distribution}.

  For the estimates of the differences, 
  the proof is similar. Indeed, we apply difference-analogue of Proposition \ref{prop:product_modelled_distributions},
  the estimate \eqref{eq:composition_modelled_distribution} and \eqref{eq:schauder_estimate_modelled_distribution},
  which can be found in \cite[Proposition 4.10]{hairer_theory_2014},
  \cite[Proposition 3.11]{HAIRER_2015} and \cite[Theorem 5.12]{hairer_theory_2014} respectively.
\end{proof}
\begin{definition}
\label{def:def_of_X_W_N_Y_N_for_model}
  Given a model $\rmodel$ realizing $K$ and $N \in \N$, 
with $\regX^{\rmodel}, \regW^{\rmodel}_{N}$ and $ \regY^{\rmodel}_{N}$ as in Definition~\ref{def:def_of_modelled_distributions_for_AH}, 
  we set 
  \begin{equation*}
     X^{\rmodel} \defby \reconst^{\rmodel} \regX^{\rmodel}, \quad
     W^{\rmodel}_{N} \defby \reconst^{\rmodel} \regW^{\rmodel}_{N}
  \end{equation*}
  and
  \begin{equation*}
    Y^{\rmodel}_{N}
    \defby \reconst^{\rmodel} \regY^{\rmodel}_{N}
    + F(W_{N}^{\rmodel}) \Big\{ \abs{\nabla [ H_N \conv (X^{\rmodel})]}^2 + [\Delta(G_N - G)] \conv  X^{\rmodel} \Big\}.
  \end{equation*}
As noted in Definition~\ref{def:kernel_cG_N}, one has $W_N^{\rmodel} = G_N \conv X^{\rmodel}$.
\end{definition}
\begin{lemma}\label{lem:cancellation_for_W_can}
  Let $\epsilon \in (0, 1)$. To simplify notation, we write
  $X^{\can,\epsilon} \defby X^{\rmodel^{\can, \epsilon}}$ here for instance.
  Then, one has the following identity:
  \begin{multline*}
    \abs{\nabla W_{N}^{\can,\epsilon}}^2 + \Delta W_{N}^{\can,\epsilon}
    = - \xi_{\epsilon} \\
    + \sum_{\substack{\tau_1, \tau_2 \in \cT_-, \\
    \abs{\tau_1}_+ + \abs{\tau_2}_+ > -2}}
     c(\tau_1) c(\tau_2) \nabla(K \conv \reconst^{\can,\epsilon} \tau_1^{\rkernel, \can,\epsilon}) \cdot
     \nabla (K \conv \reconst^{\can,\epsilon} \tau_2^{\rkernel, \can,\epsilon}) \\
    +2   \nabla [K \conv X^{\can,\epsilon}] \cdot
     \nabla [H_N \conv ( X^{\can,\epsilon})]
     +
     \abs{\nabla [ H_N \conv (X^{\can,\epsilon})]}^2 + [\Delta(G_N - G)] \conv X^{\can,\epsilon}. 
  \end{multline*}
\end{lemma}
\begin{proof}
  One has
    $W_{N}^{\can} = K \conv X^{\can} + H_N \conv X^{\can}$
  and 
  \begin{multline*}
    \abs{\nabla W_{N}^{\can}}^2 = \sum_{\tau_1, \tau_2 \in \cT_-} c(\tau_1) c(\tau_2)
    \nabla [K \conv \reconst^{\can} \tau_1^{\rkernel, \can}] \cdot
    \nabla [K \conv \reconst^{\can} \tau_2^{\rkernel, \can}]  \\
    + 2 \nabla [K \conv X^{\can}] \cdot
     \nabla [H_N \conv  X^{\can}]
     + \abs{\nabla H_N \conv  X^{\can}}^2.
  \end{multline*}
  Furthermore,
  \begin{equation*}
    \Delta W_{N}^{\can} = - \sum_{\tau \in \cT_-} c(\tau) \reconst^{\can} \tau^{\rkernel, \can}
    + [\Delta (G_N - G)] \conv X^{\can}.
  \end{equation*}
  Now it remains to observe
  \begin{multline*}
    \sum_{\tau_1, \tau_2 \in \cT_-} c(\tau_1) c(\tau_2)
    \nabla [K \conv \reconst^{\can} \tau_1^{\rkernel, \can}] \cdot
    \nabla [K \conv \reconst^{\can} \tau_2^{\rkernel, \can}]
    - \sum_{\tau \in \cT_-} c(\tau) \reconst^{\can} \tau^{\rkernel, \can} \\
    = -\xi_{\epsilon}
    +\sum_{\substack{\tau_1, \tau_2 \in \cT_-, \\
    \abs{\tau_1}_+ + \abs{\tau_2}_+ > -2}}
     c(\tau_1) c(\tau_2) \nabla(K \conv \reconst^{\can} \tau_1^{\rkernel, \can}) \cdot
     \nabla (K \conv \reconst^{\can} \tau_2^{\rkernel, \can}). \qedhere
  \end{multline*}
\end{proof}
\subsection{BPHZ renormalization for $\regX$}
The goal of this section is to show $X^{\rmodel^{\BPHZ, \epsilon}} = X^{\rmodel^{\can, \epsilon}} - c_{\epsilon}$
(Proposition \ref{prop:bphz_vs_can_for_X}). To this end, our first goal is to obtain
the basis expansion for modelled distributions
$\tau^{\rkernel, \rmodel} \in \cT_-$, which will be given in Lemma \ref{lem:tau_K_expansion}.
\begin{lemma}\label{lem:Delta_to_T_0}
  For every $\tau_1, \tau_2 \in \cT_-$ with $\abs{\tau_1}_+, \abs{\tau_2}_+ < -1$ and $i, j \in \{1, \ldots, d\}$, one has
  \begin{equation*}
    \Delta^{\circ}_+ [\rint_i(\tau_1)] = \rint_i(\tau_1) \otimes \unit_+, \quad
    \Delta^{\circ}_+ [\rint_i(\tau_1) \rint_j(\tau_2)] = [\rint_i(\tau_1) \rint_j(\tau_2)] \otimes \unit_+.
  \end{equation*}
  In particular, the constant map $x \mapsto \rint_i(\tau_1) \rint_j(\tau_2)$
  belongs to $\mdspace^{\infty}_{\abs{\tau_1} + \abs{\tau_2} + 2}(\rspace, \rmodel)$ for any model $\rmodel = (\Pi, \Gamma)$
  and
  \begin{equation*}
    \reconst [ \rint_i(\tau_1) \rint_j(\tau_2) ] = \Pi_x[\rint_i(\tau_1) \rint_j(\tau_2)],
  \end{equation*}
  where the right-hand side is independent of $x$.
\end{lemma}
\begin{proof}
  In view of the recursive formula
  \cite[Proposition~4.17]{bruned_2019}, one can prove the claim by induction on $\abs{\cdot}_+$.
  Indeed, suppose one is going to prove $\Delta^{\circ}_+ \tau = \tau \otimes \punit$,
  where $\tau = \rint_i(\tau_1) \rint_j(\tau_2)$ and $\Delta^{\circ}_+ \tau_k = \tau_k \otimes \punit$.
  By Lemma \ref{lem:product_is_regular},
  $\Delta^{\circ}_+ \tau = \Delta^{\circ}_+[\rint_i(\tau_1)] \Delta^{\circ}_+[\rint_j(\tau_2)]$.
  Therefore, it suffices to show
  $\Delta^{\circ}_+[\rint_i(\tau_1)] = [\rint_i(\tau_1)] \otimes \punit$.
  By \cite[Proposition~4.17]{bruned_2019}, one has
  \begin{equation*}
    \Delta^{\circ}_+ \rint_i(\tau_1) = (\rint_i \otimes \Id) \Delta \tau_1 +
    \sum_{k: \abs{\tau}_+ + 1 - \abs{k} > 0} \frac{X^k}{k!} \otimes \hat{\rint}_{e_i + k}(\tau_1).
  \end{equation*}
  It remains to observe that $(\rint_i \otimes \Id) \Delta \tau_1 = [\rint_i(\tau_1)] \otimes \punit$ by hypothesis of the induction
  and that the set over which $k$ ranges is empty.
\end{proof}
\begin{definition}\label{def:remove_cT}
  We use some notations from Section~\ref{subsec:terminologies_reg_str}.
  Let $\tau \in \rbasis(\cT)$ and let $e$ be an edge of $\tau$ with $\ft(e) = \rint$.
  By removing the edge $e$, we
  obtain a decorated forest with two connected components. We denote by
  \begin{equation*}
    \Remove(\tau;e)
  \end{equation*}
  the component containing the root of $\tau$, with decoration inherited from $\tau$.
  For instance,
  \begin{equation*}
    \Remove(
    \begin{smalldrawtree}
      \node[black node]{} [grow'=up]
        child {node[black node]{}
          child {node[black node]{}
            child {node[noise node]{}}
            child {node[noise node]{}}
          }
          child {node[black node]{}
            child {node[noise node]{} edge from parent[black]}
            child {node[noise node]{} edge from parent[black]}
            edge from parent[purple]
          }
        }
        child {node[black node]{}
          child {node[noise node]{} }
          child {node[black node]{}
            child {node[noise node]{}}
            child {node[noise node]{} }
          }
        };
    \end{smalldrawtree};
    \begin{drawtree}
      \node[]{} [grow'=up]
        child {node[]{}
          edge from parent[purple]
        };
    \end{drawtree}
    )
    =
    \begin{smalldrawtree}
      \node[black node]{} [grow'=up]
        child {node[black node]{}
          child {node[black node]{}
            child {node[noise node]{}}
            child {node[noise node]{}}
          }
        }
        child {node[black node]{}
          child {node[noise node]{} }
          child {node[black node]{}
            child {node[noise node]{}}
            child {node[noise node]{}}
          }
        };
    \end{smalldrawtree},
  \end{equation*}
  where \begin{tikzpicture}\draw[fill=white,line width=1.2pt]  circle(0.5ex);\end{tikzpicture} represents the noise $\Xi$.
  We set
  \begin{align*}
    \Remove(\rbasis(\cT)) &\defby \set{\Remove(\tau; e) \given \tau \in \rbasis(\cT), e \in E_{\tau} \text{ with }
    \ft(e) = \rint}, \\
    \Remove^{\fn}(\rbasis(\cT)) &\defby \set{(T, 0)^{\fn, 0}_{\fe} \given (T, 0)^{0, 0}_{\fe} \in \Remove(\rbasis(\cT))}.
  \end{align*}
\end{definition}
\begin{lemma}\label{lem:tau_K_expansion}
  Suppose $\rmodel = (\Pi, \Gamma)$ is a model realizing $K$.
  Then, one has a claim for $\tau \in \cT_-$ as follows.
  \begin{enumerate}[leftmargin=*]
    \item If $\tau = \Xi$ or $\tau = \nabla \rint(\tau_1) \cdot \nabla \rint(\tau_2)$ with
    $\abs{\tau_1}_+, \abs{\tau_2}_+ < -1$, then $\tau^{\rkernel, \rmodel} = \tau$.
    \item If $\tau = \nabla \rint(\tau_1) \cdot \nabla \rint(\tau_2)$ with $\abs{\tau_1}_+ > -1$ and
    $\abs{\tau_2}_+ < -1$, then one has the expansion
    \begin{equation}\label{eq:tau_K_expansion}
       \tau^{\rkernel, \rmodel} (x) = \tau + \sum_{\sigma \in \fV(\tau)} a_{\tau, \sigma}^{\rmodel}(x) \sigma,
    \end{equation}
    with the following properties:
    \begin{itemize}[leftmargin=*]
      \item $\fV(\tau)$ is a finite subset of $\Remove^{\fn}(\rbasis(\cT))$ that is independent of $\rmodel$,
      \item one has
      \begin{align*}
      \MoveEqLeft[2]
        a_{\tau, \sigma}^{\rmodel}(x) \\
        = &\sum_{\substack{j \in \{1, \ldots, d\}, n \in \N_0,
        \rho \in \cT_-, \\
        l_1, \ldots, l_n \in \N_0^d, \sigma_1, \ldots, \sigma_n \in \Remove^{\fn}(\rbasis(\cT)) \\
        \abs{\sigma_k}_+ + 2 - l_k > 0,
        -1 < \abs{\rho}_+ < \abs{\tau}_+},
        }
        \hspace{-2cm} c_{\tau, \sigma, \rho}^{l_1, \ldots, l_n, \sigma_1, \ldots, \sigma_n}(\cP)
        [\partial_j K \conv \Pi_x \rho^{\rkernel, \rmodel}(x)]
        \prod_{k=1}^n [\partial^{l_k} K \conv \Pi_x \sigma_k](x)\\
        +&\sum_{\substack{n \in \N_0,
        \rho \in \cT_-, \\
        l, l_1, \ldots, l_n \in \N_0^d, \sigma_1, \ldots, \sigma_n \in \Remove^{\fn}(\rbasis(\cT)) \\
        \abs{\sigma_k}_+ + 2 - l_k > 0,
        -1 < \abs{\rho}_+ < \abs{\tau}_+},
        }
        \hspace{-2cm}
        c_{\tau, \sigma, \rho,l}^{l_1, \ldots, l_n, \sigma_1, \ldots, \sigma_n}(\reconst)
        [\partial^l K \conv (\reconst^{\rmodel} \rho^{\rkernel, \rmodel} - \Pi_x \rho^{\rkernel, \rmodel}(x))](x) \\
        &\hspace{8cm} \times\prod_{k=1}^n [\partial^{l_k} K \conv \Pi_x \sigma_k](x),
      \end{align*}
      where the sum is actually finite and the constants
      \begin{equation*}
      c_{\tau, \sigma, \rho}^{l_1, \ldots, l_n, \sigma_1, \ldots, \sigma_n}(\cP)
      \quad \text{and} \quad
      c_{\tau, \sigma, \rho, l}^{l_1, \ldots, l_n, \sigma_1, \ldots, \sigma_n}(\reconst)
      \end{equation*}
      are independent of $\rmodel$.
    \end{itemize}
  \end{enumerate}
\end{lemma}
\begin{proof}
  To see the claim (a), if $\abs{\tau}_+ < -1$,  thanks to Lemma \ref{lem:Delta_to_T_0}, the identity \eqref{eq:def_of_curl_K}
  becomes
  \begin{equation*}
    \rkernel \tau = \rint \tau + (K \conv \reconst \tau)(x) \unit
  \end{equation*}
  and hence $\rderivative_i \rkernel \tau = \rint_i \tau$.
  The claim (b) seems complicated but can be proven easily by induction. Suppose that one has
  $\tau = \nabla \rint(\tau_1) \cdot \nabla \rint(\tau_2)$ such that $\tau_1$  has
  the expansions of the form \eqref{eq:tau_K_expansion} and $\tau_2^{\rkernel} = \tau_2$. Furthermore, one has
  $-1 < \abs{\tau_1}_+ < 0$ since $\abs{\tau}_+ < 0$.
  Therefore, one has
  \begin{equation}\label{eq:tau_1_K_expansion}
    \tau_1^{\rkernel} = \tau_1 + \sum_{ \sigma \in \Remove^{\fn}(\rbasis(\cT)) } a_{\sigma} \sigma, \quad \tau_2^{\rkernel} = \tau_2
  \end{equation}
  where $a_{\sigma}$ has the desired property.
  By the definition \eqref{eq:def_of_curl_K} of $\rkernel$ , one has
  \begin{multline*}
    \rderivative_i \rkernel \tau^{\rkernel}_1 (x)
    = \rint_i \tau_1 + \sum_{ \sigma \in \Remove^{\fn}(\rbasis(\cT)) } a_{\sigma}(x) \rint_i(\sigma)
    +[\partial_i K \conv \Pi_x \tau_1](x) \unit \\
    + \sum_{\substack{ \sigma \in \Remove^{\fn}(\rbasis(\cT)) , l \in\N_0^d \\ \abs{\sigma} + 1 -  |l|  > 0} } a_{\sigma}(x) [\partial^{\bm{e}_i + l} K \conv \Pi_x \sigma](x) \frac{X^l}{l!}
    + \sum_{\abs{l} < \gamma_{\tau_1} + 1} [\partial^{\bm{e}_i + l} K \conv (\reconst \tau_1^{\rkernel} -
    \Pi_x \tau_1^{\rkernel})](x) \frac{X^l}{l!},
  \end{multline*}
  where $\gamma_{\tau_1}$ is chosen so that $\tau_1^{\rkernel} \in \mdspace^{\gamma_{\tau_1}}(\rspace, \rmodel)$,
  see Remark \ref{rem:tau_K_and_tau}.
  Since $\rderivative_i \rkernel \tau_2^{\rkernel} = \rint_i \tau_2$ as shown in the part (a), one has
  \begin{equation*}
    \rint_i(\sigma) \rint_i(\tau_2), X^l \rint_i(\tau_2) \in \Remove^{\fn}(\rbasis(\cT)).
  \end{equation*}
  Since $\abs{\tau_1}_+ < \abs{\tau}_+$, we
  complete the induction.
\end{proof}

We recall an explicit formula of the BPHZ realization.
\begin{definition}[{\cite[Theorem~6.18]{bruned_2019}}]
  Let $\hat{\rspace}_-$ be the free algebra generated by $\rspace$ under the forest product. (In fact, recalling $H^R_1$ from
  Definition~\ref{def:basis_of_H_R}, we have $\hat{\rspace}_- = H^R_1$.) We define the algebra homomorphism $g_{\epsilon}^-:\hat{\rspace}_- \to \R$ 
  characterized by
  \begin{equation*}
    g_{\epsilon}^- (\rinjection_{\circ} \tau) \defby \expect[\rreal^{\can, \epsilon} \tau (0)],
  \end{equation*}
  where $\rinjection_{\circ}: \rspace \to \hat{\rspace}_-$ is the natural injection. Then, we have
  \begin{equation}\label{eq:bphz_realization}
    \rreal^{\BPHZ, \epsilon} = (g_{\epsilon}^- \hat{\rantipode}_- \otimes \rreal^{\can, \epsilon} \Delta^{\circ}_-).
  \end{equation}
\end{definition}

In view of the identity \eqref{eq:bphz_realization} and
Lemma \ref{lem:tau_K_expansion}, we need to understand $(g^-_{\epsilon} \hat{\rantipode}_- \otimes \ecanonicalmodel)
\Delta_-^{\circ} \tau$ for $\tau \in \cT_-$ and $\tau \in \Remove^{\fn}(\rbasis(\cT))$.
As one can easily guess from the definition of $g^-_{\epsilon}$, it is necessary to estimate
$\expect[\ecanonicalmodel \tau(0)]$ for such $\tau$.
The following simple lemma is a consequence of the symmetry of the noise $\xi$.
\begin{lemma}\label{lem:expectation_remove_cT}
  For $\tau \in \Remove(\rbasis(\cT))$, one has $\expect[\ecanonicalmodel \tau(0)] = 0$.
\end{lemma}
\begin{proof}
  Let $\tau = (T, 0)^{0,0}_{\fe} \in \Remove(\rbasis(\cT))$.
    Let
    $\rreal^{\text{minus}}$ be the canonical realization for $\xi_{\epsilon}(-\cdot)$.
    Since $\xi \dequal \xi(- \cdot)$, one has $\rreal^{\text{minus}} \sigma \dequal \ecanonicalmodel \sigma$ for every
    $\sigma \in \rspace$. If we set
    \begin{equation*}
      n(T) \defby
       \# \set{e \in E_{T} \given \ft(e) = \rint},
    \end{equation*}
    by using the identity
    \begin{equation*}
      \partial_i K \conv [f(-\cdot)] = - [\partial_i K \conv f](- \cdot),
    \end{equation*}
    where the fact $K = K(-\cdot)$ is used,
    one has $\rreal^{\text{minus}} \tau = (-1)^{n(T)} \ecanonicalmodel \tau$.
    However, since $\tau \in \Remove(\rbasis(\cT))$, $n(T)$ is odd. Therefore, one
    has
    \begin{equation*}
      \rreal^{\text{minus}} \tau  \dequal \ecanonicalmodel \tau \quad \text{and} \quad
      \rreal^{\text{minus}} \tau = - \ecanonicalmodel \tau,
    \end{equation*}
    and concludes $\expect[\ecanonicalmodel \tau(0)] = 0$.
\end{proof}
\begin{lemma}\label{lem:Delta_minus_expansion}
  For $\tau = (F, \hat{F})^{\fn, \fo}_{\fe} \in \rbasis(\cT) \cup \Remove^{\fn}(\rbasis(\cT))$ and $x \in \R^d$, one has
  \begin{equation*}
    \Delta^{\circ}_- \tau = \tau \otimes \unit + \unit_- \otimes \tau + \ker(g^-_{\epsilon} \hat{\rantipode}_-
    \otimes \Pi^{\can, \epsilon}_x) \cap \ker(g^-_{\epsilon} \hat{\rantipode}_-
    \otimes \ecanonicalmodel).
  \end{equation*}
\end{lemma}
\begin{proof}
  Recall from Definition~\ref{def:terminology_tree}-(a) that edges are oriented.
  We call an edge $e = (a, b)$ a leaf if $b$ is not followed by any edge.
  We call a node $a$ of $F$ true if there exists an edge $e = (a, b)$ such that $\ft(e) = \rint$.
  We denote by $N^{\text{true}}$ the set of all true nodes of $F$. For a subforest $G$ of $F$, we set
  \begin{equation*}
    N_G^j \defby \set{a \in N_G \cap N^{\text{true}} \given \text{ there exist exactly $j$ outgoing edges in $G$ at $a$}}.
  \end{equation*}

  Recalling the coproduct formula \eqref{eq:coproduct_formula}, one has
  \begin{multline*}
    \Delta^{\circ}_- \tau = \tau \otimes \rrootred_{\abs{\tau}_+} \unit + \unit_- \otimes \tau \\
    + \sum_{G \subseteq F, G \neq \emptyset} \sum_{\fn_G \neq \fn, \epsilon^F_G} \frac{1}{\epsilon^F_G !}
    \binom{\fn}{\fn_G}
    (G, 0)^{\fn_G + \pi \epsilon^F_G, 0}_{\fe} \otimes \rcont (F, \indic_G)^{\fn - \fn_G, \pi(\epsilon^F_A - \fe \indic_G)}_{\fe \indic_{
    E_F \setminus E_G
    } + \epsilon^F_G},
  \end{multline*}
  where $\rrootred_{\alpha}$ is defined in Definition~\ref{def:canonical_model}.
  However, note that $\ecanonicalmodel \rrootred_{\alpha} \unit = \ecanonicalmodel \unit$.
  We fix $G \neq \emptyset$, $\fn_G \neq \fn$ and $\epsilon^F_G$ and set
  \begin{equation*}
    \tau_1 \defby (G, 0)^{\fn_G + \pi \epsilon^F_G, 0}_{\fe}, \quad \tau_2 \defby \rcont (F, \indic_G)^{
    \fn- \fn_G, \pi(\epsilon^F_A - \fe \indic_G)}_{\fe \indic_{
    E_F \setminus E_G
    }+ \epsilon^F_G}
  \end{equation*}
  We will prove $(g^-_{\epsilon} \hat{\rantipode}_- \otimes \Pi^{\can, \epsilon}_x) (\tau_1 \otimes \tau_2) = 0$ by
  considering various cases, which will complete the proof.
  When a case is studied, we exclude all cases considered before.

  \begin{enumerate}[1., leftmargin=*]
    \item Suppose that $G \neq F$ and that a connected component $T$ of $G$ satisfies $N_T^0  = \emptyset$ and
    $N_T^1 = N_F^1 \cap N_G$. Then,
    the forest $\tau_2$ contains a leaf $(a, \rho_T)$ of edge type $\rint$ and hence
    $\Pi^{\can, \epsilon}_x \tau_2 = \ecanonicalmodel \tau_2 = 0$.
    \item Suppose $G$ contains a leaf of edge type $\rint$. Then, in view of the recursive formula \eqref{eq:antipode_minus_recursive},
    this is also the case for each forest appearing in $\hat{\rantipode}_- \tau_1$ and hence
    $g^-_{\epsilon} \hat{\rantipode}_- \tau_1 = 0$.
    \item Suppose $N_G^0 \neq \emptyset$. If the case 2 is excluded, then a connected component of $\tau_1$ is
    of the form $\bullet^{\fn_1, 0}$ and hence $\tau_1 = 0$ (as an element of $\rspace_-$).
    \item Suppose $\tau_1$ contains a connected component $\tau_3 = (T,0)^{\fn, 0}_{\fe}$ such that $\# N^1_T \geq 2$.
    Let $a \in N^1_T$.
    \begin{itemize}
      \item If $a$ is the root of $T$, then $\tau_3 = \rint_i (\tau_4)$ and hence $\tau_1 = 0$ (as an element of $\rspace_-$).
      \item If $a$ is not the root of $T$, one can merge two consecutive edges $(a_1, a)$ and $(a, a_2)$ into
      a single edge $(a_1, a_2)$ to obtain a new tree $\tau_5 \in \rtree_{\circ}$ with
      $\abs{\tau_3}_- = \abs{\tau_5}_- + 1$. Since $\abs{\sigma}_- \geq -2 + \delta$ for every $\sigma \in \rtree_{\circ}$,
      if $\#(N_T^1 \setminus \{\rho_T\}) \geq 2$, then $\abs{\tau_3}_- > 0$ and hence
      $\tau_1 = 0$ (as an element of $\rspace_-$).
    \end{itemize}
    \item Suppose that $\tau_1$ contains a connected component $\tau_6 = (T_6, 0)^{\fn_6, 0}_{\fe}$ such that
    $N_{T_6}^0 = N_{T_6}^1 = \emptyset$. Then, $T_1 = T_6 = F$ and $\tau_1 \in \rbasis(\cT)$.
    However, this implies $\fn = \fn_G = 0$, which is excluded.
    \item Therefore, it remains to consider the case where every connnected component $\tau_7 = (T_7, 0)^{\fn_7, 0}_{\fe}$
    of $\tau_1$ satisfies $\# N_{T_7}^1 = 1$ and $N^0_{T_7} = \emptyset$ and all leaves of $\tau_7$ are of type $\Xi$,
    namely $\tau_7 \in \Remove^{\fn}(\rbasis(\cT))$.
    If $\fn_7 \neq 0$ on $N_{T_7}$, then $\abs{\tau_7}_- > 0$. Thus, we suppose $\fn_7 = 0$.
    We will show
    $g^-_{\epsilon} \hat{\rantipode}_- \tau_7 = 0$, which implies $g^-_{\epsilon} \hat{\rantipode}_- \tau_1 = 0$ since
    the character $g^-_{\epsilon} \hat{\rantipode}_-$ is multiplicative.
    To apply the recursive formula \eqref{eq:antipode_minus_recursive}, consider the expansion
    \begin{equation*}
      \hat{\Delta}_- \tau_7 - \tau_7 \otimes \unit_- = \unit \otimes \tau_7 + \sum_{\tau_8} c_{\tau_8} \tau_8 \otimes \tau_9.
    \end{equation*}
    Then, one has
    \begin{equation*}
      g^-_{\epsilon} \hat{\rantipode}_- \tau_7 = - \expect[\ecanonicalmodel \tau_7 (0)]
      - \sum_{\tau_8} c_{\tau_8} \times \big(g^-_{\epsilon} \hat{\rantipode}_- \tau_8\big) \times
      \expect[\ecanonicalmodel \tau_9(0)].
    \end{equation*}
    By the same reasoning as before, one can suppose that every component $\tau_{10} = (T_{10}, 0)^{0,0}_{\fe}$ of $\tau_8$
    belongs to $\Remove(\rbasis(\cT))$.
    However, since $T_{10}$ has a strictly smaller number of edges than
    $T_7$ does, one can assume $g^-_{\epsilon} \hat{\rantipode}_- \tau_8 = 0$ by induction.
    Therefore, it remains to show $\expect[\ecanonicalmodel \tau_7 (0)] = 0$.
    But this was shown in Lemma \ref{lem:expectation_remove_cT}.
    \qedhere
  \end{enumerate}
\end{proof}
\begin{corollary}\label{cor:bphz_character_for_cT}
  If $\tau \in \Remove(\rbasis(\cT))$, then $g^-_{\epsilon} \hat{\rantipode}_- \tau = 0$.
  If $\tau \in \cT_-$, then
  \begin{equation*}
    g^-_{\epsilon} \hat{\rantipode}_- \tau = - \expect[\ecanonicalmodel \tau(0)].
  \end{equation*}
\end{corollary}
\begin{proof}
  The claim for $\tau \in \Remove(\rbasis(\cT))$ is proved in the proof of Lemma \ref{lem:Delta_minus_expansion},
  see the case 6. If $\tau \in \cT_-$, by Lemma \ref{lem:Delta_minus_expansion} one has
  \begin{equation*}
    \ebphzmodel \tau = \ecanonicalmodel \tau + g^-_{\epsilon} \hat{\rantipode} \tau.
  \end{equation*}
  However, since $\abs{\tau}_- < 0$, one has $\expect[\ebphzmodel \tau(0)] = 0$ by definition,
  which completes the proof.
\end{proof}
\begin{proposition}\label{prop:bphz_vs_can_for_X}
  For $\tau \in \cT_-$, one has
  \begin{align}
    \Pi_x^{\rmodel^{\BPHZ, \epsilon}} \tau^{\rkernel, \rmodel^{\BPHZ, \epsilon}}(x) &=
    \Pi_x^{\rmodel^{\can, \epsilon}} \tau^{\rkernel, \rmodel^{\can, \epsilon}}(x)
    - \expect[\ecanonicalmodel \tau (0)], \quad x \in \R^d,  \label{eq:bphz_vs_can_for_X_realization}\\
    \reconst^{\rmodel^{\BPHZ,\epsilon}} \tau^{\rkernel, \rmodel^{\BPHZ,\epsilon}} &=
    \reconst^{\rmodel^{\can,\epsilon}} \tau^{\rkernel, \rmodel^{\can,\epsilon}} - \expect[\ecanonicalmodel \tau (0)] \notag.
  \end{align}
  In particular,
  \begin{equation*}
    X^{\rmodel^{\BPHZ, \epsilon}} = X^{\rmodel^{\can, \epsilon}} - c_{\epsilon}.
  \end{equation*}
  where
  \begin{equation}\label{eq:def_of_c_epsilon}
    c_{\epsilon} \defby \sum_{\tau \in \cT_-} c(\tau) \expect[\ecanonicalmodel \tau (0)].
  \end{equation}
\end{proposition}
\begin{proof}
  To simplify notation, we write $\reconst^{\BPHZ} \defby \reconst^{\rmodel^{\BPHZ, \epsilon}}$ here, for instance.
  Since
  \begin{equation*}
    \reconst^{\#} \tau^{\rkernel, \#} (x) = [\Pi^{\#}_x \tau^{\rkernel, \#}(x)](x), \quad \# \in \{\can, \BPHZ\},
  \end{equation*}
  it suffices to prove \eqref{eq:bphz_vs_can_for_X_realization}.
  By Lemma \ref{lem:tau_K_expansion}, one has the expansion
  \begin{equation*}
    \tau^{\rkernel, \BPHZ}(x) = \tau + \sum_{\sigma} a_{\tau, \sigma}^{\BPHZ}(x) \sigma.
  \end{equation*}
  In the expression of $a_{\tau, \sigma}^{\BPHZ}$ given in Lemma \ref{lem:tau_K_expansion}, every $\rho$
  in the sum satisfies $\abs{\rho}_+ < \abs{\tau}_+$. Therefore, one can assume
  $a_{\sigma}^{\BPHZ} = a_{\sigma}^{\can}$ by induction.
  By Lemma \ref{lem:Delta_minus_expansion} and Corollary \ref{cor:bphz_character_for_cT},
  \begin{equation*}
    \Delta^{\circ}_- \tau^{\rkernel, \BPHZ}(x) = \tau \otimes \unit + \unit_- \otimes \tau
    + \sum_{\sigma} a_{\sigma}^{\can}(x) \unit_- \otimes \sigma
    + \ker(g^-_{\epsilon} \hat{\rantipode}_- \otimes \Pi^{\can}_x).
  \end{equation*}
  Furthermore, by \cite[Theorem 6.16]{bruned_2019}, one has
  \begin{equation*}
    \Pi_x^{\BPHZ} = (g^-_{\epsilon} \hat{\rantipode}_- \otimes \Pi^{\can}_{x}) \Delta^{\circ}_-.
  \end{equation*}
  Therefore,
  \begin{align*}
    \Pi^{\BPHZ}_x \tau^{\rkernel, \BPHZ}(x)
    &= g^-_{\epsilon} \hat{\rantipode}_- \tau + \Pi^{\can}_x \tau
    + \sum_{\sigma} a_{\sigma}^{\can}(x) \Pi^{\can}_x \sigma \\
    &= -\expect[\ecanonicalmodel \tau(0)] + \Pi^{\can}_x \tau^{\rkernel, \can}(x),
  \end{align*}
  where we applied Corollary \ref{cor:bphz_character_for_cT} to get the last equality.
\end{proof}
\subsection{BPHZ renormalization for $\regY_{N}$}
The goal of this section is to compare $\regY_{N}^{\rmodel^{\can, \epsilon}}$ and
$\regY_{N}^{\rmodel^{\BPHZ, \epsilon}}$,
as we did for $\regX$ in the previous section. Again, we need to obtain the basis expansion for $\regY_N$.
\begin{lemma}\label{lem:expansion_for_cY}
  Let $\tau_1, \tau_2 \in \cT_-$, $i \in \{1, \ldots, d\}$ and $N \in \N$.
  Let $\rmodel$ be a model realizing $K$.
  Assume $\abs{\tau_1}_+ + \abs{\tau_2}_+ > -2$.
  Then, for $x \in \R^d$, one has
  \begin{multline*}
    \rproj_{<\delta} \big\{F(\regW_{N}^{\rmodel})(x) \rprod
    \rderivative_i[\rkernel \tau_1^{\rkernel, \rmodel}](x) \rprod
    \rderivative_i[\rkernel \tau_2^{\rkernel, \rmodel}](x)\big\} \\
    = \rproj_{<\delta}  \Big\{  \sum_{k \in \N_0} \frac{D^k F(W_{N}^{\rmodel}(x))}{k!}
    \Big(\sum_{\tau \in \cT_-} \rint \tau \Big)^{\rprod k} \rprod
    \rderivative_i[\rkernel \tau_1^{\rkernel, \rmodel}](x) \rprod
    \rderivative_i[\rkernel \tau_2^{\rkernel, \rmodel}](x) \Big\}
  \end{multline*}
  and
  \begin{multline*}
    \rproj_{<\delta} \big\{F(\regW_{N}^{\rmodel})(x) \rprod
    \rderivative_i [\rkernel^{\rmodel} \regX^{\rmodel}](x)
    \star R_2[\partial_i \{H_N \conv (\reconst^{\rmodel} \regX^{\rmodel})\} ](x) \big\} \\
    = \rproj_{< \delta}  \Big\{  \sum_{k \in \N_0}
    \frac{D^k F(W_{N}^{\rmodel}(x))}{k!} \partial_i[H_N \conv (X^{\rmodel})](x)
    \Big(\sum_{\tau \in \cT_-} \rint \tau \Big)^{\rprod k}
    \rprod \rderivative_i [\rkernel^{\rmodel} \regX^{\rmodel}](x)  \Big\} .
  \end{multline*}
\end{lemma}
\begin{proof}
  By Lemma \ref{lem:tau_K_expansion}, one has
  \begin{equation*}
    \regW_{N}^{\rmodel}(x) = \sum_{\tau \in \cT_-} \rint \tau + W_{N}^{\rmodel}(x) \unit +
    \regW_{N}^{\rmodel, +}(x),
  \end{equation*}
  where $\regW_{N}^{\rmodel, +}(x) \in \oplus_{\alpha \geq 1} \rspace_{\alpha}$.
  Recalling Definition \ref{def:function_and_modelled_distritbusion}, one has
  \begin{equation*}
    F(\regW_{N}^{\rmodel})(x) =
    \sum_{k \in \N_0} \frac{D^k F(W_{N}^{\rmodel}(x))}{k!} \Big(
    \sum_{\tau \in \cT_-} \rint \tau + \regW_{N}^{\rmodel, +}(x) \Big)^{\rprod k}.
  \end{equation*}
  Since Lemma \ref{lem:tau_K_expansion} implies that
  \begin{equation*}
    \rderivative_i[\rkernel \tau_1^{\rkernel, \rmodel}](x) \rprod
    \rderivative_i[\rkernel \tau_2^{\rkernel, \rmodel}](x)
  \end{equation*}
  is $\oplus_{\alpha \geq -1 + \delta} \rspace_{\alpha}$-valued, one can ignore the contribution from
  $\regW_{N}^{\rmodel, +}(x)$ when the projection $\rproj_{< \delta}$ is applied.
  This observation proves the claimed identities.
\end{proof}
\begin{lemma}\label{lem:BPHZ_vs_can_for_Y}
  Let $N \in \N$.   To simplify notation, we write $X^{\BPHZ, \epsilon} \defby X^{\rmodel^{\BPHZ, \epsilon}}$ here, for instance.
  Then, one has
  \begin{multline*}
    \reconst^{{\BPHZ, \epsilon}} \regY_{N}^{{\BPHZ, \epsilon}}
    = F(W_{N}^{{\BPHZ, \epsilon}}) \\
    \times
    \Big\{
    \sum_{\substack{\tau_1, \tau_2 \in \cT_-, \\
    \abs{\tau_1}_+ + \abs{\tau_2}_+ > -2}}
     c(\tau_1) c(\tau_2) \nabla(K \conv \reconst^{{\can, \epsilon}} \tau_1^{\rkernel, {\can, \epsilon}}) \cdot
     \nabla (K \conv \reconst^{{\can, \epsilon}} \tau_2^{\rkernel, {\can, \epsilon}}) \\
    +2   \nabla [K \conv X^{{\can, \epsilon}}] \cdot
     \nabla [H_N \conv X^{\can, \epsilon}] \Big\}
  \end{multline*}
\end{lemma}
\begin{proof}
  To simplify notation, we also write $\Pi_x^{\BPHZ, \epsilon} \defby \Pi_x^{\rmodel^{\BPHZ, \epsilon}}$ here, for instance.
  One has
  \begin{equation*}
    \reconst^{\BPHZ, \epsilon} \regY_{N}^{\BPHZ, \epsilon}(x)
  = [\Pi_x^{\BPHZ, \epsilon} \regY_{N}^{\BPHZ, \epsilon}(x)](x).
  \end{equation*}
  In view of Lemma \ref{lem:tau_K_expansion}, Proposition \ref{prop:bphz_vs_can_for_X} and
  Lemma \ref{lem:expansion_for_cY}, it suffices to show
  \begin{align}
    \Pi^{\BPHZ, \epsilon}_x [\rint(\tau_1) \cdots \rint(\tau_n) \rint_i(\tau_{n+1})]
    &= \Pi^{\can, \epsilon}_x [\rint(\tau_1) \cdots \rint(\tau_n) \rint_i(\tau_{n+1})] \notag , \\
    \label{eq:pi_BPHZ_vs_can}
    \Pi^{\BPHZ, \epsilon}_x [\rint(\tau_1) \cdots \rint(\tau_n) \rint_i(\tau_{n+1}) \rint_i(\tau_{n+2})]
    &= \Pi^{\can, \epsilon}_x [\rint(\tau_1) \cdots \rint(\tau_n) \rint_i(\tau_{n+1}) \rint_i(\tau_{n+2})],
  \end{align}
  for $\tau_1, \ldots, \tau_n, \tau_{n+1} \in \cT_-$ and $\tau_{n+2} \in \Remove(\rbasis(\cT))$.
  We only prove the second identity of \eqref{eq:pi_BPHZ_vs_can}.
  We set
  \begin{equation*}
    \bm{\tau} \defby (F, 0)^{0,0}_{\fe} \defby \rint(\tau_1) \cdots \rint(\tau_n) \rint_i(\tau_{n+1}) \rint_i(\tau_{n+2}),
    \quad (F_j, 0)^{0, 0}_{\fe} \defby \tau_j.
  \end{equation*}
  The proof of \eqref{eq:pi_BPHZ_vs_can} follows the argument in the proof of Lemma \ref{lem:Delta_minus_expansion}.
  We claim
  \begin{multline}\label{eq:Delta_minus_for_bm_tau}
    \Delta^{\circ}_- \bm{\tau}
    = \unit_- \otimes \bm{\tau} + \sum_{J \subseteq \{1, \ldots, n\}} \big[\rint_i (\tau_{n+1}) \prod_{j \in J}
    \rint(\tau_j) \big] \otimes \big[\rint_i(\tau_{n+2})\prod_{j \notin J} \rint(\tau_j)\big] \\
    + \sum_{J \subseteq \{1, \ldots, n\}} \big[\rint_i (\tau_{n+1}) \rint_i(\tau_{n+2}) \prod_{j \in J}
    \rint(\tau_j) \big] \otimes \prod_{j \notin J} \rint(\tau_j).
  \end{multline}
  Indeed, let $\sigma \otimes \sigma'$ be a basis appearing in  the coproduct formula
  \eqref{eq:coproduct_formula} for $\Delta^{\circ}_- \bm{\tau}$. If we set $(G, 0)^{\fn, 0}_{\fe} \defby \sigma$ and
  $\sigma_k \defby (G \cap F_j, 0)^{\fn, 0}_{\fe}$, by repeating the argument in the proof of
  Lemma \ref{lem:Delta_minus_expansion}, the forest $\sigma_k$ is either $\emptyset$, $\tau_k$ or
  $\Remove(\rho_k; e_k)$ for some $\rho_k$ and $e_k$.
  \begin{itemize}
    \item If $\sigma_k = \emptyset$, then $\sigma = 0$ in $\rspace_-$ unless $(\rho_{\bm{\tau}}, \rho_{\tau_k})
    \notin E_{\sigma}$.
    \item If $\sigma_k = \tau_k$, then $\sigma'$ has a leaf of type $\rint$ unless
    $(\rho_{\bm{\tau}}, \rho_{\tau_k}) \in E_{\sigma}$.
    \item If $\sigma_k = \Remove(\rho_k; e_k)$, then $\abs{\sigma}_+ > 0$ and hence $\sigma = 0$ in $\rspace_-$.
  \end{itemize}
  Therefore, the claimed identity \eqref{eq:Delta_minus_for_bm_tau} is established.
  It remains to show
  \begin{equation}\label{eq:bphz_character_I_and_derivative_I}
    g_{\epsilon}^- \hat{\rantipode}_-  \big[\rint_i (\tau_{n+1}) \prod_{j \in J}
    \rint(\tau_j) \big] = 0,
    \quad
    g_{\epsilon}^- \hat{\rantipode}_-  \big[\rint_i (\tau_{n+1}) \rint_i(\tau_{n+2}) \prod_{j \in J}
    \rint(\tau_j) \big] = 0.
  \end{equation}
  Without loss of generality, we can suppose $J = \{1, \ldots, n\}$.
  The proof is based on induction. We only consider the first identity of \eqref{eq:bphz_character_I_and_derivative_I}. As for the case $n=0$,
  the first identity of \eqref{eq:bphz_character_I_and_derivative_I} is shown in Lemma \ref{lem:expectation_remove_cT}.
  Similarly to \eqref{eq:Delta_minus_for_bm_tau}, one can show
  \begin{equation*}
    \hat{\Delta}_- \bm{\tau}
    = \unit_- \otimes \bm{\tau} + \sum_{J \subseteq \{1, \ldots, n\}} \big[\rint_i (\tau_{n+1}) \prod_{j \in J}
    \rint(\tau_j) \big] \otimes \prod_{j \notin J} \rint(\tau_j)
  \end{equation*}
  In view of the recursive formula \eqref{eq:antipode_minus_recursive} and the hypothesis of the induction,
  it remains to show
  \begin{equation*}
    \expect[\ecanonicalmodel \bm{\tau}(0)] = 0.
  \end{equation*}
  However, this can be proved as in Lemma \ref{lem:expectation_remove_cT}, since
  $\bm{\tau}$ has an odd number of edges $e$ such that $\ft(e) = \rint$ and $\abs{\fe(e)} = 1$.
\end{proof}
\begin{proposition}\label{prop:identitiy_for_W_and_Y}
Let $c_{\varepsilon}$ be as in \eqref{eq:def_of_c_epsilon}. We then have 
\begin{equation*}
  Y_{N}^{\rmodel^{\BPHZ, \epsilon}} = F(W_N^{\rmodel^{\BPHZ, \varepsilon}}) 
  (\xi_{\varepsilon} - c_{\varepsilon} + \abs{\nabla W_{N}^{\rmodel^{\BPHZ, \epsilon}}}^2 + \Delta W_{N}^{\rmodel^{\BPHZ, \epsilon}}).
\end{equation*}
\end{proposition}
\begin{proof}
 To simplify notation, we write $X^{\BPHZ, \epsilon} \defby X^{\rmodel^{\BPHZ, \epsilon}}$ here, for instance.
aaaa

  By Proposition \ref{prop:bphz_vs_can_for_X}, 
  $X^{{\BPHZ, \epsilon}}$ and $X^{{\can, \epsilon}}$ are only different by constant, so that $W_{N}^{{\BPHZ, \epsilon}} = W_{N}^{{\can, \epsilon}}$.
  Therefore, by Lemma \ref{lem:cancellation_for_W_can} and Lemma \ref{lem:BPHZ_vs_can_for_Y},
\begin{calc}
\newline
By definition:
  \begin{equation*}
    Y^{{\BPHZ, \epsilon}}_{N} 
    \defby \reconst^{{\BPHZ, \epsilon}} \regY^{{\BPHZ, \epsilon}}_{N}
    + F(W_{N}^{{\BPHZ, \epsilon}}) \Big\{ \abs{\nabla [ H_N \conv (X^{{\BPHZ, \epsilon}})]}^2 + [\Delta(G_N - G)] \conv  X^{{\BPHZ, \epsilon}} \Big\}.
  \end{equation*}
By Lemma~\ref{lem:BPHZ_vs_can_for_Y}
  \begin{multline*}
    \reconst^{{\BPHZ, \epsilon}} \regY_{N}^{{\BPHZ, \epsilon}}
    = F(W_{N}^{{\BPHZ, \epsilon}}) \\
    \times
    \Big\{
    \sum_{\substack{\tau_1, \tau_2 \in \cT_-, \\
    \abs{\tau_1}_+ + \abs{\tau_2}_+ > -2}}
     c(\tau_1) c(\tau_2) \nabla(K \conv \reconst^{{\can, \epsilon}} \tau_1^{\rkernel, {\can, \epsilon}}) \cdot
     \nabla (K \conv \reconst^{{\can, \epsilon}} \tau_2^{\rkernel, {\can, \epsilon}}) \\
    +2   \nabla [K \conv X^{{\can, \epsilon}}] \cdot
     \nabla [H_N \conv X^{\can, \epsilon}] \Big\}
  \end{multline*}
By Lemma~\ref{lem:cancellation_for_W_can}
  \begin{multline*}
    \abs{\nabla W_{N}^{\can,\epsilon}}^2 + \Delta W_{N}^{\can,\epsilon}
    = - \xi_{\epsilon} \\
    + \sum_{\substack{\tau_1, \tau_2 \in \cT_-, \\
    \abs{\tau_1}_+ + \abs{\tau_2}_+ > -2}}
     c(\tau_1) c(\tau_2) \nabla(K \conv \reconst^{\can,\epsilon} \tau_1^{\rkernel, \can,\epsilon}) \cdot
     \nabla (K \conv \reconst^{\can,\epsilon} \tau_2^{\rkernel, \can,\epsilon}) \\
    +2   \nabla [K \conv X^{\can,\epsilon}] \cdot
     \nabla [H_N \conv ( X^{\can,\epsilon})]
     +
     \abs{\nabla [ H_N \conv (X^{\can,\epsilon})]}^2 + [\Delta(G_N - G)] \conv X^{\can,\epsilon}. 
  \end{multline*}
\end{calc}    
  \begin{align*}
    \MoveEqLeft[3]
    Y_N^{{\BPHZ, \epsilon}} 
    -  F(W_N^{{\BPHZ, \varepsilon}}) 
  (\xi_{\varepsilon}  + \abs{\nabla W_{N}^{{\BPHZ, \epsilon}}}^2 + \Delta W_{N}^{{\BPHZ, \epsilon}}) \\
  &= F(W_N^{{\BPHZ, \varepsilon}}) \times \big( [\Delta(G_N - G)] \conv (X^{{\BPHZ, \varepsilon}} - X^{{\can, \varepsilon}}) \big) \\
  &= F(W_N^{{\BPHZ, \varepsilon}}) \times \big( [\Delta (G_N - G)] \conv c_{\varepsilon} \big) = - F(W_N^{{\BPHZ, \varepsilon}}) c_{\varepsilon}. \qedhere
  \end{align*}
\end{proof}
\subsection{Stochastic estimates and Besov regularity}\label{subsec:stochastic_estimates}
Proposition \ref{prop:modelled_distributions_for_AH} gives pathwise estimates for the modelled distributions
$\regX$, $\regW_{N}$ and $\regY_{N}$. Here we give stochastic estimates for $X$ and $Y_{N}$ in
suitable Besov spaces.
For this sake, we will use the notation on wavelets from Appendix~\ref{subsec:estimates_in_Besov}.
We fix $k \in \N$ such that $k > \frac{5d}{2} + 2$, and
we consider the orthonormal basis $\{\Psi^{n, G}_m\}$ given by \eqref{eq:wavelet_basis}.
We set $\Psi \defby \Psi^{0, (\ff, \ldots, \ff)}_0$.
\begin{definition}
  Let $\rmodel = (\Pi, \Gamma), \overline{\rmodel} = (\overline{\Pi}, \overline{\Gamma}) \in \rmodelsp(\rspace, K)$.
  Given a compact set $\cptset \subseteq \R^d$, we set
  \begin{align*}
    \lmodelnorm{\rmodel}_{\cptset}
    &\defby \sup_{\tau = (T, 0)^{\fn, 0}_{\fe} \in \rbasis(\rspace) \cap \rspace_{< 0}}
    \sup_{n \in \N} \sup_{x \in \cptset \cap 2^{-n} \Z^d} 2^{n \abs{\tau}_+ }
    \abs{\inp{\Pi_x \tau}{2^{nd} \Psi(2^n(\cdot - x))}_{\R^d}}, \\
    \lmodelnorm{\rmodel; \overline{\rmodel}}_{\cptset}
    &\defby \sup_{\tau = (T, 0)^{\fn, 0}_{\fe} \in \rbasis(\rspace) \cap \rspace_{< 0}}
    \sup_{n \in \N} \sup_{x \in \cptset \cap 2^{-n} \Z^d} 2^{n \abs{\tau}_+  }
    \abs{\inp{\Pi_x \tau - \overline{\Pi}_x \tau}{2^{nd} \Psi(2^n(\cdot - x))}_{\R^d}}.
  \end{align*}
\end{definition}
\begin{lemma}\label{lem:model_norm_bounded_by_lattice_model_norm}
  For each $\gamma \in \R$, there exist a constant $C \in (0, \infty)$ and an integer $k \in \N$ such that the
  following estimates hold uniformly over $\rmodel, \overline{\rmodel} \in \rmodelsp(\rspace, K)$ and compact
  sets $\cptset \subseteq \R^d$:
  \begin{equation*}
    \vertiii{\rmodel}_{\gamma; \cptset} \leq C (1 + \lmodelnorm{\rmodel}_{\cptset})^k,
    \quad
    \vertiii{\rmodel; \overline{\rmodel}}_{\gamma; \cptset}
    \leq C (1 + \lmodelnorm{\rmodel}_{\cptset})^k (\lmodelnorm{\rmodel; \overline{\rmodel}}_{\cptset}
    + \lmodelnorm{\rmodel; \overline{\rmodel}}_{\cptset}^k).
  \end{equation*}
\end{lemma}
\begin{proof}
  Using the recursive formula \cite[Proposition~4.17]{bruned_2019}, one can prove the claim as in
  \cite[Lemma 2.3]{labbe_continuous_2019}.
\end{proof}
\begin{lemma}\label{lem:L_p_norm_of_lattice_model_norm}
  Let $L \in [1, \infty)$ and set $Q_L \defby [-L, L]^d$. Let $p \in 2\N$.
  Under Assumption \ref{assump:convergence_of_BPHZ},
  if $p \delta' > d + 1$, one has
  \begin{equation*}
    \expect[ \lmodelnorm{\rmodel^{\BPHZ}}_{Q_L}^{p}]
     \le   C_p^{\BPHZ} L^d,
    \quad \expect[\lmodelnorm{\rmodel^{\BPHZ}; \rmodel^{\BPHZ, \epsilon}}_{Q_L}^{p}]
     \le  \bm{\epsilon}^{\BPHZ}_p(\epsilon) L^d.
  \end{equation*}
\end{lemma}
\begin{proof}
  The proof is essentially the repetition of \cite[Lemma 4.11]{labbe_continuous_2019}.
  Set
  \begin{equation*}
    \rbasis_0(\rspace) \defby \set{\tau = (T, 0)^{\fn, 0}_{\fe} \in \rbasis(\rspace) \given \abs{\tau}_+ < 0}.
  \end{equation*}
  If we write $\Psi^{\lambda}_x \defby \lambda^{-d} \Psi(\lambda^{-1}(\cdot -x))$, one has
  \begin{align*}
    \expect[\lmodelnorm{\rmodel^{\BPHZ}}_{Q_L}^{p}]
    &= \expect[\sup_{\tau \in \rbasis_0(\rspace)} \sup_{n \in \N} \sup_{x \in Q_L \cap 2^{-n} \Z^d}
    2^{ n \abs{\tau}_+ p} \abs{\inp{\Pi_x \tau}{\Psi_x^{2^{-n}}}_{\R^d}}^{p}] \\
    &\lesssim \sum_{\tau \in \rbasis_0(\rspace)} \sum_{n \in \N} 2^{nd} L^d 2^{ n \abs{\tau}_+ p}
    \expect[\abs{\inp{\Pi_0 \tau}{\Psi_0^{2^{-n}}}_{\R^d}}^{p}],
  \end{align*}
  where the stationarity of the noise $\xi$ and the estimate
  $\# (Q_L \cap 2^{-n} \Z^d) \lesssim 2^{nd} L^d$ are used.
  By Assumption \ref{assump:convergence_of_BPHZ},
  \begin{equation*}
    \expect[\abs{\inp{\Pi_0 \tau}{\Psi_0^{2^{-n}}}_{\R^d}}^{p}]
    \lesssim_{\scalefcn} C_p^{\BPHZ} 2^{-np(\abs{\tau}_+ + \delta')}.
  \end{equation*}
  Therefore,
  \begin{equation*}
    \expect[\lmodelnorm{\rmodel^{\BPHZ}}_{Q_L}^{p}]
    \lesssim C^{\BPHZ}_p L^d \abs{\rbasis_0(\rspace)} (2^{p \delta' - d} - 1)^{-1}.
  \end{equation*}
  The estimate for the second claimed inequality is similar.
\end{proof}
\begin{lemma}\label{lem:estimate_of_H_N_conv_X}
  Let $\cptset \subseteq \R^d$ be a compact set and $\sigma \in (0, \infty)$. Then, there exists a constant $C \in (0, \infty)$
  such that  for all $N\in\N$
  \begin{equation*}
    \norm{H_N \conv X}_{C^2(\cptset)} \leq C 2^{3 N} \norm{X}_{\csp^{-2, \sigma}(\R^d)}.
  \end{equation*}
\end{lemma}
\begin{proof}
Let $\phi \in C_c^{\infty}(\R^d)$ be such that $\phi \equiv 1$ on $\cptset$.
By Lemma \ref{lem:localization}, one has
\begin{equation*}
  \norm{H_N \conv X}_{C^2(\cptset)} \lesssim \norm{\phi (H_N \conv X)}_{\csp^2(\R^d)}
  \lesssim_{\sigma} \norm{H_N \conv X}_{\csp^{2, \sigma}(\R^d)}.
\end{equation*}
It remains to apply Corollary \ref{cor:estimates_of_G_N_and_H_N}.
\end{proof}
Recall from Definition~\ref{def:degree} that we have, for instance, $\abs{\Xi}_+ = -2 + \delta +\kappa$
for some $\kappa \in (0, \delta')$.
\begin{proposition}\label{prop:stochastic_estimates_of_X_and_Y}
   Under Assumption~\ref{assump:convergence_of_BPHZ},
   there exist a deterministic integer
   $k = k(\delta_-) \in \N$ such that for all $\sigma \in (0, \infty)$, $p \in 2\N$ with
   $p > (d+1)/ \min\{\delta' - \kappa, \sigma\}$ and $N \in \N$ we have the following:
  \begin{align*}
    \expect[\norm{X^{\rmodel^{\BPHZ}}}_{B_{p,p}^{-2 + \delta + \kappa/2, \sigma}(\R^d)}^p]
&    \lesssim_{\delta, \delta', \kappa, \sigma, p}  C_{kp}^{\BPHZ},  \\
    \expect[\norm{Y^{\rmodel^{\BPHZ}}_N}_{B_{p,p}^{-1 + \delta + \kappa/2, \sigma}(\R^d)}^p]
&    \lesssim_{\delta, \delta', \kappa, \sigma, p}  C_{kp}^{\BPHZ} 2^{kp N}
  \end{align*}
  and
  \begin{align*}
    \expect[\norm{X^{\rmodel^{\BPHZ}} - X^{\rmodel^{\BPHZ, \epsilon}}}_{B_{p,p}^{-2 + \delta + \kappa/2, \sigma}(\R^d)}^p]
&    \lesssim_{\delta, \delta', \kappa, \sigma, p}
    C_{kp}^{\BPHZ} [\bm{\epsilon}^{\BPHZ}_{kp}(\epsilon)+\bm{\epsilon}^{\BPHZ}_{p}(\epsilon)], \\
    \expect[\norm{Y^{\rmodel^{\BPHZ}}_N - Y^{\rmodel^{\BPHZ, \epsilon}}_N}_{B_{p,p}^{-1 + \delta + \kappa/2, \sigma}(\R^d)}^p]
&    \lesssim_{\delta, \delta', \kappa/2, \sigma, p} C_{kp}^{\BPHZ} 2^{k p N}
    [\bm{\epsilon}^{\BPHZ}_{kp}(\epsilon)+\bm{\epsilon}^{\BPHZ}_{p}(\epsilon)].
  \end{align*}
\end{proposition}
\begin{proof}
  Set $\rmodel \defby \rmodel^{\BPHZ}$. In the proof, we drop superscripts for $\BPHZ$.
  The natural numbers $k, l, \gamma$ depend only on $\rspace$ and they vary from line to line.
  We will not write down the dependence on $\rspace, \delta, \delta_-, p, \sigma$.
  Recall the notation $\Psi^{n, G}_m$ from \eqref{eq:wavelet_basis}.

  Suppose we are given a modelled distribution $f \in \mdspace^{\gamma}_{\alpha}(\rspace, \rmodel)$ with
  $\alpha < 0 < \gamma$.
  We decompose
  \begin{multline*}
    \inp{\reconst f}{2^{nd/2} \Psi^{n, G}_m}_{\R^d} \\=
    \inp{\reconst f - \Pi_{2^{-n}m} f(2^{-n}m)}{2^{nd/2} \Psi^{n, G}_m}_{\R^d}
    + \inp{\Pi_{2^{-n}m} f(2^{-n} m)}{2^{nd/2} \Psi^{n, G}_m}_{\R^d}.
  \end{multline*}
  Using \eqref{eq:modelled_dist_approximation}, the first term is bounded by a constant times
  \begin{equation*}
    2^{-n \gamma} \vertiii{f}_{\gamma; B(2^{-n}m, l)} \vertiii{\rmodel}_{\gamma; B(2^{-n}m, l)}.
  \end{equation*}
  To estimate the second term, consider the basis expansion
  \begin{equation*}
    f(x) = \sum_{\sigma} a_{\sigma}(x) \sigma.
  \end{equation*}
  One has $\abs{ a_\sigma (2^{-n} m)} \leq \vertiii{f}_{\gamma; B(2^{-n} m, l)}$ and
  \begin{equation*}
    \abs{\inp{\Pi_{2^{-n}m} \sigma}{2^{nd/2}\Psi^{n, G}_m}_{\R^d}}
    \lesssim 2^{-n \alpha} \vertiii{\rmodel}_{\gamma; B(2^{-n}m, l)}.
  \end{equation*}
  Therefore,
  \begin{equation}\label{eq:modelled_dist_tested_agianst_wavelet}
    \abs{\inp{\reconst f}{2^{nd/2} \Psi^{n, G}_m}_{\R^d}} \lesssim 2^{-n \alpha}
    \vertiii{f}_{\gamma; B(2^{-n} m, l)} \vertiii{\rmodel}_{\gamma; B(2^{-n}m, l)}.
  \end{equation}

  Applying the estimate \eqref{eq:modelled_dist_tested_agianst_wavelet} to $\regX$ and $\regY_N$,
  by Proposition~\ref{prop:besov_wavelet}, we get
  \begin{multline*}
    \norm{X}_{B_{p, p}^{-2 + \delta+\kappa/2, \sigma}(\R^d)}^p \\
    \lesssim \sum_{n \in \N_0} 2^{-n(d+\kappa/2)} \sum_{G \in G^n, m \in \Z^d}
    w_{\sigma}(2^{-n}m)^p
    \vertiii{\regX}_{2; B(2^{-n} m, l)}^p \vertiii{\rmodel}_{2; B(2^{-n}m, l)}^p,
  \end{multline*}
  \begin{multline*}
    \norm{Y_N}_{B_{p, p}^{-1+\delta+\kappa/2, \sigma}(\R^d)}^p \\
    \lesssim \sum_{n \in \N_0} 2^{-n(d+\kappa/2)} \sum_{G \in G^n, m \in \Z^d}
    w_{\sigma}(2^{-n}m)^p
    \vertiii{\regY_N}_{\delta; B(2^{-n} m, l)}^p \vertiii{\rmodel}_{\delta; B(2^{-n}m, l)}^p.
  \end{multline*}
  To estimate $\norm{X}_{B_{p, p}^{-2+\delta+\kappa/2, \sigma}(\R^d)}$, we use Lemma \ref{prop:modelled_distributions_for_AH}
  and stationarity to obtain
  \begin{equation*}
    \expect[\norm{X}_{B_{p, p}^{-2+\delta+\kappa/2, \sigma}(\R^d)}^p]
    \lesssim
    \sum_{n \in \N_0} 2^{-n(d + (\delta - \delta_-))} \sum_{G \in G^n, m \in \Z^d}
    w_{\sigma}(2^{-n}m)^p
    \expect[(1 + \vertiii{\rmodel}_{\gamma;B(0, l)})^{kp}].
  \end{equation*}
  Since
  \begin{equation*}
    \sum_{m \in \Z^d} w_{\sigma}(2^{-n}m)^p
    \lesssim \int_{\R^d} (1 + \abs{2^{-n} x}^2)^{-\frac{p \sigma}{2}} dx = 2^{nd} \norm{w_{\sigma}}_{L^p(\R^d)}^p,
  \end{equation*}
  and by Lemma \ref{lem:model_norm_bounded_by_lattice_model_norm} and by Lemma \ref{lem:L_p_norm_of_lattice_model_norm}
  \begin{equation*}
    \expect[(1 + \vertiii{\rmodel}_{\gamma;B(0, l)})^{kp}]
    \lesssim C_{k' p}^{\BPHZ}
  \end{equation*}
  for some $k' \in \N$, we conclude
  \begin{equation*}
    \expect[\norm{X}_{B_{p, p}^{-2+\delta+\kappa/2, \sigma}(\R^d)}^p]
    \lesssim C_{kp}^{\BPHZ}.
  \end{equation*}
  The estimate of $Y_N$ is similar by using Lemma \ref{lem:estimate_of_H_N_conv_X}.
  The estimates of the differences can be proved similarly by using
  \cite[(3.4)]{hairer_theory_2014}.
\end{proof}


\change{
\begin{proof}[Proof of Theorem~\ref{thm:convergence_X_Y_BPHZ}]
The claim on the convergence follows from Proposition \ref{prop:stochastic_estimates_of_X_and_Y} and by applying Besov embeddings: 
Observe that $W_{N}^{\rmodel^{\BPHZ, \epsilon}} = G_N * X^{\rmodel^{\BPHZ, \epsilon}} = G_N * X^\epsilon = W_N^\epsilon$  (see Definition~\ref{def:def_of_X_W_N_Y_N_for_model}) and 
$Y_N^\epsilon = Y_{N}^{\rmodel^{\BPHZ, \epsilon}}$
by Proposition~\ref{prop:identitiy_for_W_and_Y}.

  To show \eqref{eqn:def_a_AH}, let $p \in 2\N$ be such that $d/p < \kappa/2$.
  By Proposition \ref{prop:stochastic_estimates_of_X_and_Y} and the Besov embedding,
  for some $k' \in \N$,
  \begin{equation*}
    \expect[\norm{Y_N^{\rmodel^{\BPHZ}}}_{\csp^{-1 + \delta, \sigma}(\R^d)}^p]
    \lesssim_{p, \delta, \sigma}
    \expect[\norm{Y_N^{\rmodel^{\BPHZ}}}_{B^{-1 + \delta+d/p, \sigma}_{p, p}(\R^d)}^p]
    \lesssim_{p, \delta, \kappa, \sigma} 2^{p k' N}.
  \end{equation*}
  Therefore, if $k > k'$,
  \begin{equation*}
    \sum_{N \in \N} 2^{-kp N} \expect[\norm{Y_N^{\rmodel^{\BPHZ}}}_{\csp^{-1 + \delta, \sigma}(\R^d)}^p] < \infty.
    \qedhere
  \end{equation*}
\end{proof}
}

%
%
%

\printbibliography[heading=bibintoc]
\end{document}